\pdfoutput=1 
\documentclass{article}
\usepackage[utf8]{inputenc}

\usepackage{amsfonts}
\usepackage{amssymb}
\usepackage{amsmath}
\usepackage{amsthm}
\usepackage{mathtools}
\usepackage{cleveref}
\usepackage{url}
\usepackage{tikz}
\usepackage{tikz-cd}
\usepackage{adjustbox}
\usepackage{enumitem}
\usepackage{subcaption}
\usepackage{graphicx}
\usepackage{afterpage}
\usepackage{changepage}
\usepackage{authblk}
\usepackage{longtable}
\usepackage{bm}
\usepackage{multicol}
\usepackage{wrapfig}

\usetikzlibrary{knots}

\usepackage[pass]{geometry}
\usepackage[
    backend=biber,
    style=alphabetic,
    sorting=nyt,
    doi=false,
    url=false,
    isbn=false,
    maxbibnames=10,
]{biblatex}
\DeclareFieldFormat[misc]{title}{``#1"}

\theoremstyle{plain}
    \newtheorem{maintheorem}{Theorem}
    \newtheorem{theorem}{Theorem}[section]
    \newtheorem{proposition}[theorem]{Proposition}
    \newtheorem{corollary}[theorem]{Corollary}
    \newtheorem{lemma}[theorem]{Lemma}

\theoremstyle{definition}
    \newtheorem{definition}[theorem]{Definition}

    \newtheorem{example}[theorem]{Example}
    \newtheorem{question}[theorem]{Question}

\theoremstyle{remark}
	\newtheorem{remark}[theorem]{Remark}%

\crefname{maintheorem}{Theorem}{Theorems}
\crefname{claim}{Claim}{Claims}

\newtheoremstyle{restated}
    {\topsep}{\topsep} %%% space between body and thm
    {\itshape}         %%% Thm body font
    {}                 %%% Indent amount (empty = no indent)
    {\bfseries}        %%% Thm head font
    {.}                %%% Punctuation after thm head
    { }                %%% Space after thm head
    {\thmname{#1} \ref{#3} {\normalfont(Restated)}}
\theoremstyle{restated}
    \newtheorem{restate-theorem}{Theorem}
    \newtheorem{restate-proposition}{Proposition}
    \newtheorem{restate-corollary}{Corollary}
    
\numberwithin{equation}{section}
\numberwithin{table}{section}

\newcommand{\ZZ}{\mathbb{Z}}
\newcommand{\QQ}{\mathbb{Q}}
\newcommand{\RR}{\mathbb{R}}

\newcommand{\FF}{\mathbb{F}}

\newcommand{\id}{\mathit{id}}
\newcommand{\isom}{\cong}
\newcommand{\htpy}{\simeq}

\newcommand{\del}{\partial}

\renewcommand{\emptyset}{\varnothing}

\renewcommand{\epsilon}{\varepsilon}

\DeclareMathOperator{\Hom}{Hom}

\DeclareMathOperator{\Tor}{Tor}

\DeclareMathOperator{\Mod}{Mod}

\DeclareMathOperator{\fchar}{char}

\DeclareMathOperator{\qdeg}{qdeg}

\DeclareMathOperator{\Cob}{Cob}
\DeclareMathOperator{\Mat}{Mat}
\DeclareMathOperator{\Kom}{Kom}
\DeclareMathOperator{\Kob}{Kob}
\DeclareMathOperator{\Diag}{Diag}

\newcommand{\CKh}{\mathit{CKh}}

\newcommand{\Kh}{\mathit{Kh}}

\newcommand{\ca}{\alpha}
\newcommand{\cb}{\beta}

\newcommand{\lin}{\text{in}}
\newcommand{\lout}{\text{out}}

\newcommand{\sfA}{\mathsf{A}}
\newcommand{\sfB}{\mathsf{B}}
\newcommand{\sfC}{\mathsf{C}}
\newcommand{\sfD}{\mathsf{D}}
\newcommand{\sfE}{\mathsf{E}}
\newcommand{\sfX}{\mathsf{X}}
\newcommand{\sfY}{\mathsf{Y}}

\newcommand{\thickcirc}{\tikz{\draw[line width=1.2pt] (0,.2em) circle(.4em);}}
\newcommand{\dotcircI}{\tikz{\draw[line width=.5pt, dotted]  (0,.2em) circle(0.4em) node [font=\tiny] {$+$};}}
\newcommand{\dotcircX}{\tikz{\draw[line width=.5pt, dotted]  (0,.2em) circle(0.4em) node [font=\tiny] {$-$};}}

\newcommand{\larc}{\tikz{\draw[line width=1.2pt] (0,-.3em) arc(300:420:0.6em);}}
\newcommand{\rarc}{\tikz{\draw[line width=1.2pt] (0,-.3em) arc(240:120:0.6em);}}
\newcommand{\udarc}{\tikz{\draw[line width=1.2pt] (0,0) arc(150:30:0.6em); \draw[line width=1.2pt] (0,1em) arc(210:330:0.6em);}}

\title{
    A diagrammatic approach to the Rasmussen invariant via tangles and cobordisms
}

\author{KeeTaek Kim, Taketo Sano}

\newcommand{\addresses}{{
  \bigskip
  \noindent
  KeeTaek Kim, \textsc{Korea Advanced Institute of Science and Technology (KAIST), Daejeon, South Korea.}\par\nopagebreak
  \noindent
  \textit{E-mail address}: \url{0xorxor0@kaist.ac.kr}
  
  \bigskip
  
  \noindent
  Taketo Sano, \textsc{RIKEN iTHEMS, Wako, Saitama 351-0198, Japan .}\par\nopagebreak
  \noindent
  \textit{E-mail address}: \url{taketo.sano@riken.jp}
}}

\begin{document}
    \maketitle
    \begin{abstract}
    We introduce a diagrammatic approach to Rasmussen's $s$-invariant, based on Bar-Natan's reformulation of Khovanov homology for tangles and cobordisms. This method enables a local computation of $s$ from a tangle decomposition of a knot diagram. As an application, we compute the $s$-invariants of all 3-strand pretzel knots.
\end{abstract}
    
    \setcounter{tocdepth}{2}
    % \tableofcontents
    
    \section{Introduction}
\label{sec:intro}

Rasmussen's \textit{$s$-invariant} is an integer-valued knot concordance invariant derived from a deformed version of Khovanov homology \cite{Rasmussen:2010}. It provided a combinatorial reproof of the \textit{Milnor conjecture} \cite{Milnor:1968}, which was originally proved by Kronheimer and Mrowka using gauge theory \cite{KM:1993}. A notable advantage of $s$ is its algorithmic computability, as demonstrated in Piccirillo's resolution of a long-standing open problem: the Conway knot is not slice \cite{Piccirillo:2020}.

Although $s$ is directly computable and efficient algorithms have been developed \cite{Schuetz:2021}, many interesting examples remain beyond the reach of current implementations. These include, for instance, the 222-crossing link introduced in \cite{FGMW:2010}, which was considered a potential counterexample to the four-dimensional smooth Poincar{\'e} conjecture.
% ; and Whitehead doubles of $(p, p + 1)$-torus knots, which may show the linear independence of $s$-invariants over the prime field of characteristic $p$ \cite[Questions 6.1, 6.2]{LS:2014}, \cite[Conjecture 1.3]{Schuetz:2022}. 
Beyond direct computation, the $s$-invariant is known only for certain families of knots, including (i) torus knots \cite{Rasmussen:2010}, (ii) (quasi-) positive knots \cite{Rasmussen:2010,Livingston:2004}, (iii) (quasi-) alternating knots \cite{Lee:2005,Manolescu-Ozvath:2008}, and (iv) homogeneous knots \cite{Abe:2011}.

The standard definition of the $s$-invariant requires the structure of the entire homology group, or at least its component in homological grading $0$. This global dependence makes structural analysis difficult and limits computational efficiency. It is therefore desirable to develop a method that enables \textit{local computation} of $s$ via tangle decompositions of knot diagrams, which is suitable for both pen-and-paper calculations and algorithmic implementations. To this end, we introduce a \textit{diagrammatic approach} to the $s$-invariant, based on Bar-Natan’s reformulation of Khovanov homology for tangles and cobordisms \cite{BarNatan:2004}. We begin with a brief overview of the key ideas underlying our approach.

\begin{figure}[t]
    \centering
    \begin{subfigure}[t]{0.3\linewidth}
        \centering
        \tikzset{every picture/.style={line width=0.75pt}} %set default line width to 0.75pt        

\begin{tikzpicture}[x=0.75pt,y=0.75pt,yscale=-.8,xscale=.8]
%uncomment if require: \path (0,110); %set diagram left start at 0, and has height of 110

\clip (0,0) rectangle + (100, 100);

%Curve Lines [id:da4534858391601907] 
\draw [line width=1.5]    (70.58,41.39) .. controls (63.5,-18) and (4,25.5) .. (42.36,77.01) ;
%Curve Lines [id:da9744085105392458] 
\draw [line width=1.5]    (53.15,86.58) .. controls (94,115) and (123,32) .. (35.5,45.5) ;
%Curve Lines [id:da5896580142187507] 
\draw [line width=1.5]    (25.13,46.95) .. controls (-12.3,54.35) and (15.68,141.29) .. (69.06,54.44) ;
%Straight Lines [id:da9402680077717771] 
\draw    (54.5,13) -- (47.5,13) ;
\draw [shift={(45.5,13)}, rotate = 360] [color={rgb, 255:red, 0; green, 0; blue, 0 }  ][line width=0.75]    (10.93,-4.9) .. controls (6.95,-2.3) and (3.31,-0.67) .. (0,0) .. controls (3.31,0.67) and (6.95,2.3) .. (10.93,4.9)   ;

\end{tikzpicture}
        \caption{}
        \label{fig:intro-a}
    \end{subfigure}
    \begin{subfigure}[t]{0.3\linewidth}
        \centering
        \tikzset{every picture/.style={line width=0.75pt}} %set default line width to 0.75pt        

\begin{tikzpicture}[x=0.75pt,y=0.75pt,yscale=-.8,xscale=.8]
%uncomment if require: \path (0,110); %set diagram left start at 0, and has height of 110

\clip (0,0) rectangle + (100, 100);

%Shape: Polygon Curved [id:ds35396286239104846] 
\draw  [color={rgb, 255:red, 208; green, 2; blue, 27 }  ,draw opacity=1 ][line width=1.5]  (49.5,13) .. controls (62.5,12) and (67.5,20) .. (69.5,37) .. controls (71.5,54) and (79.27,44.04) .. (89.5,55) .. controls (99.73,65.96) and (98.5,74) .. (96.5,83) .. controls (94.5,92) and (82.5,100) .. (65.5,91) .. controls (48.5,82) and (55.5,83) .. (40.5,92) .. controls (25.5,101) and (16.5,91) .. (10.5,84) .. controls (4.5,77) and (5.5,62) .. (12.5,57) .. controls (19.5,52) and (30.5,52) .. (32.5,37) .. controls (34.5,22) and (36.5,14) .. (49.5,13) -- cycle ;
%Shape: Polygon Curved [id:ds6741754652220737] 
\draw  [color={rgb, 255:red, 0; green, 116; blue, 255 }  ,draw opacity=1 ][line width=1.5]  (49.5,44) .. controls (61.5,44) and (75,41) .. (68.5,63) .. controls (62,85) and (45,87) .. (35,67) .. controls (25,47) and (37.5,44) .. (49.5,44) -- cycle ;

% Text Node
\draw (72,5.4) node [anchor=north west][inner sep=0.75pt]    {$\textcolor[rgb]{0.82,0.01,0.11}{X}$};
% Text Node
\draw (70.5,58.4) node [anchor=north west][inner sep=0.75pt]  [color={rgb, 255:red, 0; green, 117; blue, 255 }  ,opacity=1 ]  {$Y$};

\end{tikzpicture}
        \caption{}
        \label{fig:intro-b}
    \end{subfigure}
    \begin{subfigure}[t]{0.35\linewidth}
        \centering
        \tikzset{every picture/.style={line width=0.75pt}} %set default line width to 0.75pt        

\begin{tikzpicture}[x=0.75pt,y=0.75pt,yscale=-.8,xscale=.8]
%uncomment if require: \path (0,111); %set diagram left start at 0, and has height of 111

\clip (0,0) rectangle + (150, 100);

%Curve Lines [id:da354239760487419] 
\draw [color={rgb, 255:red, 208; green, 2; blue, 27 }  ,draw opacity=1 ][line width=1.5]    (98.22,13.41) .. controls (107.84,12.45) and (111.54,20.08) .. (113.02,36.28) .. controls (114.5,52.49) and (120.25,43) .. (127.82,53.44) .. controls (135.38,63.88) and (134.47,71.55) .. (132.99,80.13) .. controls (131.51,88.71) and (128.99,89.47) .. (124.14,90.98) ;
%Curve Lines [id:da7468146028110322] 
\draw [color={rgb, 255:red, 0; green, 116; blue, 255 }  ,draw opacity=1 ][line width=1.5]    (98.22,42.96) .. controls (107.1,42.96) and (117.09,40.1) .. (112.28,61.07) .. controls (107.47,82.03) and (104.49,78.54) .. (99.21,78.03) ;
%Curve Lines [id:da456514121160539] 
\draw [color={rgb, 255:red, 208; green, 2; blue, 27 }  ,draw opacity=1 ][line width=1.5]    (98.22,13.41) .. controls (71.82,13.41) and (9,1.67) .. (9,56) .. controls (9,110.33) and (96.96,90.49) .. (124.14,90.98) ;
%Shape: Ellipse [id:dp6743408953422624] 
\draw  [fill={rgb, 255:red, 0; green, 0; blue, 0 }  ,fill opacity=1 ] (26.3,58.45) .. controls (26.3,56.76) and (27.67,55.39) .. (29.35,55.39) .. controls (31.04,55.39) and (32.41,56.76) .. (32.41,58.45) .. controls (32.41,60.13) and (31.04,61.5) .. (29.35,61.5) .. controls (27.67,61.5) and (26.3,60.13) .. (26.3,58.45) -- cycle ;
%Curve Lines [id:da5604215334542274] 
\draw [color={rgb, 255:red, 208; green, 2; blue, 27 }  ,draw opacity=1 ][line width=1.5]  [dash pattern={on 1.69pt off 2.76pt}]  (124.14,90.98) .. controls (108.4,90.99) and (108.4,81.84) .. (102.3,82.23) .. controls (96.2,82.61) and (91.24,87.94) .. (91.57,88.71) .. controls (91.89,89.47) and (75.23,88.71) .. (69.89,81.84) .. controls (64.55,74.98) and (65.67,60.11) .. (70.85,55.35) .. controls (76.03,50.58) and (84.17,50.58) .. (85.65,36.28) .. controls (87.13,21.99) and (88.28,16.13) .. (95.31,13.97) ;
%Curve Lines [id:da025390541887372886] 
\draw [color={rgb, 255:red, 0; green, 116; blue, 255 }  ,draw opacity=1 ][line width=1.5]  [dash pattern={on 1.69pt off 2.76pt}]  (100.46,78.03) .. controls (95.91,78.45) and (91.02,70.52) .. (87.43,61.26) .. controls (83.83,51.99) and (87.65,43.7) .. (95.54,43.06) ;
%Curve Lines [id:da030811963727374936] 
\draw [color={rgb, 255:red, 74; green, 144; blue, 226 }  ,draw opacity=1 ][fill={rgb, 255:red, 255; green, 255; blue, 255 }  ,fill opacity=0.8 ][line width=1.5]    (98.22,42.96) .. controls (71.82,42.96) and (49.78,48.32) .. (49.3,59.99) .. controls (48.82,71.67) and (65.84,72.77) .. (100.46,78.03) ;
%Shape: Ellipse [id:dp0036598793116868755] 
\draw  [color={rgb, 255:red, 0; green, 0; blue, 0 }  ,draw opacity=1 ][fill={rgb, 255:red, 255; green, 255; blue, 255 }  ,fill opacity=1 ] (57.92,60.71) .. controls (57.92,59.02) and (59.28,57.66) .. (60.97,57.66) .. controls (62.66,57.66) and (64.02,59.02) .. (64.02,60.71) .. controls (64.02,62.4) and (62.66,63.76) .. (60.97,63.76) .. controls (59.28,63.76) and (57.92,62.4) .. (57.92,60.71) -- cycle ;

% Text Node
\draw (117.65,6.76) node [anchor=north west][inner sep=0.75pt]    {$\textcolor[rgb]{0.82,0.01,0.11}{X}$};
% Text Node
\draw (113.91,63.95) node [anchor=north west][inner sep=0.75pt]  [color={rgb, 255:red, 0; green, 117; blue, 255 }  ,opacity=1 ]  {$Y$};

\end{tikzpicture}
        \caption{}
        \label{fig:intro-c}
    \end{subfigure}
    \caption{The Lee cycle $\ca(K)$ regarded as a dotted cobordism.}
    \label{fig:intro-lee-cycle-cob}    
\end{figure}

Given a knot diagram $K$, its \textit{Lee cycle} $\ca(K)$ is a cycle in a deformed version of the Khovanov complex, defined by a specific coloring of the Seifert circles of $K$ by $X$ or $Y$ (see \Cref{fig:intro-a,fig:intro-b}). In Bar-Natan's framework, the Lee cycle is reinterpreted as a chain map
\[
    \ca(K): [\varnothing] \to [K],
\]
represented by a \textit{dotted cobordism}, as illustrated in \Cref{fig:intro-c}. Here, $[K]$ denotes the \textit{formal Khovanov complex} of $K$, a chain complex in the category of $(1+1)$-dimensional dotted cobordisms.
Now, let $H$ denote the sphere with two dots:
\begin{center}
    \tikzset{every picture/.style={line width=0.75pt}} %set default line width to 0.75pt        

\begin{tikzpicture}[x=0.75pt,y=0.75pt,yscale=-1,xscale=1]
%uncomment if require: \path (0,47); %set diagram left start at 0, and has height of 47

%Shape: Arc [id:dp6319807885803715] 
\draw  [draw opacity=0] (96.94,23.75) .. controls (96.29,26) and (89.45,27.76) .. (81.11,27.76) .. controls (72.35,27.76) and (65.24,25.81) .. (65.24,23.41) .. controls (65.24,23.26) and (65.27,23.11) .. (65.32,22.97) -- (81.11,23.41) -- cycle ; \draw  [color={rgb, 255:red, 128; green, 128; blue, 128 }  ,draw opacity=1 ] (96.94,23.75) .. controls (96.29,26) and (89.45,27.76) .. (81.11,27.76) .. controls (72.35,27.76) and (65.24,25.81) .. (65.24,23.41) .. controls (65.24,23.26) and (65.27,23.11) .. (65.32,22.97) ;  
%Shape: Arc [id:dp056615061649111675] 
\draw  [draw opacity=0][dash pattern={on 0.84pt off 2.51pt}] (65.63,22.33) .. controls (66.9,20.27) and (73.43,18.7) .. (81.28,18.7) .. controls (89.93,18.7) and (96.96,20.6) .. (97.15,22.96) -- (81.28,23.06) -- cycle ; \draw  [color={rgb, 255:red, 128; green, 128; blue, 128 }  ,draw opacity=1 ][dash pattern={on 0.84pt off 2.51pt}] (65.63,22.33) .. controls (66.9,20.27) and (73.43,18.7) .. (81.28,18.7) .. controls (89.93,18.7) and (96.96,20.6) .. (97.15,22.96) ;  
%Shape: Circle [id:dp39111359424094494] 
\draw   (65.12,23.41) .. controls (65.12,14.57) and (72.28,7.41) .. (81.11,7.41) .. controls (89.95,7.41) and (97.11,14.57) .. (97.11,23.41) .. controls (97.11,32.24) and (89.95,39.4) .. (81.11,39.4) .. controls (72.28,39.4) and (65.12,32.24) .. (65.12,23.41) -- cycle ;
%Shape: Ellipse [id:dp9734470792692242] 
\draw  [fill={rgb, 255:red, 0; green, 0; blue, 0 }  ,fill opacity=1 ] (74.63,14.68) .. controls (74.63,13.55) and (75.55,12.63) .. (76.68,12.63) .. controls (77.82,12.63) and (78.73,13.55) .. (78.73,14.68) .. controls (78.73,15.81) and (77.82,16.73) .. (76.68,16.73) .. controls (75.55,16.73) and (74.63,15.81) .. (74.63,14.68) -- cycle ;
%Shape: Ellipse [id:dp1260497035726349] 
\draw  [fill={rgb, 255:red, 0; green, 0; blue, 0 }  ,fill opacity=1 ] (84.23,14.68) .. controls (84.23,13.55) and (85.15,12.63) .. (86.28,12.63) .. controls (87.42,12.63) and (88.33,13.55) .. (88.33,14.68) .. controls (88.33,15.81) and (87.42,16.73) .. (86.28,16.73) .. controls (85.15,16.73) and (84.23,15.81) .. (84.23,14.68) -- cycle ;

% Text Node
\draw (14,14.61) node [anchor=north west][inner sep=0.75pt]    {$H\ =\ $};

\end{tikzpicture}
\end{center}
We define the \textit{$H$-divisibility} $d_H(K)$ of $\ca(K)$ as the maximal $k \geq 0$ such that there exists a chain map $z$ for which $H^k z$ is homotopic (modulo torsion) to $\ca(K)$. It is then shown that the $s$-invariant can be characterized as follows:

\begin{maintheorem}
\label{thm:s-reformulate}
\[
    s(K) = 2d_H(K) + w(K) - r(K) + 1.
\]
\end{maintheorem}

Here, $w$ denotes the writhe and $r$ the number of Seifert circles. \Cref{thm:s-reformulate} states that $s$ is essentially determined by the number of doubly dotted spheres that can be factored out from the dotted cobordism $\ca(K)$ up to chain homotopy. This formula is identical to those obtained in \cite{Sano:2020, Sano-Sato:2023}, but is now formalized within the category of dotted cobordisms. This shift in perspective enables the extension of Lee cycles to tangle diagrams.

For a tangle diagram $T$, the Lee cycle is similarly defined as a chain map
\[
    \ca(T): [T^\lin] \to [T]
\]
where $T^\lin$ denotes the crossingless tangle diagram obtained by removing the circle components from the Seifert resolution of $T$ (see \Cref{fig:tangle-lee-cycle-c} in \Cref{sec:reform-s}). Note that we cannot define $\ca(T)$ as an `element' of the complex $[T]$, since $[T]$ is not an ordinary chain complex with an underlying $\ZZ$-module structure. This extension enables the local computation of $s$ from a tangle decomposition.

\begin{figure}[t]
    \centering
    \input{tikzpictures/intro-lee-decomp}
    \caption{Tangle decomposition of the Lee cycle $\ca(K)$.}
    \label{fig:intro-lee-decomp}
\end{figure}

\begin{maintheorem}
\label{thm:lee-cycle-decomp}
    Suppose that a knot diagram $K$ admits a tangle decomposition
    \[
        K = D(T_1, T_2, \ldots, T_d)
    \]
    for some $d$-input planar arc diagram $D$ and tangle diagrams $T_i$. Let $T_0$ denote the crossingless unlink diagram given by the composition
    \[
        T_0 = D(T_1^\lin, T_2^\lin, \ldots, T_d^\lin).
    \]
    Then we have
    \[
        \ca(K) = D(\ca(T_1), \ca(T_2), \ldots, \ca(T_d)) \circ \ca(T_0).
    \]
\end{maintheorem}

\Cref{fig:intro-lee-decomp} illustrates \Cref{thm:lee-cycle-decomp} in the case where $K$ is a 3-strand pretzel knot, represented as $d = 3$ twist tangles connected by external arcs. Together with \Cref{thm:s-reformulate,thm:lee-cycle-decomp}, the computation of $s$ reduces to determining the $H$-divisibility of each cycle 
\[
    \ca(T_i) \htpy H^{k_i} z_i,
\]
and that of the reduced composition cycle
\[
    D(z_1, z_2, \ldots, z_d) \circ \ca(T_0).
\]

The effectiveness of this approach is demonstrated through the computation of $s$-invariants for 3-strand pretzel knots. Recall that a 3-strand pretzel link $P(p, q, r)$ ($p, q, r \in \ZZ$) is a knot if and only if either all of $p, q, r$ are odd, or exactly one of them is even. We refer to the former as the \textit{odd} type and the latter as the \textit{even} type.  

\begin{maintheorem}
\label{thm:pretzel}
    (O) For any odd pretzel knot $P = P(p, q, r)$ with $p > 0$,
    \[
        s(P) = \begin{cases}
            -2 & \text{if\ } p + q > 0,\ p + r > 0\text{\ and\ } q + r > 0,\\
            2  & \text{if\ } p + q < 0 \text{\ and\ } p + r < 0, \\
            0  & \text{otherwise}. 
        \end{cases}
    \]
    (E) For any even pretzel knot $P = P(p, q, r)$ with even $p > 0$,
    \[
        s(P) = \begin{cases}
            q + r - 2 & \text{if\ } p + q > 0,\ p + r > 0\text{\ and\ } q + r > 0, \\
            q + r & \text{otherwise}.
        \end{cases}
    \]
\end{maintheorem}

\begin{figure}
    \centering
    \begin{subfigure}{.45\textwidth}
        \centering
        \tikzset{every picture/.style={line width=0.75pt}} %set default line width to 0.75pt        

\begin{tikzpicture}[x=0.5pt,y=0.5pt,yscale=-1,xscale=1]

%Straight Lines [id:da4085929601326327] 
\draw [color={rgb, 255:red, 245; green, 166; blue, 35 }  ,draw opacity=1 ][line width=1.5]  [dash pattern={on 1.69pt off 2.76pt}]  (240,196) -- (0,196) ;
%Straight Lines [id:da5991417279053142] 
\draw [color={rgb, 255:red, 245; green, 166; blue, 35 }  ,draw opacity=1 ][line width=1.5]  [dash pattern={on 1.69pt off 2.76pt}]  (76,270) -- (76,30) ;
%Straight Lines [id:da26272709805947836] 
\draw [color={rgb, 255:red, 245; green, 166; blue, 35 }  ,draw opacity=1 ][line width=1.5]  [dash pattern={on 1.69pt off 2.76pt}]  (240,270) -- (0,30) ;
%Straight Lines [id:da3038111845839939] 
\draw [color={rgb, 255:red, 245; green, 166; blue, 35 }  ,draw opacity=1 ][line width=1.5]    (76,108) -- (76,30) ;
%Straight Lines [id:da8833062771962256] 
\draw [color={rgb, 255:red, 245; green, 166; blue, 35 }  ,draw opacity=1 ][line width=1.5]    (240,196) -- (165.4,196) ;
%Straight Lines [id:da5206321045872273] 
\draw [color={rgb, 255:red, 245; green, 166; blue, 35 }  ,draw opacity=1 ][line width=1.5]    (166.4,196.4) -- (77,107) ;
%Straight Lines [id:da045636047293270754] 
\draw [color={rgb, 255:red, 245; green, 166; blue, 35 }  ,draw opacity=1 ][line width=1.5]    (76,196) -- (0,196) ;
%Straight Lines [id:da7251110044927165] 
\draw [color={rgb, 255:red, 245; green, 166; blue, 35 }  ,draw opacity=1 ][line width=1.5]    (76,270) -- (76,196) ;

%Straight Lines [id:da9877204124523966] 
\draw    (0,150) -- (256,150) ;
\draw [shift={(258,150)}, rotate = 180] [color={rgb, 255:red, 0; green, 0; blue, 0 }  ][line width=0.75]    (10.93,-3.29) .. controls (6.95,-1.4) and (3.31,-0.3) .. (0,0) .. controls (3.31,0.3) and (6.95,1.4) .. (10.93,3.29)   ;
%Straight Lines [id:da2640379809030683] 
\draw    (120,270) -- (120,14) ;
\draw [shift={(120,12)}, rotate = 90] [color={rgb, 255:red, 0; green, 0; blue, 0 }  ][line width=0.75]    (10.93,-3.29) .. controls (6.95,-1.4) and (3.31,-0.3) .. (0,0) .. controls (3.31,0.3) and (6.95,1.4) .. (10.93,3.29)   ;

% Text Node
\draw (163,105.5) node  [font=\scriptsize]  {$-2$};
% Text Node
\draw (192.5,106) node  [font=\scriptsize]  {$-2$};
% Text Node
\draw (223,105.5) node  [font=\scriptsize]  {$-2$};
% Text Node
\draw (163,75.5) node  [font=\scriptsize]  {$-2$};
% Text Node
\draw (192.5,76) node  [font=\scriptsize]  {$-2$};
% Text Node
\draw (223,75.5) node  [font=\scriptsize]  {$-2$};
% Text Node
\draw (163,45.5) node  [font=\scriptsize]  {$-2$};
% Text Node
\draw (192.5,46) node  [font=\scriptsize]  {$-2$};
% Text Node
\draw (223,45.5) node  [font=\scriptsize]  {$-2$};
% Text Node
\draw (257.5,141.6) node [anchor=south] [inner sep=0.75pt]    {$q$};
% Text Node
\draw (107.5,254.5) node  [font=\scriptsize]  {$0$};
% Text Node
\draw (107.5,165) node  [font=\scriptsize]  {$0$};
% Text Node
\draw (107.5,195.5) node  [font=\scriptsize]  {$0$};
% Text Node
\draw (107.5,224.5) node  [font=\scriptsize]  {$0$};
% Text Node
\draw (133,105.5) node  [font=\scriptsize]  {$-2$};
% Text Node
\draw (133,75.5) node  [font=\scriptsize]  {$-2$};
% Text Node
\draw (133,45.5) node  [font=\scriptsize]  {$-2$};
% Text Node
\draw (163,134) node  [font=\scriptsize]  {$-2$};
% Text Node
\draw (192.5,134.5) node  [font=\scriptsize]  {$-2$};
% Text Node
\draw (223,134) node  [font=\scriptsize]  {$-2$};
% Text Node
\draw (133,134) node  [font=\scriptsize]  {$-2$};
% Text Node
\draw (77.5,165) node  [font=\scriptsize]  {$0$};
% Text Node
\draw (47.5,165) node  [font=\scriptsize]  {$0$};
% Text Node
\draw (17.5,165) node  [font=\scriptsize]  {$0$};
% Text Node
\draw (47.5,254.5) node  [font=\scriptsize]  {$2$};
% Text Node
\draw (47.5,224.5) node  [font=\scriptsize]  {$2$};
% Text Node
\draw (17.5,254.5) node  [font=\scriptsize]  {$2$};
% Text Node
\draw (17.5,224.5) node  [font=\scriptsize]  {$2$};
% Text Node
\draw (132.5,22.6) node [anchor=south] [inner sep=0.75pt]    {$r$};
% Text Node
\draw (137.5,165.5) node  [font=\scriptsize]  {$0$};
% Text Node
\draw (107.5,135.5) node  [font=\scriptsize]  {$0$};
% Text Node
\draw (103,105.5) node  [font=\scriptsize]  {$-2$};
% Text Node
\draw (103,75.5) node  [font=\scriptsize]  {$-2$};
% Text Node
\draw (103,45.5) node  [font=\scriptsize]  {$-2$};
% Text Node
\draw (163,164) node  [font=\scriptsize]  {$-2$};
% Text Node
\draw (192.5,164.5) node  [font=\scriptsize]  {$-2$};
% Text Node
\draw (223,164) node  [font=\scriptsize]  {$-2$};
% Text Node
\draw (137.5,254.5) node  [font=\scriptsize]  {$0$};
% Text Node
\draw (137.5,195.5) node  [font=\scriptsize]  {$0$};
% Text Node
\draw (137.5,224.5) node  [font=\scriptsize]  {$0$};
% Text Node
\draw (167.5,254.5) node  [font=\scriptsize]  {$0$};
% Text Node
\draw (167.5,195.5) node  [font=\scriptsize]  {$0$};
% Text Node
\draw (167.5,224.5) node  [font=\scriptsize]  {$0$};
% Text Node
\draw (197.5,254.5) node  [font=\scriptsize]  {$0$};
% Text Node
\draw (197.5,224.5) node  [font=\scriptsize]  {$0$};
% Text Node
\draw (227.5,255.5) node  [font=\scriptsize]  {$0$};
% Text Node
\draw (78.5,135.5) node  [font=\scriptsize]  {$0$};
% Text Node
\draw (48.5,135.5) node  [font=\scriptsize]  {$0$};
% Text Node
\draw (18.5,135.5) node  [font=\scriptsize]  {$0$};
% Text Node
\draw (77.5,105.5) node  [font=\scriptsize]  {$0$};
% Text Node
\draw (48.5,105.5) node  [font=\scriptsize]  {$0$};
% Text Node
\draw (18.5,105.5) node  [font=\scriptsize]  {$0$};
% Text Node
\draw (48.5,75.5) node  [font=\scriptsize]  {$0$};
% Text Node
\draw (18.5,75.5) node  [font=\scriptsize]  {$0$};
% Text Node
\draw (17.5,45.5) node  [font=\scriptsize]  {$0$};
% Text Node
\draw (77.5,75.5) node  [font=\scriptsize]  {$0$};
% Text Node
\draw (76.5,45.5) node  [font=\scriptsize]  {$0$};
% Text Node
\draw (47.5,45.5) node  [font=\scriptsize]  {$0$};
% Text Node
\draw (197.5,195.5) node  [font=\scriptsize]  {$0$};
% Text Node
\draw (227.5,195.5) node  [font=\scriptsize]  {$0$};
% Text Node
\draw (228.5,224.5) node  [font=\scriptsize]  {$0$};
% Text Node
\draw (77.5,254.5) node  [font=\scriptsize]  {$0$};
% Text Node
\draw (77.5,195.5) node  [font=\scriptsize]  {$0$};
% Text Node
\draw (77.5,224.5) node  [font=\scriptsize]  {$0$};
% Text Node
\draw (47.5,194.5) node  [font=\scriptsize]  {$0$};
% Text Node
\draw (17.5,194.5) node  [font=\scriptsize]  {$0$};

\end{tikzpicture}
        \caption{Odd, $p = 3$.}
    \end{subfigure}
    \begin{subfigure}{.45\textwidth}
        \centering
        \tikzset{every picture/.style={line width=0.75pt}} %set default line width to 0.75pt        

\begin{tikzpicture}[x=0.5pt,y=0.5pt,yscale=-1,xscale=1]

%Straight Lines [id:da13934912292490942] 
\draw [color={rgb, 255:red, 245; green, 166; blue, 35 }  ,draw opacity=1 ][line width=1.5]  [dash pattern={on 1.69pt off 2.76pt}]  (240,180) -- (0,180) ;
%Straight Lines [id:da6083637605620927] 
\draw [color={rgb, 255:red, 245; green, 166; blue, 35 }  ,draw opacity=1 ][line width=1.5]  [dash pattern={on 1.69pt off 2.76pt}]  (90,270) -- (90,30) ;
%Straight Lines [id:da3749391967027246] 
\draw [color={rgb, 255:red, 245; green, 166; blue, 35 }  ,draw opacity=1 ][line width=1.5]  [dash pattern={on 1.69pt off 2.76pt}]  (240,270) -- (0,30) ;
%Straight Lines [id:da9769193407625179] 
\draw [color={rgb, 255:red, 245; green, 166; blue, 35 }  ,draw opacity=1 ][line width=1.5]    (90,120) -- (90,30) ;
%Straight Lines [id:da7986017093442541] 
\draw [color={rgb, 255:red, 245; green, 166; blue, 35 }  ,draw opacity=1 ][line width=1.5]    (240,180) -- (150,180) ;
%Straight Lines [id:da29009190285321296] 
\draw [color={rgb, 255:red, 245; green, 166; blue, 35 }  ,draw opacity=1 ][line width=1.5]    (150,180) -- (90,120) ;

% %Shape: Right Triangle [id:dp3132823507634923] 
% \draw  [draw opacity=0][fill={rgb, 255:red, 248; green, 231; blue, 28 }  ,fill opacity=0.2 ] (120,150) -- (0,30) -- (120,30) -- cycle ;
% %Shape: Right Triangle [id:dp40570116708619275] 
% \draw  [draw opacity=0][fill={rgb, 255:red, 248; green, 231; blue, 28 }  ,fill opacity=0.2 ] (240,270) -- (120,150) -- (240,150) -- cycle ;
%Straight Lines [id:da9892770383836467] 
\draw    (0,150) -- (256,150) ;
\draw [shift={(258,150)}, rotate = 180] [color={rgb, 255:red, 0; green, 0; blue, 0 }  ][line width=0.75]    (10.93,-3.29) .. controls (6.95,-1.4) and (3.31,-0.3) .. (0,0) .. controls (3.31,0.3) and (6.95,1.4) .. (10.93,3.29)   ;
%Straight Lines [id:da8215822754287251] 
\draw    (120,270) -- (120,14) ;
\draw [shift={(120,12)}, rotate = 90] [color={rgb, 255:red, 0; green, 0; blue, 0 }  ][line width=0.75]    (10.93,-3.29) .. controls (6.95,-1.4) and (3.31,-0.3) .. (0,0) .. controls (3.31,0.3) and (6.95,1.4) .. (10.93,3.29)   ;

% Text Node
\draw (167,105.5) node  [font=\scriptsize]  {$4$};
% Text Node
\draw (196.5,106) node  [font=\scriptsize]  {$6$};
% Text Node
\draw (227,105.5) node  [font=\scriptsize]  {$8$};
% Text Node
\draw (167,75.5) node  [font=\scriptsize]  {$6$};
% Text Node
\draw (196.5,76) node  [font=\scriptsize]  {$8$};
% Text Node
\draw (223,75.5) node  [font=\scriptsize]  {$10$};
% Text Node
\draw (167,45.5) node  [font=\scriptsize]  {$8$};
% Text Node
\draw (192.5,46) node  [font=\scriptsize]  {$10$};
% Text Node
\draw (223,45.5) node  [font=\scriptsize]  {$12$};
% Text Node
\draw (257.5,141.6) node [anchor=south] [inner sep=0.75pt]    {$q$};
% Text Node
\draw (104.5,254.5) node  [font=\scriptsize]  {$-8$};
% Text Node
\draw (104.5,165) node  [font=\scriptsize]  {$-2$};
% Text Node
\draw (104.5,195.5) node  [font=\scriptsize]  {$-4$};
% Text Node
\draw (104.5,224.5) node  [font=\scriptsize]  {$-6$};
% Text Node
\draw (137,105.5) node  [font=\scriptsize]  {$2$};
% Text Node
\draw (137,75.5) node  [font=\scriptsize]  {$4$};
% Text Node
\draw (137,45.5) node  [font=\scriptsize]  {$6$};
% Text Node
\draw (167,134) node  [font=\scriptsize]  {$2$};
% Text Node
\draw (196.5,134.5) node  [font=\scriptsize]  {$4$};
% Text Node
\draw (227,134) node  [font=\scriptsize]  {$6$};
% Text Node
\draw (137,134) node  [font=\scriptsize]  {$0$};
% Text Node
\draw (74.5,165) node  [font=\scriptsize]  {$-4$};
% Text Node
\draw (44.5,165) node  [font=\scriptsize]  {$-6$};
% Text Node
\draw (14.5,165) node  [font=\scriptsize]  {$-8$};
% Text Node
\draw (44.5,254.5) node  [font=\scriptsize]  {$-12$};
% Text Node
\draw (44.5,224.5) node  [font=\scriptsize]  {$-10$};
% Text Node
\draw (14.5,254.5) node  [font=\scriptsize]  {$-14$};
% Text Node
\draw (14.5,224.5) node  [font=\scriptsize]  {$-12$};
% Text Node
\draw (132.5,22.6) node [anchor=south] [inner sep=0.75pt]    {$r$};
% Text Node
\draw (137.5,165.5) node  [font=\scriptsize]  {$0$};
% Text Node
\draw (107.5,135.5) node  [font=\scriptsize]  {$0$};
% Text Node
\draw (107,105.5) node  [font=\scriptsize]  {$0$};
% Text Node
\draw (107,75.5) node  [font=\scriptsize]  {$2$};
% Text Node
\draw (107,45.5) node  [font=\scriptsize]  {$4$};
% Text Node
\draw (167,165) node  [font=\scriptsize]  {$0$};
% Text Node
\draw (196.5,165.5) node  [font=\scriptsize]  {$2$};
% Text Node
\draw (227,165) node  [font=\scriptsize]  {$4$};
% Text Node
\draw (134.5,254.5) node  [font=\scriptsize]  {$-6$};
% Text Node
\draw (134.5,195.5) node  [font=\scriptsize]  {$-2$};
% Text Node
\draw (134.5,224.5) node  [font=\scriptsize]  {$-4$};
% Text Node
\draw (164.5,254.5) node  [font=\scriptsize]  {$-4$};
% Text Node
\draw (167.5,195.5) node  [font=\scriptsize]  {$0$};
% Text Node
\draw (164.5,224.5) node  [font=\scriptsize]  {$-2$};
% Text Node
\draw (194.5,254.5) node  [font=\scriptsize]  {$-2$};
% Text Node
\draw (197.5,224.5) node  [font=\scriptsize]  {$0$};
% Text Node
\draw (227.5,255.5) node  [font=\scriptsize]  {$0$};
% Text Node
\draw (75.5,135.5) node  [font=\scriptsize]  {$-2$};
% Text Node
\draw (45.5,135.5) node  [font=\scriptsize]  {$-4$};
% Text Node
\draw (13.5,135.5) node  [font=\scriptsize]  {$-6$};
% Text Node
\draw (77.5,105.5) node  [font=\scriptsize]  {$0$};
% Text Node
\draw (45.5,105.5) node  [font=\scriptsize]  {$-2$};
% Text Node
\draw (15.5,105.5) node  [font=\scriptsize]  {$-4$};
% Text Node
\draw (48.5,75.5) node  [font=\scriptsize]  {$0$};
% Text Node
\draw (15.5,77.5) node  [font=\scriptsize]  {$-2$};
% Text Node
\draw (17.5,45.5) node  [font=\scriptsize]  {$0$};
% Text Node
\draw (77.5,75.5) node  [font=\scriptsize]  {$2$};
% Text Node
\draw (76.5,45.5) node  [font=\scriptsize]  {$4$};
% Text Node
\draw (47.5,45.5) node  [font=\scriptsize]  {$2$};
% Text Node
\draw (197.5,195.5) node  [font=\scriptsize]  {$2$};
% Text Node
\draw (227.5,195.5) node  [font=\scriptsize]  {$4$};
% Text Node
\draw (228.5,222.5) node  [font=\scriptsize]  {$2$};
% Text Node
\draw (74.5,254.5) node  [font=\scriptsize]  {$-10$};
% Text Node
\draw (74.5,195.5) node  [font=\scriptsize]  {$-6$};
% Text Node
\draw (74.5,224.5) node  [font=\scriptsize]  {$-8$};
% Text Node
\draw (44.5,194.5) node  [font=\scriptsize]  {$-8$};
% Text Node
\draw (14.5,194.5) node  [font=\scriptsize]  {$-10$};

\end{tikzpicture}
        \caption{Even, $p = 2$.}
        \label{fig:table-s-even}
    \end{subfigure}
    \caption{$s$-invariant for pretzel knots $P(p, q, r)$.}
    \label{fig:table-s}
\end{figure}

\Cref{fig:table-s} illustrates \Cref{thm:pretzel} for both odd and even 3-strand pretzel knots. For fixed $p > 0$, the value $s(P(p, q, r))$ is placed at each odd lattice point $(q, r)$. The orange lines indicate the equations $p + q = 0$, $p + r = 0$, and $q + r = 0$. Note that whenever one of these equalities holds, the pretzel knot $P(p, q, r)$ is ribbon, and hence $s = 0$.

We remark that \Cref{thm:pretzel} is not a new result: Suzuki \cite{Suzuki:2010} directly computed the $s$-invariant for all odd 3-strand pretzel knots, and for a subset of even ones, using Turner's spectral sequence \cite{Turner:2008}. Lewark \cite{Lewark:2014} later determined the $s$-invariants for even 3-strand pretzel knots via Manion's formula for the Khovanov homology of 3-strand pretzel links \cite{Manion:2014}. In contrast to their approaches, our computation is purely combinatorial and elementary: it combines the formal properties of \textit{slice-torus invariants} with the diagrammatic method developed in this paper.

We expect that our method can be applied to a broader class of knots, beginning with those composed of rational tangles, such as Montesinos knots. We also anticipate that it may be useful in studying the behavior of $s$ under various operations, including knot concordances, cobordisms, crossing changes, and satellite constructions.
In a subsequent paper, we further extend our framework to the \textit{equivariant Rasmussen invariant} $(\underline{s}, \overline{s})$ for strongly invertible knots, introduced in earlier work by the second author \cite{Sano:2025}.

\medskip
\noindent 
\textbf{Organization.} This paper is organized as follows. In \Cref{sec:prelim}, we review Bar-Natan’s reformulation of Khovanov homology for tangles \cite{BarNatan:2004} and summarize the relevant results used in the subsequent sections. In \Cref{sec:reform-s}, we reformulate the $s$-invariant within Bar-Natan’s framework, extend the definition of Lee cycles to tangle diagrams, and prove \Cref{thm:s-reformulate,thm:lee-cycle-decomp}. In \Cref{sec:twist-tangles}, we carry out explicit computations for twist tangles. In \Cref{sec:pretzel}, we prove \Cref{thm:pretzel} by applying our diagrammatic method. Finally, in \Cref{sec:appendix}, we present a detailed proof of \Cref{prop:ca-under-reidemeister}, which establishes the invariance of $s$ in this reformulation.

\medskip
\noindent 
\textbf{Acknowledgements.}
TS dedicates this paper to his Ph.D. advisor, Mikio Furuta, on the occasion of his retirement from the University of Tokyo in March 2025. He is especially grateful for Professor Furuta’s encouragement to pursue this research and for the opportunity to present an early version of this work at a seminar in May 2023. He also thanks Nobuo Iida and Jin Miyazawa for organizing that seminar.
KK appreciates continuous support of his advisor, JungHwan Park.
The authors thank Tetsuya Abe and Nobuo Iida for their helpful comments and corrections on an early version of this paper.
TS was supported by JSPS KAKENHI Grant Numbers 23K12982 and academist crowdfunding.
KK is partially supported by the Samsung Science and Technology Foundation (SSTF-BA2102-02) and the NRF
 grant RS-2025-00542968.
    \section{Khovanov homology for tangles}
\label{sec:prelim}

In this section, we review the reformulation of Khovanov homology for tangles given by Bar-Natan in \cite{BarNatan:2004}. Furthermore, we collect and summarize results given in \cite{Naot:2006,BarNatan:2007,Khovanov:2022} that will be used in the subsequent sections. Hereafter, all geometric objects we consider are assumed to be smooth. 

\subsection{Category of dotted cobordisms}
\label{subsec:cob3}

\begin{definition}
    A set of even number of points on the boundary of the unit disk $B \subset \del D^2$ is called a \textit{set of boundary points}. In particular, $B = \varnothing$ is denoted $\underline{0}$, and $B = \{(0, \pm 1)\}$ is denoted $\underline{2}$. 
\end{definition}

In the following, we fix a set of boundary points $B$. 

\begin{definition}
    A preadditive category $\Cob^3(B)$ is defined as follows: an object of $\Cob^3(B)$ is an unoriented compact $1$-manifold $X$ properly embedded in $D^2$ such that $\partial X = B$. A morphism $f: X \rightarrow Y$ in $\Cob^3(B)$ is a $\ZZ$-linear combination of boundary-fixing isotopy classes of cobordisms from $X$ to $Y$. Here, a cobordism from $X$ to $Y$ is an unoriented compact $2$-manifold $S$ properly embedded in $D^2 \times I$ such that $S \cap \partial (D^2 \times I) = X \times \{0\} \cup B \times I \cup Y \times \{1\}$. 
\end{definition}

\begin{definition}
    The \textit{local relations} on $\Cob^3(B)$ are the following three relations (S), (T), (4Tu) defined on each hom-set of $\Cob^3(B)$:
    \begin{center}
        \input{tikzpictures/loc-relations}
    \end{center}
    Here, each picture is assumed to be \textit{local}, i.e.\ it depicts the intersection of a $3$-ball and a cobordism in $D^2 \times I$. The quotient category, $\Cob^3(B)$ modulo local relations is denoted $\Cob^3_{/l}(B)$. 
\end{definition}

\begin{definition}
    $\Cob^3(B)$ is endowed a graded category structure by the \textit{quantum degree function}: for any cobordism $C$, its quantum degree is defined by 
    \[
        \qdeg C = \chi(C) - \frac{1}{2}|B|.
    \]
    Since $\qdeg$ is preserved under the local relations, the grading structure is inherited to $\Cob^3_{/l}(B)$. For any object $X$ and $m \in \ZZ$, the grading shift of $X$ by $m$ will be denoted $X\{m\}$.
\end{definition}

Although the main part of \cite{BarNatan:2004} is based on this category $\Cob^3_{/l}(B)$, there is a more handier version, called the \textit{dotted category}, from which the original Khovanov homology and its deformations can be fully recovered.

\begin{definition}
    A preadditive category $\Cob^3_\bullet(B)$ is defined similarly as $\Cob^3(B)$, except each cobordism is allowed to have finitely many dots in its interior that can move freely within their components. 
\end{definition}

\begin{definition}
    The \textit{local relations} on $\Cob^3_\bullet(B)$ are the following three relations (S), (S$_\bullet$), (NC) defined on each hom-set of $\Cob^3_\bullet(B)$:
    \begin{center}
        \input{tikzpictures/loc-relation-dot}
    \end{center}
    The quotient category, $\Cob^3_\bullet(B)$ modulo local relations is denoted $\Cob^3_{\bullet/l}(B)$. 
\end{definition}

\begin{definition}
\label{def:cob-grad}
    $\Cob^3_\bullet(B)$ is endowed a graded category structure by the \textit{quantum degree function}: each dot is declared to have degree $-2$, and for any cobordism $C$ with $k$ dots, 
    \[
        \qdeg C = \chi(C) - \frac{1}{2}|B| - 2k.
    \]
    The grading structure is inherited to $\Cob^3_{\bullet/l}(B)$.
\end{definition}

A geometric interpretation of the dot $\bullet$ can be given in the context of the \textit{reduced theory}, as explained in \Cref{subsec:reduced-khovanov}. The undotted and the dotted categories are related as follows.

\begin{proposition}
    As graded categories, $\Cob^3(B)$ embeds into $\Cob^3_{\bullet}(B)$, and the local relations in $\Cob^3(B)$ are preserved in $\Cob^3_{\bullet}(B)$. Thus there is a natural functor from $\Cob^3_{/l}(B)$ to $\Cob^3_{\bullet/l}(B)$ such that the following diagram commutes:
    \[
    \begin{tikzcd}
        \Cob^3(B) \arrow[r, hook] \arrow[d, two heads] 
        & \Cob^3_\bullet(B) \arrow[d, two heads] \\
        \Cob^3_{/l}(B) \arrow[r, dashed]
        & \Cob^3_{\bullet/l}(B)
    \end{tikzcd}
    \]
\end{proposition}

\begin{proof}
    It is easy to see that (T), (4Tu) can be derived from (S), (S$_\bullet$), (NC). 
\end{proof}

In $\Cob^3_{\bullet/l}(B)$, let $h, t$ denote the following closed dotted cobordisms:
\begin{center}
    \tikzset{every picture/.style={line width=0.75pt}} %set default line width to 0.75pt        

\begin{tikzpicture}[x=0.75pt,y=0.75pt,yscale=-1,xscale=1]
%uncomment if require: \path (0,86); %set diagram left start at 0, and has height of 86

%Shape: Arc [id:dp15897402593249266] 
\draw  [draw opacity=0] (89.94,40.75) .. controls (89.29,43) and (82.45,44.76) .. (74.11,44.76) .. controls (65.35,44.76) and (58.24,42.81) .. (58.24,40.41) .. controls (58.24,40.26) and (58.27,40.11) .. (58.32,39.97) -- (74.11,40.41) -- cycle ; \draw  [color={rgb, 255:red, 128; green, 128; blue, 128 }  ,draw opacity=1 ] (89.94,40.75) .. controls (89.29,43) and (82.45,44.76) .. (74.11,44.76) .. controls (65.35,44.76) and (58.24,42.81) .. (58.24,40.41) .. controls (58.24,40.26) and (58.27,40.11) .. (58.32,39.97) ;  
%Shape: Arc [id:dp6833264157790081] 
\draw  [draw opacity=0][dash pattern={on 0.84pt off 2.51pt}] (58.63,39.33) .. controls (59.9,37.27) and (66.43,35.7) .. (74.28,35.7) .. controls (82.93,35.7) and (89.96,37.6) .. (90.15,39.96) -- (74.28,40.06) -- cycle ; \draw  [color={rgb, 255:red, 128; green, 128; blue, 128 }  ,draw opacity=1 ][dash pattern={on 0.84pt off 2.51pt}] (58.63,39.33) .. controls (59.9,37.27) and (66.43,35.7) .. (74.28,35.7) .. controls (82.93,35.7) and (89.96,37.6) .. (90.15,39.96) ;  
%Shape: Circle [id:dp7493091039430603] 
\draw   (58.12,40.41) .. controls (58.12,31.57) and (65.28,24.41) .. (74.11,24.41) .. controls (82.95,24.41) and (90.11,31.57) .. (90.11,40.41) .. controls (90.11,49.24) and (82.95,56.4) .. (74.11,56.4) .. controls (65.28,56.4) and (58.12,49.24) .. (58.12,40.41) -- cycle ;
%Shape: Ellipse [id:dp6860921357132487] 
\draw  [fill={rgb, 255:red, 0; green, 0; blue, 0 }  ,fill opacity=1 ] (67.63,31.68) .. controls (67.63,30.55) and (68.55,29.63) .. (69.68,29.63) .. controls (70.82,29.63) and (71.73,30.55) .. (71.73,31.68) .. controls (71.73,32.81) and (70.82,33.73) .. (69.68,33.73) .. controls (68.55,33.73) and (67.63,32.81) .. (67.63,31.68) -- cycle ;
%Shape: Ellipse [id:dp483675740291299] 
\draw  [fill={rgb, 255:red, 0; green, 0; blue, 0 }  ,fill opacity=1 ] (77.23,31.68) .. controls (77.23,30.55) and (78.15,29.63) .. (79.28,29.63) .. controls (80.42,29.63) and (81.33,30.55) .. (81.33,31.68) .. controls (81.33,32.81) and (80.42,33.73) .. (79.28,33.73) .. controls (78.15,33.73) and (77.23,32.81) .. (77.23,31.68) -- cycle ;
%Shape: Arc [id:dp5400793746532146] 
\draw  [draw opacity=0] (249.56,23.31) .. controls (249,25.25) and (243.11,26.77) .. (235.92,26.77) .. controls (228.36,26.77) and (222.23,25.09) .. (222.23,23.01) .. controls (222.23,22.88) and (222.26,22.76) .. (222.3,22.64) -- (235.92,23.01) -- cycle ; \draw  [color={rgb, 255:red, 128; green, 128; blue, 128 }  ,draw opacity=1 ] (249.56,23.31) .. controls (249,25.25) and (243.11,26.77) .. (235.92,26.77) .. controls (228.36,26.77) and (222.23,25.09) .. (222.23,23.01) .. controls (222.23,22.88) and (222.26,22.76) .. (222.3,22.64) ;  
%Shape: Arc [id:dp3207985778436975] 
\draw  [draw opacity=0][dash pattern={on 0.84pt off 2.51pt}] (222.57,22.08) .. controls (223.67,20.31) and (229.29,18.95) .. (236.06,18.95) .. controls (243.51,18.95) and (249.57,20.59) .. (249.75,22.62) -- (236.06,22.71) -- cycle ; \draw  [color={rgb, 255:red, 128; green, 128; blue, 128 }  ,draw opacity=1 ][dash pattern={on 0.84pt off 2.51pt}] (222.57,22.08) .. controls (223.67,20.31) and (229.29,18.95) .. (236.06,18.95) .. controls (243.51,18.95) and (249.57,20.59) .. (249.75,22.62) ;  
%Shape: Ellipse [id:dp5941715166163359] 
\draw   (222.13,23.01) .. controls (222.13,15.4) and (228.3,9.22) .. (235.92,9.22) .. controls (243.53,9.22) and (249.71,15.4) .. (249.71,23.01) .. controls (249.71,30.63) and (243.53,36.8) .. (235.92,36.8) .. controls (228.3,36.8) and (222.13,30.63) .. (222.13,23.01) -- cycle ;
%Shape: Ellipse [id:dp1657843871417639] 
\draw  [fill={rgb, 255:red, 0; green, 0; blue, 0 }  ,fill opacity=1 ] (230.33,15.49) .. controls (230.33,14.51) and (231.12,13.72) .. (232.1,13.72) .. controls (233.07,13.72) and (233.87,14.51) .. (233.87,15.49) .. controls (233.87,16.47) and (233.07,17.26) .. (232.1,17.26) .. controls (231.12,17.26) and (230.33,16.47) .. (230.33,15.49) -- cycle ;
%Shape: Ellipse [id:dp5151756206642133] 
\draw  [fill={rgb, 255:red, 0; green, 0; blue, 0 }  ,fill opacity=1 ] (238.61,15.49) .. controls (238.61,14.51) and (239.4,13.72) .. (240.37,13.72) .. controls (241.35,13.72) and (242.14,14.51) .. (242.14,15.49) .. controls (242.14,16.47) and (241.35,17.26) .. (240.37,17.26) .. controls (239.4,17.26) and (238.61,16.47) .. (238.61,15.49) -- cycle ;

%Shape: Arc [id:dp9257785879079623] 
\draw  [draw opacity=0] (249.56,60.11) .. controls (249,62.05) and (243.11,63.57) .. (235.92,63.57) .. controls (228.36,63.57) and (222.23,61.89) .. (222.23,59.81) .. controls (222.23,59.68) and (222.26,59.56) .. (222.3,59.44) -- (235.92,59.81) -- cycle ; \draw  [color={rgb, 255:red, 128; green, 128; blue, 128 }  ,draw opacity=1 ] (249.56,60.11) .. controls (249,62.05) and (243.11,63.57) .. (235.92,63.57) .. controls (228.36,63.57) and (222.23,61.89) .. (222.23,59.81) .. controls (222.23,59.68) and (222.26,59.56) .. (222.3,59.44) ;  
%Shape: Arc [id:dp5930238767137886] 
\draw  [draw opacity=0][dash pattern={on 0.84pt off 2.51pt}] (222.57,58.88) .. controls (223.67,57.11) and (229.29,55.75) .. (236.06,55.75) .. controls (243.51,55.75) and (249.57,57.39) .. (249.75,59.42) -- (236.06,59.51) -- cycle ; \draw  [color={rgb, 255:red, 128; green, 128; blue, 128 }  ,draw opacity=1 ][dash pattern={on 0.84pt off 2.51pt}] (222.57,58.88) .. controls (223.67,57.11) and (229.29,55.75) .. (236.06,55.75) .. controls (243.51,55.75) and (249.57,57.39) .. (249.75,59.42) ;  
%Shape: Ellipse [id:dp9593516230095376] 
\draw   (222.13,59.81) .. controls (222.13,52.2) and (228.3,46.02) .. (235.92,46.02) .. controls (243.53,46.02) and (249.71,52.2) .. (249.71,59.81) .. controls (249.71,67.43) and (243.53,73.6) .. (235.92,73.6) .. controls (228.3,73.6) and (222.13,67.43) .. (222.13,59.81) -- cycle ;
%Shape: Ellipse [id:dp3050748545457964] 
\draw  [fill={rgb, 255:red, 0; green, 0; blue, 0 }  ,fill opacity=1 ] (230.33,52.29) .. controls (230.33,51.31) and (231.12,50.52) .. (232.1,50.52) .. controls (233.07,50.52) and (233.87,51.31) .. (233.87,52.29) .. controls (233.87,53.27) and (233.07,54.06) .. (232.1,54.06) .. controls (231.12,54.06) and (230.33,53.27) .. (230.33,52.29) -- cycle ;
%Shape: Ellipse [id:dp5389470679739694] 
\draw  [fill={rgb, 255:red, 0; green, 0; blue, 0 }  ,fill opacity=1 ] (238.61,52.29) .. controls (238.61,51.31) and (239.4,50.52) .. (240.37,50.52) .. controls (241.35,50.52) and (242.14,51.31) .. (242.14,52.29) .. controls (242.14,53.27) and (241.35,54.06) .. (240.37,54.06) .. controls (239.4,54.06) and (238.61,53.27) .. (238.61,52.29) -- cycle ;

%Shape: Arc [id:dp31048787170952385] 
\draw  [draw opacity=0] (195.94,39.76) .. controls (195.29,42) and (188.45,43.77) .. (180.11,43.77) .. controls (171.35,43.77) and (164.24,41.82) .. (164.24,39.41) .. controls (164.24,39.26) and (164.27,39.12) .. (164.32,38.98) -- (180.11,39.41) -- cycle ; \draw  [color={rgb, 255:red, 128; green, 128; blue, 128 }  ,draw opacity=1 ] (195.94,39.76) .. controls (195.29,42) and (188.45,43.77) .. (180.11,43.77) .. controls (171.35,43.77) and (164.24,41.82) .. (164.24,39.41) .. controls (164.24,39.26) and (164.27,39.12) .. (164.32,38.98) ;  
%Shape: Arc [id:dp4135788123972073] 
\draw  [draw opacity=0][dash pattern={on 0.84pt off 2.51pt}] (164.63,38.33) .. controls (165.9,36.27) and (172.43,34.71) .. (180.28,34.71) .. controls (188.93,34.71) and (195.96,36.6) .. (196.15,38.96) -- (180.28,39.06) -- cycle ; \draw  [color={rgb, 255:red, 128; green, 128; blue, 128 }  ,draw opacity=1 ][dash pattern={on 0.84pt off 2.51pt}] (164.63,38.33) .. controls (165.9,36.27) and (172.43,34.71) .. (180.28,34.71) .. controls (188.93,34.71) and (195.96,36.6) .. (196.15,38.96) ;  
%Shape: Circle [id:dp087883636425651] 
\draw   (164.12,39.41) .. controls (164.12,30.58) and (171.28,23.42) .. (180.11,23.42) .. controls (188.95,23.42) and (196.11,30.58) .. (196.11,39.41) .. controls (196.11,48.25) and (188.95,55.41) .. (180.11,55.41) .. controls (171.28,55.41) and (164.12,48.25) .. (164.12,39.41) -- cycle ;
%Shape: Ellipse [id:dp890068764944708] 
\draw  [fill={rgb, 255:red, 0; green, 0; blue, 0 }  ,fill opacity=1 ] (170.63,30.49) .. controls (170.63,29.35) and (171.55,28.44) .. (172.68,28.44) .. controls (173.82,28.44) and (174.73,29.35) .. (174.73,30.49) .. controls (174.73,31.62) and (173.82,32.54) .. (172.68,32.54) .. controls (171.55,32.54) and (170.63,31.62) .. (170.63,30.49) -- cycle ;
%Shape: Ellipse [id:dp5723981912476287] 
\draw  [fill={rgb, 255:red, 0; green, 0; blue, 0 }  ,fill opacity=1 ] (178.23,30.49) .. controls (178.23,29.35) and (179.15,28.44) .. (180.28,28.44) .. controls (181.42,28.44) and (182.33,29.35) .. (182.33,30.49) .. controls (182.33,31.62) and (181.42,32.54) .. (180.28,32.54) .. controls (179.15,32.54) and (178.23,31.62) .. (178.23,30.49) -- cycle ;
%Shape: Ellipse [id:dp3158081729720502] 
\draw  [fill={rgb, 255:red, 0; green, 0; blue, 0 }  ,fill opacity=1 ] (185.83,30.49) .. controls (185.83,29.35) and (186.75,28.44) .. (187.88,28.44) .. controls (189.02,28.44) and (189.93,29.35) .. (189.93,30.49) .. controls (189.93,31.62) and (189.02,32.54) .. (187.88,32.54) .. controls (186.75,32.54) and (185.83,31.62) .. (185.83,30.49) -- cycle ;

% Text Node
\draw (204.5,30.9) node [anchor=north west][inner sep=0.75pt]    {$-$};
% Text Node
\draw (21,31.61) node [anchor=north west][inner sep=0.75pt]    {$h\ =\ $};
% Text Node
\draw (127,31.61) node [anchor=north west][inner sep=0.75pt]    {$t\ =\ $};
% Text Node
\draw (97,31.61) node [anchor=north west][inner sep=0.75pt]    {$,$};

\end{tikzpicture}
\end{center}
We have $\qdeg h = -2,\ \qdeg t = -4$. The relation (NC) implies the following relation:
\begin{center}
    \tikzset{every picture/.style={line width=0.75pt}} %set default line width to 0.75pt        

\begin{tikzpicture}[x=0.75pt,y=0.75pt,yscale=-.75,xscale=.75]
%uncomment if require: \path (0,67); %set diagram left start at 0, and has height of 67

%Shape: Rectangle [id:dp1880103673271405] 
\draw  [dash pattern={on 0.84pt off 2.51pt}] (10.93,11.56) -- (56.56,11.56) -- (56.56,55.49) -- (10.93,55.49) -- cycle ;
%Shape: Rectangle [id:dp3671255456183994] 
\draw  [dash pattern={on 0.84pt off 2.51pt}] (122.41,11.06) -- (169.08,11.06) -- (169.08,55.99) -- (122.41,55.99) -- cycle ;
%Shape: Rectangle [id:dp14439302115506092] 
\draw  [dash pattern={on 0.84pt off 2.51pt}] (224.95,12.06) -- (269.54,12.06) -- (269.54,54.99) -- (224.95,54.99) -- cycle ;
%Shape: Circle [id:dp4867055193593717] 
\draw  [fill={rgb, 255:red, 0; green, 0; blue, 0 }  ,fill opacity=1 ] (23.5,34.2) .. controls (23.5,32.43) and (24.93,31) .. (26.7,31) .. controls (28.47,31) and (29.91,32.43) .. (29.91,34.2) .. controls (29.91,35.97) and (28.47,37.41) .. (26.7,37.41) .. controls (24.93,37.41) and (23.5,35.97) .. (23.5,34.2) -- cycle ;
%Shape: Circle [id:dp9537965570069136] 
\draw  [fill={rgb, 255:red, 0; green, 0; blue, 0 }  ,fill opacity=1 ] (36,34.2) .. controls (36,32.43) and (37.43,31) .. (39.2,31) .. controls (40.97,31) and (42.41,32.43) .. (42.41,34.2) .. controls (42.41,35.97) and (40.97,37.41) .. (39.2,37.41) .. controls (37.43,37.41) and (36,35.97) .. (36,34.2) -- cycle ;
%Shape: Circle [id:dp6536374189751153] 
\draw  [fill={rgb, 255:red, 0; green, 0; blue, 0 }  ,fill opacity=1 ] (142.54,33.52) .. controls (142.54,31.75) and (143.97,30.32) .. (145.74,30.32) .. controls (147.51,30.32) and (148.95,31.75) .. (148.95,33.52) .. controls (148.95,35.29) and (147.51,36.73) .. (145.74,36.73) .. controls (143.97,36.73) and (142.54,35.29) .. (142.54,33.52) -- cycle ;

% Text Node
\draw (68,25.32) node [anchor=north west][inner sep=0.75pt]    {$=$};
% Text Node
\draw (183,25.32) node [anchor=north west][inner sep=0.75pt]    {$+$};
% Text Node
\draw (103,25.32) node [anchor=north west][inner sep=0.75pt]    {$h$};
% Text Node
\draw (207,25.32) node [anchor=north west][inner sep=0.75pt]    {$t$};

\end{tikzpicture}
\end{center}
Here, the dotted square represents some part of a surface, and the components represented by $h, t$ are assumed to be placed near the square in each term. However, the following lemma imply that, in fact, they can be placed anywhere in the complement of the surface. 

\begin{lemma}
\label{lem:pass-dot-sphere}
    In $\Cob^3_{\bullet/l}(B)$, any closed component (possibly with dots) can be unknotted and can pass through any other component. 
\end{lemma}

\begin{proof}
    By repeatedly applying (NC), any closed component can be transformed into a sum of disjoint unions of dotted spheres, and then recombined as an unknotted surface. Thus it suffices to prove that a dotted sphere can pass through any other component. Let $S$ be a sphere with $n$ dots. If $n = 0, 1$, the claim is trivial from (S) and (S$_\bullet$). If $n \geq 2$, consider the following isotopy:
    \begin{center}
        \tikzset{every picture/.style={line width=0.75pt}} %set default line width to 0.75pt        

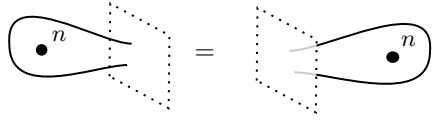
\begin{tikzpicture}[x=0.75pt,y=0.75pt,yscale=-.75,xscale=.75]
%uncomment if require: \path (0,95); %set diagram left start at 0, and has height of 95

%Shape: Circle [id:dp4920417855202871] 
\draw  [fill={rgb, 255:red, 0; green, 0; blue, 0 }  ,fill opacity=1 ] (54,45.2) .. controls (54,43.43) and (55.43,42) .. (57.2,42) .. controls (58.97,42) and (60.41,43.43) .. (60.41,45.2) .. controls (60.41,46.97) and (58.97,48.41) .. (57.2,48.41) .. controls (55.43,48.41) and (54,46.97) .. (54,45.2) -- cycle ;
%Curve Lines [id:da9325994926166838] 
\draw    (35.49,41.12) .. controls (32.49,0.12) and (96.49,41.12) .. (117.7,41) ;
%Curve Lines [id:da10467713182482452] 
\draw    (35.49,41.12) .. controls (35.49,80.12) and (93.49,55.53) .. (114.7,55.41) ;
%Shape: Parallelogram [id:dp2680788725045189] 
\draw  [dash pattern={on 0.84pt off 2.51pt}] (102.85,62.99) -- (103.37,13.99) -- (143.15,35.42) -- (142.63,84.42) -- cycle ;
%Shape: Ellipse [id:dp9593488305455407] 
\draw  [fill={rgb, 255:red, 0; green, 0; blue, 0 }  ,fill opacity=1 ] (297.05,50.07) .. controls (297.05,48.3) and (295.4,46.86) .. (293.36,46.86) .. controls (291.32,46.86) and (289.67,48.3) .. (289.67,50.07) .. controls (289.67,51.84) and (291.32,53.27) .. (293.36,53.27) .. controls (295.4,53.27) and (297.05,51.84) .. (297.05,50.07) -- cycle ;
%Curve Lines [id:da05535218230010608] 
\draw    (318.37,45.98) .. controls (321.82,4.98) and (248.13,45.98) .. (223.7,45.86) ;
%Curve Lines [id:da8204055620252977] 
\draw    (318.37,45.98) .. controls (318.37,84.98) and (251.59,60.39) .. (227.16,60.27) ;
%Shape: Parallelogram [id:dp23076736541509546] 
\draw  [fill={rgb, 255:red, 255; green, 255; blue, 255 }  ,fill opacity=0.8 ][dash pattern={on 0.84pt off 2.51pt}] (201.85,65.85) -- (202.37,16.86) -- (242.15,38.28) -- (241.63,87.28) -- cycle ;

% Text Node
\draw (62,30.61) node [anchor=north west][inner sep=0.75pt]  [font=\small]  {$n$};
% Text Node
\draw (297,33.61) node [anchor=north west][inner sep=0.75pt]  [font=\small]  {$n$};
% Text Node
\draw (158,42.61) node [anchor=north west][inner sep=0.75pt]    {$=$};

\end{tikzpicture}
    \end{center}
    By applying (NC), the left hand side transforms into:
    \begin{center}
        \tikzset{every picture/.style={line width=0.75pt}} %set default line width to 0.75pt        

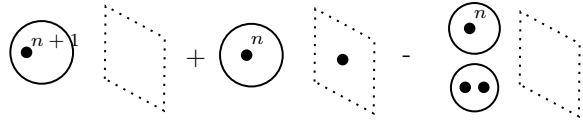
\begin{tikzpicture}[x=0.75pt,y=0.75pt,yscale=-.75,xscale=.75]
%uncomment if require: \path (0,95); %set diagram left start at 0, and has height of 95

%Shape: Circle [id:dp8626270800256467] 
\draw  [fill={rgb, 255:red, 0; green, 0; blue, 0 }  ,fill opacity=1 ] (47,45.2) .. controls (47,43.43) and (48.43,42) .. (50.2,42) .. controls (51.97,42) and (53.41,43.43) .. (53.41,45.2) .. controls (53.41,46.97) and (51.97,48.41) .. (50.2,48.41) .. controls (48.43,48.41) and (47,46.97) .. (47,45.2) -- cycle ;
%Shape: Parallelogram [id:dp5016293375822631] 
\draw  [dash pattern={on 0.84pt off 2.51pt}] (102.85,62.99) -- (103.37,13.99) -- (143.15,35.42) -- (142.63,84.42) -- cycle ;
%Shape: Circle [id:dp6414927470087627] 
\draw   (39.33,45.2) .. controls (39.33,33.56) and (48.77,24.13) .. (60.41,24.13) .. controls (72.05,24.13) and (81.49,33.56) .. (81.49,45.2) .. controls (81.49,56.84) and (72.05,66.28) .. (60.41,66.28) .. controls (48.77,66.28) and (39.33,56.84) .. (39.33,45.2) -- cycle ;
%Shape: Circle [id:dp016322107839386035] 
\draw  [fill={rgb, 255:red, 0; green, 0; blue, 0 }  ,fill opacity=1 ] (195,47.07) .. controls (195,45.3) and (196.43,43.86) .. (198.2,43.86) .. controls (199.97,43.86) and (201.41,45.3) .. (201.41,47.07) .. controls (201.41,48.84) and (199.97,50.27) .. (198.2,50.27) .. controls (196.43,50.27) and (195,48.84) .. (195,47.07) -- cycle ;
%Shape: Parallelogram [id:dp18414886576506528] 
\draw  [dash pattern={on 0.84pt off 2.51pt}] (243.85,64.85) -- (244.37,15.86) -- (284.15,37.28) -- (283.63,86.28) -- cycle ;
%Shape: Circle [id:dp5151272639890471] 
\draw   (180.33,47.07) .. controls (180.33,35.43) and (189.77,25.99) .. (201.41,25.99) .. controls (213.05,25.99) and (222.49,35.43) .. (222.49,47.07) .. controls (222.49,58.71) and (213.05,68.14) .. (201.41,68.14) .. controls (189.77,68.14) and (180.33,58.71) .. (180.33,47.07) -- cycle ;
%Shape: Circle [id:dp9784702121638151] 
\draw  [fill={rgb, 255:red, 0; green, 0; blue, 0 }  ,fill opacity=1 ] (259.59,50.07) .. controls (259.59,48.3) and (261.03,46.86) .. (262.8,46.86) .. controls (264.57,46.86) and (266,48.3) .. (266,50.07) .. controls (266,51.84) and (264.57,53.27) .. (262.8,53.27) .. controls (261.03,53.27) and (259.59,51.84) .. (259.59,50.07) -- cycle ;
%Shape: Circle [id:dp5672130094085165] 
\draw  [fill={rgb, 255:red, 0; green, 0; blue, 0 }  ,fill opacity=1 ] (345,29.07) .. controls (345,27.3) and (346.43,25.86) .. (348.2,25.86) .. controls (349.97,25.86) and (351.41,27.3) .. (351.41,29.07) .. controls (351.41,30.84) and (349.97,32.27) .. (348.2,32.27) .. controls (346.43,32.27) and (345,30.84) .. (345,29.07) -- cycle ;
%Shape: Parallelogram [id:dp09537640901309308] 
\draw  [dash pattern={on 0.84pt off 2.51pt}] (381.85,66.85) -- (382.37,17.86) -- (422.15,39.28) -- (421.63,88.28) -- cycle ;
%Shape: Circle [id:dp42325307810301016] 
\draw   (334.33,29.07) .. controls (334.33,19.63) and (341.98,11.99) .. (351.41,11.99) .. controls (360.84,11.99) and (368.49,19.63) .. (368.49,29.07) .. controls (368.49,38.5) and (360.84,46.14) .. (351.41,46.14) .. controls (341.98,46.14) and (334.33,38.5) .. (334.33,29.07) -- cycle ;
%Shape: Circle [id:dp6807605006983658] 
\draw   (335.59,69) .. controls (335.59,60.21) and (342.71,53.09) .. (351.5,53.09) .. controls (360.29,53.09) and (367.41,60.21) .. (367.41,69) .. controls (367.41,77.79) and (360.29,84.91) .. (351.5,84.91) .. controls (342.71,84.91) and (335.59,77.79) .. (335.59,69) -- cycle ;
%Shape: Circle [id:dp38516486549179874] 
\draw  [fill={rgb, 255:red, 0; green, 0; blue, 0 }  ,fill opacity=1 ] (342,68.7) .. controls (342,66.93) and (343.43,65.5) .. (345.2,65.5) .. controls (346.97,65.5) and (348.41,66.93) .. (348.41,68.7) .. controls (348.41,70.47) and (346.97,71.91) .. (345.2,71.91) .. controls (343.43,71.91) and (342,70.47) .. (342,68.7) -- cycle ;
%Shape: Circle [id:dp8237417114585771] 
\draw  [fill={rgb, 255:red, 0; green, 0; blue, 0 }  ,fill opacity=1 ] (354.5,68.7) .. controls (354.5,66.93) and (355.93,65.5) .. (357.7,65.5) .. controls (359.47,65.5) and (360.91,66.93) .. (360.91,68.7) .. controls (360.91,70.47) and (359.47,71.91) .. (357.7,71.91) .. controls (355.93,71.91) and (354.5,70.47) .. (354.5,68.7) -- cycle ;

% Text Node
\draw (51,30.61) node [anchor=north west][inner sep=0.75pt]  [font=\scriptsize]  {$n+1$};
% Text Node
\draw (199,32.47) node [anchor=north west][inner sep=0.75pt]  [font=\scriptsize]  {$n$};
% Text Node
\draw (155,41.21) node [anchor=north west][inner sep=0.75pt]   [align=left] {+};
% Text Node
\draw (349,14.47) node [anchor=north west][inner sep=0.75pt]  [font=\scriptsize]  {$n$};
% Text Node
\draw (301,42.21) node [anchor=north west][inner sep=0.75pt]   [align=left] {\mbox{-}};

\end{tikzpicture}
    \end{center}
    Similarly, for the right hand side, we obtain a similar picture with the spheres placed on the right side of the wall. Thus the claim follows by induction. 
\end{proof}

In particular, \Cref{lem:pass-dot-sphere} implies that each hom-module of $\Cob^3_{\bullet/l}(B)$ admits a $\ZZ[h, t]$-module structure, and $\Cob^3_{\bullet/l}(B)$ becomes a $\ZZ[h, t]$-linear category. 

Following \cite{BHP:2023} and \cite{Khovanov:2022}, we also allow \textit{hollow dots} $\circ$ and \textit{stars} $\star$ to be put on a surface in $\Cob^3_{\bullet/l}(B)$. First, a hollow dot represents%
\footnote{
    Our definition of the hollow dot $\circ$ differs by an overall sign from the one defined in \cite{BHP:2023} and in \cite{Khovanov:2004}. This convention suits better for the expressions given in \Cref{prop:ca-under-reidemeister,prop:ca-under-cob}. See also the discussion at the end of this subsection regarding duality. 
} 
\begin{center}
    \tikzset{every picture/.style={line width=0.75pt}} %set default line width to 0.75pt        

\begin{tikzpicture}[x=0.75pt,y=0.75pt,yscale=-.75,xscale=.75]
%uncomment if require: \path (0,67); %set diagram left start at 0, and has height of 67

%Shape: Rectangle [id:dp6657862526390232] 
\draw  [dash pattern={on 0.84pt off 2.51pt}] (10.93,11.56) -- (56.56,11.56) -- (56.56,55.49) -- (10.93,55.49) -- cycle ;
%Shape: Rectangle [id:dp05316274891441941] 
\draw  [dash pattern={on 0.84pt off 2.51pt}] (101.41,12.06) -- (148.08,12.06) -- (148.08,56.99) -- (101.41,56.99) -- cycle ;
%Shape: Circle [id:dp6568279118439123] 
\draw  [color={rgb, 255:red, 0; green, 0; blue, 0 }  ,draw opacity=1 ][fill={rgb, 255:red, 255; green, 255; blue, 255 }  ,fill opacity=1 ][line width=0.75]  (29.5,34.2) .. controls (29.5,32.43) and (30.93,31) .. (32.7,31) .. controls (34.47,31) and (35.91,32.43) .. (35.91,34.2) .. controls (35.91,35.97) and (34.47,37.41) .. (32.7,37.41) .. controls (30.93,37.41) and (29.5,35.97) .. (29.5,34.2) -- cycle ;
%Shape: Circle [id:dp4996287769828862] 
\draw  [fill={rgb, 255:red, 0; green, 0; blue, 0 }  ,fill opacity=1 ] (121.54,34.52) .. controls (121.54,32.75) and (122.97,31.32) .. (124.74,31.32) .. controls (126.51,31.32) and (127.95,32.75) .. (127.95,34.52) .. controls (127.95,36.29) and (126.51,37.73) .. (124.74,37.73) .. controls (122.97,37.73) and (121.54,36.29) .. (121.54,34.52) -- cycle ;
%Shape: Rectangle [id:dp696397315587286] 
\draw  [dash pattern={on 0.84pt off 2.51pt}] (209.95,11.91) -- (254.54,11.91) -- (254.54,54.84) -- (209.95,54.84) -- cycle ;

% Text Node
\draw (68,25.32) node [anchor=north west][inner sep=0.75pt]    {$=$};
% Text Node
\draw (167,24.24) node [anchor=north west][inner sep=0.75pt]    {$-$};
% Text Node
\draw (194,24.17) node [anchor=north west][inner sep=0.75pt]    {$h$};

\end{tikzpicture}
\end{center}
Then we have relations
\begin{center}
    \tikzset{every picture/.style={line width=0.75pt}} %set default line width to 0.75pt        

\begin{tikzpicture}[x=0.75pt,y=0.75pt,yscale=-.75,xscale=.75]
%uncomment if require: \path (0,179); %set diagram left start at 0, and has height of 179

%Shape: Rectangle [id:dp09151331113716654] 
\draw  [dash pattern={on 0.84pt off 2.51pt}] (10.93,11.56) -- (56.56,11.56) -- (56.56,55.49) -- (10.93,55.49) -- cycle ;
%Shape: Rectangle [id:dp7313549892497327] 
\draw  [dash pattern={on 0.84pt off 2.51pt}] (128.41,11.06) -- (175.08,11.06) -- (175.08,55.99) -- (128.41,55.99) -- cycle ;
%Shape: Rectangle [id:dp04296904551935177] 
\draw  [dash pattern={on 0.84pt off 2.51pt}] (230.95,12.06) -- (275.54,12.06) -- (275.54,54.99) -- (230.95,54.99) -- cycle ;
%Shape: Circle [id:dp6215879466612017] 
\draw  [color={rgb, 255:red, 0; green, 0; blue, 0 }  ,draw opacity=1 ][fill={rgb, 255:red, 255; green, 255; blue, 255 }  ,fill opacity=1 ] (23.5,34.2) .. controls (23.5,32.43) and (24.93,31) .. (26.7,31) .. controls (28.47,31) and (29.91,32.43) .. (29.91,34.2) .. controls (29.91,35.97) and (28.47,37.41) .. (26.7,37.41) .. controls (24.93,37.41) and (23.5,35.97) .. (23.5,34.2) -- cycle ;
%Shape: Circle [id:dp3391062439777357] 
\draw  [color={rgb, 255:red, 0; green, 0; blue, 0 }  ,draw opacity=1 ][fill={rgb, 255:red, 255; green, 255; blue, 255 }  ,fill opacity=1 ] (36,34.2) .. controls (36,32.43) and (37.43,31) .. (39.2,31) .. controls (40.97,31) and (42.41,32.43) .. (42.41,34.2) .. controls (42.41,35.97) and (40.97,37.41) .. (39.2,37.41) .. controls (37.43,37.41) and (36,35.97) .. (36,34.2) -- cycle ;
%Shape: Circle [id:dp18360081161742248] 
\draw  [color={rgb, 255:red, 0; green, 0; blue, 0 }  ,draw opacity=1 ][fill={rgb, 255:red, 255; green, 255; blue, 255 }  ,fill opacity=1 ] (148.54,33.52) .. controls (148.54,31.75) and (149.97,30.32) .. (151.74,30.32) .. controls (153.51,30.32) and (154.95,31.75) .. (154.95,33.52) .. controls (154.95,35.29) and (153.51,36.73) .. (151.74,36.73) .. controls (149.97,36.73) and (148.54,35.29) .. (148.54,33.52) -- cycle ;
%Shape: Rectangle [id:dp21450024325531214] 
\draw  [dash pattern={on 0.84pt off 2.51pt}] (9.93,68.56) -- (55.56,68.56) -- (55.56,112.49) -- (9.93,112.49) -- cycle ;
%Shape: Rectangle [id:dp9287078924646119] 
\draw  [dash pattern={on 0.84pt off 2.51pt}] (129.95,70.06) -- (174.54,70.06) -- (174.54,112.99) -- (129.95,112.99) -- cycle ;
%Shape: Circle [id:dp7148310372598236] 
\draw  [color={rgb, 255:red, 0; green, 0; blue, 0 }  ,draw opacity=1 ][fill={rgb, 255:red, 0; green, 0; blue, 0 }  ,fill opacity=1 ] (22.5,91.2) .. controls (22.5,89.43) and (23.93,88) .. (25.7,88) .. controls (27.47,88) and (28.91,89.43) .. (28.91,91.2) .. controls (28.91,92.97) and (27.47,94.41) .. (25.7,94.41) .. controls (23.93,94.41) and (22.5,92.97) .. (22.5,91.2) -- cycle ;
%Shape: Circle [id:dp5750392281765165] 
\draw  [color={rgb, 255:red, 0; green, 0; blue, 0 }  ,draw opacity=1 ][fill={rgb, 255:red, 255; green, 255; blue, 255 }  ,fill opacity=1 ] (35,91.2) .. controls (35,89.43) and (36.43,88) .. (38.2,88) .. controls (39.97,88) and (41.41,89.43) .. (41.41,91.2) .. controls (41.41,92.97) and (39.97,94.41) .. (38.2,94.41) .. controls (36.43,94.41) and (35,92.97) .. (35,91.2) -- cycle ;
%Shape: Arc [id:dp3828175310813955] 
\draw  [draw opacity=0] (53.81,147.4) .. controls (53.07,150) and (45.16,152.04) .. (35.51,152.04) .. controls (25.37,152.04) and (17.15,149.78) .. (17.15,147) .. controls (17.15,146.83) and (17.18,146.66) .. (17.24,146.5) -- (35.51,147) -- cycle ; \draw  [color={rgb, 255:red, 128; green, 128; blue, 128 }  ,draw opacity=1 ] (53.81,147.4) .. controls (53.07,150) and (45.16,152.04) .. (35.51,152.04) .. controls (25.37,152.04) and (17.15,149.78) .. (17.15,147) .. controls (17.15,146.83) and (17.18,146.66) .. (17.24,146.5) ;  
%Shape: Arc [id:dp731135441497379] 
\draw  [draw opacity=0][dash pattern={on 0.84pt off 2.51pt}] (17.6,145.75) .. controls (19.07,143.37) and (26.62,141.56) .. (35.71,141.56) .. controls (45.7,141.56) and (53.83,143.75) .. (54.06,146.48) -- (35.71,146.6) -- cycle ; \draw  [color={rgb, 255:red, 128; green, 128; blue, 128 }  ,draw opacity=1 ][dash pattern={on 0.84pt off 2.51pt}] (17.6,145.75) .. controls (19.07,143.37) and (26.62,141.56) .. (35.71,141.56) .. controls (45.7,141.56) and (53.83,143.75) .. (54.06,146.48) ;  
%Shape: Circle [id:dp10360601726537078] 
\draw   (17.01,147) .. controls (17.01,136.78) and (25.29,128.5) .. (35.51,128.5) .. controls (45.73,128.5) and (54.01,136.78) .. (54.01,147) .. controls (54.01,157.22) and (45.73,165.5) .. (35.51,165.5) .. controls (25.29,165.5) and (17.01,157.22) .. (17.01,147) -- cycle ;
%Shape: Circle [id:dp8442682750286817] 
\draw  [color={rgb, 255:red, 0; green, 0; blue, 0 }  ,draw opacity=1 ][fill={rgb, 255:red, 255; green, 255; blue, 255 }  ,fill opacity=1 ] (32.31,135.7) .. controls (32.31,133.93) and (33.74,132.5) .. (35.51,132.5) .. controls (37.28,132.5) and (38.71,133.93) .. (38.71,135.7) .. controls (38.71,137.47) and (37.28,138.91) .. (35.51,138.91) .. controls (33.74,138.91) and (32.31,137.47) .. (32.31,135.7) -- cycle ;

% Text Node
\draw (67,25.32) node [anchor=north west][inner sep=0.75pt]    {$=$};
% Text Node
\draw (189,25.32) node [anchor=north west][inner sep=0.75pt]    {$+$};
% Text Node
\draw (93,25.32) node [anchor=north west][inner sep=0.75pt]    {$\ -h$};
% Text Node
\draw (213,25.32) node [anchor=north west][inner sep=0.75pt]    {$t$};
% Text Node
\draw (67,82.32) node [anchor=north west][inner sep=0.75pt]    {$=$};
% Text Node
\draw (111,83.32) node [anchor=north west][inner sep=0.75pt]    {$t$};
% Text Node
\draw (67,137.4) node [anchor=north west][inner sep=0.75pt]    {$=\ 1$};

\end{tikzpicture}
\end{center}
and (NC) can be rewritten as 
\begin{center}
    \input{tikzpictures/hollow-dot-NC}
\end{center}
Next, a star $\star$ represents
\begin{center}
    \tikzset{every picture/.style={line width=0.75pt}} %set default line width to 0.75pt        

\begin{tikzpicture}[x=0.75pt,y=0.75pt,yscale=-.75,xscale=.75]
%uncomment if require: \path (0,67); %set diagram left start at 0, and has height of 67

%Shape: Rectangle [id:dp9087049366158031] 
\draw  [dash pattern={on 0.84pt off 2.51pt}] (10.93,11.56) -- (56.56,11.56) -- (56.56,55.49) -- (10.93,55.49) -- cycle ;
%Shape: Rectangle [id:dp9319588449248473] 
\draw  [dash pattern={on 0.84pt off 2.51pt}] (96.41,12.06) -- (143.08,12.06) -- (143.08,56.99) -- (96.41,56.99) -- cycle ;
%Shape: Rectangle [id:dp6235300642447601] 
\draw  [dash pattern={on 0.84pt off 2.51pt}] (181.95,13.06) -- (226.54,13.06) -- (226.54,55.99) -- (181.95,55.99) -- cycle ;
%Shape: Circle [id:dp6530489371527614] 
\draw  [fill={rgb, 255:red, 0; green, 0; blue, 0 }  ,fill opacity=1 ] (116.54,34.52) .. controls (116.54,32.75) and (117.97,31.32) .. (119.74,31.32) .. controls (121.51,31.32) and (122.95,32.75) .. (122.95,34.52) .. controls (122.95,36.29) and (121.51,37.73) .. (119.74,37.73) .. controls (117.97,37.73) and (116.54,36.29) .. (116.54,34.52) -- cycle ;
%Shape: Circle [id:dp6471212892520815] 
\draw  [color={rgb, 255:red, 0; green, 0; blue, 0 }  ,draw opacity=1 ][fill={rgb, 255:red, 255; green, 255; blue, 255 }  ,fill opacity=1 ] (201.04,34.52) .. controls (201.04,32.75) and (202.47,31.32) .. (204.24,31.32) .. controls (206.01,31.32) and (207.45,32.75) .. (207.45,34.52) .. controls (207.45,36.29) and (206.01,37.73) .. (204.24,37.73) .. controls (202.47,37.73) and (201.04,36.29) .. (201.04,34.52) -- cycle ;

% Text Node
\draw (68,25.32) node [anchor=north west][inner sep=0.75pt]    {$=$};
% Text Node
\draw (157,26.32) node [anchor=north west][inner sep=0.75pt]    {$+$};
% Text Node
\draw (26,26) node [anchor=north west][inner sep=0.75pt]    {$\star $};

\end{tikzpicture}
\end{center}
Then we have
\begin{center}
    \tikzset{every picture/.style={line width=0.75pt}} %set default line width to 0.75pt        

\begin{tikzpicture}[x=0.75pt,y=0.75pt,yscale=-.75,xscale=.75]
%uncomment if require: \path (0,67); %set diagram left start at 0, and has height of 67

%Shape: Rectangle [id:dp2767788229067313] 
\draw  [dash pattern={on 0.84pt off 2.51pt}] (10.93,11.56) -- (56.56,11.56) -- (56.56,55.49) -- (10.93,55.49) -- cycle ;
%Shape: Rectangle [id:dp7686094481876052] 
\draw  [dash pattern={on 0.84pt off 2.51pt}] (122.41,11.06) -- (169.08,11.06) -- (169.08,55.99) -- (122.41,55.99) -- cycle ;

% Text Node
\draw (68,25.32) node [anchor=north west][inner sep=0.75pt]    {$=$};
% Text Node
\draw (99,25.32) node [anchor=north west][inner sep=0.75pt]    {$\mathcal{D}$};
% Text Node
\draw (20,26) node [anchor=north west][inner sep=0.75pt]    {$\star \star$};

\end{tikzpicture}
\end{center}
where $\mathcal{D} = h^2 + 4t$ is the discriminant of $X^2 - hX - t$. Using (NC), a star on a surface is equivalent to adding a handle to that point. With these relations, we obtain the following evaluation formulae for closed cobordisms.

\begin{proposition}[{\cite{Khovanov:2022}}]
\label{lem:eval-dot-surface}
    Let $\Sigma_g$ denote the closed orientable surface of genus $g$ standardly embedded in $D^2 \times I$, and $\Sigma^\bullet_g$ the surface $\Sigma_g$ with a single dot. In $\Cob^3_{\bullet/l}(B)$, these surfaces are evaluated as 
    \[
        \Sigma_g = \begin{cases}
            0 & \text{$g$ is even,} \\
            2 \mathcal{D}^{\frac{g - 1}{2}} & \text{$g$ is odd}
        \end{cases}
    \]
    and 
    \[
        \Sigma^\bullet_g = \begin{cases}
            \mathcal{D}^{\frac{g}{2}} & \text{$g$ is even,} \\
            h \mathcal{D}^{\frac{g - 1}{2}} & \text{$g$ is odd.}
        \end{cases}
    \]
    Furthermore, the surface $\Sigma_g$ with arbitrary number of dots can be evaluated as a polynomial in $h$ and $t$, by using $(\bullet \bullet) = h (\bullet) + t(\ )$ inductively and applying the above formulae. 
\end{proposition}

The following operation, called \textit{delooping}, is a generalization of \cite[Lemma 4.1]{BarNatan:2007}, and will be used throughout this paper. First, recall that any preadditive category $\mathcal{A}$ can be turned into an additive category by taking its additive closure. Concretely, the additive closure of $\mathcal{A}$ is given by the category $\Mat(\mathcal{A})$ consisting of objects: formal direct sums of objects of $\mathcal{A}$, and morphisms: matrices with entries morphisms of $\mathcal{A}$ with appropriate endpoints. 
 
\begin{proposition}[Delooping]
\label{prop:delooping}
    The following matrix of morphisms give mutually inverse isomorphisms in $\Mat(\Cob^3_{\bullet/l}(B))$:
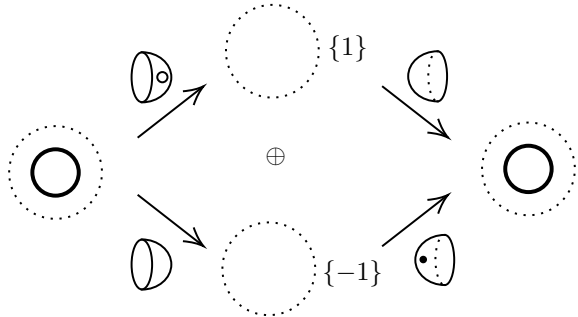
    \begin{center}
        \input{tikzpictures/delooping}
    \end{center}
\end{proposition}

\begin{proof}
    Immediate from the local relations. 
\end{proof}

Hereafter, objects $\varnothing\{1\}, \varnothing\{-1\}$ will be denoted $\dotcircI, \dotcircX$ respectively. 

\subsection{Recovering the $(1+1)$-TQFT}

% Now, consider the case when $B = \underline{0}$. 

The \textit{universal Khovanov homology}, also called the \textit{$U(2)$-equivariant Khovanov homology}, $\Kh_{h, t}$, is given by the Frobenius extension $(R_{h, t}, A_{h, t})$, where $R_{h, t}$ is the graded ring $\ZZ[h, t]$ with $\deg h = -2,\ \deg t = -4$ and $A_{h, t}$ is the graded Frobenius algebra 
\[
     R_{h, t}[X]/(X^2 - hX - t)
\]
determined by the counit $\epsilon$ with $\epsilon(1) = 0,\ \epsilon(X) = 1$ and endowed with the grading $\deg(1) = 0, \deg(X) = -2$. Setting $(h, t) = (0, 0)$ recovers the original Khovanov homology. We shall show that both $R_{h, t}$ and $A_{h, t}$ can be recovered by the \textit{tautological functor} in $\Cob^3_{\bullet/l}(\underline{0})$. 

Hereafter, let $\Hom_{\bullet/l}$ denote the hom-functor in the category $\Cob^3_{\bullet/l}(\underline{0})$. The hom-module
\[
    \Hom_{\bullet/l}(\emptyset, \emptyset)
\]
admits a commutative graded ring structure, where the multiplication is the disjoint union and the unit is the empty cobordism $\emptyset$.

\begin{proposition}
\label{prop:Zht-ring-isom}
    There is a canonical graded ring isomorphism
    \[
        \Hom_{\bullet/l}(\emptyset, \emptyset) \isom R_{h, t}.
    \]
\end{proposition}

\begin{proof}
    An obvious graded ring homomorphism $i$ from the right to the left is given by mapping $h, t$ to the corresponding cobordisms. Furthermore, \Cref{lem:eval-dot-surface} shows that $i$ is surjective. The inverse map is given by the TQFT $\mathcal{F}_{h, t}$ obtained from the Frobenius algebra $A_{h, t}$. Indeed, $\mathcal{F}_{h, t}$ preserves the local relations, and hence induces a map from the left to the right, and we have 
    \[
        \mathcal{F}_{h, t}(i(h)) = \epsilon(X^2) = h
    \]
    and 
    \[
        \mathcal{F}_{h, t}(i(t)) = \epsilon(X^3) - \epsilon(X^2)^2 = t.
    \]
\end{proof}

\begin{proposition}
\label{prop:cob-tautological}
    The delooping isomorphism of \Cref{prop:delooping} gives an isomorphism of graded $R_{h, t}$-modules
    \[
        \Hom_{\bullet/l}(\varnothing, \bigcirc) \isom A_{h, t}\{1\}
    \]
    that maps
    \begin{center}
        \tikzset{every picture/.style={line width=0.75pt}} %set default line width to 0.75pt        

\begin{tikzpicture}[x=0.75pt,y=0.75pt,yscale=-1,xscale=1]
%uncomment if require: \path (0,41); %set diagram left start at 0, and has height of 41

%Shape: Arc [id:dp12726597247618732] 
\draw  [draw opacity=0][dash pattern={on 0.84pt off 2.51pt}] (25.7,33.61) .. controls (23.49,31.96) and (21.86,26.52) .. (21.86,20.07) .. controls (21.86,13.63) and (23.48,8.2) .. (25.7,6.54) -- (27.13,20.07) -- cycle ; \draw  [dash pattern={on 0.84pt off 2.51pt}] (25.7,33.61) .. controls (23.49,31.96) and (21.86,26.52) .. (21.86,20.07) .. controls (21.86,13.63) and (23.48,8.2) .. (25.7,6.54) ;  
%Shape: Arc [id:dp06536334966692847] 
\draw  [draw opacity=0] (26.55,6.1) .. controls (26.74,6.04) and (26.93,6.02) .. (27.13,6.02) .. controls (30.04,6.02) and (32.4,12.31) .. (32.4,20.07) .. controls (32.4,27.84) and (30.04,34.13) .. (27.13,34.13) .. controls (26.93,34.13) and (26.74,34.11) .. (26.55,34.05) -- (27.13,20.07) -- cycle ; \draw   (26.55,6.1) .. controls (26.74,6.04) and (26.93,6.02) .. (27.13,6.02) .. controls (30.04,6.02) and (32.4,12.31) .. (32.4,20.07) .. controls (32.4,27.84) and (30.04,34.13) .. (27.13,34.13) .. controls (26.93,34.13) and (26.74,34.11) .. (26.55,34.05) ;  
%Shape: Arc [id:dp9081360277319049] 
\draw  [draw opacity=0] (27.71,6.02) .. controls (27.51,6.02) and (27.32,6.02) .. (27.13,6.02) .. controls (18,6.02) and (10.61,12.31) .. (10.61,20.07) .. controls (10.61,27.84) and (18,34.13) .. (27.13,34.13) .. controls (27.32,34.13) and (27.51,34.13) .. (27.71,34.12) -- (27.13,20.07) -- cycle ; \draw   (27.71,6.02) .. controls (27.51,6.02) and (27.32,6.02) .. (27.13,6.02) .. controls (18,6.02) and (10.61,12.31) .. (10.61,20.07) .. controls (10.61,27.84) and (18,34.13) .. (27.13,34.13) .. controls (27.32,34.13) and (27.51,34.13) .. (27.71,34.12) ;  

%Shape: Arc [id:dp8270869968566329] 
\draw  [draw opacity=0][dash pattern={on 0.84pt off 2.51pt}] (117.7,34.61) .. controls (115.49,32.96) and (113.86,27.52) .. (113.86,21.07) .. controls (113.86,14.63) and (115.48,9.2) .. (117.7,7.54) -- (119.13,21.07) -- cycle ; \draw  [dash pattern={on 0.84pt off 2.51pt}] (117.7,34.61) .. controls (115.49,32.96) and (113.86,27.52) .. (113.86,21.07) .. controls (113.86,14.63) and (115.48,9.2) .. (117.7,7.54) ;  
%Shape: Arc [id:dp23191640421443382] 
\draw  [draw opacity=0] (118.55,7.1) .. controls (118.74,7.04) and (118.93,7.02) .. (119.13,7.02) .. controls (122.04,7.02) and (124.4,13.31) .. (124.4,21.07) .. controls (124.4,28.84) and (122.04,35.13) .. (119.13,35.13) .. controls (118.93,35.13) and (118.74,35.11) .. (118.55,35.05) -- (119.13,21.07) -- cycle ; \draw   (118.55,7.1) .. controls (118.74,7.04) and (118.93,7.02) .. (119.13,7.02) .. controls (122.04,7.02) and (124.4,13.31) .. (124.4,21.07) .. controls (124.4,28.84) and (122.04,35.13) .. (119.13,35.13) .. controls (118.93,35.13) and (118.74,35.11) .. (118.55,35.05) ;  
%Shape: Arc [id:dp07137268684926101] 
\draw  [draw opacity=0] (119.71,7.02) .. controls (119.51,7.02) and (119.32,7.02) .. (119.13,7.02) .. controls (110,7.02) and (102.61,13.31) .. (102.61,21.07) .. controls (102.61,28.84) and (110,35.13) .. (119.13,35.13) .. controls (119.32,35.13) and (119.51,35.13) .. (119.71,35.12) -- (119.13,21.07) -- cycle ; \draw   (119.71,7.02) .. controls (119.51,7.02) and (119.32,7.02) .. (119.13,7.02) .. controls (110,7.02) and (102.61,13.31) .. (102.61,21.07) .. controls (102.61,28.84) and (110,35.13) .. (119.13,35.13) .. controls (119.32,35.13) and (119.51,35.13) .. (119.71,35.12) ;  
%Shape: Ellipse [id:dp4659238662024515] 
\draw  [fill={rgb, 255:red, 0; green, 0; blue, 0 }  ,fill opacity=1 ] (105.49,20.82) .. controls (105.49,19.96) and (106.19,19.27) .. (107.04,19.27) .. controls (107.9,19.27) and (108.6,19.96) .. (108.6,20.82) .. controls (108.6,21.68) and (107.9,22.37) .. (107.04,22.37) .. controls (106.19,22.37) and (105.49,21.68) .. (105.49,20.82) -- cycle ;

% Text Node
\draw (39,10.4) node [anchor=north west][inner sep=0.75pt]    {$\mapsto \ 1,$};
% Text Node
\draw (130,11.4) node [anchor=north west][inner sep=0.75pt]    {$\mapsto \ X.$};

\end{tikzpicture}
    \end{center}
    Under this identification, the tautological functor
    \[
        \Hom_{\bullet/l}(\emptyset, -)\colon \Cob^3_{\bullet/l}(\underline{0}) \rightarrow {R_{h, t}\Mod}
    \]
    coincides with the (1+1)-TQFT $\mathcal{F}_{h, t}$ obtained from the Frobenius algebra $A_{h, t}$. Here, $R_{h, t}\Mod$ denotes the category of $R_{h, t}$-modules.
\end{proposition}

\begin{remark}
    To be precise, the above functor is considered in the additive closure $\Mat(\Cob^3_{\bullet/l}(\underline{0}))$. 
\end{remark}

\begin{proof}
    Let $\mathcal{T}$ denote the tautological functor. \Cref{prop:Zht-ring-isom} states that $\mathcal{T}(\varnothing) \isom R_{h, t}$. The first statement is immediate from \Cref{prop:delooping}, for we have 
    \begin{align*}
        \mathcal{T}(\bigcirc) 
        &\isom \mathcal{T}(\varnothing)\{1\} \oplus \mathcal{T}(\varnothing)\{-1\} \\
        &\isom R_{h, t}\{1\} \oplus R_{h, t}\{-1\} \\
        &\isom A_{h, t}\{1\}
    \end{align*}
    as graded $R_{h, t}$-modules. Similarly for $k \geq 1$, we have $\mathcal{T}(\bigcirc^{\sqcup k}) \isom (A_{h, t})^{\otimes k}$. This is exactly how the TQFT $\mathcal{F}_{h, t}$ acts on the objects. It remains to check the correspondences for the four elementary cobordisms. 
    
    First, for the merge cobordism, consider the following diagram in $\Mat(\Cob^3_{\bullet/l}(\underline{0}))$: 
    \begin{center}
        \input{tikzpictures/TQFT-m}
    \end{center}
    Here, direct sums are expressed by parentheses, and the vertical arrows are the delooping isomorphisms. The dashed horizontal arrow is a matrix of morphisms in $\Cob^3_{\bullet/l}(\underline{0})$, defined so that the square commutes. Consider the matrix entry, say, corresponding to the morphism $\dotcircX \dotcircX  \rightarrow \dotcircX$. It is the composition 
    \begin{center}
        \input{tikzpictures/TQFT-mXX}
    \end{center}
    which forms a sphere with two dots, namely $h$. Proceeding similarly, one can verify that the matrix $m$ is given by 
    \[
        \begin{pmatrix}
            1 & 0 & 0 & t \\
            0 & 1 & 1 & h
        \end{pmatrix}
    \]
    which is exactly the representation matrix of the multiplication $m$ with respect to the basis $\{1, X\}$. The other three cases can be proved similarly. 
\end{proof}

Hereafter, we use the following symbolic notations for elementary cobordisms corresponding to the operations of the Frobenius algebra $A_{h, t}$:
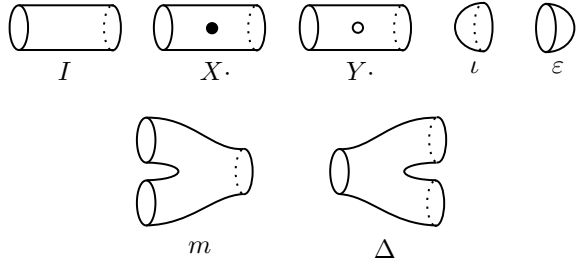
\begin{center}
    \input{tikzpictures/elem-cob-symbols}    
\end{center}
Note that $X$ has 
\[
    X^2 = hX + t,\quad
    \Delta(X) = X \otimes X + t(1 \otimes 1)
\]
whereas $Y$ has
\[
    Y^2 = -hY + t,\quad
    \Delta(Y) = Y \otimes Y + t (1 \otimes 1).
\]
We also use notations such as $\iota_X$ and $\epsilon_Y$ to represent elementary cobordisms with the corresponding dots added. The delooping isomorphism will be expressed as 
\[
\begin{tikzcd}[column sep=5em, ampersand replacement=\&]
    \thickcirc \arrow[shift left]{r}{ \begin{pmatrix} \epsilon_Y \\ \epsilon \end{pmatrix}} \& \begin{pmatrix} \dotcircI \\ \dotcircX \end{pmatrix} \arrow[shift left]{l}{ \begin{pmatrix} \iota & \iota_X \end{pmatrix}}.
    \end{tikzcd}
\]

\begin{remark}
    There is another version of delooping,
    \[
\begin{tikzcd}[column sep=5em, ampersand replacement=\&]
    \thickcirc \arrow[shift left]{r}{ \begin{pmatrix} \epsilon_X \\ \epsilon \end{pmatrix}} \& \begin{pmatrix} \dotcircI \\ \dotcircX \end{pmatrix} \arrow[shift left]{l}{ \begin{pmatrix} \iota & \iota_Y \end{pmatrix}}
    \end{tikzcd}
\]
    obtained by reversing the arrows of \Cref{prop:delooping}. With this identification, the TQFT obtained from the tautological functor becomes $\mathcal{F}_{-h, t}$, which is obtained from the Frobenius algebra 
    \[
        A_{-h, t} = R[Y]/(Y^2 + hY - t).
    \]
    There is a non-degenerate pairing
    \[
        \beta = \epsilon m\colon 
        A_{h, t} \otimes A_{-h, t} \to R
    \]
    and $\{Y, 1\}$ gives a basis of $A_{-h, t}$ dual to $\{1, X\}$ of $A_{h, t}$ with respect to $\beta$. This gives a Frobenius algebra isomorphism
    \[
        A_{-h, t} \isom A_{h, t}^*; \quad 
        1 \mapsto X^*,\quad 
        Y \mapsto 1^*
    \]
    where $A_{h, t}^* = \Hom_R(A_{h, t}, R)$ is the \textit{dual Frobenius algebra} of $A_{h, t}$, whose operations are given by the duals of the operations of $A_{h, t}$.
\end{remark}

\subsection{Khovanov homology for tangles}

\afterpage{
\begin{figure}[t]
    \centering
    \resizebox{.95\textwidth}{!}{
        \input{tikzpictures/ckh-tangle}
    }
    \caption{The formal Khovanov bracket [T]}
    \label{fig:formal-ckh}
\end{figure}
}

Given an oriented tangle diagram $T$ with boundary $\del T = B$, the \textit{formal Khovanov bracket} $[T]$ is defined as a ($\ZZ$-graded) chain complex in the additive category $\Mat(\Cob^3_{/l}(B))$, or equivalently an object in the category
\[
    \Kob(B) := \Kom(\Mat(\Cob^3_{/l}(B))).
\]
which reads, the category of chain complexes in the additive closure of the preadditive category $\Cob^3_{/l}(B)$. 
The construction of $[T]$ is formally identical to that of the ordinary Khovanov chain complex, except that each `chain group' is replaced by a formal direct sum of crossingless  tangle diagrams, and each differential is replaced by a matrix with entries linear combinations of embedded cobordisms. See \Cref{fig:formal-ckh} for an example. 

The complex $[T]$ is also endowed a bigrading. Let $n^+, n^-$ denote the number of positive and negative crossings of $T$ respectively. The \textit{homological grading} of $[T]$ is shifted by $-n^-$, so that the sequence starts from $[T]^{-n^-}$ and ends with $[T]^{n^+}$. Moreover, each object $[T]^k$ of homological grading $k$ will have its quantum grading (see \Cref{def:cob-grad}) shifted by $k + n^+ - n^-$. Thus $[T]$ can be written as
\[
    [T] = \left\{\ [T]^{-n^-}\{n^+ - 2n^-\} \to \cdots \to [T]^{n^+}\{2n^+ - n^-\} \ \right\}
\]
With this bigrading, one can check that the differential $d$ has bidegree $(1, 0)$. 

% The complex $[T]$ can be also regarded as an object of the dotted version
% \[
%     \Kob_\bullet(B) := \Kom(\Mat(\Cob^3_{\bullet/l}(B))).
% \]
Concepts such as chain homotopies and chain homotopy equivalences make sense, and the corresponding homotopy category will be denoted $\Kob_{/h}(B)$. Bar-Natan proves that the complex $[T]$ is a tangle invariant as an object in the homotopy category. 

\begin{theorem}[{\cite[Theorem 1]{BarNatan:2004}}]
\label{thm:formal-ckh-invariance}
    Suppose $T, T'$ are tangle diagrams with $\del T = \del T' = B$ that are related by a single Reidemeister move. Then there is a corresponding chain homotopy equivalence in $\Kob(B)$,
    \[
        f: [T] \xrightarrow{ \htpy } [T']
    \]
    of bidegree $(0, 0)$. 
\end{theorem}

The proof is given by explicitly constructing chain maps in both directions together with explicit chain homotopies and show that they satisfy the desired equations, for each of the three Reidemeister moves. All of the above concepts make sense in the dotted versions, and we will work in the category
\[
    \Kob_\bullet(B) := \Kom(\Mat(\Cob^3_{_\bullet/l}(B))).
\]

We shall extend the tautological functor of \Cref{prop:cob-tautological} to $\Kob_\bullet(\underline{0})$, and show that it naturally gives rise to the universal Khovanov homology functor $\Kh_{h, t}$. To make the statement precise, let us recall the concept of \textit{hom-complexes} in any $R$-linear category $\mathcal{A}$. Given two chain complexes $C, C'$ in $\mathcal{A}$, the hom-set 
\[
    \Hom_\mathcal{A}(C, C')
\]
admits an $R$-chain complex structure, whose $k$-th chain module is given by morphisms of degree $k$ from $C$ to $C'$, and the differential $\delta$ is given by 
\[
    \delta f = d_{C'} \circ f - (-1)^{\deg f} f \circ d_C.
\]
One can see that a $k$-cycle $f$ in the hom-complex is precisely a degree $k$ chain map $f: C \to C'$, and a $k$-boundary $f = \delta h$ is a degree $k$ null-homotopic chain map with degree $k - 1$ null homotopy $h$.

Returning to our case, consider the formal bracket $[\varnothing] \in \Kob_\bullet(\underline{0})$ for the empty link diagram $\varnothing$, which is given by
\[
    [\varnothing]^k = \begin{cases}
        \varnothing & k = 0, \\
        0 & k \neq 0
    \end{cases}
\]
with the trivial differential. The following proposition is immediate from \Cref{prop:cob-tautological}.

\begin{proposition}
\label{prop:kom-tautological}
    The tautological functor
    \[
        \Hom_{\bullet/l}([\emptyset], -)\colon \Kob_\bullet(\underline{0}) \longrightarrow \Kom({R_{h, t}\Mod})
    \]
    coincides with the functor induced from the TQFT $\mathcal{F}_{h, t}$. 
\end{proposition}

Thus for a link diagram $L$, its \textit{universal Khovanov complex} $\CKh_{h, t}(L)$ can be regarded as
\[
    \CKh_{h, t}(L) = \Hom_{\bullet/l}([\emptyset], [L]).
\]
Under this identification, a $k$-cycle $z$ in $\CKh_{h, t}^k(L)$ can be regarded as a degree $k$ chain map 
\[
    z: [\varnothing] \to [L]
\]
and its homology class $[z]$ in $\Kh^k_{h, t}(L)$ as the chain homotopy class of $z$. Furthermore, it automatically follows from \Cref{thm:formal-ckh-invariance} that the chain homotopy type of $\CKh_{h, t}(L)$ is an invariant of the corresponding link. 
This change of view will be essential in extending the  definition of the Lee classes to tangle diagrams in \Cref{sec:reform-s}. 

% \subsection{Functoriality of Khovanov homology}

Now, let $\Diag(B)$ denote the category\footnote{
    In \cite{BarNatan:2004}, the domain category of $[\cdot]$ is denoted $\Cob^4(B)$. Here we used the notation and definition used in \cite{BHL:2019} for clarity and to avoid the genericity argument for the morphisms. 
} 
consisting of objects -- oriented tangle diagrams $T$ with $\del T = B$, and morphisms -- oriented movies between tangle diagrams that fix the boundary $B$. Let $\Diag_{/i}(B)$ denote the quotient category, $\Diag(B)$ modulo the 15 \textit{movie moves} of Cartar and Saito \cite{CS:1993}. In particular when $B = \underline{0}$, the objects of $\Diag_{/i}(\underline{0})$ are oriented link diagrams, and the category is equivalent to the category of oriented links in $\RR^3$ (\cite[Proposition 2.14]{BHL:2019}). 
The formal Khovanov bracket $[\cdot]$ extends to a functor
\[
    [\cdot]: \Diag(B) \to \Kob(B)
\]
by mapping each elementary move (a Reidemeister move or a Morse move) to the corresponding chain map in $\Kob(B)$. Here, the maps for the Reidemeister moves are those constructed in \Cref{thm:formal-ckh-invariance}, and the maps for the Morse moves are given by the obvious cobordisms. The following theorem states the \textit{isotopy invariance} of the functor $[\cdot]$, up to sign.

\begin{theorem}[{\cite[Theorem 4, 5]{BarNatan:2004}}]
\label{thm:formal-ckh-functoriality}
    The formal Khovanov bracket descends to a functor
    \[
        [\cdot]: \Diag_{/i}(B) \to \Kob_{/\pm h}(B).
    \]
    Here, $\Kob_{/\pm h}(B)$ denotes the projectivization of the homotopy category $\Kob_{/h}(B)$. Similar statement holds for the dotted category.
\end{theorem}

In particular when $B = \underline{0}$, 

\begin{proposition}
    The $U(2)$-equivariant Khovanov homology functor $\Kh_{h, t}$ decomposes into a sequence of functors
    \[
        \begin{tikzcd}
        \Diag(\underline{0}) \arrow[r, "{[\cdot]}"] & \Kob(\underline{0}) \arrow[r, "\mathcal{F}_{h, t}"] & \Kom({R_{h, t}\Mod}) \arrow[r, "H"] & {R_{h, t}\Mod}.
        \end{tikzcd}
    \]
    and descends to
    \[
    \begin{tikzcd}
        \Diag_{/i}(\underline{0}) \arrow[r, "{[\cdot]}"] & \Kob_{/\pm h}(\underline{0}) \arrow[r, "\mathcal{F}_{h, t}"] & \Kom_{/\pm h}({R_{h, t}\Mod}) \arrow[r, "H"] & {R_{h, t}\Mod} / \pm.
    \end{tikzcd}
    \]    
    Similar statement holds for the dotted category.
\end{proposition}

% In \cite{Sano:2020-b}, the author proved that the sign indeterminacy can be fixed by adjusting the signs of the cobordism maps so that the signs of the \textit{Lee classes} are fixed. 
% Later in \Cref{thm:adjust-cob-map} we prove that this argument can be generalized to tangles.

\subsection{Reduced Khovanov homology}
\label{subsec:reduced-khovanov}

Here, we briefly explain that the \textit{reduced Khovanov homology}, originally introduced in \cite{Khovanov:2003}, can also be recovered from a certain tautological functor. Consider the dotted category $\Cob^3_\bullet(\underline{2})$, where $\underline{2}$ denotes $\{(0, \pm 1)\} \subset \del D^2$. To distinguish from the unreduced setting, we write  
\begin{center}
    \tikzset{every picture/.style={line width=0.75pt}} %set default line width to 0.75pt        

\begin{tikzpicture}[x=0.75pt,y=0.75pt,yscale=-1,xscale=1]
%uncomment if require: \path (0,47); %set diagram left start at 0, and has height of 47

%Shape: Arc [id:dp6319807885803715] 
\draw  [draw opacity=0] (96.94,23.75) .. controls (96.29,26) and (89.45,27.76) .. (81.11,27.76) .. controls (72.35,27.76) and (65.24,25.81) .. (65.24,23.41) .. controls (65.24,23.26) and (65.27,23.11) .. (65.32,22.97) -- (81.11,23.41) -- cycle ; \draw  [color={rgb, 255:red, 128; green, 128; blue, 128 }  ,draw opacity=1 ] (96.94,23.75) .. controls (96.29,26) and (89.45,27.76) .. (81.11,27.76) .. controls (72.35,27.76) and (65.24,25.81) .. (65.24,23.41) .. controls (65.24,23.26) and (65.27,23.11) .. (65.32,22.97) ;  
%Shape: Arc [id:dp056615061649111675] 
\draw  [draw opacity=0][dash pattern={on 0.84pt off 2.51pt}] (65.63,22.33) .. controls (66.9,20.27) and (73.43,18.7) .. (81.28,18.7) .. controls (89.93,18.7) and (96.96,20.6) .. (97.15,22.96) -- (81.28,23.06) -- cycle ; \draw  [color={rgb, 255:red, 128; green, 128; blue, 128 }  ,draw opacity=1 ][dash pattern={on 0.84pt off 2.51pt}] (65.63,22.33) .. controls (66.9,20.27) and (73.43,18.7) .. (81.28,18.7) .. controls (89.93,18.7) and (96.96,20.6) .. (97.15,22.96) ;  
%Shape: Circle [id:dp39111359424094494] 
\draw   (65.12,23.41) .. controls (65.12,14.57) and (72.28,7.41) .. (81.11,7.41) .. controls (89.95,7.41) and (97.11,14.57) .. (97.11,23.41) .. controls (97.11,32.24) and (89.95,39.4) .. (81.11,39.4) .. controls (72.28,39.4) and (65.12,32.24) .. (65.12,23.41) -- cycle ;
%Shape: Ellipse [id:dp9734470792692242] 
\draw  [fill={rgb, 255:red, 0; green, 0; blue, 0 }  ,fill opacity=1 ] (74.63,14.68) .. controls (74.63,13.55) and (75.55,12.63) .. (76.68,12.63) .. controls (77.82,12.63) and (78.73,13.55) .. (78.73,14.68) .. controls (78.73,15.81) and (77.82,16.73) .. (76.68,16.73) .. controls (75.55,16.73) and (74.63,15.81) .. (74.63,14.68) -- cycle ;
%Shape: Ellipse [id:dp1260497035726349] 
\draw  [fill={rgb, 255:red, 0; green, 0; blue, 0 }  ,fill opacity=1 ] (84.23,14.68) .. controls (84.23,13.55) and (85.15,12.63) .. (86.28,12.63) .. controls (87.42,12.63) and (88.33,13.55) .. (88.33,14.68) .. controls (88.33,15.81) and (87.42,16.73) .. (86.28,16.73) .. controls (85.15,16.73) and (84.23,15.81) .. (84.23,14.68) -- cycle ;

% Text Node
\draw (14,14.61) node [anchor=north west][inner sep=0.75pt]    {$H\ =\ $};

\end{tikzpicture}
\end{center}
instead of $h$. Every object $X$ of $\Cob^3_\bullet(\underline{2})$ contains a unique arc $e$, called the \textit{special arc} of $X$, such that $\del e = \underline{2}$. Furthermore, every cobordism $S$ between objects $X, Y$ contains a unique component $S$, called the \textit{special component} of $S$, that contains the both of the special arcs of $X$ and $Y$ in its boundary. In addition to the three local relations, we impose two more relations (R1) and (R2), given by 
\begin{center}
    \tikzset{every picture/.style={line width=0.75pt}} %set default line width to 0.75pt        

\begin{tikzpicture}[x=0.75pt,y=0.75pt,yscale=-.75,xscale=.75]
%uncomment if require: \path (0,82); %set diagram left start at 0, and has height of 82

%Shape: Rectangle [id:dp8333514108188844] 
\draw  [dash pattern={on 4.5pt off 4.5pt}] (55.93,12.56) -- (101.56,12.56) -- (101.56,56.49) -- (55.93,56.49) -- cycle ;
%Shape: Rectangle [id:dp19497325531006915] 
\draw  [dash pattern={on 4.5pt off 4.5pt}] (165.41,13.06) -- (212.08,13.06) -- (212.08,57.99) -- (165.41,57.99) -- cycle ;
%Shape: Circle [id:dp6574774305818769] 
\draw  [fill={rgb, 255:red, 0; green, 0; blue, 0 }  ,fill opacity=1 ] (75.54,34.52) .. controls (75.54,32.75) and (76.97,31.32) .. (78.74,31.32) .. controls (80.51,31.32) and (81.95,32.75) .. (81.95,34.52) .. controls (81.95,36.29) and (80.51,37.73) .. (78.74,37.73) .. controls (76.97,37.73) and (75.54,36.29) .. (75.54,34.52) -- cycle ;
%Shape: Arc [id:dp06666842893981129] 
\draw  [draw opacity=0] (398.56,20.31) .. controls (398,22.25) and (392.11,23.77) .. (384.92,23.77) .. controls (377.36,23.77) and (371.23,22.09) .. (371.23,20.01) .. controls (371.23,19.88) and (371.26,19.76) .. (371.3,19.64) -- (384.92,20.01) -- cycle ; \draw  [color={rgb, 255:red, 128; green, 128; blue, 128 }  ,draw opacity=1 ] (398.56,20.31) .. controls (398,22.25) and (392.11,23.77) .. (384.92,23.77) .. controls (377.36,23.77) and (371.23,22.09) .. (371.23,20.01) .. controls (371.23,19.88) and (371.26,19.76) .. (371.3,19.64) ;  
%Shape: Arc [id:dp6303245775179046] 
\draw  [draw opacity=0][dash pattern={on 0.84pt off 2.51pt}] (371.57,19.08) .. controls (372.67,17.31) and (378.29,15.95) .. (385.06,15.95) .. controls (392.51,15.95) and (398.57,17.59) .. (398.75,19.62) -- (385.06,19.71) -- cycle ; \draw  [color={rgb, 255:red, 128; green, 128; blue, 128 }  ,draw opacity=1 ][dash pattern={on 0.84pt off 2.51pt}] (371.57,19.08) .. controls (372.67,17.31) and (378.29,15.95) .. (385.06,15.95) .. controls (392.51,15.95) and (398.57,17.59) .. (398.75,19.62) ;  
%Shape: Ellipse [id:dp647174213061882] 
\draw   (371.13,20.01) .. controls (371.13,12.4) and (377.3,6.22) .. (384.92,6.22) .. controls (392.53,6.22) and (398.71,12.4) .. (398.71,20.01) .. controls (398.71,27.63) and (392.53,33.8) .. (384.92,33.8) .. controls (377.3,33.8) and (371.13,27.63) .. (371.13,20.01) -- cycle ;
%Shape: Ellipse [id:dp05342748249245888] 
\draw  [fill={rgb, 255:red, 0; green, 0; blue, 0 }  ,fill opacity=1 ] (379.33,12.49) .. controls (379.33,11.51) and (380.12,10.72) .. (381.1,10.72) .. controls (382.07,10.72) and (382.87,11.51) .. (382.87,12.49) .. controls (382.87,13.47) and (382.07,14.26) .. (381.1,14.26) .. controls (380.12,14.26) and (379.33,13.47) .. (379.33,12.49) -- cycle ;
%Shape: Ellipse [id:dp4584754734512062] 
\draw  [fill={rgb, 255:red, 0; green, 0; blue, 0 }  ,fill opacity=1 ] (387.61,12.49) .. controls (387.61,11.51) and (388.4,10.72) .. (389.37,10.72) .. controls (390.35,10.72) and (391.14,11.51) .. (391.14,12.49) .. controls (391.14,13.47) and (390.35,14.26) .. (389.37,14.26) .. controls (388.4,14.26) and (387.61,13.47) .. (387.61,12.49) -- cycle ;

%Shape: Arc [id:dp022650617392291683] 
\draw  [draw opacity=0] (398.56,57.11) .. controls (398,59.05) and (392.11,60.57) .. (384.92,60.57) .. controls (377.36,60.57) and (371.23,58.89) .. (371.23,56.81) .. controls (371.23,56.68) and (371.26,56.56) .. (371.3,56.44) -- (384.92,56.81) -- cycle ; \draw  [color={rgb, 255:red, 128; green, 128; blue, 128 }  ,draw opacity=1 ] (398.56,57.11) .. controls (398,59.05) and (392.11,60.57) .. (384.92,60.57) .. controls (377.36,60.57) and (371.23,58.89) .. (371.23,56.81) .. controls (371.23,56.68) and (371.26,56.56) .. (371.3,56.44) ;  
%Shape: Arc [id:dp0533700800002167] 
\draw  [draw opacity=0][dash pattern={on 0.84pt off 2.51pt}] (371.57,55.88) .. controls (372.67,54.11) and (378.29,52.75) .. (385.06,52.75) .. controls (392.51,52.75) and (398.57,54.39) .. (398.75,56.42) -- (385.06,56.51) -- cycle ; \draw  [color={rgb, 255:red, 128; green, 128; blue, 128 }  ,draw opacity=1 ][dash pattern={on 0.84pt off 2.51pt}] (371.57,55.88) .. controls (372.67,54.11) and (378.29,52.75) .. (385.06,52.75) .. controls (392.51,52.75) and (398.57,54.39) .. (398.75,56.42) ;  
%Shape: Ellipse [id:dp12728445276652467] 
\draw   (371.13,56.81) .. controls (371.13,49.2) and (377.3,43.02) .. (384.92,43.02) .. controls (392.53,43.02) and (398.71,49.2) .. (398.71,56.81) .. controls (398.71,64.43) and (392.53,70.6) .. (384.92,70.6) .. controls (377.3,70.6) and (371.13,64.43) .. (371.13,56.81) -- cycle ;
%Shape: Ellipse [id:dp7627968194286492] 
\draw  [fill={rgb, 255:red, 0; green, 0; blue, 0 }  ,fill opacity=1 ] (379.33,49.29) .. controls (379.33,48.31) and (380.12,47.52) .. (381.1,47.52) .. controls (382.07,47.52) and (382.87,48.31) .. (382.87,49.29) .. controls (382.87,50.27) and (382.07,51.06) .. (381.1,51.06) .. controls (380.12,51.06) and (379.33,50.27) .. (379.33,49.29) -- cycle ;
%Shape: Ellipse [id:dp2700126717227459] 
\draw  [fill={rgb, 255:red, 0; green, 0; blue, 0 }  ,fill opacity=1 ] (387.61,49.29) .. controls (387.61,48.31) and (388.4,47.52) .. (389.37,47.52) .. controls (390.35,47.52) and (391.14,48.31) .. (391.14,49.29) .. controls (391.14,50.27) and (390.35,51.06) .. (389.37,51.06) .. controls (388.4,51.06) and (387.61,50.27) .. (387.61,49.29) -- cycle ;

%Shape: Arc [id:dp00682089804826469] 
\draw  [draw opacity=0] (334.94,37.76) .. controls (334.29,40) and (327.45,41.77) .. (319.11,41.77) .. controls (310.35,41.77) and (303.24,39.82) .. (303.24,37.41) .. controls (303.24,37.26) and (303.27,37.12) .. (303.32,36.98) -- (319.11,37.41) -- cycle ; \draw  [color={rgb, 255:red, 128; green, 128; blue, 128 }  ,draw opacity=1 ] (334.94,37.76) .. controls (334.29,40) and (327.45,41.77) .. (319.11,41.77) .. controls (310.35,41.77) and (303.24,39.82) .. (303.24,37.41) .. controls (303.24,37.26) and (303.27,37.12) .. (303.32,36.98) ;  
%Shape: Arc [id:dp054052321872980036] 
\draw  [draw opacity=0][dash pattern={on 0.84pt off 2.51pt}] (303.63,36.33) .. controls (304.9,34.27) and (311.43,32.71) .. (319.28,32.71) .. controls (327.93,32.71) and (334.96,34.6) .. (335.15,36.96) -- (319.28,37.06) -- cycle ; \draw  [color={rgb, 255:red, 128; green, 128; blue, 128 }  ,draw opacity=1 ][dash pattern={on 0.84pt off 2.51pt}] (303.63,36.33) .. controls (304.9,34.27) and (311.43,32.71) .. (319.28,32.71) .. controls (327.93,32.71) and (334.96,34.6) .. (335.15,36.96) ;  
%Shape: Circle [id:dp730969299255903] 
\draw   (303.12,37.41) .. controls (303.12,28.58) and (310.28,21.42) .. (319.11,21.42) .. controls (327.95,21.42) and (335.11,28.58) .. (335.11,37.41) .. controls (335.11,46.25) and (327.95,53.41) .. (319.11,53.41) .. controls (310.28,53.41) and (303.12,46.25) .. (303.12,37.41) -- cycle ;
%Shape: Ellipse [id:dp9507107202889884] 
\draw  [fill={rgb, 255:red, 0; green, 0; blue, 0 }  ,fill opacity=1 ] (309.63,28.49) .. controls (309.63,27.35) and (310.55,26.44) .. (311.68,26.44) .. controls (312.82,26.44) and (313.73,27.35) .. (313.73,28.49) .. controls (313.73,29.62) and (312.82,30.54) .. (311.68,30.54) .. controls (310.55,30.54) and (309.63,29.62) .. (309.63,28.49) -- cycle ;
%Shape: Ellipse [id:dp17038327512066886] 
\draw  [fill={rgb, 255:red, 0; green, 0; blue, 0 }  ,fill opacity=1 ] (317.23,28.49) .. controls (317.23,27.35) and (318.15,26.44) .. (319.28,26.44) .. controls (320.42,26.44) and (321.33,27.35) .. (321.33,28.49) .. controls (321.33,29.62) and (320.42,30.54) .. (319.28,30.54) .. controls (318.15,30.54) and (317.23,29.62) .. (317.23,28.49) -- cycle ;
%Shape: Ellipse [id:dp2810743072438544] 
\draw  [fill={rgb, 255:red, 0; green, 0; blue, 0 }  ,fill opacity=1 ] (324.83,28.49) .. controls (324.83,27.35) and (325.75,26.44) .. (326.88,26.44) .. controls (328.02,26.44) and (328.93,27.35) .. (328.93,28.49) .. controls (328.93,29.62) and (328.02,30.54) .. (326.88,30.54) .. controls (325.75,30.54) and (324.83,29.62) .. (324.83,28.49) -- cycle ;

% Text Node
\draw (113,26.32) node [anchor=north west][inner sep=0.75pt]    {$=$};
% Text Node
\draw (61,16) node [anchor=north west][inner sep=0.75pt]    {$*$};
% Text Node
\draw (169.41,16) node [anchor=north west][inner sep=0.75pt]    {$*$};
% Text Node
\draw (342.5,27.9) node [anchor=north west][inner sep=0.75pt]    {$=$};
% Text Node
\draw (11,26) node [anchor=north west][inner sep=0.75pt]   [align=left] {(R1)};
% Text Node
\draw (139,28.24) node [anchor=north west][inner sep=0.75pt]    {$H$};
% Text Node
\draw (252,24) node [anchor=north west][inner sep=0.75pt]   [align=left] {(R2)};

\end{tikzpicture}
\end{center}
In the picture of (R1), the sheet with the asterisk $*$ denotes the special component. (R2) is equivalent to setting $t = 0$. For now, let $\widetilde{\Cob}{}^3_{\bullet/l}(\underline{2})$ denote the corresponding quotient category.

\begin{remark}
    As explained in \cite[Section 5]{Naot:2006}, in the reduced setting, the dot $\bullet$ can be given a geometric interpretation as follows. Given a dotted cobordism $S$, replace each dot on $S$ with one end of a tube whose other end is connected to the special component of $S$. With this interpretation, $H$ turns into a handle attached to the special component:
    \begin{center}
        \tikzset{every picture/.style={line width=0.75pt}} %set default line width to 0.75pt        

\begin{tikzpicture}[x=0.75pt,y=0.75pt,yscale=-.8,xscale=.8]
%uncomment if require: \path (0,72); %set diagram left start at 0, and has height of 72

%Shape: Rectangle [id:dp6600111918616832] 
\draw  [dash pattern={on 4.5pt off 4.5pt}] (65.41,14.06) -- (112.08,14.06) -- (112.08,58.99) -- (65.41,58.99) -- cycle ;
%Shape: Arc [id:dp5168701500807074] 
\draw  [draw opacity=0][fill={rgb, 255:red, 255; green, 255; blue, 255 }  ,fill opacity=1 ] (88.03,49.02) .. controls (89.73,49.13) and (91.47,49.19) .. (93.24,49.19) .. controls (111.89,49.19) and (127,42.69) .. (127,34.68) .. controls (127,26.67) and (111.89,20.18) .. (93.24,20.18) .. controls (91.47,20.18) and (89.73,20.24) .. (88.03,20.35) -- (93.24,34.68) -- cycle ; \draw   (88.03,49.02) .. controls (89.73,49.13) and (91.47,49.19) .. (93.24,49.19) .. controls (111.89,49.19) and (127,42.69) .. (127,34.68) .. controls (127,26.67) and (111.89,20.18) .. (93.24,20.18) .. controls (91.47,20.18) and (89.73,20.24) .. (88.03,20.35) ;  
%Shape: Arc [id:dp07972852627583737] 
\draw  [draw opacity=0] (86.26,43.47) .. controls (87.31,43.5) and (88.37,43.52) .. (89.45,43.52) .. controls (105.78,43.52) and (119.01,39.56) .. (119.01,34.68) .. controls (119.01,29.8) and (105.78,25.85) .. (89.45,25.85) .. controls (88.37,25.85) and (87.31,25.86) .. (86.26,25.9) -- (89.45,34.68) -- cycle ; \draw   (86.26,43.47) .. controls (87.31,43.5) and (88.37,43.52) .. (89.45,43.52) .. controls (105.78,43.52) and (119.01,39.56) .. (119.01,34.68) .. controls (119.01,29.8) and (105.78,25.85) .. (89.45,25.85) .. controls (88.37,25.85) and (87.31,25.86) .. (86.26,25.9) ;  

% Text Node
\draw (67.41,17.46) node [anchor=north west][inner sep=0.75pt]    {$*$};
% Text Node
\draw (14,27.24) node [anchor=north west][inner sep=0.75pt]    {$H = $};

\end{tikzpicture}
    \end{center}
    Furthermore, the relation (S$_\bullet$) becomes trivial, (NC) becomes a special case of (4Tu), and the two additional relations (R1), (R2) also become trivial.     
\end{remark}

Let $e$ denote the object consisting only of the special arc. The hom-set
\[
    \Hom_{\bullet/l}(e, e)
\]
admits a commutative ring structure, where the multiplication is given by horizontally concatenating the special components while taking the disjoint union of the remaining closed components, and the unit is given by the product $e \times I$. Analogous to \Cref{prop:Zht-ring-isom}, there is a graded ring isomorphism
\[
    \Hom_{\bullet/l}(e, e) \isom R_H
\]
where $R_H$ is the graded polynomial ring $R_H = \ZZ[H]$ with $\deg H = -2$. Furthermore, analogous to \Cref{prop:cob-tautological}, the tautological functor
\[
    \Hom_{\bullet/l}(e, -)\colon \widetilde{\Cob}{}^3_{\bullet/l}(\underline{2}) \rightarrow {R_H}\Mod
\]
coincides with the non-closed $(1+1)$-TQFT $\mathcal{F}'_H$ obtained from the Frobenius algebra
\[
    A_H = R[X]/(X(X - H)),
\]
with additional correspondences mapping the special arc $e$ to the the rank $1$ $A_H$-submodule
\[
    R'_H = R_H \langle X \rangle \{2\}
\]
and the merge and the split cobordisms involving the special arc to 
\begin{align*}
    m': R'_H \otimes A_H \to R'_H; \quad
    X \otimes 1 \mapsto X,\quad 
    X \otimes X \mapsto X \otimes H
\end{align*}
and
\begin{align*}
    \Delta': R'_H \to R'_H \otimes A_H; \quad
    X \mapsto X \otimes X.
\end{align*}
Next, let $\widetilde{\Kob}(\underline{2})$ denote the category
\[
     \Kom(\Mat(\widetilde{\Cob}{}^3_{\bullet/l}(\underline{2}))).
\]
Analogous to \Cref{prop:kom-tautological}, the above tautological functor induces a functor 
\[
    \Hom_{\bullet/l}([e], -)\colon \widetilde{\Kob}(\underline{2}) \longrightarrow \Kom({R_H}\Mod)
\]
that coincides with the functor induced from $\mathcal{F}'_H$. 

Given a pointed link diagram $L$, let $L^\times$ denote the  $2$-end tangle diagram obtained by cutting the strand at the base point of $L$. Then the above tautological functor applied to the complex $[L^\times]$ gives the \textit{reduced $U(1)$-equivariant Khovanov complex}, also called the \textit{reduced Bar-Natan complex} of $L$. Setting $H = 0$ recovers the original reduced Khovanov homology. Combining these observations with \Cref{thm:formal-ckh-functoriality}, we obtain,

\begin{proposition}
    The reduced $U(1)$-Khovanov homology functor $\widetilde{\Kh}_H$ decomposes into a sequence of functors
    \[
        \begin{tikzcd}
        \Diag(\underline{2}) \arrow[r, "{[\cdot]}"] & \widetilde{\Kob}(\underline{2}) \arrow[r, "\mathcal{F}'_H"] & \Kom({R_H}\Mod) \arrow[r, "H^*"] & {R_H}\Mod
        \end{tikzcd}
    \]
    and descends to
    \[
    \begin{tikzcd}
        \Diag_{/i}(\underline{2}) \arrow[r, "{[\cdot]}"] & \widetilde{\Kob}_{/\pm h}(\underline{2}) \arrow[r, "\mathcal{F}'_H"] & \Kom_{/\pm h}({R_H}\Mod) \arrow[r, "H^*"] & {R_H}\Mod/\pm.
    \end{tikzcd}    
    \]
\end{proposition}

The unreduced and the reduced theories can be related as follows. Consider the functor
\[
    I\colon \Cob^3(\underline{0}) \to \Cob^3(\underline{2})
\]
that inserts the special arc $e$ to the objects and the special component $e \times I$ to the morphisms (such functor is obtained from a specific \textit{planar arc diagram}, as explained in \Cref{subsec:planar-algebra}). 

\begin{proposition}
    The functor $I$ induces a full and faithful functor 
    \[
        I: \Cob^3_{\bullet/l, {t = 0}}(\underline{0}) \longrightarrow \widetilde{\Cob}{}^3_{\bullet/l}(\underline{2}),
    \]
    and the following diagram commutes
    \[
    \begin{tikzcd}[row sep=3em, column sep=4em]
\Cob^3_{\bullet/l, t = 0}(\underline{0}) \arrow[r, "I"] \arrow[d, "\mathcal{F}_{h, 0}"'] & \widetilde{\Cob}{}^3_{\bullet/l}(\underline{2}) \arrow[d, "\mathcal{F}'_H"] \\
{{R_h}\Mod} \arrow[r, "h = H"'] & {{R_H}\Mod}.
    \end{tikzcd}
    \]
\end{proposition}

\begin{remark}
    As explained in \cite{Naot:2006} and in more detail in \cite{ILM:2021}, we may also recover the unreduced complex $\CKh_{h, t}$ from the reduced complex $\widetilde{\CKh}_H$ as follows. Let $H$ act on the Frobenius algebra $A_{h, t}$ as 
    \[
        H = 2X - h
    \]
    and regard $A_{h, t}$ as an $R_H$-module. Then for a $2$-end tangle diagram $T$, there is an isomorphism
    \[
        \widetilde{\CKh}_H(T) \otimes_{\ZZ[H]} A_{h, t} \isom \CKh_{h, t}(\bar{T})
    \]
    where $\bar{T}$ is the closure of $T$. 
\end{remark}

\subsection{Planar algebra structure}
\label{subsec:planar-algebra}

The advantage of generalizing Khovanov homology to tangles is that it enables \textit{local treatments} of link diagrams, which is justified by the fact that the above defined categories form \textit{planar algebras}, and the formal Khovanov bracket $[\cdot]$ gives a \textit{planar algebra morphism}. Here we review the basic concepts. See \cite[Section 5]{BarNatan:2004} or \cite[Appendix A.4, A.5]{CMW:2009} for details. 
\begin{definition}    
    A \textit{$d$-input planar arc diagram} $D$ is a disk  with $d$ smaller disks $D_1, ..., D_d$ removed, together with a collection of disjoint properly embedded oriented arcs and circles that has boundary in $\del D$. The outer disk is called the \textit{output disk}, and the inner disks are called the \textit{input disks}.
\end{definition}

\begin{figure}[t]
    \centering
    \tikzset{every picture/.style={line width=0.75pt}} %set default line width to 0.75pt        

\begin{tikzpicture}[x=0.75pt,y=0.75pt,yscale=-.9,xscale=.9]
%uncomment if require: \path (0,151); %set diagram left start at 0, and has height of 151

%Shape: Circle [id:dp2777223627013291] 
\draw  [fill={rgb, 255:red, 241; green, 241; blue, 241 }  ,fill opacity=1 ][dash pattern={on 4.5pt off 4.5pt}] (13.38,70.46) .. controls (13.38,37.4) and (40.18,10.6) .. (73.24,10.6) .. controls (106.31,10.6) and (133.11,37.4) .. (133.11,70.46) .. controls (133.11,103.52) and (106.31,130.33) .. (73.24,130.33) .. controls (40.18,130.33) and (13.38,103.52) .. (13.38,70.46) -- cycle ;
%Shape: Circle [id:dp6642696064348876] 
\draw  [fill={rgb, 255:red, 255; green, 255; blue, 255 }  ,fill opacity=1 ][dash pattern={on 4.5pt off 4.5pt}] (28.91,70.46) .. controls (28.91,60.82) and (36.72,53) .. (46.37,53) .. controls (56.01,53) and (63.83,60.82) .. (63.83,70.46) .. controls (63.83,80.11) and (56.01,87.92) .. (46.37,87.92) .. controls (36.72,87.92) and (28.91,80.11) .. (28.91,70.46) -- cycle ;
%Shape: Ellipse [id:dp4823478759253361] 
\draw  [fill={rgb, 255:red, 255; green, 255; blue, 255 }  ,fill opacity=1 ][dash pattern={on 4.5pt off 4.5pt}] (84.1,70.46) .. controls (84.1,62.7) and (90.39,56.41) .. (98.15,56.41) .. controls (105.91,56.41) and (112.2,62.7) .. (112.2,70.46) .. controls (112.2,78.22) and (105.91,84.51) .. (98.15,84.51) .. controls (90.39,84.51) and (84.1,78.22) .. (84.1,70.46) -- cycle ;
%Shape: Circle [id:dp28357815974161094] 
\draw  [fill={rgb, 255:red, 0; green, 0; blue, 0 }  ,fill opacity=1 ] (50.76,14.14) .. controls (50.76,12.87) and (51.78,11.85) .. (53.05,11.85) .. controls (54.32,11.85) and (55.34,12.87) .. (55.34,14.14) .. controls (55.34,15.41) and (54.32,16.43) .. (53.05,16.43) .. controls (51.78,16.43) and (50.76,15.41) .. (50.76,14.14) -- cycle ;
%Shape: Circle [id:dp06513125218762328] 
\draw  [fill={rgb, 255:red, 0; green, 0; blue, 0 }  ,fill opacity=1 ] (92.6,14.43) .. controls (92.6,13.16) and (93.63,12.13) .. (94.89,12.13) .. controls (96.16,12.13) and (97.18,13.16) .. (97.18,14.43) .. controls (97.18,15.69) and (96.16,16.72) .. (94.89,16.72) .. controls (93.63,16.72) and (92.6,15.69) .. (92.6,14.43) -- cycle ;
%Shape: Ellipse [id:dp7962160929963763] 
\draw  [fill={rgb, 255:red, 0; green, 0; blue, 0 }  ,fill opacity=1 ] (130.81,70.46) .. controls (130.81,69.2) and (131.84,68.17) .. (133.11,68.17) .. controls (134.37,68.17) and (135.4,69.2) .. (135.4,70.46) .. controls (135.4,71.73) and (134.37,72.76) .. (133.11,72.76) .. controls (131.84,72.76) and (130.81,71.73) .. (130.81,70.46) -- cycle ;
%Shape: Ellipse [id:dp26734275128823015] 
\draw  [fill={rgb, 255:red, 0; green, 0; blue, 0 }  ,fill opacity=1 ] (107.5,117.41) .. controls (107.5,116.14) and (108.53,115.12) .. (109.79,115.12) .. controls (111.06,115.12) and (112.09,116.14) .. (112.09,117.41) .. controls (112.09,118.68) and (111.06,119.7) .. (109.79,119.7) .. controls (108.53,119.7) and (107.5,118.68) .. (107.5,117.41) -- cycle ;
%Shape: Ellipse [id:dp041553542179075365] 
\draw  [fill={rgb, 255:red, 0; green, 0; blue, 0 }  ,fill opacity=1 ] (11.21,70.46) .. controls (11.21,69.2) and (12.24,68.17) .. (13.5,68.17) .. controls (14.77,68.17) and (15.8,69.2) .. (15.8,70.46) .. controls (15.8,71.73) and (14.77,72.76) .. (13.5,72.76) .. controls (12.24,72.76) and (11.21,71.73) .. (11.21,70.46) -- cycle ;
%Shape: Ellipse [id:dp31003496150667154] 
\draw  [fill={rgb, 255:red, 0; green, 0; blue, 0 }  ,fill opacity=1 ] (34.71,117.31) .. controls (34.71,116.04) and (35.74,115.02) .. (37,115.02) .. controls (38.27,115.02) and (39.3,116.04) .. (39.3,117.31) .. controls (39.3,118.57) and (38.27,119.6) .. (37,119.6) .. controls (35.74,119.6) and (34.71,118.57) .. (34.71,117.31) -- cycle ;
%Shape: Circle [id:dp6419972624056253] 
\draw  [fill={rgb, 255:red, 0; green, 0; blue, 0 }  ,fill opacity=1 ] (32.3,56.54) .. controls (32.3,55.28) and (33.33,54.25) .. (34.59,54.25) .. controls (35.86,54.25) and (36.89,55.28) .. (36.89,56.54) .. controls (36.89,57.81) and (35.86,58.84) .. (34.59,58.84) .. controls (33.33,58.84) and (32.3,57.81) .. (32.3,56.54) -- cycle ;
%Shape: Ellipse [id:dp48852865023244907] 
\draw  [fill={rgb, 255:red, 0; green, 0; blue, 0 }  ,fill opacity=1 ] (53.75,56.04) .. controls (53.75,54.78) and (54.78,53.75) .. (56.04,53.75) .. controls (57.31,53.75) and (58.34,54.78) .. (58.34,56.04) .. controls (58.34,57.31) and (57.31,58.34) .. (56.04,58.34) .. controls (54.78,58.34) and (53.75,57.31) .. (53.75,56.04) -- cycle ;
%Shape: Ellipse [id:dp4706102417257356] 
\draw  [fill={rgb, 255:red, 0; green, 0; blue, 0 }  ,fill opacity=1 ] (33.3,83.48) .. controls (33.3,82.21) and (34.32,81.19) .. (35.59,81.19) .. controls (36.86,81.19) and (37.88,82.21) .. (37.88,83.48) .. controls (37.88,84.75) and (36.86,85.77) .. (35.59,85.77) .. controls (34.32,85.77) and (33.3,84.75) .. (33.3,83.48) -- cycle ;
%Shape: Ellipse [id:dp4330890106797717] 
\draw  [fill={rgb, 255:red, 0; green, 0; blue, 0 }  ,fill opacity=1 ] (54.75,82.98) .. controls (54.75,81.72) and (55.78,80.69) .. (57.04,80.69) .. controls (58.31,80.69) and (59.33,81.72) .. (59.33,82.98) .. controls (59.33,84.25) and (58.31,85.27) .. (57.04,85.27) .. controls (55.78,85.27) and (54.75,84.25) .. (54.75,82.98) -- cycle ;
%Shape: Ellipse [id:dp8174842978112602] 
\draw  [fill={rgb, 255:red, 0; green, 0; blue, 0 }  ,fill opacity=1 ] (95.46,56.41) .. controls (95.46,55.15) and (96.48,54.12) .. (97.75,54.12) .. controls (99.02,54.12) and (100.04,55.15) .. (100.04,56.41) .. controls (100.04,57.68) and (99.02,58.71) .. (97.75,58.71) .. controls (96.48,58.71) and (95.46,57.68) .. (95.46,56.41) -- cycle ;
%Shape: Circle [id:dp333533488935827] 
\draw  [fill={rgb, 255:red, 0; green, 0; blue, 0 }  ,fill opacity=1 ] (95.46,84.51) .. controls (95.46,83.25) and (96.48,82.22) .. (97.75,82.22) .. controls (99.02,82.22) and (100.04,83.25) .. (100.04,84.51) .. controls (100.04,85.78) and (99.02,86.81) .. (97.75,86.81) .. controls (96.48,86.81) and (95.46,85.78) .. (95.46,84.51) -- cycle ;
%Curve Lines [id:da6986998455454837] 
\draw    (56.04,56.04) .. controls (69.31,43.28) and (90.57,48.43) .. (97.75,56.41) ;
%Curve Lines [id:da31798456433788824] 
\draw    (59.33,82.98) .. controls (70.81,90.17) and (86.28,92.25) .. (97.75,84.51) ;
%Curve Lines [id:da5014371436800616] 
\draw    (34.59,56.54) .. controls (36.39,34.79) and (50.36,34.5) .. (53.05,14.14) ;
%Curve Lines [id:da32226508052763947] 
\draw    (130.81,70.46) .. controls (126.46,71.09) and (92.26,25.32) .. (94.89,14.43) ;
%Curve Lines [id:da4193420366926577] 
\draw    (107.5,117.41) .. controls (102.24,103.64) and (48.86,101.64) .. (37,117.31) ;
%Curve Lines [id:da5241731791295293] 
\draw    (13.5,70.46) .. controls (21.92,68.72) and (25.91,90.67) .. (35.59,85.77) ;

% Text Node
\draw (98.15,70.46) node  [font=\scriptsize]  {$D_{2}$};
% Text Node
\draw (46.37,70.46) node  [font=\scriptsize]  {$D_{1}$};

\end{tikzpicture}
    \caption{A $2$-input planar arc diagram.}
    \label{fig:planar-diagram}
\end{figure}

\begin{definition}
\label{def:planar-alg}
    A \textit{planar algebra} $\mathcal{P}$ is a collection of sets $\{ \mathcal{P}(B) \}_B$ indexed by sets of boundary points $B \subset \del D^2$, together with operations $\{ \mathcal{O}(D) \}_D$ indexed by planar arc diagrams $D$, such that for each $d$-input arc diagram $D$ we have
    \[
        \mathcal{O}(D): \mathcal{P}(B_1) \times \cdots \times \mathcal{P}(B_d) \longrightarrow \mathcal{P}(B),
    \]
    and satisfies the \textit{identity rule} for each $B$,
    \[
        \mathcal{O}(I_B) = \id_{\mathcal{P}(B)}
    \]
    and the \textit{associativity rule} in a natural way. Hereafter, $\mathcal{O}(D)$ is simply denoted $D$.
\end{definition}

\begin{definition}
\label{def:planar-mor}
    A \textit{planar algebra morphism} $\Phi = \{ \Phi_B \}_B$ between planar algebras $\mathcal{P}$ and $\mathcal{Q}$ is a collection of maps
    \[
        \Phi_B: \mathcal{P}(B) \to \mathcal{Q}(B)
    \]
    such that for each $d$-input arc diagram $D$,
    \[
        \Phi_B \circ D = D \circ (\Phi_{B_1} \times \cdots \times \Phi_{B_d}).
    \]
\end{definition}

In \Cref{def:planar-alg}, if we allow $\{\mathcal{P}(B)\}_B$ to be a collection of categories and $\{ \mathcal{O}(D) \}_D$ to be a collection of functors that are compatible with the compositions, we obtain the concept of \textit{planar algebra of categories} (see \cite[Definition 3]{Webster:2007}). Similarly, a planar algebra morphism can be defined for planar algebra of categories by allowing $\{ \Phi_B \}_B$ to be a collection of functors. The following collections of categories admit planar algebra structures in the obvious way: 
\[
    \Cob^3_{(\bullet)}, \ 
    \Kob^3_{(\bullet)}, \ 
    \Diag.
\]
Furthermore, the formal Khovanov bracket
\[
    [\cdot]: \Diag(B) \to \Kob_{(\bullet)}(B)
\]
induces a planar algebra morphism between the two planar algebras of categories. The planar algebra structures are inherited to the respective quotient categories
\[
    \Cob^3_{(\bullet)/l}, \ 
    \Kob^3_{(\bullet)/h}, \ 
    \Diag_{/i},
\]
and the formal Khovanov bracket induces a planar algebra morphism
\[
    [\cdot]: \Diag_{/i} \to \Kob_{(\bullet)/\pm h}.
\]

\subsection{Gaussian elimination}
\label{subsec:gaussian}

As shown in \cite{BarNatan:2007}, \textit{Gaussian elimination}, one of the fundamental methods in linear algebra, can be applied to reduce a chain complex in any additive category. First, we recall the definition of a \textit{strong deformation retraction}.  
\begin{definition}
    In any additive category, a chain map $r\colon \Omega \rightarrow \Omega'$ between chain complexes is a \textit{strong deformation retraction}\footnote{
        Conditions (iii) - (v) are called \textit{side conditions}. Some authors only require conditions (i), (ii) for the definition of strong deformations retractions. If (i), (ii) are satisfied, then the homotopy $h$ can be modified so that it also satisfies the side conditions. 
    }, if there is a chain map $i\colon \Omega' \rightarrow \Omega$ (called the \textit{inclusion}) and a homotopy $h$ on $\Omega$, satisfying (i) $ri = 1$, (ii) $ir - 1 = dh + hd$, (iii) $hi = 0$, (iv) $rh = 0$ and (v) $h^2 = 0$. Such $\Omega'$ is called a \textit{strong deformation retract} of $\Omega$. 
    \[
        \begin{tikzcd}
        \Omega 
            \arrow[r, "r", shift left] 
            \arrow["h"', loop, distance=2em, in=305, out=235] & 
        \Omega' 
            \arrow[l, "i", shift left]
        \end{tikzcd}        
    \]
\end{definition}

\begin{proposition}[Gaussian elimination]
\label{prop:gauss-elim}
    Suppose the top row of the following diagram is a part of a chain complex $\Omega$ such that the morphism $X \xrightarrow{a} Y$ is invertible. Then there is a strong deformation retract $\Omega'$ of $\Omega$ described by the bottom row, with a strong deformation retraction indicated by the downward vertical arrows: 
    \[
\begin{tikzcd}
[row sep=6em, column sep=4.5em, ampersand replacement=\&]
U 
    \arrow[r, "{\begin{pmatrix} \alpha \\ \beta \end{pmatrix}}"] 
    \arrow[d, equal] 
\& \begin{pmatrix} X \\ Y \end{pmatrix} 
    \arrow[rr, "{\begin{pmatrix} a & b \\ c & d \end{pmatrix}}", shift left] 
    \arrow[d, "{\begin{pmatrix} 0 & 1 \end{pmatrix}}"', shift right] 
\& \& \begin{pmatrix} Z \\ W \end{pmatrix} 
    \arrow[r, "{\begin{pmatrix} \gamma & \delta \end{pmatrix}}"] 
    \arrow[d, "{\begin{pmatrix} -ca^{-1} & 1 \end{pmatrix}}"', shift right] 
    \arrow[ll, "{\begin{pmatrix} -a & 0 \\ 0 & 0 \end{pmatrix}}", dashed, shift left]
\& V 
    \arrow[d, equal] 
\\ U 
    \arrow[r, "\beta"] 
\& X' 
    \arrow[rr, "d - c a^{-1} b"] 
    \arrow[u, "{\begin{pmatrix} -a^{-1}b \\ 1 \end{pmatrix}}"', dashed, shift right]
\& \& Y' 
    \arrow[r, "\delta"] 
    \arrow[u, "{\begin{pmatrix} 0 \\ 1 \end{pmatrix}}"', dashed, shift right]
\& V.
\end{tikzcd}
    \]
    The associated inclusion and homotopy are indicated by the dashed arrows.
\end{proposition}

The morphism $d - c a^{-1} b$ is the \textit{Schur complement} of the matrix $\begin{pmatrix} a & b \\ c & d \end{pmatrix}$. One may regard the $c a^{-1} b$ term as produced by the reversal of the arrow $a$
\[
\begin{tikzcd}[row sep=3em, column sep=3.5em]
X \arrow[rd, "c"', pos=.8]  & Z \arrow[l, "a^{-1}"', dashed] \\
Y \arrow[ru, "b", pos=.8] & W.
\end{tikzcd}
\]
Also note the $-c a^{-1}$ entry in the right vertical arrow of the retraction, which is also caused by the reversal of $a$. One may recall from Morse theory, where a pair of critical points $p, q$ of a Morse function $f$ on a manifold $M$ that are connected by a single flow line from $p$ to $q$ can be canceled by modifying the gradient-like vector field so that the flow line is reversed from $q$ to $p$ (see \cite[Figure 5.2]{Milnor:1965}). 

In actual computations, when applying Gaussian elimination on an invertible arrow $a\colon X \to Z$, we collect \textit{all objects} $Y$ that has a non-trivial arrow $b$ into $Z$ and \textit{all objects} $W$ that has a non-trivial arrow $c$ out from $X$, and for every pair $(Y, W)$ subtract the morphism $c a^{-1} b$ from $d\colon Y \to W$. This process can be repeated as long as there is an invertible arrow left in the reduced complex. 

Given a tangle diagram $T \in \Diag(B)$, by repeatedly applying delooping and Gaussian elimination to the chain complex $[T]$, it can be reduced into a complex $E$ with differential $d_E$ such that:
\begin{enumerate}
    \item No object of $E$ contains a loop.
    \item No entry of $d_E$ is invertible. 
\end{enumerate}
We call such procedure, and also the corresponding homotopy equivalence, a \textit{reduction} of the complex $[T]$. In particular when $B = \underline{0}$, $[T]$ reduces to a complex $E$ such that:
\begin{enumerate}
    \item Every object of $E$ is a direct sum of empty diagrams.
    \item Every entry of $d_E$ is a non-invertible homogeneous polynomial in $\ZZ[h, t]$. 
\end{enumerate}

\cite[Section 6]{BarNatan:2007} demonstrates this procedure, where the Khovanov homology ($h = t = 0$) of the figure eight knot is efficiently computed by hand. Analogous procedure also works in the reduced setting, i.e.\ when $B = \underline{2}$. 
    \section{Reformulation of the Rasmussen invariant}
\label{sec:reform-s}

Hereafter, we impose an additional local relation to the category $\Cob_{\bullet/l}(B)$,
\begin{center}
    \tikzset{every picture/.style={line width=0.75pt}} %set default line width to 0.75pt        

\begin{tikzpicture}[x=0.75pt,y=0.75pt,yscale=-1,xscale=1]
%uncomment if require: \path (0,87); %set diagram left start at 0, and has height of 87

%Shape: Arc [id:dp8991209021386461] 
\draw  [draw opacity=0] (111.56,24.31) .. controls (111,26.25) and (105.11,27.77) .. (97.92,27.77) .. controls (90.36,27.77) and (84.23,26.09) .. (84.23,24.01) .. controls (84.23,23.88) and (84.26,23.76) .. (84.3,23.64) -- (97.92,24.01) -- cycle ; \draw  [color={rgb, 255:red, 128; green, 128; blue, 128 }  ,draw opacity=1 ] (111.56,24.31) .. controls (111,26.25) and (105.11,27.77) .. (97.92,27.77) .. controls (90.36,27.77) and (84.23,26.09) .. (84.23,24.01) .. controls (84.23,23.88) and (84.26,23.76) .. (84.3,23.64) ;  
%Shape: Arc [id:dp5130200525224077] 
\draw  [draw opacity=0][dash pattern={on 0.84pt off 2.51pt}] (84.57,23.08) .. controls (85.67,21.31) and (91.29,19.95) .. (98.06,19.95) .. controls (105.51,19.95) and (111.57,21.59) .. (111.75,23.62) -- (98.06,23.71) -- cycle ; \draw  [color={rgb, 255:red, 128; green, 128; blue, 128 }  ,draw opacity=1 ][dash pattern={on 0.84pt off 2.51pt}] (84.57,23.08) .. controls (85.67,21.31) and (91.29,19.95) .. (98.06,19.95) .. controls (105.51,19.95) and (111.57,21.59) .. (111.75,23.62) ;  
%Shape: Ellipse [id:dp09019278146354637] 
\draw   (84.13,24.01) .. controls (84.13,16.4) and (90.3,10.22) .. (97.92,10.22) .. controls (105.53,10.22) and (111.71,16.4) .. (111.71,24.01) .. controls (111.71,31.63) and (105.53,37.8) .. (97.92,37.8) .. controls (90.3,37.8) and (84.13,31.63) .. (84.13,24.01) -- cycle ;
%Shape: Ellipse [id:dp7060167442475777] 
\draw  [fill={rgb, 255:red, 0; green, 0; blue, 0 }  ,fill opacity=1 ] (92.33,16.49) .. controls (92.33,15.51) and (93.12,14.72) .. (94.1,14.72) .. controls (95.07,14.72) and (95.87,15.51) .. (95.87,16.49) .. controls (95.87,17.47) and (95.07,18.26) .. (94.1,18.26) .. controls (93.12,18.26) and (92.33,17.47) .. (92.33,16.49) -- cycle ;
%Shape: Ellipse [id:dp25425461936506777] 
\draw  [fill={rgb, 255:red, 0; green, 0; blue, 0 }  ,fill opacity=1 ] (100.61,16.49) .. controls (100.61,15.51) and (101.4,14.72) .. (102.37,14.72) .. controls (103.35,14.72) and (104.14,15.51) .. (104.14,16.49) .. controls (104.14,17.47) and (103.35,18.26) .. (102.37,18.26) .. controls (101.4,18.26) and (100.61,17.47) .. (100.61,16.49) -- cycle ;

%Shape: Arc [id:dp6880148267466573] 
\draw  [draw opacity=0] (111.56,61.11) .. controls (111,63.05) and (105.11,64.57) .. (97.92,64.57) .. controls (90.36,64.57) and (84.23,62.89) .. (84.23,60.81) .. controls (84.23,60.68) and (84.26,60.56) .. (84.3,60.44) -- (97.92,60.81) -- cycle ; \draw  [color={rgb, 255:red, 128; green, 128; blue, 128 }  ,draw opacity=1 ] (111.56,61.11) .. controls (111,63.05) and (105.11,64.57) .. (97.92,64.57) .. controls (90.36,64.57) and (84.23,62.89) .. (84.23,60.81) .. controls (84.23,60.68) and (84.26,60.56) .. (84.3,60.44) ;  
%Shape: Arc [id:dp056731527356467804] 
\draw  [draw opacity=0][dash pattern={on 0.84pt off 2.51pt}] (84.57,59.88) .. controls (85.67,58.11) and (91.29,56.75) .. (98.06,56.75) .. controls (105.51,56.75) and (111.57,58.39) .. (111.75,60.42) -- (98.06,60.51) -- cycle ; \draw  [color={rgb, 255:red, 128; green, 128; blue, 128 }  ,draw opacity=1 ][dash pattern={on 0.84pt off 2.51pt}] (84.57,59.88) .. controls (85.67,58.11) and (91.29,56.75) .. (98.06,56.75) .. controls (105.51,56.75) and (111.57,58.39) .. (111.75,60.42) ;  
%Shape: Ellipse [id:dp22803045933391064] 
\draw   (84.13,60.81) .. controls (84.13,53.2) and (90.3,47.02) .. (97.92,47.02) .. controls (105.53,47.02) and (111.71,53.2) .. (111.71,60.81) .. controls (111.71,68.43) and (105.53,74.6) .. (97.92,74.6) .. controls (90.3,74.6) and (84.13,68.43) .. (84.13,60.81) -- cycle ;
%Shape: Ellipse [id:dp7195049859109318] 
\draw  [fill={rgb, 255:red, 0; green, 0; blue, 0 }  ,fill opacity=1 ] (92.33,53.29) .. controls (92.33,52.31) and (93.12,51.52) .. (94.1,51.52) .. controls (95.07,51.52) and (95.87,52.31) .. (95.87,53.29) .. controls (95.87,54.27) and (95.07,55.06) .. (94.1,55.06) .. controls (93.12,55.06) and (92.33,54.27) .. (92.33,53.29) -- cycle ;
%Shape: Ellipse [id:dp18308218918006358] 
\draw  [fill={rgb, 255:red, 0; green, 0; blue, 0 }  ,fill opacity=1 ] (100.61,53.29) .. controls (100.61,52.31) and (101.4,51.52) .. (102.37,51.52) .. controls (103.35,51.52) and (104.14,52.31) .. (104.14,53.29) .. controls (104.14,54.27) and (103.35,55.06) .. (102.37,55.06) .. controls (101.4,55.06) and (100.61,54.27) .. (100.61,53.29) -- cycle ;

%Shape: Arc [id:dp9272105409186626] 
\draw  [draw opacity=0] (47.94,41.76) .. controls (47.29,44) and (40.45,45.77) .. (32.11,45.77) .. controls (23.35,45.77) and (16.24,43.82) .. (16.24,41.41) .. controls (16.24,41.26) and (16.27,41.12) .. (16.32,40.98) -- (32.11,41.41) -- cycle ; \draw  [color={rgb, 255:red, 128; green, 128; blue, 128 }  ,draw opacity=1 ] (47.94,41.76) .. controls (47.29,44) and (40.45,45.77) .. (32.11,45.77) .. controls (23.35,45.77) and (16.24,43.82) .. (16.24,41.41) .. controls (16.24,41.26) and (16.27,41.12) .. (16.32,40.98) ;  
%Shape: Arc [id:dp8373795595166753] 
\draw  [draw opacity=0][dash pattern={on 0.84pt off 2.51pt}] (16.63,40.33) .. controls (17.9,38.27) and (24.43,36.71) .. (32.28,36.71) .. controls (40.93,36.71) and (47.96,38.6) .. (48.15,40.96) -- (32.28,41.06) -- cycle ; \draw  [color={rgb, 255:red, 128; green, 128; blue, 128 }  ,draw opacity=1 ][dash pattern={on 0.84pt off 2.51pt}] (16.63,40.33) .. controls (17.9,38.27) and (24.43,36.71) .. (32.28,36.71) .. controls (40.93,36.71) and (47.96,38.6) .. (48.15,40.96) ;  
%Shape: Circle [id:dp7641306816787177] 
\draw   (16.12,41.41) .. controls (16.12,32.58) and (23.28,25.42) .. (32.11,25.42) .. controls (40.95,25.42) and (48.11,32.58) .. (48.11,41.41) .. controls (48.11,50.25) and (40.95,57.41) .. (32.11,57.41) .. controls (23.28,57.41) and (16.12,50.25) .. (16.12,41.41) -- cycle ;
%Shape: Ellipse [id:dp2949956012879268] 
\draw  [fill={rgb, 255:red, 0; green, 0; blue, 0 }  ,fill opacity=1 ] (22.63,32.49) .. controls (22.63,31.35) and (23.55,30.44) .. (24.68,30.44) .. controls (25.82,30.44) and (26.73,31.35) .. (26.73,32.49) .. controls (26.73,33.62) and (25.82,34.54) .. (24.68,34.54) .. controls (23.55,34.54) and (22.63,33.62) .. (22.63,32.49) -- cycle ;
%Shape: Ellipse [id:dp21044418370086182] 
\draw  [fill={rgb, 255:red, 0; green, 0; blue, 0 }  ,fill opacity=1 ] (30.23,32.49) .. controls (30.23,31.35) and (31.15,30.44) .. (32.28,30.44) .. controls (33.42,30.44) and (34.33,31.35) .. (34.33,32.49) .. controls (34.33,33.62) and (33.42,34.54) .. (32.28,34.54) .. controls (31.15,34.54) and (30.23,33.62) .. (30.23,32.49) -- cycle ;
%Shape: Ellipse [id:dp015687192068096767] 
\draw  [fill={rgb, 255:red, 0; green, 0; blue, 0 }  ,fill opacity=1 ] (37.83,32.49) .. controls (37.83,31.35) and (38.75,30.44) .. (39.88,30.44) .. controls (41.02,30.44) and (41.93,31.35) .. (41.93,32.49) .. controls (41.93,33.62) and (41.02,34.54) .. (39.88,34.54) .. controls (38.75,34.54) and (37.83,33.62) .. (37.83,32.49) -- cycle ;

% Text Node
\draw (55.5,31.9) node [anchor=north west][inner sep=0.75pt]    {$=$};

\end{tikzpicture}
\end{center}
or equivalently $t = 0$, and write 
\begin{center}
    \tikzset{every picture/.style={line width=0.75pt}} %set default line width to 0.75pt        

\begin{tikzpicture}[x=0.75pt,y=0.75pt,yscale=-1,xscale=1]
%uncomment if require: \path (0,47); %set diagram left start at 0, and has height of 47

%Shape: Arc [id:dp6319807885803715] 
\draw  [draw opacity=0] (96.94,23.75) .. controls (96.29,26) and (89.45,27.76) .. (81.11,27.76) .. controls (72.35,27.76) and (65.24,25.81) .. (65.24,23.41) .. controls (65.24,23.26) and (65.27,23.11) .. (65.32,22.97) -- (81.11,23.41) -- cycle ; \draw  [color={rgb, 255:red, 128; green, 128; blue, 128 }  ,draw opacity=1 ] (96.94,23.75) .. controls (96.29,26) and (89.45,27.76) .. (81.11,27.76) .. controls (72.35,27.76) and (65.24,25.81) .. (65.24,23.41) .. controls (65.24,23.26) and (65.27,23.11) .. (65.32,22.97) ;  
%Shape: Arc [id:dp056615061649111675] 
\draw  [draw opacity=0][dash pattern={on 0.84pt off 2.51pt}] (65.63,22.33) .. controls (66.9,20.27) and (73.43,18.7) .. (81.28,18.7) .. controls (89.93,18.7) and (96.96,20.6) .. (97.15,22.96) -- (81.28,23.06) -- cycle ; \draw  [color={rgb, 255:red, 128; green, 128; blue, 128 }  ,draw opacity=1 ][dash pattern={on 0.84pt off 2.51pt}] (65.63,22.33) .. controls (66.9,20.27) and (73.43,18.7) .. (81.28,18.7) .. controls (89.93,18.7) and (96.96,20.6) .. (97.15,22.96) ;  
%Shape: Circle [id:dp39111359424094494] 
\draw   (65.12,23.41) .. controls (65.12,14.57) and (72.28,7.41) .. (81.11,7.41) .. controls (89.95,7.41) and (97.11,14.57) .. (97.11,23.41) .. controls (97.11,32.24) and (89.95,39.4) .. (81.11,39.4) .. controls (72.28,39.4) and (65.12,32.24) .. (65.12,23.41) -- cycle ;
%Shape: Ellipse [id:dp9734470792692242] 
\draw  [fill={rgb, 255:red, 0; green, 0; blue, 0 }  ,fill opacity=1 ] (74.63,14.68) .. controls (74.63,13.55) and (75.55,12.63) .. (76.68,12.63) .. controls (77.82,12.63) and (78.73,13.55) .. (78.73,14.68) .. controls (78.73,15.81) and (77.82,16.73) .. (76.68,16.73) .. controls (75.55,16.73) and (74.63,15.81) .. (74.63,14.68) -- cycle ;
%Shape: Ellipse [id:dp1260497035726349] 
\draw  [fill={rgb, 255:red, 0; green, 0; blue, 0 }  ,fill opacity=1 ] (84.23,14.68) .. controls (84.23,13.55) and (85.15,12.63) .. (86.28,12.63) .. controls (87.42,12.63) and (88.33,13.55) .. (88.33,14.68) .. controls (88.33,15.81) and (87.42,16.73) .. (86.28,16.73) .. controls (85.15,16.73) and (84.23,15.81) .. (84.23,14.68) -- cycle ;

% Text Node
\draw (14,14.61) node [anchor=north west][inner sep=0.75pt]    {$H\ =\ $};

\end{tikzpicture}
\end{center}
instead of $h$, as in \Cref{subsec:reduced-khovanov}. We override the notation and denote the new category as $\Cob_{\bullet/l}(B)$. The local relations become:
\begin{center}
    \tikzset{every picture/.style={line width=0.75pt}} %set default line width to 0.75pt        

\begin{tikzpicture}[x=0.75pt,y=0.75pt,yscale=-.75,xscale=.75]
%uncomment if require: \path (0,67); %set diagram left start at 0, and has height of 67

%Shape: Rectangle [id:dp21707422911588292] 
\draw  [dash pattern={on 0.84pt off 2.51pt}] (10.93,11.06) -- (56.56,11.06) -- (56.56,54.99) -- (10.93,54.99) -- cycle ;
%Shape: Rectangle [id:dp15493943713249336] 
\draw  [dash pattern={on 0.84pt off 2.51pt}] (115.41,11.06) -- (162.08,11.06) -- (162.08,55.99) -- (115.41,55.99) -- cycle ;
%Shape: Circle [id:dp7365920452212693] 
\draw  [fill={rgb, 255:red, 0; green, 0; blue, 0 }  ,fill opacity=1 ] (23.5,34.2) .. controls (23.5,32.43) and (24.93,31) .. (26.7,31) .. controls (28.47,31) and (29.91,32.43) .. (29.91,34.2) .. controls (29.91,35.97) and (28.47,37.41) .. (26.7,37.41) .. controls (24.93,37.41) and (23.5,35.97) .. (23.5,34.2) -- cycle ;
%Shape: Circle [id:dp9469298444570525] 
\draw  [fill={rgb, 255:red, 0; green, 0; blue, 0 }  ,fill opacity=1 ] (36,34.2) .. controls (36,32.43) and (37.43,31) .. (39.2,31) .. controls (40.97,31) and (42.41,32.43) .. (42.41,34.2) .. controls (42.41,35.97) and (40.97,37.41) .. (39.2,37.41) .. controls (37.43,37.41) and (36,35.97) .. (36,34.2) -- cycle ;
%Shape: Circle [id:dp3791562173926655] 
\draw  [fill={rgb, 255:red, 0; green, 0; blue, 0 }  ,fill opacity=1 ] (135.54,33.52) .. controls (135.54,31.75) and (136.97,30.32) .. (138.74,30.32) .. controls (140.51,30.32) and (141.95,31.75) .. (141.95,33.52) .. controls (141.95,35.29) and (140.51,36.73) .. (138.74,36.73) .. controls (136.97,36.73) and (135.54,35.29) .. (135.54,33.52) -- cycle ;
%Shape: Rectangle [id:dp7808305222084078] 
\draw  [dash pattern={on 0.84pt off 2.51pt}] (204.93,11.06) -- (250.56,11.06) -- (250.56,54.99) -- (204.93,54.99) -- cycle ;
%Shape: Rectangle [id:dp659071778039395] 
\draw  [dash pattern={on 0.84pt off 2.51pt}] (322.41,11.06) -- (369.08,11.06) -- (369.08,55.99) -- (322.41,55.99) -- cycle ;
%Shape: Circle [id:dp3577246373145724] 
\draw  [color={rgb, 255:red, 0; green, 0; blue, 0 }  ,draw opacity=1 ][fill={rgb, 255:red, 255; green, 255; blue, 255 }  ,fill opacity=1 ] (217.5,34.06) .. controls (217.5,32.29) and (218.93,30.85) .. (220.7,30.85) .. controls (222.47,30.85) and (223.91,32.29) .. (223.91,34.06) .. controls (223.91,35.83) and (222.47,37.26) .. (220.7,37.26) .. controls (218.93,37.26) and (217.5,35.83) .. (217.5,34.06) -- cycle ;
%Shape: Circle [id:dp9847973386629165] 
\draw  [color={rgb, 255:red, 0; green, 0; blue, 0 }  ,draw opacity=1 ][fill={rgb, 255:red, 255; green, 255; blue, 255 }  ,fill opacity=1 ] (230,34.06) .. controls (230,32.29) and (231.43,30.85) .. (233.2,30.85) .. controls (234.97,30.85) and (236.41,32.29) .. (236.41,34.06) .. controls (236.41,35.83) and (234.97,37.26) .. (233.2,37.26) .. controls (231.43,37.26) and (230,35.83) .. (230,34.06) -- cycle ;
%Shape: Circle [id:dp7713082279762818] 
\draw  [color={rgb, 255:red, 0; green, 0; blue, 0 }  ,draw opacity=1 ][fill={rgb, 255:red, 255; green, 255; blue, 255 }  ,fill opacity=1 ] (342.54,34.38) .. controls (342.54,32.61) and (343.97,31.17) .. (345.74,31.17) .. controls (347.51,31.17) and (348.95,32.61) .. (348.95,34.38) .. controls (348.95,36.15) and (347.51,37.58) .. (345.74,37.58) .. controls (343.97,37.58) and (342.54,36.15) .. (342.54,34.38) -- cycle ;
%Shape: Rectangle [id:dp927748882474282] 
\draw  [dash pattern={on 0.84pt off 2.51pt}] (415.93,11.06) -- (461.56,11.06) -- (461.56,54.99) -- (415.93,54.99) -- cycle ;
%Shape: Circle [id:dp6250183398486859] 
\draw  [color={rgb, 255:red, 0; green, 0; blue, 0 }  ,draw opacity=1 ][fill={rgb, 255:red, 0; green, 0; blue, 0 }  ,fill opacity=1 ] (428.5,34.56) .. controls (428.5,32.79) and (429.93,31.35) .. (431.7,31.35) .. controls (433.47,31.35) and (434.91,32.79) .. (434.91,34.56) .. controls (434.91,36.33) and (433.47,37.76) .. (431.7,37.76) .. controls (429.93,37.76) and (428.5,36.33) .. (428.5,34.56) -- cycle ;
%Shape: Circle [id:dp7505548716310184] 
\draw  [color={rgb, 255:red, 0; green, 0; blue, 0 }  ,draw opacity=1 ][fill={rgb, 255:red, 255; green, 255; blue, 255 }  ,fill opacity=1 ] (441,34.56) .. controls (441,32.79) and (442.43,31.35) .. (444.2,31.35) .. controls (445.97,31.35) and (447.41,32.79) .. (447.41,34.56) .. controls (447.41,36.33) and (445.97,37.76) .. (444.2,37.76) .. controls (442.43,37.76) and (441,36.33) .. (441,34.56) -- cycle ;

% Text Node
\draw (68,26.21) node [anchor=north west][inner sep=0.75pt]    {$=$};
% Text Node
\draw (177,26.21) node [anchor=north west][inner sep=0.75pt]    {$,$};
% Text Node
\draw (95,26.21) node [anchor=north west][inner sep=0.75pt]    {$H$};
% Text Node
\draw (262,26.21) node [anchor=north west][inner sep=0.75pt]    {$=$};
% Text Node
\draw (288,26.21) node [anchor=north west][inner sep=0.75pt]    {$-H$};
% Text Node
\draw (387,26.21) node [anchor=north west][inner sep=0.75pt]    {$,$};
% Text Node
\draw (470,26.21) node [anchor=north west][inner sep=0.75pt]    {$=\ 0.$};

\end{tikzpicture}
\end{center}
The ring $R$ and the Frobenius algebra $A$ we obtain from the new $\Cob_{\bullet/l}(B)$ are
\[
    R = \ZZ[H],\ 
    A = R[X]/(X^2 - HX)
\]
whose corresponding link homology theory is the $U(1)$-equivariant Khovanov homology. With $Y = X - H$ in $A$, we have 
\begin{gather*}
    X^2 = HX,\quad Y^2 = -HY,\quad XY = 0,\\
    \Delta (X) = X \otimes X,\quad \Delta (Y) = Y \otimes Y,\\
    \epsilon(X) = \epsilon(Y) = 1.
\end{gather*}

\subsection{The link invariant $s_H$}
\label{subsec:lee-class}

Herein, we assume that $L$ is a link diagram, and is regarded as an object in $\Diag(\underline{0})$. In view of \Cref{prop:kom-tautological}, we shall redefine the \textit{Lee cycle} $\ca(L)$ of $L$ as a chain map from $[\varnothing]$ to $[L]$, and reformulate the link invariant $s_H$ introduced in \cite{Sano:2020}.

Let $L^\lout$ denote the \textit{Seifert resolution} of $L$, i.e.\ the diagram obtained by resolving each crossing of $L$ in the orientation preserving way. Let
\[
    S_L: \varnothing \to L^\lout
\]
be the cobordism consisting of a cup $D$ for each circle $\gamma \subset L^\lout$ such that $\del D = \gamma$. The Lee cycle $\ca(L)$ is defined by adding a dot or a hollow dot to each component of $S_L$, as follows. The \textit{standard checkerboard coloring} of $L^\lout$ is the alternative coloring on the complement $D^2 \setminus L^\lout$ by white or black such that a neighborhood of $\del D^2$ is colored white. The \textit{$XY$-labeling} on $L^\lout$ is a labeling of each component $\gamma$ of $L^\lout$ by $X$ or $Y$ defined as follows: while moving a point along $\gamma$ in the given orientation, if it sees a black region to the left, label it $X$, otherwise label it $Y$. For each component $S_i$ of $S_L$, add a dot $\bullet$ or a hollow dot $\circ$ to $S_i$ depending on whether its boundary is labeled $X$ or $Y$. This defines a dotted cobordism 
\[
    \ca(L)\colon \varnothing \to L^\lout
\]
in $\Cob^3_{\bullet}(\underline{0})$. See \Cref{fig:intro-lee-cycle-cob} in \Cref{sec:intro} for an example, where the labels $\mathbf{a}, \mathbf{b}$ are to be replaced with $X, Y$. 

% \begin{figure}[t]
%     \centering
%     \input{tikzpictures/lee-cycle}
%     \caption{The $XY$-labeling for a link diagram $L$.}
%     \label{fig:lee-cycle}
%     %
%     \vspace{1.5em}
%     %
%     \input{tikzpictures/lee-cycle-cob}
%     \caption{The Lee cycle $\ca(L)$.}
%     \label{fig:lee-cycle-cob}    
% \end{figure}

\begin{proposition}
\label{prop:lee-cycle-is-cycle}
    The cobordism $\ca(L)$ defines a chain map
    \[
        \ca(L): [\varnothing] \to [L]
    \]
    of homological degree $0$. 
\end{proposition}

\begin{proof}
    That $\ca(L)$ is a chain map is equivalent to $d \circ \ca(L) = 0$, which is equivalent to $f \circ \ca(L) = 0$ for each arrow $f$ going out of $L^\lout$ in the complex $[L]$. Observe that each $f$ is a saddle cobordism with one end of the boundary a pair of circles of $L^\lout$ with different labels. Thus $f$ merges two components of $\ca(L)$ with different types of dots, giving $(\bullet \circ) = 0$. That $\ca(L)$ has homological degree $0$ is obvious from the fact that $L^\lout$ is a homological grading $0$ summand in the complex $[L]$. 
\end{proof}

\begin{definition}
\label{def:lee-cycle}
    The chain map $\ca(L)$ is called the \textit{Lee cycle} of $L$, and its chain homotopy class $[\ca(L)]$ the \textit{Lee class} of $L$.
\end{definition}

Note that for the orientation reversed diagram $L^r$ of $L$, its Lee cycle $\ca(L^r)$ is obtained from $\ca(L)$ by swapping the dots and the hollow dots. The cycle $\ca(L^r)$ will be denoted $\cb(L)$.

\begin{proposition}
\label{prop:qdeg-lee-cycle}
    \[
        \qdeg \ca(L) = \qdeg \cb(L) = w(L) - r(L).
    \]
    Here, $w(L)$ denotes the writhe, and $r(L)$ the number of Seifert circles of $L$.
\end{proposition}

% \begin{proof}
%     With $\qdeg(\bullet) = \qdeg(\circ) = -2$, we have 
%     \[
%         \qdeg \ca(L) 
%         = w(L) + r(L)(\chi(D^2) - 2) = w(L) - r(L).
%     \]
% \end{proof}

Up to chain homotopy, the Lee cycle $\ca(L)$ may be expressed as $H^k z$ with $k > 0$ and some cycle $z$. For such $z$ to be unique in a certain sense, we introduce the concept of \textit{stable (homotopy) equivalence}. 

\begin{definition}
\label{def:stab-eqiuv}
    Two chain maps $z, z'$ in $\Kob_\bullet(B)$ are \textit{stably (homotopy) equivalent}, denoted $z \htpy_s z'$, if there exists an integer $k \geq 0$ such that $H^k z \htpy H^k z'$.
\end{definition}

Obviously, stable equivalence is an equivalence relation on each hom-set of $\Kob_\bullet(B)$, and the $q$-degree is invariant under stable equivalence. 
% The quotient category will be denoted $\Kob_{\bullet / sh}(B)$. 

\begin{definition}
\label{def:H-divisiblity}
    The \textit{$H$-divisibility} of the Lee cycle $\ca(L)$ the maximal $H$-divisibility of a cycle that is stably equivalent to $\ca(L)$, i.e.\ 
    \[
        d_H(L) = \max_k\{\ \exists z\colon [\varnothing] \to [L],\ (H^k) z \htpy_s \ca(L)\ \}.
    \]
\end{definition}

Note that $d_H(L)$ is always finite, since the $q$-degree is bounded from above in the corresponding hom-set. The following proposition shows that the $H$-divisibilities can be compared using chain maps. 

\begin{proposition}
    Suppose $L, L'$ are two link diagrams, and
    \[
        f: [L] \to [L']
    \]
    is a chain map such that 
    \[
        f \circ \ca(L) \htpy_s H^k \ca(L')
    \]
    for some $k \in \ZZ$. Then 
    \[
        d_H(L) \leq k + d_H(L').
    \]
    Moreover, when $f$ is a chain homotopy equivalence, then the above inequality becomes an equality. 
\end{proposition}

The following proposition states the variance of the $H$-divisibility of the Lee class under the Reidemeister moves. Although the corresponding propositions are proved in \cite[Proposition 2.13]{Sano:2020} and in \cite[Proposition 2.16]{Sano-Sato:2023} in the ordinary context of Khovanov homology, a proof is rewritten in \Cref{sec:appendix} for the current context, effectively using delooping (\Cref{prop:delooping}) and Gaussian elimination (\Cref{prop:gauss-elim}). We encourage the reader to go through the proof, as it demonstrates how the $H$-multipliers arise while applying these techniques. 

\begin{proposition}
\label{prop:ca-under-reidemeister}
    Suppose $L, L'$ are link diagrams related by a Reidemeister move. Under the corresponding chain homotopy equivalence $f$, the Lee cycles correspond as 
    \begin{align*}
        f \circ \ca(L) &\htpy \epsilon H^j \ca(L') \\
        f \circ \cb(L) &\htpy \epsilon' H^j \cb(L')
    \end{align*}
    where $j \in \{0, \pm 1\}$ is given by%
    \footnote{
        Here, $\delta$ denotes the \textit{difference operator}, i.e.\ $\delta f(x, y) = f(y) - f(x)$.
    } 
    \[
        j = \frac{\delta w(L, L') - \delta r(L, L')}{2}
    \]
    and $\epsilon, \epsilon'$ are signs satisfying 
    \[
        \epsilon\epsilon' = (-1)^j.
    \]
\end{proposition}

\begin{definition}
    The stable equivalence class of the cycle $z$ satisfying 
    \[
        (H^{d_H(L)}) z \htpy_s \ca(L)
    \]
    is called the \textit{refined Lee class} of $L$, and will be denoted $\ca_s(L)$.
\end{definition}

The refined Lee class $\ca_s(L)$ is an invariant of the link in the following sense.

\begin{proposition}
\label{prop:lee-class-invariance}
    Suppose $L, L'$ are link diagrams related by a sequence of Reidemeister moves. With the corresponding chain homotopy equivalence $f$ given in \Cref{thm:formal-ckh-invariance}, the following diagram commutes up to sign and stable equivalence:
    \[
        \begin{tikzcd}[row sep=3em, column sep=5em]
        {[\varnothing]} \arrow[d, equal] \arrow[r, "\ca_s(L)"] & {[L]} \arrow[d, "f"] \\
        {[\varnothing]} \arrow[r, "\ca_s(L')"]          & {[L']}.              
        \end{tikzcd}
    \]
\end{proposition}

\begin{proof}
    Let $g$ be the homotopy inverse of $f$ given in \Cref{thm:formal-ckh-invariance}. First, we claim that $f \circ \ca_s(L)$ is not (stably) $H$-divisible, otherwise
    \[
        \ca_s(L) \htpy g \circ (f \circ \ca_s(L))
    \]
    would be $H$-divisible, contradicting the maximality of $d_H(L)$. Similarly, we have that $g \circ \ca_s(L')$ is not $H$-divisible. From \Cref{prop:ca-under-reidemeister}, we have 
    \[
        f \circ \ca(L) \htpy (\pm H^j) \ca(L')
    \]
    for some $j \in \ZZ$, and hence
    \[
        (H^k) f \circ \ca_s(L) \htpy_s (\pm H^{k' + j}) \ca_s(L')
    \]
    where $k = d_H(L)$ and $k' = d_H(L')$. We claim that $k = k' + j$. Otherwise, if $k < k' + j$, we would have 
    \[
        \ca_s(L) \htpy g  \circ (f \circ \ca_s(L)) \htpy_s H^{k' + j - k} g \circ \ca_s(L')
    \]
    contradicting the non-divisibility of $\ca_s(L)$. Similarly, $k > k' + j$ cannot hold. Thus $\ca_s(L)$ and $\pm \ca_s(L')$ are stably equivalent.
\end{proof}

\Cref{prop:lee-class-invariance} justifies the following definition.

\begin{definition}
\label{def:s_H}
    The link invariant $s_H$ is defined by 
    \[
        s_H(L) = 2d_H(L) + w(L) - r(L) + 1,
    \]
    which equals $\qdeg(\ca_s(L)) + 1$. 
\end{definition}

The constant is added so that $s_H(\bigcirc) = 0$. The following propositions can be obtained by translating the propositions of \cite{Sano:2020}. 

\begin{proposition} 
\label{prop:s_H-properties}
    Let $L, L'$ be link diagrams.
    \begin{enumerate}
        \item $d_H(L) = d_H(L^r)$.
        \item $d_H(L \sqcup L') \geq d_H(L) + d_H(L')$.
        \item If $L$ is positive, then $d_H(L) = 0$.
    \end{enumerate}
    Equivalently, 
    \begin{enumerate}
        \item $s_H(L) = s_H(L^r)$.
        \item $s_H(L \sqcup L') \geq s_H(L) + s_H(L') - 1$.
        \item If $L$ is positive, then $s_H(L) = n(L) - r(L) + 1$.
    \end{enumerate}
\end{proposition}

\begin{proposition}
\label{prop:ca-under-cob}
    Let $S$ be an oriented cobordism between links, represented as a movie between two link diagrams $L, L'$, such that each component has boundary in $L$. Then the corresponding chain map 
    \[
        \phi_S\colon [L] \to [L']
    \]
    gives
    \begin{align*}
        \phi_S \circ \ca(L) &\htpy \epsilon H^j \ca(L') \\
        \phi_S \circ \cb(L) &\htpy \epsilon' H^j \cb(L')
    \end{align*}
    where $j \in \ZZ$ is given by 
    \[
        j = \frac{\delta w(L, L') - \delta r(L, L') - \chi(S)}{2}
    \]
    and the signs $\epsilon, \epsilon' \in \{\pm 1\}$ satisfy 
    \[
        \epsilon\epsilon' = (-1)^j.
    \]
\end{proposition}

\begin{corollary}
\label{cor:s_H-under-cob}
    Under the assumption of \Cref{prop:ca-under-cob}, we have
    \[
        s_H(L) \leq s_H(L') - \chi(S).
    \]
    Moreover if each component of $S$ has boundary in both $L$ and $L'$, then 
    \[
        |s_H(L) - s_H(L')| \leq -\chi(S).
    \]
\end{corollary}

\begin{corollary}
    $s_H$ is a link concordance invariant. 
\end{corollary}

\begin{remark}
    The signs $\epsilon, \epsilon'$ appearing in \Cref{prop:ca-under-reidemeister,prop:ca-under-cob} can be used fix the sign indeterminacy in the functoriality of Khovanov homology for links, as in \cite{Sano:2020-b}. Is it possible to apply the same technique for tangles? See also \Cref{rem:lee-class-inv-for-tangle}. 
\end{remark}

\subsection{The reduced invariant $\tilde{s}_H$}

The reduced invariant $\tilde{s}_H$ can be defined similarly, by instead considering $2$-end tangles. Let $T$ be a $2$-end tangle diagram, whose strand connecting the two boundary points is oriented from $(0, -1)$ to $(0, 1)$. Let $T^\lout$ denote the Seifert resolution of $T$. The \textit{standard checkerboard coloring} of $T^\lout$ is the one that colors the left half of $\partial D^2$ black, so that the special arc in $T^\lout$ is labeled $X$ by the labeling rule. This gives a dotted cobordism, 
\[
    \tilde{\ca}(T)\colon e \to T^\lout
\]
except we require no dots to be placed on the special component. Analogous to \Cref{prop:lee-cycle-is-cycle}, we see that this defines a chain map
\[
    \tilde{\ca}(T)\colon [e] \to [T].
\]

\begin{definition}
    The chain map $\tilde{\ca}(T)$ is called the \textit{Lee cycle} of $T$, and its chain homotopy class the \textit{Lee class} of $T$.
\end{definition}

\begin{definition}
    The (stable) $H$-divisibility of the Lee class of $T$ is defined as 
    \[
        \tilde{d}_H(T) = \max_k\{\
            \exists z\colon [e] \to [T],\ (H^k) z \htpy_s \tilde{\ca}(T)\ 
        \}.
    \]
\end{definition}

A proposition analogous to \Cref{prop:ca-under-cob} holds for the reduced Lee class, and hence the following quantity gives an invariant of $2$-end tangles. 

\begin{definition}
    The $2$-end tangle invariant $\tilde{s}_H$ is defined by 
    \[
        \tilde{s}_H(T) = 2\tilde{d}_H(T) + w(T) - r(T).
    \]
    Here, $r(T)$ denotes the number of Seifert circles of $T$, i.e.\ the number of components of $T^\lout$ excluding the special arc. 
\end{definition}

We remark that propositions analogous to \Cref{prop:s_H-properties,prop:ca-under-cob} also hold for the reduced invariant. See \cite{Sano-Sato:2023} for further details. 

\subsection{Recovering the $s$-invariant}

\Cref{def:s_H} is formally identical to the link invariant $s_h$ defined in \cite{Sano:2020}. Indeed, we show that $s_h$, and in particular $s$, can be reformulated in the current context. Suppose we are given a commutative ring $S$ with a fixed element $h \in S$, together with a ring homomorphism $\phi\colon R \to S$ with $\phi(H) = h$. We may turn $\Cob_{\bullet/l}(B)$ into an $S$-linear category, denoted $\Cob_{\bullet/l}(B; S)$, by tensoring $S$ to the hom-modules of $\Cob_{\bullet/l}(B)$ over $R$. Note that the graded structure of $\Cob_{\bullet/l}(B)$ may collapse unless $S$ is graded and $\deg h = -2$. The corresponding category of complexes is denoted $\Kob_{\bullet/l}(B; S)$, and the Lee cycle $\ca(L; S)$ over $S$ is similarly defined as a chain map in $\Kob_{\bullet}(\underline{0}; S)$. Its $h$-divisibility is defined as 
\[
    d_h(L; S) = \max_k\{\ \exists z\colon [\varnothing] \to [L],\ (H^k) z \htpy_s \ca(L; S)\ \},
\]
and the invariant $s_h$ is given by
\[
    s_h(L; S) = 2d_h(L; S) + w(L) - r(L) + 1.
\]
Here, both $d_h$ and $s_h$ are allowed to be infinite. 

\begin{proposition}
\label{prop:d_h-by-homology}
    If $S$ is an integral domain and $h$ is non-zero, the above defined quantities $d_h$ and $s_h$ coincide with those defined in \cite{Sano:2020}. 
\end{proposition}

\begin{proof}
    For a cycle $z \in \CKh_h(L; S)$ and an integer $k \geq 0$, the following are equivalent:
    \begin{align*}
        \ca(L; S) \htpy_s h^k z 
        &\Leftrightarrow \exists l \geq 0,\ h^l \ca(L; S) \htpy h^l h^k z \\
        &\Leftrightarrow \exists l \geq 0,\ h^l [\ca(L; S)] = h^l h^k [z] \\
        &\Leftrightarrow [\ca(L; S)] \equiv h^k [z] \pmod{h \text{-}\Tor} \\
        &\Leftrightarrow [\ca(L; S)] \equiv h^k [z] \pmod{\Tor}
    \end{align*}
    Here, $[z]$ denotes the chain homotopy class (or equivalently the homology class) of $z$. The final equivalence follows from \cite[Corollary 2.10]{Sano:2020} which states that every torsion of $\Kh_h(T; S)$ is an $h$-torsion. 
\end{proof}

In particular, when $(S, h) = (F[H], H)$ for any field $F$, it is proved in \cite[Theorem 3]{Sano:2020} (for $\fchar F \neq 2$) and in \cite[Theorem 2]{Sano-Sato:2023} (including $\fchar F = 2$) that the invariant $s_h$ coincides with the $s$-invariant over $F$. In this case, the category $\Cob_{\bullet/l}(B; S)$ is obtained by tensoring $F$ to the hom-modules of $\Cob_{\bullet/l}(B)$ over $\ZZ$, so instead we write $\Cob_{\bullet/l}(B; F)$ and similarly for the other notations. We conclude, 

\begin{restate-theorem}[thm:s-reformulate]
    For a knot diagram $K$, the invariant
    \[
        s_H(K; F) = 2d_H(K; F) + w(K) - r(K) + 1
    \]
    coincides with the Rasmussen invariant $s$ over $F$. 
\end{restate-theorem}

The invariant $s_H = s_H(-; \ZZ)$ is universal in the following sense.

\begin{proposition}
\label{prop:s_H-universal}
    For any pair $(S, h)$, the following inequalities hold
    \[
        d_H \leq d_h,\quad s_H \leq s_h.
    \]
\end{proposition}

\begin{remark}
    The inequality of \Cref{prop:s_H-universal} was used in \cite{ISST:2024} to prove that the Rasmussen invariants $s^F$ of $K = 9_{46}$ among all fields $F$ are equal to $0$, by showing that $s_H(K) = s_H(K^*) = 0$. In general, if a knot $K$ satisfies $s_H(K) = -s_H(K^*)$, then it follows that all $s^F(K)$ are equal to $s_H(K)$. This is not always the case, for there are knots such that $s^\QQ \neq s^{\FF_2}$.
\end{remark}

Again, analogous statements hold for the reduced versions. In particular, the reduced invariant for $2$-end tangles
\[
    \tilde{s}_H(T; F) = \tilde{d}_H(T; F) + w(T) - r(T)
\]
coincide with the Rasmussen invariant of the closure $\bar{T}$ of $T$. In general, the reduced setting simplifies the situation for knots: when $S$ is a PID and $h$ is prime in $S$, the reduced homology $\widetilde{\Kh}_h(K; S)$ of a knot $K$ has rank $1$ and the invariant $\tilde{s}_h$ is \textit{slice-torus}. In contrast, the unreduced homology $\Kh_h(K; S)$ has rank $2$ and the invariant $s_h$ is not necessarily slice-torus, as exemplified by the case $(S, h) = (\ZZ, 3)$. See \cite{Sano-Sato:2023} for further details. 

\subsection{Lee cycle for tangle diagrams}
\label{subsec:lee-cycle-tangles}

Next, we extend the definition of Lee cycles to tangle diagrams. 
% , so that the Lee cycle of a link diagram can be reconstructed from those of its tangle pieces in a manageable way. 
First, consider the Frobenius extension $(\bar{R}, \bar{A})$ given by $\bar{R} = \ZZ[H^\pm]$ and $\bar{A} = A \otimes_R \bar{R}$. Borrowing the idea from \cite{Khovanov:2022}, define idempotents 
\[
    e_X = X/H,\quad 
    e_Y = Y/(-H)
\]
in $\bar{A}$. Then we have
\begin{gather*}
    e_X^2 = 1,\quad e_Y^2 = 1,\quad e_X e_Y = 0, \\ 
    X e_X = X,\quad Y e_Y = Y,\quad Y e_X = X e_Y = 0.
\end{gather*}
Corresponding to the idempotents $e_X, e_Y$, the \textit{reduced dot} $\mathrlap{/}\bullet$ and the \textit{reduced hollow dot} $\mathrlap{/}\circ$ are defined as
\begin{center}
    \tikzset{every picture/.style={line width=0.75pt}} %set default line width to 0.75pt        

\begin{tikzpicture}[x=0.75pt,y=0.75pt,yscale=-.75,xscale=.75]
%uncomment if require: \path (0,67); %set diagram left start at 0, and has height of 67

%Shape: Rectangle [id:dp771364038701766] 
\draw  [dash pattern={on 0.84pt off 2.51pt}] (10.93,11.56) -- (56.56,11.56) -- (56.56,55.49) -- (10.93,55.49) -- cycle ;
%Shape: Rectangle [id:dp9274977007582983] 
\draw  [dash pattern={on 0.84pt off 2.51pt}] (128.41,12.06) -- (175.08,12.06) -- (175.08,56.99) -- (128.41,56.99) -- cycle ;
%Shape: Rectangle [id:dp7095875967515791] 
\draw  [dash pattern={on 0.84pt off 2.51pt}] (213.95,13.06) -- (258.54,13.06) -- (258.54,55.99) -- (213.95,55.99) -- cycle ;
%Shape: Circle [id:dp8298032122243446] 
\draw  [fill={rgb, 255:red, 0; green, 0; blue, 0 }  ,fill opacity=1 ] (148.54,34.52) .. controls (148.54,32.75) and (149.97,31.32) .. (151.74,31.32) .. controls (153.51,31.32) and (154.95,32.75) .. (154.95,34.52) .. controls (154.95,36.29) and (153.51,37.73) .. (151.74,37.73) .. controls (149.97,37.73) and (148.54,36.29) .. (148.54,34.52) -- cycle ;
%Shape: Circle [id:dp9367598816280328] 
\draw  [color={rgb, 255:red, 0; green, 0; blue, 0 }  ,draw opacity=1 ][fill={rgb, 255:red, 255; green, 255; blue, 255 }  ,fill opacity=1 ] (233.04,34.52) .. controls (233.04,32.75) and (234.47,31.32) .. (236.24,31.32) .. controls (238.01,31.32) and (239.45,32.75) .. (239.45,34.52) .. controls (239.45,36.29) and (238.01,37.73) .. (236.24,37.73) .. controls (234.47,37.73) and (233.04,36.29) .. (233.04,34.52) -- cycle ;
%Shape: Circle [id:dp4395547045715563] 
\draw  [fill={rgb, 255:red, 0; green, 0; blue, 0 }  ,fill opacity=1 ] (30.54,33.52) .. controls (30.54,31.75) and (31.97,30.32) .. (33.74,30.32) .. controls (35.51,30.32) and (36.95,31.75) .. (36.95,33.52) .. controls (36.95,35.29) and (35.51,36.73) .. (33.74,36.73) .. controls (31.97,36.73) and (30.54,35.29) .. (30.54,33.52) -- cycle ;
%Shape: Rectangle [id:dp732979996168103] 
\draw  [dash pattern={on 0.84pt off 2.51pt}] (346.41,10.91) -- (393.08,10.91) -- (393.08,55.84) -- (346.41,55.84) -- cycle ;
%Shape: Circle [id:dp8068057692973063] 
\draw  [color={rgb, 255:red, 0; green, 0; blue, 0 }  ,draw opacity=1 ][fill={rgb, 255:red, 255; green, 255; blue, 255 }  ,fill opacity=1 ] (366.54,33.38) .. controls (366.54,31.61) and (367.97,30.17) .. (369.74,30.17) .. controls (371.51,30.17) and (372.95,31.61) .. (372.95,33.38) .. controls (372.95,35.15) and (371.51,36.58) .. (369.74,36.58) .. controls (367.97,36.58) and (366.54,35.15) .. (366.54,33.38) -- cycle ;
%Straight Lines [id:da47558233870033517] 
\draw    (40,28.06) -- (28.54,39.52) ;
%Straight Lines [id:da4693965240871042] 
\draw    (242.5,28.06) -- (231.04,39.52) ;

% Text Node
\draw (66,25.32) node [anchor=north west][inner sep=0.75pt]    {$=$};
% Text Node
\draw (189,26.32) node [anchor=north west][inner sep=0.75pt]    {$,$};
% Text Node
\draw (90,24.24) node [anchor=north west][inner sep=0.75pt]    {$H^{-1}$};
% Text Node
\draw (282.77,31.77) node    {$=$};
% Text Node
\draw (295,23.1) node [anchor=north west][inner sep=0.75pt]    {$-H^{-1}$};

\end{tikzpicture}
\end{center}

\begin{figure}[t]
    \centering
    \begin{subfigure}[t]{0.25\linewidth}
        \centering
        \tikzset{every picture/.style={line width=0.75pt}} %set default line width to 0.75pt        

\begin{tikzpicture}[x=0.75pt,y=0.75pt,yscale=-.9,xscale=.9]
%uncomment if require: \path (0,107); %set diagram left start at 0, and has height of 107

\clip (0,0) rectangle + (120, 110);

%Curve Lines [id:da15326191008464274] 
\draw [line width=1.5]    (84.58,38.39) .. controls (77.5,-21) and (18,22.5) .. (56.36,74.01) ;
%Curve Lines [id:da9013974696883018] 
\draw [line width=1.5]    (67.15,83.58) .. controls (108,112) and (137,29) .. (49.5,42.5) ;
%Curve Lines [id:da7121874323939582] 
\draw [line width=1.5]    (39.13,43.95) .. controls (1.7,51.35) and (29.68,138.29) .. (83.06,51.44) ;
%Straight Lines [id:da7418979216200584] 
\draw    (54.5,13) -- (48.5,13) ;
\draw [shift={(45.5,13)}, rotate = 360] [fill={rgb, 255:red, 0; green, 0; blue, 0 }  ][line width=0.08]  [draw opacity=0] (8.93,-4.29) -- (0,0) -- (8.93,4.29) -- cycle    ;
%Shape: Donut [id:dp17448941596299938] 
\draw  [color={rgb, 255:red, 255; green, 255; blue, 255 }  ,draw opacity=1 ][fill={rgb, 255:red, 255; green, 255; blue, 255 }  ,fill opacity=1 ,even odd rule] (21,56) .. controls (21,32.8) and (39.8,14) .. (63,14) .. controls (86.2,14) and (105,32.8) .. (105,56) .. controls (105,79.2) and (86.2,98) .. (63,98) .. controls (39.8,98) and (21,79.2) .. (21,56)(11,56) .. controls (11,27.28) and (34.28,4) .. (63,4) .. controls (91.72,4) and (115,27.28) .. (115,56) .. controls (115,84.72) and (91.72,108) .. (63,108) .. controls (34.28,108) and (11,84.72) .. (11,56) ;
%Shape: Circle [id:dp6869574614495939] 
\draw  [dash pattern={on 4.5pt off 4.5pt}] (22.24,56) .. controls (22.24,33.49) and (40.49,15.24) .. (63,15.24) .. controls (85.51,15.24) and (103.76,33.49) .. (103.76,56) .. controls (103.76,78.51) and (85.51,96.76) .. (63,96.76) .. controls (40.49,96.76) and (22.24,78.51) .. (22.24,56) -- cycle ;

%Shape: Circle [id:dp7009796048638841] 
\draw  [fill={rgb, 255:red, 0; green, 0; blue, 0 }  ,fill opacity=1 ] (47.4,17.4) .. controls (47.4,16.52) and (48.12,15.8) .. (49,15.8) .. controls (49.88,15.8) and (50.6,16.52) .. (50.6,17.4) .. controls (50.6,18.28) and (49.88,19) .. (49,19) .. controls (48.12,19) and (47.4,18.28) .. (47.4,17.4) -- cycle ;
%Shape: Circle [id:dp4029866274809353] 
\draw  [fill={rgb, 255:red, 0; green, 0; blue, 0 }  ,fill opacity=1 ] (76.6,17.6) .. controls (76.6,16.72) and (77.32,16) .. (78.2,16) .. controls (79.08,16) and (79.8,16.72) .. (79.8,17.6) .. controls (79.8,18.48) and (79.08,19.2) .. (78.2,19.2) .. controls (77.32,19.2) and (76.6,18.48) .. (76.6,17.6) -- cycle ;
%Shape: Circle [id:dp4629788763173265] 
\draw  [fill={rgb, 255:red, 0; green, 0; blue, 0 }  ,fill opacity=1 ] (102.4,56.6) .. controls (102.4,55.72) and (103.12,55) .. (104,55) .. controls (104.88,55) and (105.6,55.72) .. (105.6,56.6) .. controls (105.6,57.48) and (104.88,58.2) .. (104,58.2) .. controls (103.12,58.2) and (102.4,57.48) .. (102.4,56.6) -- cycle ;
%Shape: Circle [id:dp16114814790833754] 
\draw  [fill={rgb, 255:red, 0; green, 0; blue, 0 }  ,fill opacity=1 ] (87,88.6) .. controls (87,87.72) and (87.72,87) .. (88.6,87) .. controls (89.48,87) and (90.2,87.72) .. (90.2,88.6) .. controls (90.2,89.48) and (89.48,90.2) .. (88.6,90.2) .. controls (87.72,90.2) and (87,89.48) .. (87,88.6) -- cycle ;
%Shape: Circle [id:dp020126703215711195] 
\draw  [fill={rgb, 255:red, 0; green, 0; blue, 0 }  ,fill opacity=1 ] (19.8,61.4) .. controls (19.8,60.52) and (20.52,59.8) .. (21.4,59.8) .. controls (22.28,59.8) and (23,60.52) .. (23,61.4) .. controls (23,62.28) and (22.28,63) .. (21.4,63) .. controls (20.52,63) and (19.8,62.28) .. (19.8,61.4) -- cycle ;
%Shape: Circle [id:dp16895331144783732] 
\draw  [fill={rgb, 255:red, 0; green, 0; blue, 0 }  ,fill opacity=1 ] (36.2,89.4) .. controls (36.2,88.52) and (36.92,87.8) .. (37.8,87.8) .. controls (38.68,87.8) and (39.4,88.52) .. (39.4,89.4) .. controls (39.4,90.28) and (38.68,91) .. (37.8,91) .. controls (36.92,91) and (36.2,90.28) .. (36.2,89.4) -- cycle ;

\end{tikzpicture}
        \caption{}
        \label{fig:tangle-lee-cycle-a}
    \end{subfigure}
    \begin{subfigure}[t]{0.25\linewidth}
        \centering
        \tikzset{every picture/.style={line width=0.75pt}} %set default line width to 0.75pt        

\begin{tikzpicture}[x=0.75pt,y=0.75pt,yscale=-.9,xscale=.9]
%uncomment if require: \path (0,107); %set diagram left start at 0, and has height of 107

\clip (0,0) rectangle + (120, 110);

%Shape: Polygon Curved [id:ds17713549162592357] 
\draw  [color={rgb, 255:red, 208; green, 2; blue, 27 }  ,draw opacity=1 ][line width=0.75]  (59.5,7) .. controls (72.5,6) and (77.5,14) .. (79.5,31) .. controls (81.5,48) and (89.27,38.04) .. (99.5,49) .. controls (109.73,59.96) and (108.5,68) .. (106.5,77) .. controls (104.5,86) and (92.5,94) .. (75.5,85) .. controls (58.5,76) and (65.5,77) .. (50.5,86) .. controls (35.5,95) and (26.5,85) .. (20.5,78) .. controls (14.5,71) and (15.5,56) .. (22.5,51) .. controls (29.5,46) and (40.5,46) .. (42.5,31) .. controls (44.5,16) and (46.5,8) .. (59.5,7) -- cycle ;
%Shape: Polygon Curved [id:ds220122662274978] 
\draw  [color={rgb, 255:red, 0; green, 116; blue, 255 }  ,draw opacity=1 ][line width=1.5]  (59.5,38) .. controls (71.5,38) and (85,35) .. (78.5,57) .. controls (72,79) and (55,81) .. (45,61) .. controls (35,41) and (47.5,38) .. (59.5,38) -- cycle ;
%Shape: Donut [id:dp3569216048282726] 
\draw  [color={rgb, 255:red, 255; green, 255; blue, 255 }  ,draw opacity=1 ][fill={rgb, 255:red, 255; green, 255; blue, 255 }  ,fill opacity=1 ,even odd rule] (17.38,52) .. controls (17.38,27.91) and (36.91,8.38) .. (61,8.38) .. controls (85.09,8.38) and (104.62,27.91) .. (104.62,52) .. controls (104.62,76.09) and (85.09,95.62) .. (61,95.62) .. controls (36.91,95.62) and (17.38,76.09) .. (17.38,52)(7,52) .. controls (7,22.18) and (31.18,-2) .. (61,-2) .. controls (90.82,-2) and (115,22.18) .. (115,52) .. controls (115,81.82) and (90.82,106) .. (61,106) .. controls (31.18,106) and (7,81.82) .. (7,52) ;
%Shape: Ellipse [id:dp16911967993067456] 
\draw  [dash pattern={on 4.5pt off 4.5pt}] (18.68,52) .. controls (18.68,28.63) and (37.63,9.68) .. (61,9.68) .. controls (84.37,9.68) and (103.32,28.63) .. (103.32,52) .. controls (103.32,75.37) and (84.37,94.32) .. (61,94.32) .. controls (37.63,94.32) and (18.68,75.37) .. (18.68,52) -- cycle ;

%Shape: Ellipse [id:dp016289843171898633] 
\draw  [fill={rgb, 255:red, 0; green, 0; blue, 0 }  ,fill opacity=1 ] (44.8,11.92) .. controls (44.8,11) and (45.54,10.25) .. (46.46,10.25) .. controls (47.38,10.25) and (48.12,11) .. (48.12,11.92) .. controls (48.12,12.83) and (47.38,13.58) .. (46.46,13.58) .. controls (45.54,13.58) and (44.8,12.83) .. (44.8,11.92) -- cycle ;
%Shape: Ellipse [id:dp13230856704649308] 
\draw  [fill={rgb, 255:red, 0; green, 0; blue, 0 }  ,fill opacity=1 ] (75.12,12.12) .. controls (75.12,11.21) and (75.87,10.46) .. (76.78,10.46) .. controls (77.7,10.46) and (78.45,11.21) .. (78.45,12.12) .. controls (78.45,13.04) and (77.7,13.78) .. (76.78,13.78) .. controls (75.87,13.78) and (75.12,13.04) .. (75.12,12.12) -- cycle ;
%Shape: Ellipse [id:dp0018608289945439838] 
\draw  [fill={rgb, 255:red, 0; green, 0; blue, 0 }  ,fill opacity=1 ] (101.92,52.62) .. controls (101.92,51.71) and (102.66,50.96) .. (103.58,50.96) .. controls (104.49,50.96) and (105.24,51.71) .. (105.24,52.62) .. controls (105.24,53.54) and (104.49,54.28) .. (103.58,54.28) .. controls (102.66,54.28) and (101.92,53.54) .. (101.92,52.62) -- cycle ;
%Shape: Ellipse [id:dp07967166802575776] 
\draw  [fill={rgb, 255:red, 0; green, 0; blue, 0 }  ,fill opacity=1 ] (85.92,85.85) .. controls (85.92,84.94) and (86.67,84.19) .. (87.58,84.19) .. controls (88.5,84.19) and (89.25,84.94) .. (89.25,85.85) .. controls (89.25,86.77) and (88.5,87.52) .. (87.58,87.52) .. controls (86.67,87.52) and (85.92,86.77) .. (85.92,85.85) -- cycle ;
%Shape: Ellipse [id:dp6053029952522743] 
\draw  [fill={rgb, 255:red, 0; green, 0; blue, 0 }  ,fill opacity=1 ] (16.14,57.61) .. controls (16.14,56.69) and (16.88,55.95) .. (17.8,55.95) .. controls (18.72,55.95) and (19.46,56.69) .. (19.46,57.61) .. controls (19.46,58.53) and (18.72,59.27) .. (17.8,59.27) .. controls (16.88,59.27) and (16.14,58.53) .. (16.14,57.61) -- cycle ;
%Shape: Ellipse [id:dp13145830481472043] 
\draw  [fill={rgb, 255:red, 0; green, 0; blue, 0 }  ,fill opacity=1 ] (33.17,86.68) .. controls (33.17,85.77) and (33.91,85.02) .. (34.83,85.02) .. controls (35.75,85.02) and (36.49,85.77) .. (36.49,86.68) .. controls (36.49,87.6) and (35.75,88.35) .. (34.83,88.35) .. controls (33.91,88.35) and (33.17,87.6) .. (33.17,86.68) -- cycle ;

% Text Node
\draw (79.5,55.4) node [anchor=north west][inner sep=0.75pt]  [color={rgb, 255:red, 0; green, 117; blue, 255 }  ,opacity=1 ]  {$Y$};
% Text Node
\draw (24,26.4) node [anchor=north west][inner sep=0.75pt]  [font=\footnotesize]  {$\textcolor[rgb]{0.82,0.01,0.11}{e}\textcolor[rgb]{0.82,0.01,0.11}{_{X}}$};
% Text Node
\draw (83,25.4) node [anchor=north west][inner sep=0.75pt]  [font=\footnotesize]  {$\textcolor[rgb]{0.82,0.01,0.11}{e}\textcolor[rgb]{0.82,0.01,0.11}{_{X}}$};
% Text Node
\draw (57,79.4) node [anchor=north west][inner sep=0.75pt]  [font=\footnotesize]  {$\textcolor[rgb]{0.82,0.01,0.11}{e}\textcolor[rgb]{0.82,0.01,0.11}{_{X}}$};

\end{tikzpicture}
        \caption{}
        \label{fig:tangle-lee-cycle-b}
    \end{subfigure}
    \begin{subfigure}[t]{0.45\linewidth}
        \centering
        \input{tikzpictures/tangle-lee-cycle-c}
        \vspace{-1.1em} % wtf?
        \caption{}
        \label{fig:tangle-lee-cycle-c}
    \end{subfigure}
    \caption{The Lee cycle for a tangle diagram.}
    \label{fig:tangle-lee-cycle}    
\end{figure}

Let $T$ be a tangle diagram with $\partial T = B$. Let $T^\lout$ denote the Seifert resolution of $T$ as before, and $T^\lin$ the diagram obtained by removing all of the closed components of $T^\lout$. Note that both $T^\lin$ and $T^\lout$ are objects of $\Cob(B)$, and when $B = \underline{0}$ we have $T^\lin = \varnothing$. Let $S_T$ be the cobordism from $T^\lin$ to $T^\lout$ consisting of a cup $D$ for each circle $\gamma \subset T^\lout$ such that $\del D = \gamma$, and a sheet $\gamma \times I$ for each arc $\gamma \subset T^\lin$. Given a checkerboard coloring $c$ for $T^\lout$, the $XY$-labeling on $T^\lout$ is obtained as before, except that the arc components are labeled by $e_X$ or $e_Y$ instead of $X$ or $Y$ (see \Cref{fig:tangle-lee-cycle-a,fig:tangle-lee-cycle-b}). This gives a dotted cobordism
\[
    \ca(T, c)\colon T^\lin \to T^\lout
\]
in $\Cob^3_{\bullet}(B; \bar{R})$, by decorating each component $S_i$ of $S_T$ according to the labeling on $T^\lout$ (see \Cref{fig:tangle-lee-cycle-c}). Similar to \Cref{prop:lee-cycle-is-cycle}, it can be proved that this defines a chain map 
\[
    \ca(T, c)\colon [T^\lin] \to [T]
\] 
in $\Kob^3_{\bullet}(B; \bar{R})$.

\begin{definition}
\label{def:lee-cycle-tangle}
    The chain map $\ca(T, c) \in \Hom_{\Kob^3_{\bullet}(B; \bar{R})}([T^\lin], [T])$ is called the \textit{Lee cycle} of $T$, with respect to the checkerboard coloring $c$.
\end{definition}

Now, suppose $T$ can be expressed as $T = D(T_1, \ldots, T_d)$ with a $d$-input planar arc diagram $D$ and smaller tangle diagrams $T_i$. Note that we have $T^\lout = D({T_1}^\lout, \ldots, {T_d}^\lout)$, but not necessarily $T^\lin = D({T_1}^\lin, \ldots, {T_d}^\lin)$, since some of the arcs of ${T_i}^\lin$ may be connected by $D$ to form new circles. Let $T_0$ denote the horizontal composition $D({T_1}^\lin, \ldots, {T_d}^\lin)$. The underlying cobordism $S_T$ of $\ca(T, c)$ can be decomposed as 
\[
\begin{tikzcd}[column sep=5em]
    T^\lin 
        \arrow[r, "S_{T_0}"] & 
    T_0 
        \arrow[r, "{D(S_{T_1}, \ldots, S_{T_d})}"] & 
    T^\lout
\end{tikzcd}    
\]
with (undotted) cobordisms $S_{T_i}$ corresponding to the diagrams $T_i$. A checkerboard coloring $c$ for $T$ induces a checkerboard coloring $c_i$ for each $T_i$. Then we obtain the following decomposition of $\ca(T, c)$ (see \Cref{fig:intro-lee-decomp} in \Cref{sec:intro}). 

\begin{restate-theorem}[thm:lee-cycle-decomp]
\[
    \ca(T, c) = D(\ca(T_1, c_1), \ldots, \ca(T_d, c_d)) \circ \ca(T_0, c_0).
\]    
\end{restate-theorem}

In particular, when $T$ is a link diagram and $c$ is the standard checkerboard coloring, the induced checkerboard colorings $c_i$ will be ommitted from the notations. 

\begin{remark}
\label{rem:lee-class-inv-for-tangle}
    One may expect that, using the Lee class for tangle diagrams, an analogue of \Cref{prop:lee-class-invariance} can be proved and that the invariant $s_H$ can be extended to tangles. However, the initial diagram $T^\lin$ is generally not invariant under the Reidemeister moves, so we cannot directly compare the two chain maps $\ca(T)$ and $\ca(T')$. See Case 2 for the R2 move in \Cref{sec:appendix}.
    \[
        \begin{tikzcd}[row sep=3em, column sep=5em]
        {[T^\lin]} 
            \arrow[d, dotted, no head, "?"] 
            \arrow[r, "\ca(T)"] 
        & {[T]} 
            \arrow[d, "\htpy"]
        \\ {[T'{}^\lin]} 
            \arrow[r, "\ca(T')"]
        & {[T']}              
        \end{tikzcd}
    \]
\end{remark}

\begin{remark}
    The relation (R1) in the reduced setting could be understood that, in $\Cob^3_\bullet(\underline{2})$, it is required to put a reduced dot $\mathrlap{/}\bullet$ on the special component. This makes the definition of the reduced Lee cycle for $2$-end tangles consistent with \Cref{def:lee-cycle-tangle}.
\end{remark}

Analogous to \Cref{prop:qdeg-lee-cycle}, the following proposition holds. 

\begin{proposition}
\label{prop:qdeg-lee-cycle-tangle}
    For a tangle diagram $T$, 
    \[
        \qdeg \ca(T) = w(T) - r(T).
    \]
    Here, $w(T)$ denotes the writhe of $T$ and $r(T)$ the number of circles in $T^\lout$.
\end{proposition}

% \begin{proof}
%     Let $a$ be the number of arcs of $T^\lout$, which is equal to $\frac{|B|}{2}$. With $\qdeg({\not\bullet}) = \qdeg({\not\circ}) = 0$, we have  
%     \[
%         \qdeg \ca(T) 
%         = w + r(\chi(D^2) - 2) + a(\chi(\gamma \times I)) - \frac{|B|}{2} = w - r.
%     \]
% \end{proof}

Since $\qdeg$ is additive under both horizontal and vertical compositions of tangle diagrams, under the assumption of \Cref{thm:lee-cycle-decomp}, we have
\[
    \qdeg(\ca(T)) = \sum_{i = 0}^d \qdeg(\ca(T_i))
\]
which can also be confirmed using \Cref{prop:qdeg-lee-cycle-tangle} with
\[
    w(T) = \sum_{i = 1}^d w(T_i),\ 
    r(T) = \sum_{i = 0}^d r(T_i). 
\]

\subsection{Computing $s$ via tangle decompositions}

\begin{figure}[t]
    \centering
\[
\begin{tikzcd} [column sep = 4.5em, row sep = 2.5em]
& {[T_i^\lin]} \arrow[r, "\ca(T_i)"] \arrow[rr, "H^{k_i} z_i", bend left] \arrow[d, "D", dashed] & {[T_i]} \arrow[r, "f_i", "\htpy"'] \arrow[d, "D", dashed] & E_i \arrow[d, "D", dashed] & \\
\varnothing \arrow[r, "\ca(T_0)"] \arrow[rr, "\ca(L)", bend right] \arrow[rrrr, "H^{k+l} z", bend right = 35] & {D([T_i^\lin])} \arrow[r, "{D(\ca(T_i))}"] \arrow[rr, "H^k D(z_i)", bend right] & {D([T_i])} \arrow[r, "f", "\htpy"'] & {D(E_i)} \arrow[r, "g", "\htpy"'] & E
\end{tikzcd}
\]
    \caption{Computing $s_H(L)$ from smaller tangle pieces}
    \label{fig:compute-s-strategy}
\end{figure}

\Cref{thm:lee-cycle-decomp} can be used to estimate, or sometimes determine, the invariant $s_H$ of a link diagram $L$ from smaller tangle pieces. \Cref{fig:compute-s-strategy} shows the general process of the computation, which we explain below. 

Suppose we are given a link diagram $L$ with a tangle decomposition
\[
    L = D(T_1, \ldots, T_d),
\]
such that each complex $[T_i]$ admits a reduction (as explained in \Cref{subsec:gaussian})
\[
    f_i\colon [T_i] \xrightarrow{\htpy} E_i
\]
under which the Lee cycle $\ca(T_i)$ transforms into 
\[
    f_i \circ \ca(T_i) = \pm H^{k_i} z_i
\]
for some cycle $z_i$ and $k_i \geq 0$.  
From \Cref{thm:lee-cycle-decomp}, we have 
\[
    \ca(L) = D(\ca(T_1), \ldots, \ca(T_d)) \circ \ca(T_0),
\]
and post-composing the horizontal composition
\[
    f = D(f_1, ..., f_d)
\]
to $\ca(L)$ gives
\[
    f \circ \ca(T) = \pm H^k z_0
\]
where
\[
    z_0 = D(z_1, \ldots, z_d) \circ \ca(T_0)
\]
and $k = \sum_i k_i$. Furthermore, if there is a reduction
\[
    g\colon D(E_1, \ldots, E_d) \to E 
\]
such that the cycle $z_0$ transforms into 
\[
    g \circ z_0 = \pm H^l z
\]
for some cycle $z$ and $l \geq 0$, then we have 
\[
    gf \circ \ca(T) = \pm H^{k + l} z.
\]
Since $f, g$ are homotopy equivalences, we have 
\[
    d_H(L) \geq k + l.
\]

If $E$ is simple enough so that we can additionally prove that $l$ is maximal, then we may conclude that $d_H = k + l$, and hence $s_H$ is determined. Furthermore, by applying homotopy inverses of $f, g$, we obtain an explicit cycle $z \in \CKh_H(L)$ that realizes
\[
    s_H(L) = \qdeg(z) + 1.
\]

Strictly speaking, the cycle $\ca(L)$ is a morphism in the category $\Kob^3_{\bullet}(\underline{0})$, whereas the cycles $\ca(T_i)$ are morphisms in $\Kob^3_{\bullet}(\partial T_i; \bar{R})$, where $H$ is inverted and $H$-divisibility is trivially infinite. However, this is not a problem, if the maps $f_i$ are defined over $R$ and the composite cycle $z_0$ belongs to $\Kob^3_{\bullet}(\underline{0})$, since 
\[
    f \circ \ca(T) = \pm H^k z_0
\]
is an equation in $\Kob^3_{\bullet}(\underline{0})$. More generally, as long as we are interested in $H$-divisibility up to stable equivalence, the commutativity of \Cref{fig:compute-s-strategy} need be only up to chain homotopy, as the following lemma shows.

\begin{lemma}
    Let $f, g$ be chain maps in $\Kob^3_{\bullet}(B)$ that are chain homotopic in $\Kob^3_{\bullet}(\underline{0}; \bar{R})$. Then $f, g$ are stably equivalent in $\Kob^3_{\bullet}(B)$.
\end{lemma}

\begin{proof}
    Suppose $f \htpy g$ over $\bar{R}$, i.e.\ there is a chain homotopy $h$ in $\Kob^3_{\bullet}(\underline{0}; \bar{R})$ such that
    \[
        f - g = dh + hd.
    \]
    Since $h$ is an $\bar{R}$-linear combination of cobordisms, there is an $k \geq 0$ such that $H^k h$ is a homotopy in $\Kob^3_{\bullet}(\underline{0})$. Thus $H^k f \htpy H^k g$ over $R$. 
\end{proof}

    \section{Twist tangles}
\label{sec:twist-tangles}

For $q \in \ZZ$, let $T_q$ denote the unoriented tangle diagram obtained by adding $|q|$ half-twists, positive or negative depending on the sign of $q$, to a pair of crossingless vertical strands. 
\begin{center}
    \tikzset{every picture/.style={line width=0.75pt}} %set default line width to 0.75pt        

\begin{tikzpicture}[x=0.75pt,y=0.75pt,yscale=-1,xscale=1]
%uncomment if require: \path (0,110); %set diagram left start at 0, and has height of 110

\begin{knot}[
  % draft mode=crossings,
  clip width=3pt,
  end tolerance=1pt,
  % consider self intersections
]

%Curve Lines [id:da015620136144045005] 
\strand [line width=1.5]    (91.05,10.83) .. controls (90.61,25.83) and (71.06,24.83) .. (71.33,40.16) .. controls (71.61,55.49) and (91.33,56.16) .. (91.33,70.83) .. controls (91.33,85.49) and (71.33,84.83) .. (70.67,100.16) ;
%Curve Lines [id:da4446089835951046] 
\strand [line width=1.5]    (72.24,10.16) .. controls (72.64,25.16) and (91.59,24.83) .. (91.33,40.16) .. controls (91.08,55.49) and (71.98,55.49) .. (71.98,70.16) .. controls (71.98,84.83) and (90.05,85.49) .. (90.67,100.83) ;

% \flipcrossings{1,3,5,7,9}

\end{knot}

%Shape: Rectangle [id:dp3643844221281882] 
\draw  [color={rgb, 255:red, 255; green, 255; blue, 255 }  ,draw opacity=1 ][fill={rgb, 255:red, 255; green, 255; blue, 255 }  ,fill opacity=1 ] (67,37.17) -- (95,37.17) -- (95,70.83) -- (67,70.83) -- cycle ;
%Shape: Brace [id:dp5566677176969166] 
\draw  [color={rgb, 255:red, 128; green, 128; blue, 128 }  ,draw opacity=1 ] (104,93) .. controls (108.67,93) and (111,90.67) .. (111,86) -- (111,65.73) .. controls (111,59.06) and (113.33,55.73) .. (118,55.73) .. controls (113.33,55.73) and (111,52.4) .. (111,45.73)(111,48.73) -- (111,25.45) .. controls (111,20.78) and (108.67,18.45) .. (104,18.45) ;

% Text Node
\draw (9,44.4) node [anchor=north west][inner sep=0.75pt]    {$T_{q} \ =$};
% Text Node
\draw (72.67,44.57) node [anchor=north west][inner sep=0.75pt]    {$\vdots $};
% Text Node
\draw (125,44.4) node [anchor=north west][inner sep=0.75pt]    {$q$};

\end{tikzpicture}
\end{center}
Let $T^\pm_q$ denote the tangle diagram $T_q$ equipped with the orientation given so that each crossing has the sign indicated by the superscript. The chain homotopy type of $[T_q^\pm]$ is well-known, as originally computed in \cite[Section 6.2]{Khovanov:2000} for its closure, and subsequently in \cite[Proposition 4.1]{Thompson:2018} and in \cite[Proposition 5.1]{Schuetz:2021} for the open case. Here, we review the explicit reduction, and also state how the Lee cycles correspond under the reduction. 

Let $\sfE_0$ and $\sfE_1$ denote the following unoriented $4$-end crossingless tangle diagrams, 
\begin{center}
    \tikzset{every picture/.style={line width=0.75pt}} %set default line width to 0.75pt        

\begin{tikzpicture}[x=0.75pt,y=0.75pt,yscale=-.75,xscale=.75]
%uncomment if require: \path (0,73); %set diagram left start at 0, and has height of 73

%Shape: Circle [id:dp6062736172006213] 
\draw  [dash pattern={on 4.5pt off 4.5pt}] (206,34.77) .. controls (206,20.71) and (217.4,9.31) .. (231.46,9.31) .. controls (245.52,9.31) and (256.91,20.71) .. (256.91,34.77) .. controls (256.91,48.83) and (245.52,60.22) .. (231.46,60.22) .. controls (217.4,60.22) and (206,48.83) .. (206,34.77) -- cycle ;
%Shape: Arc [id:dp993940871537017] 
\draw  [draw opacity=0][line width=1.5]  (251.56,20.25) .. controls (246.19,23.83) and (239.17,26) .. (231.48,26) .. controls (223.83,26) and (216.84,23.85) .. (211.48,20.3) -- (231.48,2.73) -- cycle ; \draw  [line width=1.5]  (251.56,20.25) .. controls (246.19,23.83) and (239.17,26) .. (231.48,26) .. controls (223.83,26) and (216.84,23.85) .. (211.48,20.3) ;  
%Shape: Arc [id:dp4792432370176144] 
\draw  [draw opacity=0][line width=1.5]  (211.4,49.29) .. controls (216.77,45.71) and (223.79,43.54) .. (231.48,43.54) .. controls (239.13,43.54) and (246.12,45.69) .. (251.48,49.24) -- (231.48,66.81) -- cycle ; \draw  [line width=1.5]  (211.4,49.29) .. controls (216.77,45.71) and (223.79,43.54) .. (231.48,43.54) .. controls (239.13,43.54) and (246.12,45.69) .. (251.48,49.24) ;  

%Shape: Circle [id:dp4022115198378231] 
\draw  [dash pattern={on 4.5pt off 4.5pt}] (86.48,9.29) .. controls (100.54,9.29) and (111.94,20.69) .. (111.94,34.75) .. controls (111.94,48.81) and (100.54,60.21) .. (86.48,60.21) .. controls (72.42,60.21) and (61.02,48.81) .. (61.02,34.75) .. controls (61.02,20.69) and (72.42,9.29) .. (86.48,9.29) -- cycle ;
%Shape: Arc [id:dp6664921835052046] 
\draw  [draw opacity=0][line width=1.5]  (101,54.85) .. controls (97.42,49.48) and (95.25,42.46) .. (95.25,34.77) .. controls (95.25,27.12) and (97.4,20.13) .. (100.95,14.77) -- (118.52,34.77) -- cycle ; \draw  [line width=1.5]  (101,54.85) .. controls (97.42,49.48) and (95.25,42.46) .. (95.25,34.77) .. controls (95.25,27.12) and (97.4,20.13) .. (100.95,14.77) ;  
%Shape: Arc [id:dp993600306644113] 
\draw  [draw opacity=0][line width=1.5]  (71.96,14.69) .. controls (75.54,20.06) and (77.71,27.08) .. (77.71,34.77) .. controls (77.71,42.42) and (75.56,49.41) .. (72.01,54.77) -- (54.43,34.77) -- cycle ; \draw  [line width=1.5]  (71.96,14.69) .. controls (75.54,20.06) and (77.71,27.08) .. (77.71,34.77) .. controls (77.71,42.42) and (75.56,49.41) .. (72.01,54.77) ;

% Text Node
\draw (7,22.4) node [anchor=north west][inner sep=0.75pt]    {$\mathsf{E}_{0} \ =\ $};
% Text Node
\draw (149,23.4) node [anchor=north west][inner sep=0.75pt]    {$\mathsf{E}_{1} \ =\ $};
% Text Node
\draw (121,24) node [anchor=north west][inner sep=0.75pt]   [align=left] {,};

\end{tikzpicture}
\end{center}
Consider the degree $-2$ endomorphisms $a, b$ on $\sfE_1$ given by 
\begin{center}
    \tikzset{every picture/.style={line width=0.75pt}} %set default line width to 0.75pt        

\begin{tikzpicture}[x=0.75pt,y=0.75pt,yscale=-.8,xscale=.8]
%uncomment if require: \path (0,185); %set diagram left start at 0, and has height of 185

%Curve Lines [id:da8954709680110604] 
\draw [color={rgb, 255:red, 0; green, 0; blue, 0 }  ,draw opacity=1 ][line width=1.5]    (161.96,30.56) .. controls (162,29.25) and (166,32.75) .. (171.14,24.74) ;
%Curve Lines [id:da8254061701221355] 
\draw [color={rgb, 255:red, 0; green, 0; blue, 0 }  ,draw opacity=1 ][line width=1.5]    (153.5,60.56) .. controls (157.5,60.75) and (160.23,59.24) .. (166.32,74.52) ;
%Shape: Ellipse [id:dp55863236365411] 
\draw  [dash pattern={on 4.5pt off 4.5pt}] (162.5,9.33) .. controls (170.05,13.59) and (174.83,32.78) .. (173.18,52.2) .. controls (171.52,71.62) and (164.06,83.92) .. (156.5,79.66) .. controls (148.95,75.41) and (144.17,56.21) .. (145.82,36.79) .. controls (147.48,17.37) and (154.94,5.07) .. (162.5,9.33) -- cycle ;
%Curve Lines [id:da9835614185112381] 
\draw [color={rgb, 255:red, 0; green, 0; blue, 0 }  ,draw opacity=1 ][line width=1.5]    (76.96,13.56) .. controls (82.84,31.28) and (87.32,34.82) .. (96.14,24.74) ;
%Curve Lines [id:da5320450186635435] 
\draw [color={rgb, 255:red, 0; green, 0; blue, 0 }  ,draw opacity=1 ][line width=1.5]    (72.14,63.34) .. controls (80.59,57.65) and (85.23,59.24) .. (91.32,74.52) ;
%Shape: Ellipse [id:dp3792935456637039] 
\draw  [dash pattern={on 4.5pt off 4.5pt}] (87.5,9.33) .. controls (95.05,13.59) and (99.83,32.78) .. (98.18,52.2) .. controls (96.52,71.62) and (89.06,83.92) .. (81.5,79.66) .. controls (73.95,75.41) and (69.17,56.21) .. (70.82,36.79) .. controls (72.48,17.37) and (79.94,5.07) .. (87.5,9.33) -- cycle ;
%Straight Lines [id:da18875620193778764] 
\draw [line width=0.75]    (171.14,24.74) -- (96.14,24.74) ;
%Straight Lines [id:da67944913575543] 
\draw [line width=0.75]    (82.5,63.34) -- (72.14,63.34) ;
%Straight Lines [id:da33419778913212794] 
\draw [line width=0.75]    (153.5,60.56) -- (81.96,60.56) ;
%Straight Lines [id:da17152257913582103] 
\draw [line width=0.75]    (166.32,74.52) -- (91.32,74.52) ;
%Straight Lines [id:da08889557072160403] 
\draw [line width=0.75]    (161.96,30.56) -- (86.96,30.56) ;
%Straight Lines [id:da7416509830719938] 
\draw [line width=0.75]    (151.96,13.56) -- (76.96,13.56) ;
%Straight Lines [id:da42058599959712795] 
\draw [line width=0.75]    (118.46,30.75) -- (118.46,60.31) ;
%Straight Lines [id:da564261367124538] 
\draw [line width=0.75]    (126.96,30.75) -- (126.96,60.31) ;
%Curve Lines [id:da9535545082036532] 
\draw [color={rgb, 255:red, 0; green, 0; blue, 0 }  ,draw opacity=1 ][line width=1.5]    (151.96,13.56) .. controls (153.15,17.14) and (154.28,20.14) .. (155.4,22.57) ;
%Curve Lines [id:da5773154743456222] 
\draw [color={rgb, 255:red, 0; green, 0; blue, 0 }  ,draw opacity=1 ][line width=1.5]    (160.96,121.56) .. controls (161,120.25) and (165,123.75) .. (170.14,115.74) ;
%Curve Lines [id:da8834530554723561] 
\draw [color={rgb, 255:red, 0; green, 0; blue, 0 }  ,draw opacity=1 ][line width=1.5]    (152.5,151.56) .. controls (156.5,151.75) and (159.23,150.24) .. (165.32,165.52) ;
%Shape: Ellipse [id:dp03989479677387342] 
\draw  [dash pattern={on 4.5pt off 4.5pt}] (161.5,100.33) .. controls (169.05,104.59) and (173.83,123.78) .. (172.18,143.2) .. controls (170.52,162.62) and (163.06,174.92) .. (155.5,170.66) .. controls (147.95,166.41) and (143.17,147.21) .. (144.82,127.79) .. controls (146.48,108.37) and (153.94,96.07) .. (161.5,100.33) -- cycle ;
%Curve Lines [id:da5406993966095357] 
\draw [color={rgb, 255:red, 0; green, 0; blue, 0 }  ,draw opacity=1 ][line width=1.5]    (75.96,104.56) .. controls (81.84,122.28) and (86.32,125.82) .. (95.14,115.74) ;
%Curve Lines [id:da051426255711055524] 
\draw [color={rgb, 255:red, 0; green, 0; blue, 0 }  ,draw opacity=1 ][line width=1.5]    (71.14,154.34) .. controls (79.59,148.65) and (84.23,150.24) .. (90.32,165.52) ;
%Shape: Ellipse [id:dp304470596175764] 
\draw  [dash pattern={on 4.5pt off 4.5pt}] (86.5,100.33) .. controls (94.05,104.59) and (98.83,123.78) .. (97.18,143.2) .. controls (95.52,162.62) and (88.06,174.92) .. (80.5,170.66) .. controls (72.95,166.41) and (68.17,147.21) .. (69.82,127.79) .. controls (71.48,108.37) and (78.94,96.07) .. (86.5,100.33) -- cycle ;
%Straight Lines [id:da07946531551941594] 
\draw [line width=0.75]    (170.14,115.74) -- (95.14,115.74) ;
%Straight Lines [id:da9630157238751652] 
\draw [line width=0.75]    (81.5,154.34) -- (71.14,154.34) ;
%Straight Lines [id:da26400514926504803] 
\draw [line width=0.75]    (152.5,151.56) -- (80.96,151.56) ;
%Straight Lines [id:da14092409345924262] 
\draw [line width=0.75]    (165.32,165.52) -- (90.32,165.52) ;
%Straight Lines [id:da650828771480304] 
\draw [line width=0.75]    (160.96,121.56) -- (85.96,121.56) ;
%Straight Lines [id:da488963066615982] 
\draw [line width=0.75]    (150.96,104.56) -- (75.96,104.56) ;
%Straight Lines [id:da7995615845471159] 
\draw [line width=0.75]    (117.46,121.75) -- (117.46,151.31) ;
%Straight Lines [id:da30918386803971587] 
\draw [line width=0.75]    (125.96,121.75) -- (125.96,151.31) ;
%Curve Lines [id:da28268699011199605] 
\draw [color={rgb, 255:red, 0; green, 0; blue, 0 }  ,draw opacity=1 ][line width=1.5]    (150.96,104.56) .. controls (152.15,108.14) and (153.28,111.14) .. (154.4,113.57) ;
%Curve Lines [id:da9906589138130326] 
\draw [line width=0.75]    (247.17,150.79) .. controls (249.5,137.81) and (266.22,136.83) .. (268.56,151.11) ;
%Curve Lines [id:da8053921945621523] 
\draw [line width=0.75]    (241.33,150.79) .. controls (242.89,133.26) and (269.33,127.75) .. (274,151.11) ;

%Curve Lines [id:da45776286350379447] 
\draw [color={rgb, 255:red, 0; green, 0; blue, 0 }  ,draw opacity=1 ][line width=1.5]    (297.96,121.56) .. controls (298,120.25) and (302,123.75) .. (307.14,115.74) ;
%Curve Lines [id:da5961977433002985] 
\draw [color={rgb, 255:red, 0; green, 0; blue, 0 }  ,draw opacity=1 ][line width=1.5]    (289.5,151.56) .. controls (293.5,151.75) and (296.23,150.24) .. (302.32,165.52) ;
%Shape: Ellipse [id:dp008513848989451511] 
\draw  [dash pattern={on 4.5pt off 4.5pt}] (298.5,100.33) .. controls (306.05,104.59) and (310.83,123.78) .. (309.18,143.2) .. controls (307.52,162.62) and (300.06,174.92) .. (292.5,170.66) .. controls (284.95,166.41) and (280.17,147.21) .. (281.82,127.79) .. controls (283.48,108.37) and (290.94,96.07) .. (298.5,100.33) -- cycle ;
%Curve Lines [id:da4570470609350318] 
\draw [color={rgb, 255:red, 0; green, 0; blue, 0 }  ,draw opacity=1 ][line width=1.5]    (212.96,104.56) .. controls (218.84,122.28) and (223.32,125.82) .. (232.14,115.74) ;
%Curve Lines [id:da6235143743618828] 
\draw [color={rgb, 255:red, 0; green, 0; blue, 0 }  ,draw opacity=1 ][line width=1.5]    (208.14,154.34) .. controls (216.59,148.65) and (221.23,150.24) .. (227.32,165.52) ;
%Shape: Ellipse [id:dp5025223221848856] 
\draw  [dash pattern={on 4.5pt off 4.5pt}] (223.5,100.33) .. controls (231.05,104.59) and (235.83,123.78) .. (234.18,143.2) .. controls (232.52,162.62) and (225.06,174.92) .. (217.5,170.66) .. controls (209.95,166.41) and (205.17,147.21) .. (206.82,127.79) .. controls (208.48,108.37) and (215.94,96.07) .. (223.5,100.33) -- cycle ;
%Straight Lines [id:da755027405007149] 
\draw [line width=0.75]    (307.14,115.74) -- (232.14,115.74) ;
%Straight Lines [id:da035956383409978754] 
\draw [line width=0.75]    (218.5,154.34) -- (208.14,154.34) ;
%Straight Lines [id:da17221805464666873] 
\draw [line width=0.75]    (289.5,151.56) -- (217.96,151.56) ;
%Straight Lines [id:da1368580230815566] 
\draw [line width=0.75]    (302.32,165.52) -- (227.32,165.52) ;
%Straight Lines [id:da21337567679263436] 
\draw [line width=0.75]    (297.96,121.56) -- (222.96,121.56) ;
%Straight Lines [id:da8944552557818618] 
\draw [line width=0.75]    (287.96,104.56) -- (212.96,104.56) ;
%Curve Lines [id:da4104700277529103] 
\draw [color={rgb, 255:red, 0; green, 0; blue, 0 }  ,draw opacity=1 ][line width=1.5]    (287.96,104.56) .. controls (289.15,108.14) and (290.28,111.14) .. (291.4,113.57) ;

% Text Node
\draw (22,30.67) node [anchor=north west][inner sep=0.75pt]    {$a\ =\ $};
% Text Node
\draw (22,125.67) node [anchor=north west][inner sep=0.75pt]    {$b\ =\ $};
% Text Node
\draw (182.5,127.4) node [anchor=north west][inner sep=0.75pt]    {$-$};
% Text Node
\draw (185,30.77) node [anchor=north west][inner sep=0.75pt]   [align=left] {,};

\end{tikzpicture}
\end{center}
Using (NC), the two endomorphisms $a, b$ can be also written as
\begin{align*}
    a &= u_X + l_Y = u_Y + l_X, \\
    b &= u_X - l_X = u_Y - l_Y
\end{align*}
where $u_X, l_X$ are endomorphisms defined as
\begin{center}
    \tikzset{every picture/.style={line width=0.75pt}} %set default line width to 0.75pt        

\begin{tikzpicture}[x=0.75pt,y=0.75pt,yscale=-.8,xscale=.8]
%uncomment if require: \path (0,71); %set diagram left start at 0, and has height of 71

%Shape: Circle [id:dp6211665977525103] 
\draw  [dash pattern={on 4.5pt off 4.5pt}] (69,32.27) .. controls (69,18.21) and (80.4,6.81) .. (94.46,6.81) .. controls (108.52,6.81) and (119.91,18.21) .. (119.91,32.27) .. controls (119.91,46.33) and (108.52,57.72) .. (94.46,57.72) .. controls (80.4,57.72) and (69,46.33) .. (69,32.27) -- cycle ;
%Shape: Arc [id:dp6011280949412268] 
\draw  [draw opacity=0][line width=1.5]  (114.56,17.75) .. controls (109.19,21.33) and (102.17,23.5) .. (94.48,23.5) .. controls (86.83,23.5) and (79.84,21.35) .. (74.48,17.8) -- (94.48,0.23) -- cycle ; \draw  [line width=1.5]  (114.56,17.75) .. controls (109.19,21.33) and (102.17,23.5) .. (94.48,23.5) .. controls (86.83,23.5) and (79.84,21.35) .. (74.48,17.8) ;  
%Shape: Arc [id:dp43530179172757] 
\draw  [draw opacity=0][line width=1.5]  (74.4,46.79) .. controls (79.77,43.21) and (86.79,41.04) .. (94.48,41.04) .. controls (102.13,41.04) and (109.12,43.19) .. (114.48,46.74) -- (94.48,64.31) -- cycle ; \draw  [line width=1.5]  (74.4,46.79) .. controls (79.77,43.21) and (86.79,41.04) .. (94.48,41.04) .. controls (102.13,41.04) and (109.12,43.19) .. (114.48,46.74) ;  

%Shape: Circle [id:dp7534974089922698] 
\draw  [fill={rgb, 255:red, 0; green, 0; blue, 0 }  ,fill opacity=1 ] (91,23.27) .. controls (91,21.5) and (92.43,20.07) .. (94.2,20.07) .. controls (95.97,20.07) and (97.41,21.5) .. (97.41,23.27) .. controls (97.41,25.04) and (95.97,26.47) .. (94.2,26.47) .. controls (92.43,26.47) and (91,25.04) .. (91,23.27) -- cycle ;
%Shape: Circle [id:dp6670512678023024] 
\draw  [dash pattern={on 4.5pt off 4.5pt}] (209,31.27) .. controls (209,17.21) and (220.4,5.81) .. (234.46,5.81) .. controls (248.52,5.81) and (259.91,17.21) .. (259.91,31.27) .. controls (259.91,45.33) and (248.52,56.72) .. (234.46,56.72) .. controls (220.4,56.72) and (209,45.33) .. (209,31.27) -- cycle ;
%Shape: Arc [id:dp1552130861497656] 
\draw  [draw opacity=0][line width=1.5]  (254.56,16.75) .. controls (249.19,20.33) and (242.17,22.5) .. (234.48,22.5) .. controls (226.83,22.5) and (219.84,20.35) .. (214.48,16.8) -- (234.48,-0.77) -- cycle ; \draw  [line width=1.5]  (254.56,16.75) .. controls (249.19,20.33) and (242.17,22.5) .. (234.48,22.5) .. controls (226.83,22.5) and (219.84,20.35) .. (214.48,16.8) ;  
%Shape: Arc [id:dp5830620161887486] 
\draw  [draw opacity=0][line width=1.5]  (214.4,45.79) .. controls (219.77,42.21) and (226.79,40.04) .. (234.48,40.04) .. controls (242.13,40.04) and (249.12,42.19) .. (254.48,45.74) -- (234.48,63.31) -- cycle ; \draw  [line width=1.5]  (214.4,45.79) .. controls (219.77,42.21) and (226.79,40.04) .. (234.48,40.04) .. controls (242.13,40.04) and (249.12,42.19) .. (254.48,45.74) ;  

%Shape: Circle [id:dp6582951267959462] 
\draw  [fill={rgb, 255:red, 0; green, 0; blue, 0 }  ,fill opacity=1 ] (231.25,40.47) .. controls (231.25,38.7) and (232.69,37.27) .. (234.46,37.27) .. controls (236.23,37.27) and (237.66,38.7) .. (237.66,40.47) .. controls (237.66,42.24) and (236.23,43.68) .. (234.46,43.68) .. controls (232.69,43.68) and (231.25,42.24) .. (231.25,40.47) -- cycle ;

% Text Node
\draw (17,24.67) node [anchor=north west][inner sep=0.75pt]    {$u_{X} \ =\ $};
% Text Node
\draw (127,25.77) node [anchor=north west][inner sep=0.75pt]   [align=left] {,};
% Text Node
\draw (157,23.67) node [anchor=north west][inner sep=0.75pt]    {$l_{X} \ =\ $};

\end{tikzpicture}
\end{center}
and $u_Y, l_Y$ similarly using $\circ$. Let $s$ denote the genus $0$ saddle cobordism from $\sfE_0$ to $\sfE_1$ and also from $\sfE_1$ to $\sfE_0$. Note that we have $ab = ba = 0$, $sb = 0$ and $bs = 0$. 

\begin{remark}
    Borrowing the notation of \cite{Khovanov:2022}, the morphism $b$ can be expressed as a tube with a \textit{defect circle} along its meridian,
    \begin{center}
        \tikzset{every picture/.style={line width=0.75pt}} %set default line width to 0.75pt        

\begin{tikzpicture}[x=0.75pt,y=0.75pt,yscale=-.8,xscale=.8]
%uncomment if require: \path (0,95); %set diagram left start at 0, and has height of 95

%Curve Lines [id:da6882441715246497] 
\draw [color={rgb, 255:red, 0; green, 0; blue, 0 }  ,draw opacity=1 ][line width=1.5]    (161.96,34.56) .. controls (162,33.25) and (166,36.75) .. (171.14,28.74) ;
%Curve Lines [id:da09313586940764129] 
\draw [color={rgb, 255:red, 0; green, 0; blue, 0 }  ,draw opacity=1 ][line width=1.5]    (153.5,64.56) .. controls (157.5,64.75) and (160.23,63.24) .. (166.32,78.52) ;
%Shape: Ellipse [id:dp4549638243257157] 
\draw  [dash pattern={on 4.5pt off 4.5pt}] (162.5,13.33) .. controls (170.05,17.59) and (174.83,36.78) .. (173.18,56.2) .. controls (171.52,75.62) and (164.06,87.92) .. (156.5,83.66) .. controls (148.95,79.41) and (144.17,60.21) .. (145.82,40.79) .. controls (147.48,21.37) and (154.94,9.07) .. (162.5,13.33) -- cycle ;
%Curve Lines [id:da7458429450024703] 
\draw [color={rgb, 255:red, 0; green, 0; blue, 0 }  ,draw opacity=1 ][line width=1.5]    (76.96,17.56) .. controls (82.84,35.28) and (87.32,38.82) .. (96.14,28.74) ;
%Curve Lines [id:da9668179428586579] 
\draw [color={rgb, 255:red, 0; green, 0; blue, 0 }  ,draw opacity=1 ][line width=1.5]    (72.14,67.34) .. controls (80.59,61.65) and (85.23,63.24) .. (91.32,78.52) ;
%Shape: Ellipse [id:dp4303898937722468] 
\draw  [dash pattern={on 4.5pt off 4.5pt}] (87.5,13.33) .. controls (95.05,17.59) and (99.83,36.78) .. (98.18,56.2) .. controls (96.52,75.62) and (89.06,87.92) .. (81.5,83.66) .. controls (73.95,79.41) and (69.17,60.21) .. (70.82,40.79) .. controls (72.48,21.37) and (79.94,9.07) .. (87.5,13.33) -- cycle ;
%Straight Lines [id:da3680439848186754] 
\draw [line width=0.75]    (171.14,28.74) -- (96.14,28.74) ;
%Straight Lines [id:da33279807924041915] 
\draw [line width=0.75]    (82.5,67.34) -- (72.14,67.34) ;
%Straight Lines [id:da19617985260206794] 
\draw [line width=0.75]    (153.5,64.56) -- (81.96,64.56) ;
%Straight Lines [id:da24157731546754302] 
\draw [line width=0.75]    (166.32,78.52) -- (91.32,78.52) ;
%Straight Lines [id:da4143722286983248] 
\draw [line width=0.75]    (161.96,34.56) -- (86.96,34.56) ;
%Straight Lines [id:da714980385869984] 
\draw [line width=0.75]    (151.96,17.56) -- (76.96,17.56) ;
%Straight Lines [id:da059522532160527275] 
\draw [line width=0.75]    (116.46,34.75) -- (116.46,64.31) ;
%Straight Lines [id:da4045660525243089] 
\draw [line width=0.75]    (132.96,34.75) -- (132.96,64.31) ;
%Curve Lines [id:da28218089029650517] 
\draw [color={rgb, 255:red, 0; green, 0; blue, 0 }  ,draw opacity=1 ][line width=1.5]    (151.96,17.56) .. controls (153.15,21.14) and (154.28,24.14) .. (155.4,26.57) ;
%Shape: Arc [id:dp10301883211002649] 
\draw  [draw opacity=0][line width=1.5]  (132.96,45.3) .. controls (132.96,45.34) and (132.96,45.38) .. (132.96,45.42) .. controls (132.96,47.9) and (129.31,49.9) .. (124.8,49.9) .. controls (120.34,49.9) and (116.71,47.93) .. (116.65,45.48) -- (124.8,45.42) -- cycle ; \draw  [color={rgb, 255:red, 208; green, 2; blue, 27 }  ,draw opacity=1 ][line width=1.5]  (132.96,45.3) .. controls (132.96,45.34) and (132.96,45.38) .. (132.96,45.42) .. controls (132.96,47.9) and (129.31,49.9) .. (124.8,49.9) .. controls (120.34,49.9) and (116.71,47.93) .. (116.65,45.48) ;  
%Straight Lines [id:da10293738747808046] 
\draw [color={rgb, 255:red, 208; green, 2; blue, 27 }  ,draw opacity=1 ][line width=1.5]    (125.29,49.17) -- (125.29,54.57) ;

% Text Node
\draw (23,38.67) node [anchor=north west][inner sep=0.75pt]    {$b\ =\ $};

\end{tikzpicture}
    \end{center}
    which swaps $\bullet$ and $-\circ$ as they pass the defect line. We also remark that this cobordism, or the corresponding algebraic map, is used to represent \textit{crossing change maps}, as in \cite{Alishahi:2018,Ito-Yoshida:2021}. 
\end{remark}

First, we state the following lemma, which is immediate from \Cref{prop:delooping}. 

\begin{lemma}
\label{lem:Tq-lem1}
    Consider the following tangle diagrams $T, T'$ and morphisms $\Delta, m, a, b$
    \begin{center}
        \tikzset{every picture/.style={line width=0.75pt}} %set default line width to 0.75pt        

\begin{tikzpicture}[x=0.75pt,y=0.75pt,yscale=-.9,xscale=.9]
%uncomment if require: \path (0,96); %set diagram left start at 0, and has height of 96

%Shape: Ellipse [id:dp017086428089281736] 
\draw  [color={rgb, 255:red, 0; green, 0; blue, 0 }  ,draw opacity=1 ][line width=1.5]  (177.91,28.77) .. controls (177.91,23.46) and (182.41,19.16) .. (187.96,19.16) .. controls (193.5,19.16) and (198,23.46) .. (198,28.77) .. controls (198,34.07) and (193.5,38.37) .. (187.96,38.37) .. controls (182.41,38.37) and (177.91,34.07) .. (177.91,28.77) -- cycle ;
%Shape: Circle [id:dp5378187555014846] 
\draw  [dash pattern={on 4.5pt off 4.5pt}] (162.5,33.77) .. controls (162.5,19.71) and (173.9,8.31) .. (187.96,8.31) .. controls (202.02,8.31) and (213.41,19.71) .. (213.41,33.77) .. controls (213.41,47.83) and (202.02,59.22) .. (187.96,59.22) .. controls (173.9,59.22) and (162.5,47.83) .. (162.5,33.77) -- cycle ;
%Shape: Arc [id:dp576977160661786] 
\draw  [draw opacity=0][line width=1.5]  (168.89,50.71) .. controls (174.12,48.05) and (180.76,46.45) .. (187.98,46.45) .. controls (194.91,46.45) and (201.3,47.92) .. (206.43,50.39) -- (187.98,65.81) -- cycle ; \draw  [line width=1.5]  (168.89,50.71) .. controls (174.12,48.05) and (180.76,46.45) .. (187.98,46.45) .. controls (194.91,46.45) and (201.3,47.92) .. (206.43,50.39) ;  
%Shape: Circle [id:dp1605069597570039] 
\draw  [dash pattern={on 4.5pt off 4.5pt}] (37,33.77) .. controls (37,19.71) and (48.4,8.31) .. (62.46,8.31) .. controls (76.52,8.31) and (87.91,19.71) .. (87.91,33.77) .. controls (87.91,47.83) and (76.52,59.22) .. (62.46,59.22) .. controls (48.4,59.22) and (37,47.83) .. (37,33.77) -- cycle ;
%Shape: Arc [id:dp7187260614479883] 
\draw  [draw opacity=0][line width=1.5]  (42.4,48.29) .. controls (47.77,44.71) and (54.79,42.54) .. (62.48,42.54) .. controls (70.13,42.54) and (77.12,44.69) .. (82.48,48.24) -- (62.48,65.81) -- cycle ; \draw  [line width=1.5]  (42.4,48.29) .. controls (47.77,44.71) and (54.79,42.54) .. (62.48,42.54) .. controls (70.13,42.54) and (77.12,44.69) .. (82.48,48.24) ;  
%Straight Lines [id:da5888503912311841] 
\draw    (100,33) -- (147,33) ;
\draw [shift={(149,33)}, rotate = 180] [color={rgb, 255:red, 0; green, 0; blue, 0 }  ][line width=0.75]    (10.93,-4.9) .. controls (6.95,-2.3) and (3.31,-0.67) .. (0,0) .. controls (3.31,0.67) and (6.95,2.3) .. (10.93,4.9)   ;
%Straight Lines [id:da40840046117326223] 
\draw    (151,44) -- (100,44) ;
\draw [shift={(98,44)}, rotate = 360] [color={rgb, 255:red, 0; green, 0; blue, 0 }  ][line width=0.75]    (10.93,-4.9) .. controls (6.95,-2.3) and (3.31,-0.67) .. (0,0) .. controls (3.31,0.67) and (6.95,2.3) .. (10.93,4.9)   ;
%Curve Lines [id:da07083178554949143] 
\draw    (223.77,25.32) .. controls (232.08,17.65) and (256.35,7.43) .. (256.99,33.62) .. controls (257.6,58.63) and (237.61,49.89) .. (225.44,44.01) ;
\draw [shift={(223.77,43.2)}, rotate = 25.94] [color={rgb, 255:red, 0; green, 0; blue, 0 }  ][line width=0.75]    (6.56,-1.97) .. controls (4.17,-0.84) and (1.99,-0.18) .. (0,0) .. controls (1.99,0.18) and (4.17,0.84) .. (6.56,1.97)   ;

% Text Node
\draw (179,69.65) node [anchor=north west][inner sep=0.75pt]    {$T'$};
% Text Node
\draw (53,69.65) node [anchor=north west][inner sep=0.75pt]    {$T$};
% Text Node
\draw (117,11.4) node [anchor=north west][inner sep=0.75pt]    {$\Delta$};
% Text Node
\draw (117,47.4) node [anchor=north west][inner sep=0.75pt]    {$m$};
% Text Node
\draw (263,23.4) node [anchor=north west][inner sep=0.75pt]    {$a,\ b$};

\end{tikzpicture}
    \end{center}
    Here, these morphisms are assumed to have $q$-degree $0$ by shifting the gradings of the objects appropriately. Then by delooping the circle appearing in tangle $T'$, morphisms $\Delta, m$ transform into
    \[
\begin{tikzcd}[row sep=.3em] 
& & T\{2\} & & T \arrow[dd, "\oplus", phantom] \arrow[rrd, "I"] & & \\
T \arrow[rru, "Y"] \arrow[rrd, "I"'] & & \oplus & & & & T \\ & & T & & T\{-2\} \arrow[rru, "X"'] & & 
\end{tikzcd} 
    \]
    and endomorphisms $a, b$ transform into
    \[
\begin{tikzcd}[row sep=2em]
T \arrow[d, "\oplus" description, no head, phantom] \arrow[rrr, "Y"] \arrow[rrrd, "I" description, pos=.8] & & & T\{2\} \arrow[d, "\oplus" description, no head, phantom] & T \arrow[d, "\oplus" description, no head, phantom] \arrow[rrrd, "I" description, pos=.8] \arrow[rrr, "-X"] & & & T\{2\} \arrow[d, "\oplus" description, no head, phantom] \\
T\{-2\} \arrow[rrr, "X"] \arrow[rrru, dotted, "0" description, pos=.75] & & & T & T\{-2\} \arrow[rrr, "-Y"] \arrow[rrru, dotted, "0" description, pos=.75] & & & T.
\end{tikzcd}
    \]
\end{lemma}

The following proposition was originally proved in \cite[Section 6.2]{Khovanov:2000} in the context of Khovanov homology, and later in \cite[Proposition 4.1]{Thompson:2018} and \cite[Proposition 5.1]{Schuetz:2021} within Bar-Natan’s framework. We include a proof here for completeness and later use.

\begin{proposition}
\label{prop:twist-tangle-retract}
    For each $q \geq 1$, there is a strong deformation retract from the complex $[T_q]$ to the complex $E_q$ of length $q + 1$ defined as 
    \[
\begin{tikzcd}
\underline{\sfE_0} \arrow[r, "s"] & \sfE_1\{1\} \arrow[r, "b"] & \sfE_1\{3\} \arrow[r, "a"] & \sfE_1\{5\} \arrow[r, "b"] & \cdots \arrow[r, ""] & \sfE_1\{2q - 1\}.
\end{tikzcd}
    \]
    Here, the bigradings are relative with respect to the underlined object $\underline{\sfE_0}$, and must be shifted by $(0, q)$ or $(-q, -2q)$ according to the orientation. Similarly, there is a strong deformation retract from the complex $[T_{-q}]$ to the complex $E_{-q}$ of length $q + 1$ defined as 
    \[
\begin{tikzcd}[column sep = 2em]
\sfE_1\{-2q + 1\} \arrow[r, ""] & \cdots \arrow[r, "b"] & \sfE_1\{-5\} \arrow[r, "a"] & \sfE_1\{-3\} \arrow[r, "b"] & \sfE_1\{-1\} \arrow[r, "s"] & \underline{\sfE_0},
\end{tikzcd}
    \]
    where the bigradings must be shifted by $(q, 2q)$ or $(0, -q)$ according to the orientation. In each of the cases, the homological grading $0$ part is the leftmost object for the positive orientation, and the rightmost for the negative. 
\end{proposition}

\begin{proof}
    The result is trivial for $q = 1$. Let us assume that the result is true for $q \geq 1$. Let $D$ be the $2$-input planar arc diagram such that $T_{q + 1} = D(T_1, T_q)$. Then from the induction hypothesis, the complex $[T_{q+1}] = D([T_1], [T_q])$ deformation retracts to $D(\sfE_1, E_q)$, which is of the form
    \[
\begin{tikzcd}
\underline{\sfE_0} \arrow[d, "s_1"] \arrow[r, "s"] & \sfE_1 \arrow[d, "s_1"] \arrow[r, "b"] & \sfE_1 \arrow[r, "a"] \arrow[d, "s_1"] & \cdots \arrow[r] & \sfE_1 \arrow[d, "s_1"] \\
\sfE_1 \arrow[r, "-s_2"] & \sfE_{11} \arrow[r, "-b_2"] & \sfE_{11} \arrow[r, "-a_2"] & \cdots \arrow[r] & \sfE_{11}.
\end{tikzcd}
    \]
    Here, we identified $D(\sfE_0, E_i)$ and $D(E_i, \sfE_0)$ with $E_i$ for each $i = 0, 1$. The diagram $D(\sfE_1, \sfE_1)$ is denoted $\sfE_{11}$, which has the form of a circle inserted between the two arcs of $\sfE_1$. The vertical arrows are all $s_1 = D(s, I)$ and the lower horizontal arrows are given by $a_2 = D(I, a)$, $b_2 = D(I, b)$ and $s_2 = D(I, s)$. Each vertical arrow $s_1$ with target in $\sfE_{11}$ splits a circle from the upper arc, and the lower leftmost horizontal arrow $-s_2$ splits a circle from the lower arc. 

    The delooping isomorphism gives $\sfE_{11} \isom \sfE^+_1 \oplus \sfE^-_1$, where $\sfE_1^\pm$ are abbreviations of $\sfE_1\{\pm 1\}$. From \Cref{lem:Tq-lem1}, the bottom row of the above diagram can be written as 
    \[
\begin{tikzcd} [row sep=.5em] & \sfE^+_1 \arrow[r] \arrow[rdd, "-I"   description] & \sfE^+_1 \arrow[rdd, "-I"   description] \arrow[r] & \sfE^+_1 \arrow[rdd, "-I"   description] \arrow[r] & \cdots \arrow[r] \arrow[rdd, "-I"   description] & \sfE^+_1 \\
\sfE_1 \arrow[ru] \arrow[rd, "-I"   description] & & & & & \\ & \sfE^-_1 \arrow[r] & \sfE^-_1 \arrow[r] & \sfE^-_1 \arrow[r] & \cdots \arrow[r] & \sfE^-_1.
\end{tikzcd} 
    \]
    By eliminating the diagonal arrows labeled $-I$ from the left end, the complex $D(\sfE_1, E_q)$ can be further reduced to a sequence of the form
    \[
\begin{tikzcd}
\underline{\sfE_0} \arrow[r, "s"] & \sfE_1 \arrow[r, "b"] & \sfE_1 \arrow[r, "a"] & \cdots \arrow[r, ""] & \sfE_1 \arrow[d, ""] \\
& & & & \sfE_1.
\end{tikzcd}
    \]
    At the right end, if the rightmost horizontal arrow is $a$ (resp. $b$), then the added vertical arrow becomes $b$ (resp. $a$), as can be seen from the following diagrams. This proves the result for $q + 1$. 
    \[
    \begin{tikzcd}
\sfE_1 \arrow[r, "a"] & \sfE_1 \arrow[d, "u_Y" description, bend left] \arrow[dd, "I" description, bend left=50] & & \sfE_1 \arrow[r, "b"] & \sfE_1 \arrow[d, "u_Y" description, bend left] \arrow[dd, "I" description, bend left=50] \\
\sfE_1 \arrow[rd, "-I"  description] \arrow[r, "-l_Y"] & \sfE_1 & & \sfE_1 \arrow[rd, "-I"  description] \arrow[r, "l_X"] & \sfE_1 \\
\sfE_1 \arrow[r] & \sfE_1 & & \sfE_1 \arrow[r] & \sfE_1 
\end{tikzcd}
    \]

    Similarly, the complex $[T_{-q - 1}] = D([T_{-1}], [T_{-q}])$ deformation retracts to $D(E_{-1}, E_{-q})$, which is of the form
    \[
\begin{tikzcd}
\sfE_{11} \arrow[r, ""] \arrow[d, "s_1"] & \cdots \arrow[r, "-a_2"] & \sfE_{11} \arrow[r, "-b_2"] \arrow[d, "s_1"] & \sfE_{11} \arrow[r, "-s_2"] \arrow[d, "s_1"] & \sfE_1 \arrow[d, "s_1"] \\
\sfE_1 \arrow[r, ""] & \cdots \arrow[r, "a"] & \sfE_1 \arrow[r, "b"] & \sfE_1 \arrow[r, "s"] & \underline{\sfE_0}.
\end{tikzcd}
    \]
    By delooping the objects in the top row and applying reductions from the rightmost square, we see that the complex reduces to  
    \[
\begin{tikzcd}
\sfE_1 \arrow[d, ""] & & & & \\
\sfE_1 \arrow[r, ""] & \cdots \arrow[r, "a"] & \sfE_1 \arrow[r, "b"] & \sfE_1 \arrow[r, "s"] & \underline{\sfE_0}
\end{tikzcd}
    \]
    thus proving the latter result. 
\end{proof}

Next, we state the correspondence of the Lee cycles of $T^\pm_q$ under the reduction of \Cref{prop:twist-tangle-retract}. For each $q \in \ZZ \setminus \{0\}$, let $f_q: [T^\pm_q] \rightarrow E^\pm_q$ denote the reduction. Let $\ca(T^\pm_q)$ denote the Lee cycle of $T^\pm_q$, and $z^\pm_q$ the dotted cobordism obtained from $\ca(T^\pm_q)$ by removing the disk components. See \Cref{fig:tq-lee-cycles} for the case $q = \pm 3$.

\begin{figure}[t]
    \centering
    \input{tikzpictures/T_q-lee-cycles}
    \vspace{.5em}
    \tikzset{every picture/.style={line width=0.75pt}} %set default line width to 0.75pt        

\begin{tikzpicture}[x=0.75pt,y=0.75pt,yscale=-.8,xscale=.8]
%uncomment if require: \path (0,116); %set diagram left start at 0, and has height of 116

%Shape: Circle [id:dp3783461828970084] 
\draw  [dash pattern={on 4.5pt off 4.5pt}] (42.48,14.29) .. controls (56.54,14.29) and (67.94,25.69) .. (67.94,39.75) .. controls (67.94,53.81) and (56.54,65.21) .. (42.48,65.21) .. controls (28.42,65.21) and (17.02,53.81) .. (17.02,39.75) .. controls (17.02,25.69) and (28.42,14.29) .. (42.48,14.29) -- cycle ;
%Shape: Arc [id:dp2833758931658643] 
\draw  [draw opacity=0][line width=1.5]  (57,59.85) .. controls (53.42,54.48) and (51.25,47.46) .. (51.25,39.77) .. controls (51.25,32.12) and (53.4,25.13) .. (56.95,19.77) -- (74.52,39.77) -- cycle ; \draw  [color={rgb, 255:red, 74; green, 144; blue, 226 }  ,draw opacity=1 ][line width=1.5]  (57,59.85) .. controls (53.42,54.48) and (51.25,47.46) .. (51.25,39.77) .. controls (51.25,32.12) and (53.4,25.13) .. (56.95,19.77) ;  
%Shape: Arc [id:dp03409974619105005] 
\draw  [draw opacity=0][line width=1.5]  (27.96,19.69) .. controls (31.54,25.06) and (33.71,32.08) .. (33.71,39.77) .. controls (33.71,47.42) and (31.56,54.41) .. (28.01,59.77) -- (10.43,39.77) -- cycle ; \draw  [color={rgb, 255:red, 208; green, 2; blue, 27 }  ,draw opacity=1 ][line width=1.5]  (27.96,19.69) .. controls (31.54,25.06) and (33.71,32.08) .. (33.71,39.77) .. controls (33.71,47.42) and (31.56,54.41) .. (28.01,59.77) ;  

%Shape: Circle [id:dp9924117193687215] 
\draw  [dash pattern={on 4.5pt off 4.5pt}] (104.61,39.77) .. controls (104.61,25.71) and (116,14.31) .. (130.06,14.31) .. controls (144.12,14.31) and (155.52,25.71) .. (155.52,39.77) .. controls (155.52,53.83) and (144.12,65.22) .. (130.06,65.22) .. controls (116,65.22) and (104.61,53.83) .. (104.61,39.77) -- cycle ;
%Shape: Arc [id:dp26378139355045804] 
\draw  [draw opacity=0][line width=1.5]  (150.16,25.25) .. controls (144.8,28.83) and (137.77,31) .. (130.09,31) .. controls (122.44,31) and (115.44,28.85) .. (110.09,25.3) -- (130.09,7.73) -- cycle ; \draw  [color={rgb, 255:red, 208; green, 2; blue, 27 }  ,draw opacity=1 ][line width=1.5]  (150.16,25.25) .. controls (144.8,28.83) and (137.77,31) .. (130.09,31) .. controls (122.44,31) and (115.44,28.85) .. (110.09,25.3) ;  
%Shape: Arc [id:dp7786035191300683] 
\draw  [draw opacity=0][line width=1.5]  (110.01,54.29) .. controls (115.37,50.71) and (122.4,48.54) .. (130.09,48.54) .. controls (137.74,48.54) and (144.73,50.69) .. (150.08,54.24) -- (130.09,71.81) -- cycle ; \draw  [color={rgb, 255:red, 74; green, 144; blue, 226 }  ,draw opacity=1 ][line width=1.5]  (110.01,54.29) .. controls (115.37,50.71) and (122.4,48.54) .. (130.09,48.54) .. controls (137.74,48.54) and (144.73,50.69) .. (150.08,54.24) ;  

%Shape: Circle [id:dp0377432710453085] 
\draw  [dash pattern={on 4.5pt off 4.5pt}] (313.48,14.29) .. controls (327.54,14.29) and (338.94,25.69) .. (338.94,39.75) .. controls (338.94,53.81) and (327.54,65.21) .. (313.48,65.21) .. controls (299.42,65.21) and (288.02,53.81) .. (288.02,39.75) .. controls (288.02,25.69) and (299.42,14.29) .. (313.48,14.29) -- cycle ;
%Shape: Arc [id:dp8890376490818921] 
\draw  [draw opacity=0][line width=1.5]  (328,59.85) .. controls (324.42,54.48) and (322.25,47.46) .. (322.25,39.77) .. controls (322.25,32.12) and (324.4,25.13) .. (327.95,19.77) -- (345.52,39.77) -- cycle ; \draw  [color={rgb, 255:red, 74; green, 144; blue, 226 }  ,draw opacity=1 ][line width=1.5]  (328,59.85) .. controls (324.42,54.48) and (322.25,47.46) .. (322.25,39.77) .. controls (322.25,32.12) and (324.4,25.13) .. (327.95,19.77) ;  
%Shape: Arc [id:dp0540284730618702] 
\draw  [draw opacity=0][line width=1.5]  (298.96,19.69) .. controls (302.54,25.06) and (304.71,32.08) .. (304.71,39.77) .. controls (304.71,47.42) and (302.56,54.41) .. (299.01,59.77) -- (281.43,39.77) -- cycle ; \draw  [color={rgb, 255:red, 208; green, 2; blue, 27 }  ,draw opacity=1 ][line width=1.5]  (298.96,19.69) .. controls (302.54,25.06) and (304.71,32.08) .. (304.71,39.77) .. controls (304.71,47.42) and (302.56,54.41) .. (299.01,59.77) ;  

%Shape: Circle [id:dp25817316691169934] 
\draw  [dash pattern={on 4.5pt off 4.5pt}] (199.4,39.77) .. controls (199.4,25.71) and (210.79,14.31) .. (224.85,14.31) .. controls (238.91,14.31) and (250.31,25.71) .. (250.31,39.77) .. controls (250.31,53.83) and (238.91,65.22) .. (224.85,65.22) .. controls (210.79,65.22) and (199.4,53.83) .. (199.4,39.77) -- cycle ;
%Shape: Arc [id:dp22432165608472987] 
\draw  [draw opacity=0][line width=1.5]  (244.95,25.25) .. controls (239.59,28.83) and (232.56,31) .. (224.88,31) .. controls (217.23,31) and (210.23,28.85) .. (204.88,25.3) -- (224.88,7.73) -- cycle ; \draw  [color={rgb, 255:red, 208; green, 2; blue, 27 }  ,draw opacity=1 ][line width=1.5]  (244.95,25.25) .. controls (239.59,28.83) and (232.56,31) .. (224.88,31) .. controls (217.23,31) and (210.23,28.85) .. (204.88,25.3) ;  
%Shape: Arc [id:dp4761071808609787] 
\draw  [draw opacity=0][line width=1.5]  (204.8,54.29) .. controls (210.16,50.71) and (217.19,48.54) .. (224.88,48.54) .. controls (232.53,48.54) and (239.52,50.69) .. (244.87,54.24) -- (224.88,71.81) -- cycle ; \draw  [color={rgb, 255:red, 74; green, 144; blue, 226 }  ,draw opacity=1 ][line width=1.5]  (204.8,54.29) .. controls (210.16,50.71) and (217.19,48.54) .. (224.88,48.54) .. controls (232.53,48.54) and (239.52,50.69) .. (244.87,54.24) ;

% Text Node
\draw (31,75.9) node [anchor=north west][inner sep=0.75pt]    {$z_{3}^{+}$};
% Text Node
\draw (120.33,75.9) node [anchor=north west][inner sep=0.75pt]    {$z_{3}^{-}$};
% Text Node
\draw (209.66,75.9) node [anchor=north west][inner sep=0.75pt]    {$z_{-3}^{+}$};
% Text Node
\draw (299,75.9) node [anchor=north west][inner sep=0.75pt]    {$z_{-3}^{-}$};

\end{tikzpicture}
    \caption{Tangle diagrams $T^\pm_q$ and their Lee cycles for $q = \pm 3$.}
    \label{fig:tq-lee-cycles}
\end{figure}

\begin{proposition}
\label{prop:twist-tangle-lee-cycle}
    For each $q \geq 1$, we have  
    \[
        \ca(T^+_q) \xmapsto{\ f_q\ } z^+_q,\quad 
        \ca(T^-_q) \xmapsto{\ f_q\ } \pm H^{q - 1} z^-_q
    \]
    and
    \[
        \ca(T^+_{-q}) \xmapsto{\ f_{-q}\ } z^+_{-q},\quad 
        \ca(T^-_{-q}) \xmapsto{\ f_{-q}\ } z^-_{-q}.
    \]
\end{proposition}

\begin{proof}
    We only prove the first two equations. First, the statement is trivial for $q = 1$. Next, assume $q \geq 1$ and that the result it true up to $q$. Recall from the proof of \Cref{prop:twist-tangle-retract} that the map $f_{q + 1}$ is obtained by the composition 
    \[
        f_{q + 1} = g_{q + 1} \circ D(1, f_q)
    \]
    where $g_{q + 1}$ is the retraction that collapses all the squares from the left. For the positive diagram $T^+_{q + 1}$, from \Cref{thm:lee-cycle-decomp}, we have 
    \begin{align*}
        \ca(T^+_{q + 1}) 
        &= \ca(D(T^+_1, T^+_q)) \\
        &= D(\ca(T^+_1), \ca(T^+_q)) \circ \ca(\sfE_0)
    \end{align*}
    Here, we have identified $D((T^+_1)^{(i)}, (T^+_q)^{(i)}) = \sfE_0$. With $\ca(T^+_1) = z^+_1, \ca(\sfE_0) = z^+_{q + 1}$ and the induction hypothesis, post-composing $D(1, f_q)$ to $\ca(T^+_{q + 1})$ gives 
    \[
        D(z^+_1, z^+_q) \circ z^+_{q + 1} = z^+_{q + 1}.
    \]
    Since $g_{q + 1}$ acts as identity on $\underline{\sfE_0}$, we obtain the desired result. 
    For the negative diagram $T^-_{q + 1}$, we have
    \begin{align*}
        \ca(T^-_{q + 1}) 
        &= \ca(D(T^-_1, T^-_q)) \\
        &= D(\ca(T^-_1), \ca(T^-_q)) \circ \ca(\sfE_{11}).
    \end{align*}
    Here, we have identified $D((T^-_1)^{(i)}, (T^-_q)^{(i)}) = D(\sfE_1, \sfE_1) = \sfE_{11}$. From the induction hypothesis, post-composing $D(1, f_q)$ to $\ca(T^-_{q + 1})$ gives 
    \[
        \pm H^{q-2} D(z^-_1, z^-_q) \circ \ca(\sfE_{11}) = \pm H^{q-2} \ca(\sfE_{11}).
    \]
    Note that the cycle $\ca(\sfE_{11})$ has target in the rightmost bottom object $\sfE_{11}$ of the complex $E^-_{q + 1}$. One can check using \Cref{prop:gauss-elim} that post-composing $g_{q+1}$ to it gives
    \[
        \pm H z^-_{q + 1}
    \]
    for both cases of the checkerboard coloring. Thus we have 
    \[
        f_{q + 1} \circ \ca(T^-_{q + 1}) = \pm H^q z^-_{q + 1}
    \]
    as desired. The proof for the remaining equations proceeds similarly.
\end{proof}

\begin{example}
    We first demonstrate that, for any odd $q > 0$, the $s_H$-invariants of the positive $(2, q)$-torus knot and its mirror can be determined using \Cref{prop:twist-tangle-lee-cycle} (although the positive case is immediate from \Cref{prop:s_H-properties}). Write $T_{2, q} = C({T}^+_{q})$ using a closure $C$ (a $1$-input planar arc diagram) that connects the two strands of ${T}^+_{q}$ outside the disk. From \Cref{prop:twist-tangle-retract}, the complex $[T_{2, q}] = [C({T}^+_{q})] = C([{T}^+_{q}])$ can be reduced to 
    \[
\begin{tikzcd}[column sep=2.5em]
C(\underline{\sfE_0}) \arrow[r, "C(s)"] & C(\sfE_1) \arrow[r, "C(b)"] & C(\sfE_1) \arrow[r, "C(a)"] & C(\sfE_1) \arrow[r, "C(b)"] & \cdots
\end{tikzcd}
    \]
    which is
    \[
\begin{tikzcd}[column sep=3em]
\underline{\thickcirc \thickcirc} \arrow[r, "m"] & \thickcirc \arrow[r, "0", dotted] & \thickcirc \arrow[r, "U"] & \thickcirc \arrow[r, "0", dotted] & \cdots.
\end{tikzcd}
    \]
    Here, $U = X + Y$, which corresponds to the star $\star$ introduced in \Cref{sec:prelim}. From \Cref{prop:twist-tangle-lee-cycle}, we have $C(f_q) \circ \ca(T_{2, q}) = C(f_q \circ \ca({T}^+_{q})) = C(z^+_q)$. Since there are no arrows coming into the initial object, we have $d_H(T_{2, q}) = 0$. Together with $w = q$ and $r = 2$, we determine
    \[
        s_H(T_{2, q}) = q - 1.
    \]
    For the mirror, similarly write $T^*_{2, q} = C({T}^-_{-q})$ and then its complex is given by 
    \[
\begin{tikzcd}
\cdots \arrow[r, "0", dotted] & \thickcirc \arrow[r, "U"] & \thickcirc \arrow[r, "0", dotted] & \thickcirc \arrow[r, "\Delta"] & \underline{\thickcirc \thickcirc}.
\end{tikzcd}
    \]
    Again from \Cref{prop:twist-tangle-lee-cycle}, we have $C(f_{-q}) \circ \ca(T^*_{2, q}) = C(z^-_{-q})$. Furthermore, we may apply deloop-and-eliminate to the rightmost arrow
    \[
\begin{tikzcd}
\thickcirc \arrow[r, "\Delta"] & \underline{\thickcirc\thickcirc}
\end{tikzcd}
    \]
    to get 
    \[
\begin{tikzcd}
0 \arrow[r, dotted] & \underline{\thickcirc}.
\end{tikzcd}
    \]
    One can check that the post-composing this retraction to $C(z^-_{-q})$ gives $\pm H \ca(\bigcirc)$. Thus $d_H(T^*_{2, q}) = 1$, and with $w = -q$ and $r = 2$, we determine 
    \[
        s_H(T^*_{2, q}) = -q + 1.
    \]
\end{example}

    \section{Pretzel knots}
\label{sec:pretzel}

Finally we determine the $s$-invariant for all $3$-strand pretzel knots $P(p, q, r)$ by combining general arguments on \textit{slice-torus invariants} and the diagrammatic methods we have developed so far. 

\subsection{General results for slice-torus invariants}

A \emph{slice-torus invariant} $\nu$ is an $\ZZ$-valued group homomorphism from smooth knot concordance group $\mathcal{C}$, which satisfies following two conditions:
\begin{itemize}[leftmargin=5em]   
    \item[(slice)] $\nu(K) \le g_4(K)$; and 
    \item[(torus)] $\nu(T_{p,q}) = g_4(T_{p,q})$ for any positive torus knot $T_{p,q}$.
\end{itemize}
Here, $g_4$ denotes the $4$-genus of knots. Slice-torus invariants include: (i) Ozsv\'ath and Szab\'o's \textit{$\tau$-invariant} from knot Floer homology \cite{OS:2003}, (ii) Rasmussen invariants over arbitrary field $F$, normalized as $s^F/2$, (iii) the $\tilde{s}$-invariant from equivariant singular instanton Floer homology \cite{DISST:2022}, and (iv) Iida-Taniguchi's $q_M$-invariant from $\ZZ_2$-equivariant Seiberg--Witten Floer cohomology \cite{Iida-Taniguchi:2024}. (More generally, one may consider \textit{$\RR$-valued slice-torus invariants}, which contain the concordance invariants from $\mathfrak{sl}_N$ Khovanov--Rozansky homologies \cite{Wu:2009,Lobb:2009,Lobb:2012}, and the $\tau^\#$-invariant from framed instanton Floer homology \cite{Baldwin-Sivek:2021}). Studies of slice-torus invariants originates in \cite{Livingston:2004,Lewark:2014}, and have been extensively studied in \cite{Cavallo:2020,Feller:2022}.

The following properties follows from the formal properties of slice-torus invariants. 

\begin{proposition}[{\cite{Livingston:2004,Lewark:2014}}]
    Let $\nu$ be any slice-torus invariant. 
    \begin{enumerate}
        \item If $S$ is a smooth cobordism of genus $g(S)$ between knots $K$ and $K'$, 
        \[
            |\nu(K) - \nu(K')| \leq g(S).
        \]
        \item If two knots $K^+$ and $K^-$ differ by a single crossing change, from positive to negative, then
        \[
            0 \leq \nu(K^+) - \nu(K^-) \leq 1.
        \]
    \end{enumerate}
\end{proposition}

The following lemma states the effect on a slice-torus invariant of a full twist on a 2-strand tangle. Let $K$ be a knot, and let $J = \partial D$ be an unknot bounding a disk $D$ that intersects $K$ transversely at two points. We can modify $K$ by performing full twists along the two strands passing through $D$. Denote by $K_n$ the knot obtained from $K$ by applying $n \in \mathbb{Z}$ full twists along $D$.

\begin{lemma}[{\cite[Proposition 2.1]{LZ:2024}}]\label{lem:st-twist}
    Let $\nu$ be a slice-torus invariant, and let $m \geq n$ be two integers.
    \begin{enumerate}
        \item[(a)] If the two strands are oriented in opposite directions, then \[
            \nu(K_m) - \nu(K_n) \in \{0, -1\}.
        \]
        \item[(b)] If the two strands are oriented in the same direction, then 
        \[
            \nu(K_m) - \nu(K_n) \in \{m - n - 1, m - n\}.
        \]
    \end{enumerate}
\end{lemma}

\begin{figure}
    \centering
    \begin{subfigure}{.45\textwidth}
        \centering
        \tikzset{every picture/.style={line width=0.75pt}} %set default line width to 0.75pt        

\begin{tikzpicture}[x=0.75pt,y=0.75pt,yscale=-.6,xscale=.6]
%uncomment if require: \path (0,180); %set diagram left start at 0, and has height of 180

%Straight Lines [id:da6315012118684977] 
\draw    (0,150) -- (248,150) ;
\draw [shift={(250,150)}, rotate = 180] [color={rgb, 255:red, 0; green, 0; blue, 0 }  ][line width=0.75]    (10.93,-3.29) .. controls (6.95,-1.4) and (3.31,-0.3) .. (0,0) .. controls (3.31,0.3) and (6.95,1.4) .. (10.93,3.29)   ;
%Shape: Circle [id:dp08133097512373277] 
\draw  [fill={rgb, 255:red, 0; green, 0; blue, 0 }  ,fill opacity=1 ] (0,63.2) .. controls (0,61.43) and (1.43,60) .. (3.2,60) .. controls (4.97,60) and (6.41,61.43) .. (6.41,63.2) .. controls (6.41,64.97) and (4.97,66.41) .. (3.2,66.41) .. controls (1.43,66.41) and (0,64.97) .. (0,63.2) -- cycle ;
%Shape: Circle [id:dp19488342644404133] 
\draw  [fill={rgb, 255:red, 0; green, 0; blue, 0 }  ,fill opacity=1 ] (30,63.2) .. controls (30,61.43) and (31.43,60) .. (33.2,60) .. controls (34.97,60) and (36.41,61.43) .. (36.41,63.2) .. controls (36.41,64.97) and (34.97,66.41) .. (33.2,66.41) .. controls (31.43,66.41) and (30,64.97) .. (30,63.2) -- cycle ;
%Shape: Circle [id:dp4339152447269542] 
\draw  [fill={rgb, 255:red, 0; green, 0; blue, 0 }  ,fill opacity=1 ] (60,63.2) .. controls (60,61.43) and (61.43,60) .. (63.2,60) .. controls (64.97,60) and (66.41,61.43) .. (66.41,63.2) .. controls (66.41,64.97) and (64.97,66.41) .. (63.2,66.41) .. controls (61.43,66.41) and (60,64.97) .. (60,63.2) -- cycle ;
%Shape: Circle [id:dp730581601244595] 
\draw  [fill={rgb, 255:red, 0; green, 0; blue, 0 }  ,fill opacity=1 ] (90,63.2) .. controls (90,61.43) and (91.43,60) .. (93.2,60) .. controls (94.97,60) and (96.41,61.43) .. (96.41,63.2) .. controls (96.41,64.97) and (94.97,66.41) .. (93.2,66.41) .. controls (91.43,66.41) and (90,64.97) .. (90,63.2) -- cycle ;
%Shape: Circle [id:dp6480973051047133] 
\draw  [fill={rgb, 255:red, 0; green, 0; blue, 0 }  ,fill opacity=1 ] (120,93.2) .. controls (120,91.43) and (121.43,90) .. (123.2,90) .. controls (124.97,90) and (126.41,91.43) .. (126.41,93.2) .. controls (126.41,94.97) and (124.97,96.41) .. (123.2,96.41) .. controls (121.43,96.41) and (120,94.97) .. (120,93.2) -- cycle ;
%Shape: Circle [id:dp2485291341050786] 
\draw  [fill={rgb, 255:red, 0; green, 0; blue, 0 }  ,fill opacity=1 ] (150,93.2) .. controls (150,91.43) and (151.43,90) .. (153.2,90) .. controls (154.97,90) and (156.41,91.43) .. (156.41,93.2) .. controls (156.41,94.97) and (154.97,96.41) .. (153.2,96.41) .. controls (151.43,96.41) and (150,94.97) .. (150,93.2) -- cycle ;
%Shape: Circle [id:dp6994094092972528] 
\draw  [fill={rgb, 255:red, 0; green, 0; blue, 0 }  ,fill opacity=1 ] (180,93.2) .. controls (180,91.43) and (181.43,90) .. (183.2,90) .. controls (184.97,90) and (186.41,91.43) .. (186.41,93.2) .. controls (186.41,94.97) and (184.97,96.41) .. (183.2,96.41) .. controls (181.43,96.41) and (180,94.97) .. (180,93.2) -- cycle ;
%Shape: Circle [id:dp8900658770248643] 
\draw  [fill={rgb, 255:red, 0; green, 0; blue, 0 }  ,fill opacity=1 ] (210,93.2) .. controls (210,91.43) and (211.43,90) .. (213.2,90) .. controls (214.97,90) and (216.41,91.43) .. (216.41,93.2) .. controls (216.41,94.97) and (214.97,96.41) .. (213.2,96.41) .. controls (211.43,96.41) and (210,94.97) .. (210,93.2) -- cycle ;
%Shape: Rectangle [id:dp9466118018891896] 
\draw  [color={rgb, 255:red, 255; green, 255; blue, 255 }  ,draw opacity=1 ] (0,0) -- (30,0) -- (30,30) -- (0,30) -- cycle ;

% Text Node
\draw (237,122.4) node [anchor=north west][inner sep=0.75pt]    {$n$};

\end{tikzpicture}
        \caption{}
        \label{fig:nu-graph1}
    \end{subfigure}
    \begin{subfigure}{.45\textwidth}
        \centering
        \tikzset{every picture/.style={line width=0.75pt}} %set default line width to 0.75pt        

\begin{tikzpicture}[x=0.75pt,y=0.75pt,yscale=-.6,xscale=.6]
%uncomment if require: \path (0,180); %set diagram left start at 0, and has height of 180

%Straight Lines [id:da859643119698466] 
\draw    (0,150) -- (248,150) ;
\draw [shift={(250,150)}, rotate = 180] [color={rgb, 255:red, 0; green, 0; blue, 0 }  ][line width=0.75]    (10.93,-3.29) .. controls (6.95,-1.4) and (3.31,-0.3) .. (0,0) .. controls (3.31,0.3) and (6.95,1.4) .. (10.93,3.29)   ;
%Shape: Circle [id:dp3358220431898171] 
\draw  [fill={rgb, 255:red, 0; green, 0; blue, 0 }  ,fill opacity=1 ] (30,123.2) .. controls (30,121.43) and (31.43,120) .. (33.2,120) .. controls (34.97,120) and (36.41,121.43) .. (36.41,123.2) .. controls (36.41,124.97) and (34.97,126.41) .. (33.2,126.41) .. controls (31.43,126.41) and (30,124.97) .. (30,123.2) -- cycle ;
%Shape: Circle [id:dp7933734133837682] 
\draw  [fill={rgb, 255:red, 0; green, 0; blue, 0 }  ,fill opacity=1 ] (60,93.2) .. controls (60,91.43) and (61.43,90) .. (63.2,90) .. controls (64.97,90) and (66.41,91.43) .. (66.41,93.2) .. controls (66.41,94.97) and (64.97,96.41) .. (63.2,96.41) .. controls (61.43,96.41) and (60,94.97) .. (60,93.2) -- cycle ;
%Shape: Circle [id:dp920565011649488] 
\draw  [fill={rgb, 255:red, 0; green, 0; blue, 0 }  ,fill opacity=1 ] (90,63.2) .. controls (90,61.43) and (91.43,60) .. (93.2,60) .. controls (94.97,60) and (96.41,61.43) .. (96.41,63.2) .. controls (96.41,64.97) and (94.97,66.41) .. (93.2,66.41) .. controls (91.43,66.41) and (90,64.97) .. (90,63.2) -- cycle ;
%Shape: Circle [id:dp6881239873084085] 
\draw  [fill={rgb, 255:red, 0; green, 0; blue, 0 }  ,fill opacity=1 ] (120,63.2) .. controls (120,61.43) and (121.43,60) .. (123.2,60) .. controls (124.97,60) and (126.41,61.43) .. (126.41,63.2) .. controls (126.41,64.97) and (124.97,66.41) .. (123.2,66.41) .. controls (121.43,66.41) and (120,64.97) .. (120,63.2) -- cycle ;
%Shape: Circle [id:dp7431179841558968] 
\draw  [fill={rgb, 255:red, 0; green, 0; blue, 0 }  ,fill opacity=1 ] (150,33.2) .. controls (150,31.43) and (151.43,30) .. (153.2,30) .. controls (154.97,30) and (156.41,31.43) .. (156.41,33.2) .. controls (156.41,34.97) and (154.97,36.41) .. (153.2,36.41) .. controls (151.43,36.41) and (150,34.97) .. (150,33.2) -- cycle ;
%Shape: Circle [id:dp688176861965861] 
\draw  [fill={rgb, 255:red, 0; green, 0; blue, 0 }  ,fill opacity=1 ] (180,3.2) .. controls (180,1.43) and (181.43,0) .. (183.2,0) .. controls (184.97,0) and (186.41,1.43) .. (186.41,3.2) .. controls (186.41,4.97) and (184.97,6.41) .. (183.2,6.41) .. controls (181.43,6.41) and (180,4.97) .. (180,3.2) -- cycle ;
%Shape: Rectangle [id:dp4464423142969526] 
\draw  [color={rgb, 255:red, 255; green, 255; blue, 255 }  ,draw opacity=1 ] (0,0) -- (30,0) -- (30,30) -- (0,30) -- cycle ;

% Text Node
\draw (237,122.4) node [anchor=north west][inner sep=0.75pt]    {$n$};

\end{tikzpicture}
        \caption{}
        \label{fig:nu-graph2}
    \end{subfigure}
    \caption{Plotting $(n, \nu(K_n))$.}
\end{figure}

One can visualize this lemma by considering the discrete graph
\[
    \Gamma = \{ \ (n, \nu(K_n)) \mid n \in \ZZ \ \} \subset \RR^2.
\]
In the first case, the graph appears constant, except for one possible `dropping point' (\Cref{fig:nu-graph1}). In the second case, the graph appears affine with slope $1$, except for one possible `stationary point' (\Cref{fig:nu-graph2}).

\begin{proposition}
\label{prop:st-Wh}
    Let $J_n$ be an $n$-twist knot, where $n$ denotes the number of full twists. Then, for a slice-torus invariant $\nu$,
    \[
    \nu(J_n) = \begin{cases}
        1 & \text{if } n < 0, \\
        0 & \text{if } n \geq 0.
    \end{cases}
    \]
\end{proposition}

\begin{proof}
    $J_{-1} = T_{2,3}$, and $J_0$ is the unknot. Hence $\nu(J_{-1})=1$, $\nu(J_0)=0$. Other values are determined by \Cref{lem:st-twist}.
\end{proof}

Now, consider a $3$-strand pretzel knot $P(p,q,r)$ with $p, q, r \in \mathbb{Z}$. Without loss of generality, we may assume that $p > 0$, and it suffices to consider the following cases: 
\begin{itemize}[leftmargin=3em]
    \item[($O$)] $p, q, r$ are all odd, and 
    \begin{itemize}
        \item[($O_+$)] $q, r > 0$,
        \item[($O_-$)] $q, r < 0$.
    \end{itemize}
    
    \item[($E$)] $p$ is even; $q, r$ are odd, and 
    \begin{itemize}
        \item[($E_{++}$)] $q, r > 0$,
        \item[($E_{-+}$)] $q < 0 < r$,
        \item[($E_{--}$)] $q, r < 0$.
    \end{itemize}
\end{itemize}
This is because 3-strand pretzel knots possess symmetries under permutations of $(p, q, r)$, and also satisfy
\[
    -P(p, q, r) = P(-p, -q, -r),
\]
where $-P$ denotes the mirror of $P$. The case $p = 0$ is also excluded, since $P(0,q,r)$ is just the connected sum of two torus knots, $T_{2,q} \# T_{2,r}$. Also note that if $q = -r$, then $P$ is ribbon, in particular slice, and hence $\nu(P) = 0$ (see \cite[Proposition 2.1]{Lecuona:2015}). 

The following proposition states that, for any slice-torus invariant $\nu$, the value of $\nu(P(p,q,r))$ can be determined under certain conditions solely from its formal properties.

\begin{proposition}
\label{prop:st_pretzel}
    Consider the 3-strand pretzel knot $P = P(p, q, r)$ with integers $p, q, r$. Let $\nu$ be any slice-torus invariant. Depending on the five types described above, $\nu(P)$ is partially determined as follows:
    \vspace{1em}
    \begin{itemize}[leftmargin=4em]   
        \item[($O_+$)] $\nu(P) = -1$.
        \item[($O_-$)] $\nu(P) = \begin{cases}
            0 & \text{if\ } p \geq \min\{-q,-r\} \\ 
            0 \text{\ or\ } 1 & \text{otherwise}. 
        \end{cases}$
        \item[($E_{++}$)] $\nu(P) = \displaystyle \frac{q+r}{2} - 1$. 
        \item[($E_{-+}$)] $\nu(P) = \begin{cases}
            \displaystyle \frac{q+r}{2}     & \text{if\ } q + r \leq 0, \\[.75em]
            \displaystyle \frac{q+r}{2}-1 \text{\ or\ } \frac{q+r}{2} & \text{otherwise}. 
        \end{cases}$.
        \item[($E_{--}$)] $\nu(P) = \displaystyle \frac{q+r}{2}$.
    \end{itemize}
    \vspace{1em}
\end{proposition}

\begin{proof} $\ $
\begin{figure}[t]
    \centering
    \input{tikzpictures/pretzel-cut-odd}
    \caption{Genus 1 cobordism from odd $P$ to the unknot}
    \label{fig:pretzel-cut-odd}
    \vspace{2em}
    \input{tikzpictures/pretzel-cut-even}
    \caption{Genus 1 cobordism from even $P$ to $T_{2, q + r \pm 1}$.}
    \label{fig:pretzel-cut-even}
\end{figure}

    \begin{itemize}[leftmargin=4em]   
    \item[($O$)] 
        Suppose $p, q, r$ are odd. Take a genus 1 cobordism from $P$ to the unknot, described as a sequence of band surgeries and twists, as in \Cref{fig:pretzel-cut-odd}. This implies that $|\nu(P)| \leq 1$.
        
    \item[($O_+$)]
        Further suppose $q, r > 0$. Observe that $P(1,1,1) = -T_{2,3}$, so $\nu(P(1,1,1)) = -1$. $P(p,q,r)$ can be obtained by adding positive twists on $P(1,1,1)$, each of which has the effect of possibly decreasing the value of $\nu$. However, we have $\nu(P) \geq -1$ so $\nu(P) = -1$. 
        
    \item[($O_-$)]
        Next, suppose $q, r < 0$ and $p \geq \min\{-q, -r\}$. Without loss of generality, we may assume that $p \geq -q$. From $-p \leq q < 0 < -r$, we have
        \[ 
            0 = \nu(P(p,-p,r)) 
            \geq \nu(P(p,q,r))
            \geq \nu(P(p,-r,r))
            = 0. 
        \]
        
        Next, if $p < \min\{-q, -r\}$, then 
        \[
            \nu(P(p, q, r)) 
            \geq \nu(P(p, -1, r)) 
            = 0
        \]
        so $\nu(P(p, q, r))$ is either $0$ or $1$.
        
    \item[(E)]
        Suppose $p$ is even. If $q + r = 0$, then $\nu(P) = 0$. Suppose $q + r \neq 0$. Take two genus 1 cobordisms from $P$ to the $(2, q + r - 1)$-torus knot $T$ and to the $(2, q + r + 1)$-torus knot $T'$, as described in \Cref{fig:pretzel-cut-odd}. This implies 
        \[
            | \nu(P) - \nu(T) | \leq 1,\quad 
            | \nu(P) - \nu(T') | \leq 1.
        \]
        If $q + r > 0$, then both $T, T'$ are positive and $\nu(T) = \frac{q + r - 2}{2}$, $\nu(T') = \frac{q + r}{2}$, so $\nu(P)$ is either $\frac{q + r - 2}{2}$ or $\frac{q + r}{2}$. If $q + r < 0$, then both $T, T'$ are negative and $\nu(T) = \frac{q + r}{2}$, $\nu(T') = \frac{q + r + 2}{2}$, so $\nu(P)$ is either $\frac{q + r}{2}$ or $\frac{q + r + 2}{2}$. 

    \item[($E_{++}$)]
        Further suppose $q, r > 0$. Observe that $P$ can be obtained from $P(p,1,1)$ by adding positive twists, and $P(p,1,1)$ is the $p$-twist knot $J_p$. From $\nu(P(p,1,1)) = 0$, \Cref{lem:st-twist} implies $\nu(P) = \frac{q+r-2}{2}$.
    
    \item[($E_{--}$)]
        Next, suppose $q, r < 0$. One can observe that $P$ is a negative knot, which is \textit{squeezed}, hence all slice-torus invariants takes the same value (\cite[Proposition 1.2, Lemma 3.5]{LZ:2024}). In particular, the $s$-invariant can be computed from its diagram as $s(P) = q+r$. Thus $\nu(P) = \frac{q+r}{2}$.
        
    \item[($E_{-+}$)]
        Finally, suppose $q < 0 < r$. The case $q + r > 0$ is discussed above, so suppose $q + r \leq 0$. In this case, $P$ can be obtained from $P(p,q,-1)$ by adding positive twists. From the above result, we have $\nu(P(p,q,-1)) = \frac{q-1}{2}$, and \Cref{lem:st-twist} implies $\nu(P) = \frac{q+r}{2}$.
        \qedhere
    \end{itemize}    
\end{proof}

Thus, the undetermined cases are:
\begin{itemize}[leftmargin=4em]   
    \item[($O^\star_-$)] $p, q, r$ are odd, $p > 0$ and $q, r < -p$,
    \item[($E^\star_{-+}$)] $p > 0$ is even, $q < 0 < r$ are odd, and $q + r > 0$.
\end{itemize}
By further combining previously known results on $3$-strand pretzel knots and general results on slice-torus invariants, we may determine $\nu(P)$ for a large part of ($O^\star_-$), and $s(P)$ for a large part of ($E^\star_{-+}$), as explained below:

\begin{remark}
    For type ($O^\star_-$) with the additional condition $p, -q, -r \geq 3$, it follows from \cite[Corollary 1.10]{BBG:2019} that $P$ is strongly quasi-positive. From \cite[Theorem 4]{Livingston:2004}, any strongly quasi-positive knot $K$ has $\nu(K) = g(K)$. Finally, \cite[Corollary 2.7]{Kim-Lee:2007} states that $g(P) = 1$. 
\end{remark}

\begin{remark}
    For type ($E^\star_{-+}$) with the additional condition $p < -q < r$, it follows from \cite[Theorem 1.4]{Greene:2010} that $P$ is quasi-alternating. From \cite[Theorem 1]{Manolescu-Ozvath:2008}, any quasi-alternating knot $K$ is $\Kh$-thin, hence $s(K)$ coincides with the knot signature $\sigma(K)$. Finally, we may compute from \cite[Theorem 1.18]{Jabuka:2010} that $\sigma(P) = q + r$. 
\end{remark}

In the following subsections, we determine $s$ for the remaining cases solely from the diagrammatic method, without any restriction or reference to previously known results.

\subsection{Notations and general strategy}
\label{subsec:pretzel-strategy}

First, we fix notations that will be used in the following subsections. For any $p, q, r \in \ZZ$, the 3-strand pretzel link $P = P(p, q, r)$ can be expressed as 
\[
    P(p, q, r) = D(T_p, T_q, T_r)
\]
where $T_p$, $T_q$, $T_r$ are the twist tangles introduced in \Cref{sec:twist-tangles} and $D$ is the obvious $3$-input planar arc diagram $D$, as depicted in \Cref{fig:intro-lee-decomp}. By composing the reductions of \Cref{prop:twist-tangle-retract}, we get a reduced complex
\[
    [P] = D([T_p], [T_q], [T_r]) \to 
    D(E_p, E_q, E_r) =: E.
\]

The structure of the complex $E$ can be described as a graph, where vertices $E(v)$ are given by summands of $E$, indexed by triples $v = (v_1, v_2, v_3)$ with $0 \leq |v_1| \leq p$,\ $0 \leq |v_2| \leq q$ and $0 \leq |v_3| \leq r$. Only eight diagrams appear in the summands of $E$, which we denote by
\[
\begin{gathered}
    \sfA   = D(\sfE_1, \sfE_1, \sfE_1), \\ 
    \sfB_1 = D(\sfE_0, \sfE_1, \sfE_1), \
    \sfB_2 = D(\sfE_1, \sfE_0, \sfE_1), \ 
    \sfB_3 = D(\sfE_1, \sfE_1, \sfE_0), \\
    \sfC_1 = D(\sfE_0, \sfE_0, \sfE_1), \
    \sfC_2 = D(\sfE_0, \sfE_1, \sfE_0), \ 
    \sfC_3 = D(\sfE_1, \sfE_0, \sfE_0), \\
    \sfD   = D(\sfE_0, \sfE_0, \sfE_0)
\end{gathered}
\]
as in \Cref{fig:pretzel-diagrams}. We also use the shorthand notation $\sfX(v_1, v_2, v_3)$ to indicate that $E$ has diagram $\sfX$ at index $(v_1, v_2, v_3)$. 

\begin{figure}[t]
    \centering
    \input{tikzpictures/pretzel-diagrams}
    \caption{Diagrams $\sfA$ - $\sfD$ appearing in the reduced complex $E$.}
    \label{fig:pretzel-diagrams}
\end{figure}

Each edge $E(v) \to E(v')$ is given by one of the morphisms $a, b, m$ or $\Delta$ composed using $D$ with two identify morphisms. We put subscript from $1$ to $3$ to these symbols to indicate the index of the input hole of $D$, for example $a_1 = D(a, I, I)$ and $b_2 = D(I, b, I)$. Edges are also assigned signs by the standard signing rule: put $(-1)^{v_1}$ for an edge $E(v_1,v_2,v_3) \to E(v_1,v_2+1,v_3)$, put $(-1)^{v_1+v_2}$ for an edge $E(v_1,v_2,v_3) \to E(v_1,v_2,v_3+1)$. We call edges of the form 
\[
    E(v_1, v_2, v_3) \to E(v_1 + 1, v_2, v_3)
\]
\textit{horizontal edges}. If we endow the lexicographical order on the indices and align the objects vertically in each homological grading so that the horizontal edges are drawn horizontally, then all other edges will necessarily have negative slopes, as in \Cref{fig:simplify-E}. A sequence of horizontal edges for a fixed $(v_2, v_3)$ is called a \textit{horizontal sequence}, and is denoted $E(*, v_2, v_3)$. 

The Lee cycle $\ca(P)$ of $P$ is mapped to 
\[
    \pm H^{k} z
\]
for some $k \geq 0$ determined by \Cref{prop:twist-tangle-lee-cycle}, and $z$ is a cycle in $E$ of the form $D(z^\pm_p, z^\pm_q, z^\pm_r)$ with target in $E(0, 0, 0)$. The basic strategy is to further apply delooping and Gaussian elimination to $E$, until the $H$-divisibility of $z$ can be determined. 

For a diagram $\sfX$ that contains a distinguished circle $C$, we denote by $\sfX^+$ and $\sfX^-$ the two objects that appear by delooping $C$. For the diagram $\sfA$, we choose $C$ to be the lower circle, and for $\sfC_i$ $(i = 1, 2, 3)$, we choose $C$ to be the inner circle. With this notation, \Cref{lem:Tq-lem1} can be expressed as follows: 
\[
\begin{tikzcd}[row sep=.15em]
    & & & &  && \sfY^+ \\
    \sfX\arrow[rr,"\Delta"]&& \sfY & \leadsto& \sfX \arrow[rru, "Y" description] \arrow[rrd, "I" description] && \\
    & & & & & & \sfY^-
\end{tikzcd}
\]
\[
\begin{tikzcd}[row sep=.15em]
    & & & &  \sfX^+\arrow[rrd,"I" description] &&  \\
    \sfX\arrow[rr,"m"]&& \sfY & \leadsto& && \sfY\\
    & & & & \sfX^- \arrow[rru,"X" description]&&
\end{tikzcd}
\]
\[
\begin{tikzcd}[row sep=.15em]
    & & & &  \sfX^+\arrow[rrdd,"I" description] \arrow[rr,"Y"] && \sfX^+ \\
    \sfX\arrow[rr,"a"]&& \sfX & \leadsto& &&\\
    & & & & \sfX^- \arrow[rr,"X"]&& \sfX^-
\end{tikzcd}
\]
\[
\begin{tikzcd}[row sep=.15em]
    & & & &  \sfX^+\arrow[rrdd,"I" description] \arrow[rr,"-X"] && \sfX^+ \\
    \sfX\arrow[rr,"b"]&& \sfX & \leadsto& &&\\
    & & & & \sfX^- \arrow[rr,"-Y"]&& \sfX^-
\end{tikzcd}
\]
The downward diagonal arrows on the right side with labels $I$ are the ones to which Gaussian elimination can be applied. Typically, we apply this reduction on horizontal edges. For instance, suppose we have a horizontal sequence
\[
\begin{tikzcd}
    \cdots \arrow[r] 
    & \sfA \arrow[r, "a"] 
    & \sfA \arrow[r, "b"] 
    & \sfA \arrow[r] 
    & \cdots.
\end{tikzcd}
\]
After delooping one the top circle on each diagram $\sfA$, the sequence transforms into 
\[
\begin{tikzcd}
    \cdots \arrow[r] 
    & \sfA^+ \arrow[r, "Y"] \arrow[rd, "I" description] 
    & \sfA^+ \arrow[r, "-X"] \arrow[rd, "I" description] 
    & \sfA^+ \arrow[r] & \cdots \\
    \cdots \arrow[r] 
    & \sfA^- \arrow[r, "X"] 
    & \sfA^- \arrow[r, "-Y"] 
    & \sfA^- \arrow[r] & \cdots
\end{tikzcd}
\]
By applying Gaussian elimination to the two diagonal arrows labeled $I$, the center object $\sfA$ can be completely eliminated as 
\[
\begin{tikzcd}
    \cdots \arrow[r, dotted] 
    & 0 \arrow[rd, dotted] \arrow[r, dotted] 
    & 0 \arrow[rd, dotted] \arrow[r, dotted] 
    & \sfA^+ \arrow[r] & \cdots \\
    \cdots \arrow[r] 
    & \sfA^- \arrow[r, dotted] 
    & 0 \arrow[r, dotted] 
    & 0 \arrow[r, dotted] 
    & \cdots.
\end{tikzcd}    
\]
Due to \Cref{prop:gauss-elim}, each elimination will generally produce new edges, but will not affect other horizontal edges. To see that this, suppose there is a part in $E$, 
\[
    \begin{tikzcd}[row sep=.5em, every arrow/.append style={start anchor=east, end anchor=west}]
    {E(u_1, u_2, u_3)\ \ \ } \arrow[rd, shift left] & \\
    {E(v_1 - 1, v_2, v_3)} \arrow[r] \arrow[rd] 
    & {E(v_1, v_2, v_3)} \\
    & {E(w_1, w_2, w_3)}
    \end{tikzcd}
\]
where $u_1 \geq v_1$ and $v_1 - 1 \geq w_1$ and the horizontal edge is invertible. After eliminating this edge, it will produce a new edge 
\[
    \begin{tikzcd}[row sep=.5em, every arrow/.append style={start anchor=east, end anchor=west}]
    {E(u_1, u_2, u_3)\ \ \ } \arrow[rdd] \arrow[rd, dotted, shift left] & \\
    {\qquad 0 \qquad} \arrow[r, dotted] \arrow[rd, dotted] 
    & {\qquad 0 \qquad} \\
    & {E(w_1, w_2, w_3)}
    \end{tikzcd}
\]
which has $u_1 \geq v_1 > w_1$, hence is not horizontal. 

More generally, if there is a \textit{zigzag} between two objects such that the horizontal edges are invertible, then after elimination of the horizontal edges, a new steep edge between the objects will be produced. 
\[
\begin{tikzcd}[row sep=.5em]
\sfX \arrow[rd, shift left] & & & \sfX \arrow[rdddd] \arrow[rd, dotted] &      \\
{*} \arrow[r] \arrow[rd] & {*} & & 0 \arrow[r, dotted] \arrow[rd, dotted] & 0    \\
{*} \arrow[r] \arrow[rd] & {*} & \leadsto & 0 \arrow[r, dotted] \arrow[rd, dotted] & 0   \\
{*} \arrow[r] \arrow[rd] & {*} & & 0 \arrow[r, dotted] \arrow[rd, dotted] & 0    \\ & \sfY & & & \sfY
\end{tikzcd}    
\]

In many cases we shall see, the horizontal and diagonal edges consisting the zigzag are simply $\pm I$, so the newly produced will also be $\pm I$. The following lemma will be repeatedly used, which describes how the non-horizontal edges are modified under the elimination of specific horizontal edges.

\begin{lemma}
\label{lem:even-pretzel-reduce}
    Suppose we have the following parts in a complex $E$:
    \[
\begin{tikzcd}[column sep=3em]
\sfA \arrow[r, "a"] \arrow[rd, "\pm b" description] 
& \sfA \arrow[rd, "\mp b" description] 
& 
& \sfA \arrow[r, "a"] \arrow[rd, "\pm a" description] 
& \sfA \arrow[rd, "\mp a" description] & \\ 
& \sfA \arrow[r, "a"] 
& \sfA & 
& \sfA \arrow[r, "a"] 
& \sfA
\end{tikzcd}
    \]
    By delooping, these parts transform into:
    \[
\begin{tikzcd}[column sep=4em]
\sfA^+ \arrow[r] \arrow[rd, "I" description, dashed] \arrow[rdd]  \arrow[rddd] 
& \sfA^+ \arrow[rdd]  \arrow[rddd] & 
& \sfA^+ \arrow[r] \arrow[rd, "I" description, dashed] \arrow[rdd] \arrow[rddd] 
& \sfA^+ \arrow[rdd]  \arrow[rddd] & \\
\sfA^- \arrow[r] \arrow[rdd] 
& \sfA^- \arrow[rdd] & 
& \sfA^- \arrow[r] \arrow[rdd] 
& \sfA^- \arrow[rdd] & \\ 
& \sfA^+ \arrow[r] \arrow[rd, "I" description, dashed] 
& \sfA^+ & 
& \sfA^+ \arrow[r] \arrow[rd, "I" description, dashed] 
& \sfA^+ \\ 
& \sfA^- \arrow[r] 
& \sfA^- & 
& \sfA^- \arrow[r] & \sfA^-
\end{tikzcd}
    \]
    By eliminating the dashed diagonal arrows labeled $I$, these parts transform into 
    \[
\begin{tikzcd}[column sep=3em]
0 \arrow[r, dotted] \arrow[rd, dotted] \arrow[rdd, dotted] \arrow[rddd, dotted] 
& \sfA^+ \arrow[rdd, "\pm U" description, dashed] \arrow[rddd, dotted] & & 
& 0 \arrow[r, dotted] \arrow[rd, dotted] \arrow[rdd, dotted] \arrow[rddd, dotted] 
& \sfA^+ \arrow[rdd, "0" description, dotted] \arrow[rddd, dotted] & \\
\sfA^- \arrow[r, dotted] \arrow[rdd, "\mp U" description, dashed] & 
0 \arrow[rdd, dotted] & & 
& \sfA^- \arrow[r, dotted] \arrow[rdd, "0" description, dotted] 
& 0 \arrow[rdd, dotted] 
& \\ 
& 0 \arrow[r, dotted] \arrow[rd, dotted] 
& \sfA^+ & & 
& 0 \arrow[r, dotted] \arrow[rd, dotted] 
& \sfA^+ \\ 
& \sfA^- \arrow[r, dotted] 
& 0. 
& & & \sfA^- \arrow[r, dotted] 
& {0}
\end{tikzcd}
    \]
    where $U = X + Y$. A similar statement also hold when $a$ and $b$ are swapped. 
\end{lemma}

\subsection{Odd pretzel knots}

The following proposition determines $s$ for type ($O^\star_-$). 

\begin{proposition}
\label{prop:odd-pretzel}
    For odd integers $p, q, r > 0$ with $p < \min\{q, r\}$, 
    \[
        s(P(p, -q, -r)) = 2.
    \]
\end{proposition}

\begin{proof}
    Consider the tangle decomposition
    \[
        P = P(p, -q, -r) = D(T^-_p, T^+_{-q}, T^+_{-r}).
    \]
    and the reduction
    \[
        [P] = D([T^-_p], [T^+_{-q}], [T^+_{-r}]) \to D(E^-_p, E^+_{-q}, E^+_{-r}) = E
    \]
    obtained by composing the reductions of \Cref{prop:twist-tangle-retract}. The graph representing the complex $E$ has vertices $E(v)$ indexed by $v = (v_1, v_2, v_3)$ with $-p \leq v_1 \leq 0$,\ $0 \leq v_2 \leq q$ and $0 \leq v_3 \leq r$, and the diagram for $E(v_1,v_2,v_3)$ is given by $D(\sfE_i, \sfE_j, \sfE_k)$, where
    \begin{align*}
        i = 0 &\iff v_1 = -p \\
        j = 0 &\iff v_2 = q \\
        k = 0 &\iff v_3 = r.
    \end{align*}
    See \Cref{fig:simplify-E} for the case $p = 3$. The Lee cycle $\ca(P)$ of $P$ is transformed by the reduction into 
    \[
        \pm H^{p - 1} z
    \]
    where $z$ is a cycle with target in $\underline{\sfA}(0, 0, 0)$. From $d_H(P) \geq p - 1$,\ $w(P) = -p + q + r$ and $r(P) = p + q + r -2$, we have 
    \[
        s_H(P) \geq 2(p - 1) + (-p + q + r) - (p + q + r - 2) + 1 = 0.
    \]
    Thus it suffices to show that the cycle $z$ is divisible exactly once in $E$. 

\begin{figure}[t]
    \centering
    \begin{tikzcd}[row sep=.3em, column sep=4em, every arrow/.append style={gray, no head, start anchor=east, end anchor=west}, cells={nodes={font=\small}}]
\underline{-1} & \underline{0} & \underline{1} \\
{\sfA(-1, 0, 0)} 
    \arrow[r, thick, red, "a" description] 
    \arrow[rd] 
& {\underline{\sfA}(0, 0, 0)\ \ } 
    \arrow[rd] \arrow[rdd] & \\
{\sfA(-2, 1, 0)} 
    \arrow[r, thick, blue, "b" description] 
    \arrow[rdd] 
    \arrow[rddd] 
& {\sfA(-1, 1, 0)} 
    \arrow[r, thick, red, "a" description] 
    \arrow[rdd] \arrow[rddd] 
& {\sfA(0, 1, 0)} \\
{\sfA(-2, 0, 1)} 
    \arrow[r, thick, blue, "b" description] 
    \arrow[rdd] 
    \arrow[rddd] 
& {\sfA(-1, 0, 1)} 
    \arrow[r, thick, red, "a" description] 
    \arrow[rdd] 
    \arrow[rddd] 
& {\sfA(0, 0, 1)} \\
{\sfB_1(-3, 2, 0)} 
    \arrow[r, thick, orange, "\Delta" description] 
    \arrow[rddd] 
    \arrow[rdddd] 
& {\sfA(-2, 2, 0)} 
    \arrow[r, thick, blue, "b" description] 
    \arrow[rddd] 
    \arrow[rdddd] 
& {\sfA(-1, 2, 0)} \\
{\sfB_1(-3, 1, 1)} 
    \arrow[r, thick, orange, "\Delta" description] 
    \arrow[rddd] 
    \arrow[rdddd] 
& {\sfA(-2, 1, 1)} 
    \arrow[r, thick, blue, "b" description] 
    \arrow[rddd] 
    \arrow[rdddd] 
& {\sfA(-1, 1, 1)} \\
{\sfB_1(-3, 0, 2)} 
    \arrow[r, thick, orange, "\Delta" description] 
    \arrow[rddd] 
    \arrow[rdddd] 
& {\sfA(-2, 0, 2)} 
    \arrow[r, thick, blue, "b" description] 
    \arrow[rddd] 
    \arrow[rdddd] 
& {\sfA(-1, 0, 2)} \\
& {\sfB_1(-3, 3, 0)} 
    \arrow[r, thick, orange, "\Delta" description] 
& {\sfA(-2, 3, 0)} \\
& {\sfB_1(-3, 2, 1)} 
    \arrow[r, thick, orange, "\Delta" description] 
& {\sfA(-2, 2, 1)} \\
& {\sfB_1(-3, 1, 2)} 
    \arrow[r, thick, orange, "\Delta" description] 
& {\sfA(-2, 1, 2)} \\
& {\sfB_1(-3, 0, 3)} 
    \arrow[r, thick, orange, "\Delta" description] 
& {\sfA(-2, 0, 3)} \\
& & \vdots 
\end{tikzcd}
    \caption{A part of the complex $E$ described as a graph. Objects are aligned lexicographically, so that the horizontal edges are drawn horizontally. The underlined object indicates the target of the cycle $z$. }
    \label{fig:simplify-E}
\end{figure}

    Observe that each object in homological grading $\leq 0$ are either 
    \[
    E(v_1, v_2, v_3) = 
    \begin{cases}
        \sfA   & \text{if\ } v_1 > -p,\\
        \sfB_1 & \text{if\ } v_1 = -p.
    \end{cases}
    \]
    Indeed, for $v_1 + v_2 + v_3 \leq 0$ to hold, we cannot have $v_2 = q$ or $v_3 = r$; otherwise we would have
    \[
        p \geq -v_1 \geq v_2 + v_3 \geq \min\{q, r\}
    \]
    contradicting the assumption. 

    As explained in \Cref{subsec:pretzel-strategy}, each horizontal sequence of the form 
    \[
\begin{tikzcd}
    \sfB_1 \arrow[r, "\Delta_1"] 
    & \sfA \arrow[r, "b_1"] 
    & \sfA \arrow[r, "a_1"] 
    & \cdots
\end{tikzcd}
    \]
    can be eliminated up to the rightmost term $\sfA$. Thus, we may eliminate all horizontal sequences appearing in homological grading $0$, except for the one ending with $\underline{\sfA}(0, 0, 0)$, which is the target of the cycle $z$. It has a unique incoming edge from $E(-1, 0, 0)$, which can be reduced to 
    \[
\begin{tikzcd}
\cdots \arrow[r, "0", dotted] & \underline{\sfA^+}(0, 0, 0).
\end{tikzcd}
    \]
    Let us observe how the cycle $z$ is transformed under these reductions. Initially, we have \[
    \begin{tikzcd}
& & \emptyset \arrow[d, "z"] \\
\cdots \arrow[r] & {E(-1,0,0)} \arrow[r] & {\underline{\sfA}(0, 0, 0)}
\end{tikzcd}
    \]
    This part is unchanged by the elimination of the horizontal sequences. After delooping $\sfA(0, 0, 0)$, we have 
    \[
\begin{tikzcd}[row sep=.5em]
 & & & \emptyset \arrow[ld, "\epsilon_1 \circ Y_1 \circ z"', end anchor=east] \arrow[lddd, "\epsilon_1 \circ z", bend left, end anchor=east] \\
 & & {\underline{\sfA^+}(0, 0, 0)} & \\
\cdots \arrow[r] & {E(-1,0,0)} \arrow[ru] \arrow[rd] & & \\
 & & {\underline{\sfA^-}(0, 0, 0)} &
\end{tikzcd}
    \]
    Then, applying elimination to the lower diagonal arrow gives 
    \[
\begin{tikzcd}
& \emptyset 
    \arrow[ld, "\epsilon_1 \circ Y_1 \circ z"', bend right, shift right, end anchor=east] 
    \arrow[ld, "\epsilon_1 \circ Y_2 \circ z", bend left, end anchor=east] \\
{\underline{\sfA^+}(0, 0, 0)}
\end{tikzcd}
\]
    Here, the subscript $1, 2$ indicates the index of the two circles of $\sfA$, with $1$ referring to the delooped circle. We see that either one of $\epsilon_1 \circ Y_1 \circ z$ and $\epsilon_1 \circ Y_2 \circ z$ is divisible by $H$, and the other one is $0$ (depending on the original orientation of $P$). Since there are no other incoming edges to the target object, it follows that $z$ is $H$-divisible exactly once. 
\end{proof}

\begin{proof}[Proof of \Cref{thm:pretzel} (1)]
    Combine \Cref{prop:st_pretzel,prop:odd-pretzel}.
\end{proof}

\Cref{prop:odd-pretzel} can be easily generalized to $l$-strand pretzel knots. Here, we only state the result. 

\begin{proposition}
\label{prop:l-strand-odd-pretzel-2}
    Consider an $l$-strand pretzel knot of the form 
    \[
        P = P(p_1, -p_2, \ldots, -p_l)
    \]
    where $l$ is odd, all $p_i$ are positive and odd, and $p_1 < \min\{p_2, \ldots, p_l\}$. Such $P$ has $s(P) = l - 1$. 
\end{proposition}

\subsection{Even pretzel knots}

Next, we determine $s$ for the case ($E^\star_{-+}$): $p > 0$ is even, $q < 0 < r$ are odd, and $q + r > 0$. The computation is divided into two cases: $p < |q|$ and $p > |q|$. 

\begin{proposition}
\label{prop:even-pretzel-1}
    Let $p>0$ be an even integer, and $q,r>0$ be odd integers with $p<q<r$. Then
    \[
    s(P(p,-q,r)) = -q+r
    \]
\end{proposition}

\begin{proof}
    From symmetry, it suffices to prove that $s(P(-q,p,r)) = -q+r$. Consider the tangle decomposition
    \[ 
        P=P(-q,p,r)=D(T_{-q}^-, T_p^-, T_r^+),
    \]
    and the reduction
    \[ 
        [P]=D([T_{-q}^-], [T_p^-], [T_r^+]) \to D(E_{-q}^-, E_p^-, E_r^+) = E.
    \]
    The vertices $E(v)$ of $E$ are indexed by $v=(v_1,v_2,v_3)$ with $-q\leq v_1\leq 0$, $-p\leq v_2 \leq 0$, and $0 \leq v_3 \leq r$, and $E(v_1,v_2,v_3) = D(\sfE_i, \sfE_j, \sfE_k)$, where
    \begin{align*}
        i = 0 &\iff v_1 = 0 \\
        j = 0 &\iff v_2 = -p \\
        k = 0 &\iff v_3 = 0.
    \end{align*}
    The Lee cycle $\alpha(P)$ of $P$ is mapped to
    \[ 
        \pm H^{p-1}z 
    \]
    where $z$ is a cycle of $E$ with target in $\sfC_2(0, 0, 0)$. We have $d_H(P)\geq p-1$, and with $w(P) = -p-q+r$ and $r(P) = p+1$, we have
    \[ 
        s_H(P) \geq 2(p-1) + (-p-q+r) - (p+1) + 1 = -q+r-2. 
    \]
    Thus it suffice to show that the cycle $z$ is divisible exactly once in $E$.

    Hereafter, we slice $E$ into full subgraphs consisting of vertices with fixed $v_3$, each of which is denoted $E(*, *, v_3)$ and called the \textit{$v_3$-th page} of $E$. First, let us focus on the $v_3$-th page for $v_3 > 0$. The horizontal sequence at $v_2 = -p$ has the following form:
    \[
\begin{tikzcd}[row sep=.5em]
{\sfB_2(-q)} \arrow[r, "U"] & {\sfB_2(-q+1)} \arrow[r, "0", dotted] & \cdots & & \\
{\cdots} \arrow[r, dotted] & {\sfB_2(-3)} \arrow[r, "U"] & {\sfB_2(-2)} \arrow[r, "0", dotted] & {\sfB_2(-1)} \arrow[r, "\Delta_1"] & {\sfC_1(0).}
\end{tikzcd}
    \]
    Note that since $\sfB_2$ is connected, endomorphisms $a_1, b_1$ turn into $U, 0$ respectively. For the other horizontal sequences with $-p < v_2 \leq 0$, we have 
    \[
\begin{tikzcd}[row sep=.5em, ]
{\sfA(-q)} \arrow[r, "a_1"] & {\sfA(-q+1)} \arrow[r, "b_1"] & \cdots & \\
\cdots \arrow[r, "a_1"] & {\sfA(-2)} \arrow[r, "b_1"] & {\sfA(-1)} \arrow[r, "m_1"] & {\sfB_1(0).}
\end{tikzcd}
    \]
    This sequence reduces into 
    \[
\begin{tikzcd}
    \sfA^-(-q, v_2) \arrow[r, dotted] 
    & 0 \arrow[r, dotted] 
    & \cdots \arrow[r, dotted]
    & 0.
\end{tikzcd}
    \]
    By stacking these sequences and adding in the non-horizontal edges, we obtain the $v_3$-th page of the reduced complex, which is partially described as: 

\vspace{1em}
\noindent
\resizebox{\textwidth}{!}{%
\begin{tikzcd}[row sep=.3em, ampersand replacement=\&]
{\sfB_2(-q, -p)} \arrow[r, "U"] \arrow[rd] \& {\sfB_2(-q + 1, -p)} \arrow[rdd] \& {\sfB_2(-q + 2, -p)} \arrow[r, "U"] \arrow[rddd] \& {\sfB_2(-q + 3, -p)} \arrow[rdddd] \& \cdots \\
 \& {\sfA^-(-q, -p+1)} \arrow[rd] \& \& \& \\
 \& \& {\sfA^-(-q, -p+2)} \arrow[rd] \& \& \\
 \& \& \& {\sfA^-(-q, -p+3)} \arrow[rd] \& \\
 \& \& \& \& {\sfA^-(-q, -p+4)}
\end{tikzcd}
}%
\vspace{1em}

    \noindent
    Using \Cref{lem:even-pretzel-reduce}, it can be shown that this page consists of disjoint squares
    \[
\begin{tikzcd}
{\sfB_2(-q + 2i, -p)} \arrow[r, "U"] \arrow[rd, "\pm I" description] & {\sfB_2(-q + 2i + 1, -p)} \arrow[rd, "\mp I" description] &                        \\
 & {\sfA^-(-q, -p + 2i + 1)} \arrow[r, "U"] & {\sfA^-(-q, -p + 2i + 2)}
\end{tikzcd}
    \]
    for $i = 0, 1, \ldots, \frac{p - 2}{2}$, and the remaining horizontal segments at the top:
    \[
\begin{tikzcd}[column sep=1.25em]
    \sfB_2(-q+p, -p) \arrow[r, "U"] 
    & \sfB_2(-q+p+1, -p)
    & \cdots
    & {\sfB_2(-1, -p)} \arrow[r, "\Delta_1"] 
    & {\sfC_1(0, -p).}
\end{tikzcd}
    \]
    Let us denote the $i$-th square on page $v_3 > 0$ by $\Box_{i, v_3}$. 

    Next, we focus on the $0$-th page. Since $p$ is even, the rightmost part can be depicted as:
    \[
\begin{tikzcd}
\cdots \arrow[rd, "a_2" description] \\
\cdots \arrow[r] & {\sfC_2(0, -1)} \arrow[rd, "b_2 = 0" description, dotted] & \\
\cdots \arrow[r, "b_1"] & {\sfB_3(-1, -0)} \arrow[r, "\Delta_1"] & {\underline{\sfC_2}(0, 0)}.
\end{tikzcd}
    \]
    The underlined object $\underline{\sfC_2}(0, 0)$ is the target of the cycle $z$; it has only one incoming arrow from $\sfB_3(-1, 0)$. After delooping the inner circle of $\underline{\sfC_2}(0, 0)$, we have 
    \[
\begin{tikzcd} [row sep=.5em]
& {\underline{\sfC^+_2}(0, 0)} \\
{B_3(-1, 0)} \arrow[ru, "Y" description] \arrow[rd, "I" description] &                  \\ & {\underline{\sfC^-_2}(0, 0).}
\end{tikzcd}        
    \]
    Eliminating the lower diagonal edge transforms this part into:
    \[
\begin{tikzcd} [row sep=.5em]
& {\underline{\sfC^+_2}(0, 0)} \\
{0} \arrow[ru, dotted] \arrow[rd, dotted] & \\ 
& {0.}
\end{tikzcd}        
    \]

    Finally, let us observe how the cycle $z$ is transformed by these reductions. Consider the following zigzag in the original complex $E$:  
    \[
\begin{tikzcd}[row sep=1.25em]
 & \emptyset \arrow[d, "z"] \\
{\sfB_3(-1,0,0)} \arrow[r, "\Delta_1"] \arrow[rd, "-\Delta_3" description] & {\underline{\sfC_2}(0,0,0)} \\
{\sfA(-2,0,1)} \arrow[r] \arrow[rd, "b_3" description] & {\sfA(-1,0,1)} \\
{\qquad \vdots \qquad} & {\qquad \vdots \qquad} \\
{\sfA(-q, 0, q-1)} \arrow[r] \arrow[rd, "a_3" description] & {\sfA(-q+1, 0, q-1)} \\ 
& {\quad \sfA(-q, 0, q) \quad }
\end{tikzcd}
    \]
    Note that the objects on the left side has no other outgoing edges than what are shown. After delooping and the horizontal eliminations, the cycle $z$ transforms into a sum of $z_1$ with target in $\underline{\sfC^+_2}(0, 0, 0)$ and $z_2$ with target in $\sfA^+(-q, 0, q)$. First, one can show as in the proof of \Cref{prop:odd-pretzel} that $z_1$ is $H$-divisible exactly once. Next, for the term $z_2$, observe that the target $\sfA^+(-q, 0, q)$ is at rightmost square $\Box_{\frac{p - 2}{2}, q}$ on page $q$:
    \[
\begin{tikzcd}
{\sfB_2(-q + p - 2, -p, q)} \arrow[r, "U"] \arrow[rd, "\pm I" description] & {\sfB_2(-q + p - 1, -p, q)} \arrow[rd, "\mp I" description] & \\
 & {\sfA^+(-q, -1, q)} \arrow[r, "U"] & {\sfA^+(-q, 0, q).}
\end{tikzcd}        
    \]
    Note that this square is connected to other parts in different pages. However, we claim that by eliminating all of the squares $\Box_{i, v_3}$ using the two diagonal edges $\pm I$, the image of $z_2$ will vanish. Indeed, on pages $v_3 > q$, the bottom right object of the rightmost square $\Box_{\frac{p - 2}{2}, v_3}$ has homological grading
    \[
        -p - q + 2(\frac{p - 2}{2}) + 2 + v_3 > 0.
    \]
    Thus, after eliminating all of the squares, there will be no homological grading $\leq 0$ objects on pages $v_3 > q$. By a similar zigzag tracing, one can see that the claim holds. 
\end{proof}

\begin{proposition}
\label{prop:even-pretzel-2}
    Let $p>0$ be an even integer, and $q,r>0$ be odd integers with $p > q < r$. Then
    \[
        s(P(p, -q, r)) = -q + r - 2
    \]
\end{proposition}

\begin{proof}
    It suffices to prove that the special case $q = p - 1,\ r = p + 1$ has 
    \[
        s(P(p, -(p-1), p+1)) = 0
    \]
    for any $p \geq 2$. Indeed, if this holds, then for $p = q + 1$ we have $s(P(q+1,-q,q+2)) = 0$. On the other hand, from \Cref{prop:even-pretzel-1} we have $s(P(q-1,-q,q+2)) = 2$. Thus from \Cref{lem:st-twist} (1), it follows that $s(P(p,-q,q+2)) = 0$ for any $p > q$. Moreover, we have $s(P(p,-q,q)) = 0$, so from \Cref{lem:st-twist} (2), it follows that $s(P(p,-q,r)) = -q + r - 2$ for any $r > q$. From symmetry, we may equivalently show that $s(P(p-1,-p,-p-1)) = 0$.

    Consider the tangle decomposition
    \[ 
        P = P(p-1,-p,-p-1) = D(T_{p-1}^+, T_{-p}^+, T_{-p-1}^-) 
    \]
    and the reduction
    \[ 
        [P] = D([T_{p-1}^+], [T_{-p}^+], [T_{-p-1}^-]) \rightarrow D(E_{p-1}^+, E_{-p}^+, E_{-p-1}^-) = E.
    \]
    The vertices $E(v)$ of the complex $E$ are indexed by $v = (v_1,v_2,v_3)$ with $0\leq v_1 \leq p-1$, $0 \leq v_2 \leq p$, and $-p-1 \leq v_3 \leq 0$, and $E(v_1,v_2,v_3) = D(\sfE_i, \sfE_j, \sfE_k)$, where
    \begin{align*}
        i = 0 &\iff v_1 = 0 \\
        j = 0 &\iff v_2 = p \\
        k = 0 &\iff v_3 = 0.
    \end{align*}
    
    The Lee cycle $\alpha(P)$ of $P$ is mapped to a cycle $\pm z$, where $z$ is a cycle of $E$ with target in $\sfC_2(0, 0, 0)$. With $d_H(P) \geq 0$, $w(P) = p-2$ and $r(P) = p+1$, we have 
    \[ 
        s_H(P) \geq (p-2) - (p+1) + 1 = -2.
    \]
    We prove that $z$ is $H$-divisible exactly once.
    
    First, let us focus on the $v_3$-th page pf $E$ for any $v_3 > 0$. The horizontal sequence at $v_2 < p$ is of the form 
    \[
\begin{tikzcd}
{\sfB_1(0)} \arrow[r, "\Delta_1"] & {\sfA(1)} \arrow[r, "b_1"] & {\sfA(2)} \arrow[r, "a_1"] & \cdots \arrow[r] & {\sfA(p-2)} \arrow[r] & {\sfA(p-1)} 
\end{tikzcd}
    \]
    and is reduced as 
    \[
\begin{tikzcd}
{0(0)} \arrow[r, dotted] & {0(1)} \arrow[r, dotted] & {0(2)}  \arrow[r, dotted] & \cdots \arrow[r, dotted] & {0(p-2)} \arrow[r, dotted] & {\sfA^+(p-1)} 
\end{tikzcd}
    \]
    The bottom horizontal sequence at $v_2 = p$ consists of disjoint segments
\[
\begin{tikzcd}[row sep=.5em]
{\sfC_1(0)} \arrow[r, "m_1"] & {\sfB_2(1)} \arrow[r, dotted] & {\sfB_2(2)} \arrow[r, "U"] & {\sfB_2(3)} \arrow[r, dotted] & \cdots \\ 
& & \arrow[r, dotted] & {\sfB_2(p-2)} \arrow[r, "U"] & {\sfB_2(p-1)}.
\end{tikzcd}
    \]
    The leftmost segment can be reduced to 
    \[
\begin{tikzcd}
{\sfC^-_1(0)} \arrow[r, dotted] & {0(1)}.
\end{tikzcd}
    \]
    
    By stacking the above sequences and adding the non-horizontal edges, the $v_3$-th page can be partially depicted as:
    
\vspace{1em}
\noindent
\resizebox{.98\textwidth}{!}{%
\begin{tikzcd}[row sep=.5em, ampersand replacement=\&]
{\sfA^+(p - 1, 0)} \arrow[rd] \arrow[rdddd] \& \& \& \& \\
\& {\sfA^+(p - 1, 1)} \arrow[rd] \arrow[rddd, dotted] \& \& \& \\ 
\& \& {\sfA^+(p - 2, 2)} \arrow[rd] \arrow[rdd] \& \& \\ 
\& \& \& \ddots \& \\ 
\& {\sfC^-_1(0, p)} \arrow[r, dotted] \& {0(1, p)} \& {\sfB_2(2, p)} \arrow[r] \& {\cdots}.
\end{tikzcd}        
}
\vspace{1em}

    \noindent
    Again, using \Cref{lem:even-pretzel-reduce}, one can show that this complex consists of the following disjoint squares:

\vspace{1em}
\noindent
\resizebox{\textwidth}{!}{%
\begin{tikzcd}[ampersand replacement=\&]
{\sfA^+(p - 1, 0)} \arrow[r, "U"] \arrow[rd, "\pm I" description] \& {\sfA^+(p - 1, 1)} \arrow[rd, dotted] \& {\sfA^+(p - 1, 2)} \arrow[rd, "\pm I" description] \arrow[r, "U"] \& {\sfA^+(p - 1, 3)} \arrow[rd, "\mp I" description] \& \cdots \\ 
\& {\sfC^-_1(0, p)} \arrow[r, dotted] \& 0 \& {\sfB_2(2, p)} \arrow[r, "U"] \& {\sfB_2(3, p)}
\end{tikzcd}
}%
\vspace{1em}

    Next, let us focus on the $0$-th page. Since $p$ is even, the leftmost part can be depicted as:
    \[
\begin{tikzcd}
{\underline{\sfC_2}(0, 0)} \arrow[r, "m_1"] \arrow[rd, "0" description, dotted] & {\sfB_1(1, 0)} \arrow[r, dotted] & {} \\ & {\sfC_2(0, 1)} \arrow[r] \arrow[rd, "U" description] & \cdots \\ & & \cdots
\end{tikzcd}
    \]
    The underlined object $\underline{\sfC_2}(0, 0)$ is the target of the cycle $z$. Using the unique outgoing edge, it can be reduced into 
    \[
\begin{tikzcd}
{\underline{\sfC^-_2}(0, 0)} \arrow[r, dotted] & {0.}
\end{tikzcd}
    \]

    % --- Full page figures
    \begin{figure}[tp]
        \centering
\resizebox{\textwidth}{!}{%
\begin{tikzcd}[ampersand replacement=\&]
\underline{-2} \& \underline{-1} \& \underline{-0} \& \underline{1} \\
{\sfA^+(p - 1, 0, -p - 1)} \arrow[r] \arrow[rd, "\pm I" description, dashed] \& {\sfA^+(p - 1, 1, -p - 1)} \arrow[rd, dotted] \arrow[rddd, "\pm I" description, dashed, bend right=20, pos=.45] \& {\sfA^+(p - 1, 1, -p - 1)} \arrow[r] \arrow[rd, "\pm I" description, dashed] \& \cdots \\
\& {\sfC^-_1(0, p, -p - 1)} \arrow[r, dotted] \& 0 \& \cdots \\ 
\& {\sfA^+(p - 1, 0, -p)} \arrow[r] \arrow[rd] \arrow[rdd, bend right] \arrow[rdddd, "\pm I" description, dashed, bend right] \& {\underline{\sfA^+}(p - 1, 1, -p)} \arrow[rd, dotted] \& \cdots \\ 
\& \& {\sfC^-_1(0, p, -p)} \arrow[r, dotted] \& \cdots \\ 
\& \& {\sfA^+(p - 1, 0, -p + 1)} \arrow[r] \arrow[rd, "\pm I" description, dashed] \& \cdots \\ \& \& \& \cdots \\ \& \& {\underline{\sfC^-_2}(0, 0, 0)} \& 
\end{tikzcd} 
}
        \caption{The homological grading $\leq 0$ part of $E'$.}
        \label{fig:pretzel-even2-graph}
        \vspace{2em}
\resizebox{.9\textwidth}{!}{%
\begin{tikzcd}[ampersand replacement=\&]
\underline{-1} \& \underline{-0} \& {} \arrow[dddddd, no head, dotted] \& \underline{-1} \& \underline{-0} \\
{\sfA(p - 1, 1)} \arrow[rd, "-" description] \arrow[rrrr, "a_3" description, bend right=15] \& \& \& {\sfA(p - 2, 1)} \arrow[r] \arrow[rd, "+" description] \& {\sfA(p - 1, 1)} \\
{\sfA(p - 2, 2)} \arrow[r] \arrow[rd, "+" description] \arrow[rrrr, "a_3" description, bend right=15] \& {\sfA(p - 1, 2)} \& \& {\sfA(p - 3, 2)} \arrow[r] \arrow[rd, "-" description] \& {\sfA(p - 2, 2)} \\
\qquad \vdots \qquad \arrow[rd, "+" description] \& \qquad \vdots \qquad \& \& \qquad \vdots \qquad \arrow[rd, "-" description] \& \qquad \vdots \qquad \\
{\sfA(1, p-1)} \arrow[r] \arrow[rrrr, "a_3" description, bend right=15] \arrow[rd, "-" description] \& {\sfA(2, p-1)} \& \& {\sfB_1(0, p-1)} \arrow[r] \arrow[rd, "+" description] \& {\sfA(1, p-1)} \\
{\sfC_1(0, p)} \arrow[r] \arrow[rrrr, "U_3" description, bend right=15] \& {\sfB_2(1, p)} \& \& \& {\sfC_1(0, p)} \\ \& \& {} \& \& 
\end{tikzcd}
}
        \caption{Zigzag paths from $E(p-1, 1, -p-1)$ (top left) to $E(0, p, -p)$ (bottom right). The left side is the $(-p-1)$-th page and the right side $(-p)$-th page. Diagonal edges increment $v_2$, and the bent edges increment $v_3$.}
        \label{fig:pretzel-even2-zigzag}
    \end{figure}
    % ---
    
    Now, let $E'$ denote the resulting reduced complex of $E$ and $z'$ the transformed cycle. \Cref{fig:pretzel-even2-graph} depicts the homological grading $\leq 0$ part of $E'$, where the bent arrows are those that bridges between different pages. We claim:
    \begin{enumerate}
        \item $\sfA^+(p-1, 1, -p-1) \to \sfC^-_1(0, p, -p)$ is $\pm I$, 
        \item $\sfA^+(p-1, 1, -p-1) \to \sfA^+(p-1, 1, -p)$ is 0, and
        \item $\sfA^+(p-1, 0, -p) \to \sfC^-_2(0, 0, 0)$ is $\pm I$.
    \end{enumerate}
    We use \Cref{lem:st-twist} to verify these claims. For the first claim, one needs to consider all zigzags connecting $E(p-1, 1, -p-1)$ to $E(0, p, -p)$ in the original complex; there are total of $p$ such paths (see \Cref{fig:pretzel-even2-zigzag}). After delooping and elimination, one can show that: $\frac{p}{2}$ contributes to $\pm I$; $(\frac{p}{2} - 1)$ contributes to $\mp I$; and one (which has $U$ in the final step) contributes to $0$. Thus in total, we have $\pm I$. We leave the details and verification of the other two claims to the reader. 
    
    By eliminating the five dashed arrows labeled $\pm I$ in the figure, all objects except for $\underline{\sfA^+}(p - 1, 1, -p)$ in homological grading $\leq 0$ are eliminated. After this reduction, the cycle $z'$ transforms into 
    \[
\begin{tikzcd}
{\emptyset} \arrow[r, "z'"] & {\underline{\sfC^-_2}(0, 0, 0)} \arrow[r, "\pm I"] & {\sfA^+(p - 1, 0, -p)} \arrow[r, "U"] & {\underline{\sfA^+}(p - 1, 1, -p)}
\end{tikzcd}
    \]
    which yields a single $H$. 
\end{proof}

\begin{proof}[Proof of \Cref{thm:pretzel} (2)] 
    Combine \Cref{prop:st_pretzel,prop:even-pretzel-1,prop:even-pretzel-2}.
\end{proof}

\subsection{Observations}

Finally, we present some observations on the sliceness and ribbonness of 3-strand pretzel knots, as corollaries of \Cref{thm:pretzel}.

Pretzel knots give examples of knots that are topologically slice but not smoothly slice. It is known that each such knot gives rise to an \textit{exotic $\RR^4$} (see \cite[Exercise 9.4.23]{Gompf-Stipsicz:1999}). The $s$-invariant gives an obstruction to smooth sliceness, whereas the Alexander polynomial detects topological sliceness, i.e.\ if $\Delta_K(t) = 1$, then $K$ is topologically slice. The converse does not hold in general for either of the statements. 

Fintushel and Stern \cite{Fintushel-Stern:1985} proved that any non-trivial odd 3-strand pretzel knot with $\Delta_K(t) = 1$ is never smoothly slice. Thus, for odd 3-strand pretzel knots, the Alexander polynomial alone detects topologically--but--not--smoothly slice knots. Recently, Belousov \cite{Belousov:2025} gave explicit formulae for the Alexander polynomial of pretzel knots, and in particular shows that an odd pretzel knot $P(p, q, r)$ has trivial Alexander polynomial if and only if 
\[
    pq + qr + rp = -1.
\]
This includes the well known example $P(-3, 5, 7)$, which is topologically--but--not--smoothly slice, as stated in \cite{Rudolph:1993}. In general, Belousov's result together with \Cref{thm:pretzel} reproves Fintushel and Stern's result.

\begin{corollary}
    Any non-trivial odd pretzel knot $P = P(p, q, r)$ satisfying $pq + qr + rp = -1$ has $s(P) = \pm 2$. 
\end{corollary}

\begin{proof}
    We assume that $p > 0$. From the assumption, we have 
    \[
        (p + q)(p + r) = p^2 + (pq + qr + rp) = p^2 - 1.
    \]
    If $p = 1$, then either $q = -1$ or $r = -1$ must hold, which implies that $P$ is trivial, contradicting the assumption. Thus $p \geq 3$, and $p + q$, $p + r$ must have the same signs. If $p + q < 0$ and $p + r < 0$, then from \Cref{thm:pretzel} we have $s(P) = 2$.
    
    Next, suppose $p + q > 0$ and $p + r > 0$. We have
    \[
        r = -\frac{1 + pq}{p + q}, \quad 
        q + r = \frac{q^2 - 1}{p + q}.
    \]
    If $q = \pm 1$, then $r = \mp 1$ and $P$ is trivial. Thus we must have $|q| > 1$ and $q + r > 0$. In this case, we have $s(P) = -2$ from \Cref{thm:pretzel}.
\end{proof}

Next, we consider topologically slice 3-strand pretzel knots that have non-trivial Alexander polynomial. For the odd-type, Miller proves the following, which implies that the topologically--but--not--smoothly slice odd 3-strand pretzel knots are precisely the non-trivial ones with trivial Alexander polynomial. 

\begin{proposition}[{\cite[Theorem 1.5]{Miller:2017}}]
\label{prop:miller-odd}
    Let $K$ be an odd 3-strand pretzel knot with nontrivial Alexander polynomial. Then $K$ is topologically slice iff $K$ is ribbon iff $K$ is of the form $\pm P(p, q, -q)$ or $\pm P(1, q, -q-4)$ with $q > 0$.
\end{proposition}

Next we consider even 3-strand pretzel knots. Belousov's formulae implies that a non-trivial even 3-strand pretzel knot never has trivial Alexander polynomial (see \cite[Remark 2]{Belousov:2025}). Thus the Alexander polynomial cannot be used to detect topologically slice even 3-strand pretzel knots. 

Consider the following family of even pretzel knots
\[
    P_a = P\left(a, -a-2, -\frac{(a + 1)^2}{2}\right)
\]
for $a \geq 3$ odd. Analogous to \Cref{prop:miller-odd}, we have the following. 

\begin{proposition}[{\cite[Theorem 1.6]{Miller:2017}, \cite[Theorem 1.1]{Kim-Lee-Song:2022}}]
\label{prop:miller-even}
    Let $K$ be an even 3-strand pretzel knot that is not of the form $\pm P_a$ for $a \equiv 1,97 \bmod{120}$. Then $K$ is topologically slice iff $K$ is ribbon iff $K$ is of the form $\pm P(p, q, -q)$.
\end{proposition}

The family of pretzel knots $\{P_a\}$ was introduced by Lecuona \cite{Lecuona:2015}, where they proved that the \textit{slice-ribbon conjecture} holds for all 3-strand pretzel knots except for a specific subset of this family. More than three quarters of this family was proved to be not algebraically slice (hence not topologically slice), and lately by works of Miller \cite{Miller:2017} and Kim--Lee--Song \cite{Kim-Lee:2007}, only knots for $a \equiv 1,97 \bmod{120}$ remain unknown whether they are topologically slice or not. If they are indeed not topologically slice, then the slice-ribbon conjecture holds for all 3-strand pretzel knots, and \Cref{prop:miller-odd,prop:miller-even} implies that the topologically--but--not--smoothly slice 3-strand pretzel knots are precisely the non-trivial ones with trivial Alexander polynomial.

Although all knots in the family $\{P_a\}$ are conjectured to be non-slice, well-known obstructions to sliceness, such as the signature, the determinant, and the $\tau$-invariant fail to detect them (see \cite[Section 4.2]{Lecuona:2015}). We prove that the same holds for the $s$-invariant. 

\begin{corollary}
\label{cor:s-Pa}
    All knots in $\{P_a\}$ have $s = 0$. 
\end{corollary}

\begin{proof}
    Put
    \[
        (p, q, r) = \left(\frac{(a + 1)^2}{2},\ -a,\ a + 2 \right)
    \]
    and consider $P = P(p, q, r)$. 
    We have 
    \[
        p + q = \frac{(a + 1)^2}{2} - a = \frac{a^2 + 1}{2} > 0
    \]
    and also $p + r > 0$,\ $q + r = 2$. Thus, from \Cref{thm:pretzel}, we have $s(P) = 0$.
\end{proof}

\begin{remark}
    There is another family of pretzel knots introduced in \cite{Lecuona:2015}, 
    \[
        K_a = P\left( a,\ -a - 2,\ -a - \frac{a^2 + 9}{2} \right)
    \]
    for $a \geq 3$, and are proved to be non-slice \cite[Theorem 4.3]{Lecuona:2015}. Similar to the proof of \Cref{cor:s-Pa}, one can show that $s(K_a) = 0$ for all $a \geq 3$. In general, any even 3-strand pretzel knot $P = P(p, q, r)$ with $p \geq 0$ even, $p + q > 0$,\ $p + r > 0$ and $q + r = 2$ has $s(P) = 0$, so its sliceness cannot be obstructed by $s$. 
\end{remark}

\begin{question}
    Is there a slice torus invariant $\nu$ that exhibits $\nu(P_a) \neq 0$ or $\nu(K_a) \neq 0$?
\end{question}
    \appendix
\section{Proof of \Cref{prop:ca-under-reidemeister}}
\label{sec:appendix}

\begin{restate-proposition}[prop:ca-under-reidemeister]
    Suppose $L, L'$ are link diagrams related by a Reidemeister move. Under the corresponding chain homotopy equivalence $f$, the Lee cycles correspond as 
    \begin{align*}
        f \circ \ca(L) &\htpy \epsilon H^j \ca(L'), \\
        f \circ \cb(L) &\htpy \epsilon' H^j \cb(L')
    \end{align*}
    where
    \[
        j = \frac{\delta w(L, L') - \delta r(L, L')}{2}
    \]
    and $\epsilon, \epsilon'$ are signs satisfying
    \[
        \epsilon\epsilon' = (-1)^j.
    \]
\end{restate-proposition}

\begin{proof}
    First, we fix the notations common to the three moves. Let $T, T'$ be the minimal tangle parts of $L, L'$, whose complements are not involved in the move. Then $L, L'$ and the move $R$ can be written as
    \[
        L = D(T, T_1, \ldots, T_d),\ 
        L' = D(T', T_1, \ldots, T_d),\
        R = D(R_0, I, \ldots, I)
    \]
    for some planar arc diagram $D$, and tangle diagrams $T_1, \ldots, T_d$ common to $L$ and $L'$. In the proof of \cite[Theorem 1]{BarNatan:2004}, an explicit chain homotopy equivalence $F_0: [T] \to [T']$ and its inverse $G_0: [T'] \to [T]$ corresponding to the move $R_0$ are given, together with chain homotopies. Using the planar algebra structure, the chain homotopy equivalences $F, G$ for the moves $R$ and $R^{-1}$ are given by 
    \[
        F = D(F_0, I, \ldots, I),\ 
        G = D(G_0, I, \ldots, I).
    \]
    Let $\ca(L), \ca(L')$ be the Lee cycles of $L, L'$, and $\ca(T), \ca(T')$ be those of $T, T'$ respectively. Here, the checkerboard colorings for $T, T'$ are inherited from the standard checkerboard colorings for $L, L'$, and are omitted from the notations. 

    \bigskip

    \noindent
    \textbf{R1}. Consider the move
    \begin{center}
        \tikzset{every picture/.style={line width=0.75pt}} %set default line width to 0.75pt        

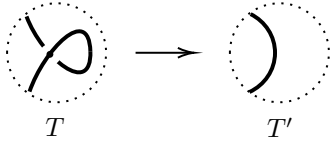
\begin{tikzpicture}[x=0.75pt,y=0.75pt,yscale=-.75,xscale=.75]
%uncomment if require: \path (0,111); %set diagram left start at 0, and has height of 111

%Shape: Circle [id:dp21965207196201753] 
\draw  [dash pattern={on 0.84pt off 2.51pt}] (167,42.3) .. controls (167,23.91) and (181.91,9) .. (200.3,9) .. controls (218.68,9) and (233.59,23.91) .. (233.59,42.3) .. controls (233.59,60.68) and (218.68,75.59) .. (200.3,75.59) .. controls (181.91,75.59) and (167,60.68) .. (167,42.3) -- cycle ;
%Shape: Arc [id:dp35395161805976716] 
\draw  [draw opacity=0][line width=1.5]  (180.93,16.27) .. controls (190.76,21.19) and (197.44,31.13) .. (197.35,42.55) .. controls (197.25,54.1) and (190.24,64.03) .. (180.15,68.72) -- (167,42.3) -- cycle ; \draw  [line width=1.5]  (180.93,16.27) .. controls (190.76,21.19) and (197.44,31.13) .. (197.35,42.55) .. controls (197.25,54.1) and (190.24,64.03) .. (180.15,68.72) ;  

%Straight Lines [id:da12768714611405063] 
\draw    (103.57,42.3) -- (140.95,42.3) ;
\draw [shift={(142.95,42.3)}, rotate = 180] [color={rgb, 255:red, 0; green, 0; blue, 0 }  ][line width=0.75]    (10.93,-3.29) .. controls (6.95,-1.4) and (3.31,-0.3) .. (0,0) .. controls (3.31,0.3) and (6.95,1.4) .. (10.93,3.29)   ;
%Shape: Ellipse [id:dp9694931346190031] 
\draw  [dash pattern={on 0.84pt off 2.51pt}] (17.17,42.3) .. controls (17.17,23.91) and (32.08,9) .. (50.47,9) .. controls (68.85,9) and (83.76,23.91) .. (83.76,42.3) .. controls (83.76,60.68) and (68.85,75.59) .. (50.47,75.59) .. controls (32.08,75.59) and (17.17,60.68) .. (17.17,42.3) -- cycle ;
%Curve Lines [id:da4260312888837898] 
\draw [line width=1.5]    (30.25,17.84) .. controls (40.61,45.68) and (74.23,72.95) .. (73.18,41.25) ;
%Shape: Ellipse [id:dp4725275903530799] 
\draw  [draw opacity=0][fill={rgb, 255:red, 255; green, 255; blue, 255 }  ,fill opacity=1 ] (46.1,36.31) .. controls (50.07,36.26) and (53.34,39.44) .. (53.39,43.42) .. controls (53.44,47.39) and (50.26,50.66) .. (46.29,50.71) .. controls (42.31,50.76) and (39.04,47.58) .. (38.99,43.6) .. controls (38.94,39.63) and (42.12,36.36) .. (46.1,36.31) -- cycle ;
%Shape: Ellipse [id:dp5319163873688867] 
\draw  [draw opacity=0][fill={rgb, 255:red, 0; green, 0; blue, 0 }  ,fill opacity=1 ] (45.95,41.05) .. controls (47.22,41.03) and (48.26,42.05) .. (48.28,43.32) .. controls (48.29,44.59) and (47.28,45.64) .. (46.01,45.65) .. controls (44.74,45.67) and (43.69,44.65) .. (43.68,43.38) .. controls (43.66,42.11) and (44.68,41.07) .. (45.95,41.05) -- cycle ;

%Curve Lines [id:da4843626611914248] 
\draw [line width=1.5]    (73.18,41.25) .. controls (72.13,9.55) and (41.87,40.33) .. (32.25,68.43) ;

% Text Node
\draw (190,83.83) node [anchor=north west][inner sep=0.75pt]    {$T'$};
% Text Node
\draw (41,83.83) node [anchor=north west][inner sep=0.75pt]    {$T$};

\end{tikzpicture}
    \end{center}
    The complex $[T]$ is given by 
    \[
        \begin{tikzcd}
        \larc \thickcirc \arrow[r, "m"] & \larc 
        \end{tikzcd}
    \]
    As stated in \Cref{lem:Tq-lem1}, delooping the circle appearing on the left gives
    \[
        \begin{tikzcd}[row sep=0em]
        \larc \dotcircI \arrow[rd, "I", dashed]   &   \\
        & \larc. \\
        \larc \dotcircX \arrow[ru, "X"'] &  
        \end{tikzcd}        
    \]
    Eliminating along the dashed arrow $I$ gives
    
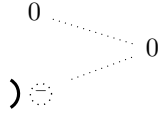
\[
        \begin{tikzcd}[row sep=0em]
        0 \arrow[rd, dotted, no head] &   \\
        & 0 \\
        \larc \dotcircX \arrow[ru, dotted, no head] &  
        \end{tikzcd}        
    \]
    which is now isomorphic to $[T']$. The morphism $F_0$ is defined as the composition of the above chain homotopy equivalences. Post-composing $F_0$ to $\ca(T)$ gives
    
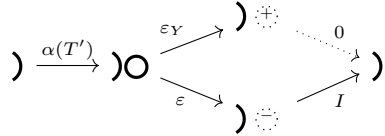
\[
        \begin{tikzcd}[row sep=0em]
        & & \larc \dotcircI \arrow[rd, "0", dotted] & \\
        \larc \arrow[r, "{\ca(T')}"] & \larc \thickcirc \arrow[ru, "\epsilon_Y"] \arrow[rd, "\epsilon"'] & & \larc \\
        & & \larc \dotcircX \arrow[ru, "I"'] &
        \end{tikzcd}
    \]
    By identifying $T^\lin = T'{}^\lin$, one can see that 
    \[
        F_0 \circ \ca(T) = \ca(T')
    \]
    for either case of the labeling on the circle, since we have
    \[
        \epsilon \iota_X = \epsilon \iota_Y = 1.
    \]
    Thus the equivalence $F$ gives
    \[
        F \circ \ca(L) = \ca(L').
    \]
    With
    \[
        w(L) = w(L') + 1,\ 
        r(L) = r(L') + 1,
    \]
    we have $\delta w - \delta r = 0$, and the equations hold. 

    \bigskip

    \noindent
    \textbf{R1'}. Consider the move
    \begin{center}
        \tikzset{every picture/.style={line width=0.75pt}} %set default line width to 0.75pt        

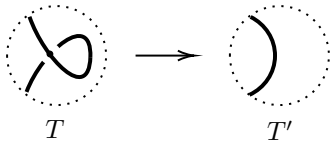
\begin{tikzpicture}[x=0.75pt,y=0.75pt,yscale=-.75,xscale=.75]
%uncomment if require: \path (0,111); %set diagram left start at 0, and has height of 111

%Shape: Circle [id:dp09240350559029897] 
\draw  [dash pattern={on 0.84pt off 2.51pt}] (167,42.3) .. controls (167,23.91) and (181.91,9) .. (200.3,9) .. controls (218.68,9) and (233.59,23.91) .. (233.59,42.3) .. controls (233.59,60.68) and (218.68,75.59) .. (200.3,75.59) .. controls (181.91,75.59) and (167,60.68) .. (167,42.3) -- cycle ;
%Shape: Arc [id:dp3477911980738484] 
\draw  [draw opacity=0][line width=1.5]  (180.93,16.27) .. controls (190.76,21.19) and (197.44,31.13) .. (197.35,42.55) .. controls (197.25,54.1) and (190.24,64.03) .. (180.15,68.72) -- (167,42.3) -- cycle ; \draw  [line width=1.5]  (180.93,16.27) .. controls (190.76,21.19) and (197.44,31.13) .. (197.35,42.55) .. controls (197.25,54.1) and (190.24,64.03) .. (180.15,68.72) ;  

%Straight Lines [id:da5916944564623559] 
\draw    (103.57,42.3) -- (140.95,42.3) ;
\draw [shift={(142.95,42.3)}, rotate = 180] [color={rgb, 255:red, 0; green, 0; blue, 0 }  ][line width=0.75]    (10.93,-3.29) .. controls (6.95,-1.4) and (3.31,-0.3) .. (0,0) .. controls (3.31,0.3) and (6.95,1.4) .. (10.93,3.29)   ;
%Shape: Ellipse [id:dp08164041434482128] 
\draw  [dash pattern={on 0.84pt off 2.51pt}] (17.17,42.3) .. controls (17.17,60.68) and (32.08,75.59) .. (50.47,75.59) .. controls (68.85,75.59) and (83.76,60.68) .. (83.76,42.3) .. controls (83.76,23.91) and (68.85,9) .. (50.47,9) .. controls (32.08,9) and (17.17,23.91) .. (17.17,42.3) -- cycle ;
%Curve Lines [id:da026720433343995542] 
\draw [line width=1.5]    (30.25,66.75) .. controls (40.61,38.92) and (74.23,11.64) .. (73.18,43.34) ;
%Shape: Ellipse [id:dp8857641719965605] 
\draw  [draw opacity=0][fill={rgb, 255:red, 255; green, 255; blue, 255 }  ,fill opacity=1 ] (46.1,48.28) .. controls (50.07,48.33) and (53.34,45.15) .. (53.39,41.18) .. controls (53.44,37.2) and (50.26,33.93) .. (46.29,33.88) .. controls (42.31,33.83) and (39.04,37.01) .. (38.99,40.99) .. controls (38.94,44.96) and (42.12,48.23) .. (46.1,48.28) -- cycle ;
%Shape: Ellipse [id:dp49231078972251985] 
\draw  [draw opacity=0][fill={rgb, 255:red, 0; green, 0; blue, 0 }  ,fill opacity=1 ] (45.95,43.54) .. controls (47.22,43.56) and (48.26,42.54) .. (48.28,41.27) .. controls (48.29,40) and (47.28,38.96) .. (46.01,38.94) .. controls (44.74,38.92) and (43.69,39.94) .. (43.68,41.21) .. controls (43.66,42.48) and (44.68,43.52) .. (45.95,43.54) -- cycle ;

%Curve Lines [id:da6812069814895134] 
\draw [line width=1.5]    (73.18,43.34) .. controls (72.13,75.04) and (41.87,44.26) .. (32.25,16.16) ;

% Text Node
\draw (190,83.83) node [anchor=north west][inner sep=0.75pt]    {$T'$};
% Text Node
\draw (41,83.83) node [anchor=north west][inner sep=0.75pt]    {$T$};

\end{tikzpicture}
    \end{center}
    The complex $[T]$ is given by 
    \[
        \begin{tikzcd}
        \larc \arrow[r, "\Delta"] & \larc \thickcirc 
        \end{tikzcd}
    \]
    Delooping the circle on the right gives
    
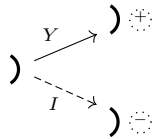
\[
        \begin{tikzcd}[row sep=0em]
        & \larc \dotcircI \\
        \larc \arrow[ru, "Y"] \arrow[rd, "I"', dashed] & \\
        & \larc \dotcircX
        \end{tikzcd}        
    \]
    and eliminating along the dashed arrow gives
    
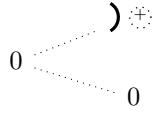
\[
        \begin{tikzcd}[row sep=0em]
        & \larc\dotcircI \\
        0 \arrow[ru, dotted, no head] \arrow[rd, dotted, no head] & \\
        & 0
        \end{tikzcd}        
    \]
    which is isomorphic to $[T']$. $F_0$ is defined by the composition of the above equivalences, and post-composing it to $\ca(T)$ gives:
    
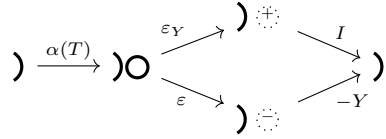
\[
        \begin{tikzcd}[row sep=0em]
        & & \larc \dotcircI \arrow[rd, "I"] & \\
        \larc \arrow[r, "{\ca(T)}"] & \larc \thickcirc \arrow[ru, "\epsilon_Y"] \arrow[rd, "\epsilon"'] & & \larc \\
        & & \larc \dotcircX \arrow[ru, "-Y"'] & 
        \end{tikzcd}
    \]
    Note the diagonal arrow labeled $-Y$, which is the $-ca^{-1}$ term of \Cref{prop:gauss-elim}, caused by the reversal of the arrow $I$. If $\ca(T), \ca(T')$ are labeled as 
    \[
        \larc^{e_Y} \thickcirc^X,\quad 
        \larc^{e_Y}
    \]
    then the above diagram shows that $F_0 \circ \ca(T) = H \ca(T')$, using $XY = 0,\ \epsilon X \iota = 1$ and $(-Y)e_Y = -Y = H e_Y$. For the other case, we can similarly see that $F_0 \circ \ca(T) = -H \ca(T')$. Thus we have 
    \[
        F \circ \ca(L) = \pm H \ca(L')
    \]
    and with
    \[
        w(L) = w(L') - 1,\ 
        r(L) = r(L') + 1
    \]
    we have $\delta w - \delta r = 2$. Note that the minus sign appears for only one of the two cases, thereby proving $\epsilon \epsilon' = -1$. 

    \bigskip

    \noindent
    \textbf{R2}. Consider the move
    \begin{center}
        \tikzset{every picture/.style={line width=0.75pt}} %set default line width to 0.75pt        

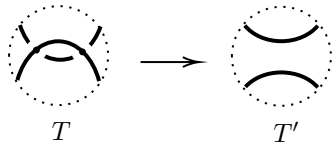
\begin{tikzpicture}[x=0.75pt,y=0.75pt,yscale=-.75,xscale=.75]
%uncomment if require: \path (0,113); %set diagram left start at 0, and has height of 113

%Straight Lines [id:da14410297930046945] 
\draw    (102.57,49.63) -- (139.95,49.63) ;
\draw [shift={(141.95,49.63)}, rotate = 180] [color={rgb, 255:red, 0; green, 0; blue, 0 }  ][line width=0.75]    (10.93,-3.29) .. controls (6.95,-1.4) and (3.31,-0.3) .. (0,0) .. controls (3.31,0.3) and (6.95,1.4) .. (10.93,3.29)   ;
%Shape: Ellipse [id:dp34585713490565784] 
\draw  [dash pattern={on 0.84pt off 2.51pt}] (14.17,45.3) .. controls (14.17,26.91) and (29.08,12) .. (47.47,12) .. controls (65.85,12) and (80.76,26.91) .. (80.76,45.3) .. controls (80.76,63.68) and (65.85,78.59) .. (47.47,78.59) .. controls (29.08,78.59) and (14.17,63.68) .. (14.17,45.3) -- cycle ;
%Shape: Arc [id:dp5227988727671137] 
\draw  [draw opacity=0][line width=1.5]  (75.72,25.64) .. controls (71.17,38.93) and (60.62,48.2) .. (48.43,48.07) .. controls (36.13,47.95) and (25.69,38.29) .. (21.49,24.71) -- (48.81,11.18) -- cycle ; \draw  [line width=1.5]  (75.72,25.64) .. controls (71.17,38.93) and (60.62,48.2) .. (48.43,48.07) .. controls (36.13,47.95) and (25.69,38.29) .. (21.49,24.71) ;  
%Shape: Ellipse [id:dp6509386249108922] 
\draw  [draw opacity=0][fill={rgb, 255:red, 255; green, 255; blue, 255 }  ,fill opacity=1 ] (70.58,43.06) .. controls (70.55,47.04) and (67.3,50.24) .. (63.32,50.2) .. controls (59.35,50.17) and (56.15,46.92) .. (56.18,42.95) .. controls (56.21,38.97) and (59.46,35.77) .. (63.44,35.8) .. controls (67.42,35.84) and (70.61,39.09) .. (70.58,43.06) -- cycle ;
%Shape: Ellipse [id:dp020569245994460972] 
\draw  [draw opacity=0][fill={rgb, 255:red, 0; green, 0; blue, 0 }  ,fill opacity=1 ] (65.84,42.81) .. controls (65.83,44.08) and (64.8,45.1) .. (63.53,45.09) .. controls (62.25,45.08) and (61.23,44.05) .. (61.24,42.78) .. controls (61.25,41.5) and (62.29,40.48) .. (63.56,40.49) .. controls (64.83,40.5) and (65.85,41.54) .. (65.84,42.81) -- cycle ;

%Shape: Ellipse [id:dp02742333596447477] 
\draw  [draw opacity=0][fill={rgb, 255:red, 255; green, 255; blue, 255 }  ,fill opacity=1 ] (39.59,41.81) .. controls (39.56,45.79) and (36.31,48.99) .. (32.33,48.96) .. controls (28.36,48.92) and (25.16,45.67) .. (25.19,41.7) .. controls (25.22,37.72) and (28.47,34.52) .. (32.45,34.56) .. controls (36.42,34.59) and (39.62,37.84) .. (39.59,41.81) -- cycle ;
%Shape: Ellipse [id:dp1643478572551702] 
\draw  [draw opacity=0][fill={rgb, 255:red, 0; green, 0; blue, 0 }  ,fill opacity=1 ] (34.85,41.56) .. controls (34.84,42.83) and (33.81,43.86) .. (32.53,43.85) .. controls (31.26,43.84) and (30.24,42.8) .. (30.25,41.53) .. controls (30.26,40.26) and (31.3,39.23) .. (32.57,39.24) .. controls (33.84,39.25) and (34.86,40.29) .. (34.85,41.56) -- cycle ;

%Shape: Arc [id:dp9697268238137627] 
\draw  [draw opacity=0][line width=1.5]  (19.77,61.77) .. controls (24.03,46.45) and (34.77,35.59) .. (47.3,35.65) .. controls (59.94,35.7) and (70.66,46.85) .. (74.7,62.42) -- (47.11,76.27) -- cycle ; \draw  [line width=1.5]  (19.77,61.77) .. controls (24.03,46.45) and (34.77,35.59) .. (47.3,35.65) .. controls (59.94,35.7) and (70.66,46.85) .. (74.7,62.42) ;  
%Shape: Circle [id:dp11996714138220577] 
\draw  [dash pattern={on 0.84pt off 2.51pt}] (164,45.75) .. controls (164,27.36) and (178.91,12.45) .. (197.3,12.45) .. controls (215.68,12.45) and (230.59,27.36) .. (230.59,45.75) .. controls (230.59,64.14) and (215.68,79.04) .. (197.3,79.04) .. controls (178.91,79.04) and (164,64.14) .. (164,45.75) -- cycle ;
%Shape: Arc [id:dp5321322738750331] 
\draw  [draw opacity=0][line width=1.5]  (173.14,66.31) .. controls (178.44,60.43) and (187.3,56.58) .. (197.34,56.6) .. controls (207.54,56.62) and (216.51,60.63) .. (221.74,66.69) -- (197.3,79.04) -- cycle ; \draw  [line width=1.5]  (173.14,66.31) .. controls (178.44,60.43) and (187.3,56.58) .. (197.34,56.6) .. controls (207.54,56.62) and (216.51,60.63) .. (221.74,66.69) ;  
%Shape: Arc [id:dp7312833580358863] 
\draw  [draw opacity=0][line width=1.5]  (221.75,25.17) .. controls (216.52,31.42) and (207.56,35.56) .. (197.38,35.6) .. controls (187.03,35.63) and (177.93,31.43) .. (172.69,25.06) -- (197.3,12.45) -- cycle ; \draw  [line width=1.5]  (221.75,25.17) .. controls (216.52,31.42) and (207.56,35.56) .. (197.38,35.6) .. controls (187.03,35.63) and (177.93,31.43) .. (172.69,25.06) ;  

% Text Node
\draw (190,87.83) node [anchor=north west][inner sep=0.75pt]    {$T'$};
% Text Node
\draw (41,87.83) node [anchor=north west][inner sep=0.75pt]    {$T$};

\end{tikzpicture}
    \end{center}
    The complex $[T]$ is given by 
    
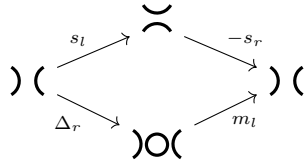
\[
    \begin{tikzcd}[row sep=0.5em]
        & \udarc\arrow[rd, "-s_r"] & \\
        \larc \ \rarc \arrow[ru, "s_l"] \arrow[rd, "\Delta_r"'] & & \larc \ \rarc \\
        & \larc \thickcirc \rarc \arrow[ru, "m_l"'] &
    \end{tikzcd}
    \]
    The subscript $l, r$ indicates which of the two resolutions are being changed. By applying delooping to the middle circle, and eliminating one of the two summands with the right end object, we get
    
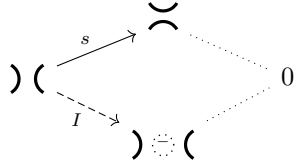
\[
    \begin{tikzcd}[row sep=0.5em]
        & \udarc \arrow[rd, dotted, no head] & \\
        \larc \ \rarc \arrow[ru, "s"] \arrow[rd, "I"', dashed] & & 0 \\
        & \larc \dotcircX \rarc \arrow[ru, dotted, no head] & 
    \end{tikzcd}
    \]
    Furthermore, eliminating along the lower diagonal arrow gives 
    
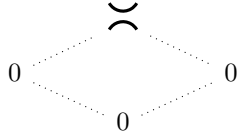
\[
    \begin{tikzcd}[row sep=0.5em]
        & \udarc \arrow[rd, no head, dotted] & \\
        0 \arrow[ru, no head, dotted] \arrow[rd, no head, dotted] & & 0 \\
        & 0 \arrow[ru, no head, dotted] &  
    \end{tikzcd}
    \]
    which is now isomorphic to $[T']$. The morphism $F_0$ is defined as the composition of the above chain homotopy equivalences. 
    Regarding the Lee cycle $\ca(T)$ of $T$, we cannot tell whether its target $T^\lout$ is 
    \[
         \udarc \text{\quad or \ } \larc \thickcirc \rarc
    \]
    for it depends on the orientation of the two strands.

    \medskip

    \noindent
    \textbf{Case 1.} If the two strands are oriented to the same direction, then 
    \[
        T^\lout = \udarc
    \]
    and $F_0$ acts as identity on $T^\lout$, so $F \circ \ca(L) = \ca(L')$ trivially holds. In this case we have $\delta w = \delta r = 0$. 
    
    \medskip

    \noindent
    \textbf{Case 2.} If the two arcs are oriented oppositely, then 
    \[
        T^\lout = \larc \thickcirc \rarc
    \]
    and $F_0 \circ \ca(T)$ is the composition
    
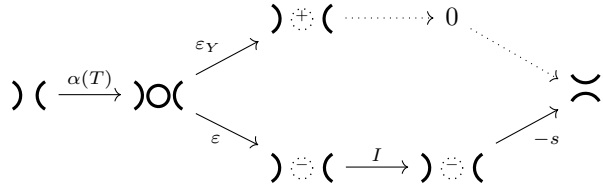
\[
        \begin{tikzcd}[row sep=0.75em]
        & & \larc \dotcircI \rarc \arrow[r, dotted] & 0 \arrow[rd, dotted] & \\
        \larc \ \rarc \arrow[r, "{\ca(T)}"] & \larc \thickcirc \rarc \arrow[ru, "\epsilon_Y"] \arrow[rd, "\epsilon"'] & & & \udarc \\
        & & \larc \dotcircX \rarc \arrow[r, "I"] & \larc \dotcircX \rarc \arrow[ru, "-s"'] & 
        \end{tikzcd}
    \]
    Again note the last diagonal arrow labeled $-s$ produced by the second elimination, where $s$ is a saddle cobordism connecting the two pairs of arcs. We see that $F_0 \circ \ca(T) = -s$. Note that $T^\lin$ and $T'{}^\lin$ are not isotopic, so we need to zoom out from the local picture and divide into subcases to compare the two cycles. 
    
    \smallskip 
    
    \noindent
    \textbf{Case 2-1.} The two arcs appearing in $T^\lout$ are parts of separate circles in $L^\lout$. In this case, the cobordism $-s$ merges two circles with the same label, so we have 
    \[
        F \circ \ca(L) = \pm H \ca(L')
    \]
    while $\delta w = 0, \delta r = -2$. 

    \smallskip 
    \noindent
    \textbf{Case 2-2.} The two arcs appearing in $T^\lout$ are parts of the same circle in $L^\lout$. In this case, $-s$ splits one circle into two circles labeled the same, so we have 
    \[
        F \circ \ca(L) = -\ca(L')
    \]
    while $\delta w = 0, \delta r = 0$. 
    
    \bigskip
    
    \noindent
    \textbf{R3}. Consider the move
    \begin{center}
        \tikzset{every picture/.style={line width=0.75pt}} %set default line width to 0.75pt        

\begin{tikzpicture}[x=0.75pt,y=0.75pt,yscale=-.75,xscale=.75]
%uncomment if require: \path (0,119); %set diagram left start at 0, and has height of 119

%Shape: Circle [id:dp3145483297824164] 
\draw  [dash pattern={on 0.84pt off 2.51pt}] (14,48.3) .. controls (14,29.91) and (28.91,15) .. (47.3,15) .. controls (65.68,15) and (80.59,29.91) .. (80.59,48.3) .. controls (80.59,66.68) and (65.68,81.59) .. (47.3,81.59) .. controls (28.91,81.59) and (14,66.68) .. (14,48.3) -- cycle ;
%Straight Lines [id:da1552399551869057] 
\draw    (100.57,49.63) -- (137.95,49.63) ;
\draw [shift={(139.95,49.63)}, rotate = 180] [color={rgb, 255:red, 0; green, 0; blue, 0 }  ][line width=0.75]    (10.93,-3.29) .. controls (6.95,-1.4) and (3.31,-0.3) .. (0,0) .. controls (3.31,0.3) and (6.95,1.4) .. (10.93,3.29)   ;
%Shape: Ellipse [id:dp9738420049846099] 
\draw  [dash pattern={on 0.84pt off 2.51pt}] (161.17,48.3) .. controls (161.17,29.91) and (176.08,15) .. (194.47,15) .. controls (212.85,15) and (227.76,29.91) .. (227.76,48.3) .. controls (227.76,66.68) and (212.85,81.59) .. (194.47,81.59) .. controls (176.08,81.59) and (161.17,66.68) .. (161.17,48.3) -- cycle ;
%Curve Lines [id:da21275279255570145] 
\draw [line width=1.5]    (34.33,77.62) .. controls (53.35,60.94) and (66.98,44.36) .. (56.27,16.32) ;
%Shape: Ellipse [id:dp6160921881071979] 
\draw  [draw opacity=0][fill={rgb, 255:red, 255; green, 255; blue, 255 }  ,fill opacity=1 ] (39.24,67.82) .. controls (37.7,64.15) and (39.42,59.93) .. (43.08,58.39) .. controls (46.75,56.85) and (50.97,58.57) .. (52.51,62.23) .. controls (54.06,65.9) and (52.34,70.12) .. (48.67,71.66) .. controls (45,73.2) and (40.78,71.48) .. (39.24,67.82) -- cycle ;
%Shape: Ellipse [id:dp7650633049675122] 
\draw  [draw opacity=0][fill={rgb, 255:red, 0; green, 0; blue, 0 }  ,fill opacity=1 ] (43.69,66.18) .. controls (43.2,65) and (43.75,63.65) .. (44.92,63.16) .. controls (46.09,62.67) and (47.44,63.22) .. (47.93,64.39) .. controls (48.42,65.56) and (47.87,66.91) .. (46.7,67.4) .. controls (45.53,67.9) and (44.18,67.35) .. (43.69,66.18) -- cycle ;

%Curve Lines [id:da7559682269821785] 
\draw [line width=1.5]    (34.27,17.05) .. controls (29.29,45.03) and (37.19,61.23) .. (60.1,78.43) ;
%Shape: Ellipse [id:dp9634203552476284] 
\draw  [draw opacity=0][fill={rgb, 255:red, 255; green, 255; blue, 255 }  ,fill opacity=1 ] (26.72,42.86) .. controls (23.81,40.15) and (23.64,35.6) .. (26.35,32.69) .. controls (29.06,29.78) and (33.62,29.61) .. (36.53,32.32) .. controls (39.44,35.03) and (39.6,39.59) .. (36.89,42.5) .. controls (34.18,45.41) and (29.63,45.57) .. (26.72,42.86) -- cycle ;
%Shape: Ellipse [id:dp28662048343708957] 
\draw  [draw opacity=0][fill={rgb, 255:red, 0; green, 0; blue, 0 }  ,fill opacity=1 ] (30.1,39.54) .. controls (29.17,38.67) and (29.12,37.22) .. (29.98,36.29) .. controls (30.85,35.36) and (32.3,35.3) .. (33.23,36.17) .. controls (34.17,37.04) and (34.22,38.49) .. (33.35,39.42) .. controls (32.49,40.35) and (31.03,40.4) .. (30.1,39.54) -- cycle ;

%Shape: Ellipse [id:dp0981698426085923] 
\draw  [draw opacity=0][fill={rgb, 255:red, 255; green, 255; blue, 255 }  ,fill opacity=1 ] (55.66,43.11) .. controls (52.75,40.41) and (52.59,35.85) .. (55.3,32.94) .. controls (58.01,30.03) and (62.56,29.86) .. (65.47,32.57) .. controls (68.39,35.28) and (68.55,39.84) .. (65.84,42.75) .. controls (63.13,45.66) and (58.58,45.82) .. (55.66,43.11) -- cycle ;
%Shape: Ellipse [id:dp8936235066098274] 
\draw  [draw opacity=0][fill={rgb, 255:red, 0; green, 0; blue, 0 }  ,fill opacity=1 ] (59.05,39.79) .. controls (58.12,38.92) and (58.06,37.47) .. (58.93,36.54) .. controls (59.8,35.61) and (61.25,35.56) .. (62.18,36.42) .. controls (63.11,37.29) and (63.16,38.74) .. (62.3,39.67) .. controls (61.43,40.6) and (59.98,40.66) .. (59.05,39.79) -- cycle ;

%Curve Lines [id:da6171433718741905] 
\draw [line width=1.5]    (14.25,45.61) .. controls (37.76,30.61) and (64.77,33.88) .. (80.87,48.87) ;
%Curve Lines [id:da05367422550409706] 
\draw [line width=1.5]    (183.57,78.14) .. controls (181.29,49.77) and (179.87,36.34) .. (204.34,17.24) ;
%Shape: Ellipse [id:dp9664170153194522] 
\draw  [draw opacity=0][fill={rgb, 255:red, 255; green, 255; blue, 255 }  ,fill opacity=1 ] (186.7,31.12) .. controls (185.14,27.46) and (186.85,23.23) .. (190.5,21.68) .. controls (194.16,20.12) and (198.39,21.82) .. (199.95,25.48) .. controls (201.51,29.14) and (199.8,33.37) .. (196.14,34.93) .. controls (192.48,36.48) and (188.25,34.78) .. (186.7,31.12) -- cycle ;
%Shape: Ellipse [id:dp9640860136086993] 
\draw  [draw opacity=0][fill={rgb, 255:red, 0; green, 0; blue, 0 }  ,fill opacity=1 ] (191.14,29.46) .. controls (190.64,28.29) and (191.19,26.94) .. (192.36,26.44) .. controls (193.53,25.94) and (194.88,26.49) .. (195.37,27.66) .. controls (195.87,28.83) and (195.33,30.18) .. (194.16,30.68) .. controls (192.99,31.17) and (191.64,30.63) .. (191.14,29.46) -- cycle ;

%Curve Lines [id:da5021004432216586] 
\draw [line width=1.5]    (182.35,18.05) .. controls (204.57,37.14) and (207.57,49.14) .. (206.57,80.14) ;
%Shape: Ellipse [id:dp5831417532886326] 
\draw  [draw opacity=0][fill={rgb, 255:red, 255; green, 255; blue, 255 }  ,fill opacity=1 ] (201.37,61.97) .. controls (198.37,59.36) and (198.05,54.81) .. (200.66,51.81) .. controls (203.28,48.81) and (207.82,48.49) .. (210.82,51.11) .. controls (213.82,53.72) and (214.14,58.26) .. (211.53,61.26) .. controls (208.92,64.26) and (204.37,64.58) .. (201.37,61.97) -- cycle ;
%Shape: Ellipse [id:dp22424907582767917] 
\draw  [draw opacity=0][fill={rgb, 255:red, 0; green, 0; blue, 0 }  ,fill opacity=1 ] (204.64,58.53) .. controls (203.68,57.7) and (203.58,56.24) .. (204.41,55.29) .. controls (205.25,54.33) and (206.7,54.23) .. (207.66,55.06) .. controls (208.62,55.9) and (208.72,57.35) .. (207.88,58.31) .. controls (207.05,59.27) and (205.6,59.37) .. (204.64,58.53) -- cycle ;

%Shape: Ellipse [id:dp6040847578820988] 
\draw  [draw opacity=0][fill={rgb, 255:red, 255; green, 255; blue, 255 }  ,fill opacity=1 ] (177.73,61.38) .. controls (174.73,58.77) and (174.41,54.22) .. (177.02,51.22) .. controls (179.63,48.22) and (184.18,47.91) .. (187.18,50.52) .. controls (190.18,53.13) and (190.5,57.68) .. (187.89,60.68) .. controls (185.27,63.68) and (180.73,63.99) .. (177.73,61.38) -- cycle ;
%Shape: Ellipse [id:dp7157580686800156] 
\draw  [draw opacity=0][fill={rgb, 255:red, 0; green, 0; blue, 0 }  ,fill opacity=1 ] (181,57.95) .. controls (180.04,57.11) and (179.94,55.66) .. (180.77,54.7) .. controls (181.61,53.74) and (183.06,53.64) .. (184.02,54.48) .. controls (184.98,55.31) and (185.08,56.76) .. (184.24,57.72) .. controls (183.41,58.68) and (181.96,58.78) .. (181,57.95) -- cycle ;

%Curve Lines [id:da18180939613990066] 
\draw [line width=1.5]    (160.68,48.13) .. controls (181.57,61.14) and (207.57,61.14) .. (227.38,49.18) ;

% Text Node
\draw (185,90.83) node [anchor=north west][inner sep=0.75pt]    {$T'$};
% Text Node
\draw (36,90.83) node [anchor=north west][inner sep=0.75pt]    {$T$};

\end{tikzpicture}
    \end{center}
    
    As explained in \cite[Section 4.3]{BarNatan:2004}, the complex $[T]$ may be regarded as the mapping cone of the saddle cobordism $s$ corresponding to the center crossing. The same applies to $[T']$. Then the two complexes are homotopy equivalent to a reduced complex $E$, depicted below. Here, morphisms $F, G, H$ and $F', G', H'$ are the chain maps and chain homotopies for the corresponding R2-moves. 

    \begin{center}
        \input{tikzpictures/R3-expand}
    \end{center}
    
    With these explicit maps, the two Lee cycles can be compared by those images in $E$. Details are left to the reader. 
\end{proof}

    \printbibliography

@article {Abe:2011,
    AUTHOR = {Abe, Tetsuya},
     TITLE = {The {R}asmussen invariant of a homogeneous knot},
   JOURNAL = {Proc. Amer. Math. Soc.},
  FJOURNAL = {Proceedings of the American Mathematical Society},
    VOLUME = {139},
      YEAR = {2011},
    NUMBER = {7},
     PAGES = {2647--2656},
      ISSN = {0002-9939,1088-6826},
   MRCLASS = {57M25},
  MRNUMBER = {2784833},
MRREVIEWER = {Kimihiko\ Motegi},
       DOI = {10.1090/S0002-9939-2010-10687-1},
       URL = {https://doi.org/10.1090/S0002-9939-2010-10687-1},
}

@article {Alishahi:2018,
    AUTHOR = {Alishahi, Akram and Dowlin, Nathan},
     TITLE = {The {L}ee spectral sequence, unknotting number, and the knight
              move conjecture},
   JOURNAL = {Topology Appl.},
  FJOURNAL = {Topology and its Applications},
    VOLUME = {254},
      YEAR = {2019},
     PAGES = {29--38},
      ISSN = {0166-8641},
   MRCLASS = {57M27 (57M25)},
  MRNUMBER = {3894208},
MRREVIEWER = {Sebastian Baader},
       DOI = {10.1016/j.topol.2018.11.020},
       URL = {https://doi.org/10.1016/j.topol.2018.11.020},
}

@article {BHL:2019,
    AUTHOR = {Baldwin, John A. and Hedden, Matthew and Lobb, Andrew},
     TITLE = {On the functoriality of {K}hovanov-{F}loer theories},
   JOURNAL = {Adv. Math.},
  FJOURNAL = {Advances in Mathematics},
    VOLUME = {345},
      YEAR = {2019},
     PAGES = {1162--1205},
      ISSN = {0001-8708,1090-2082},
   MRCLASS = {57M27 (18G60 57R58)},
  MRNUMBER = {3903915},
MRREVIEWER = {Nikolai\ N.\ Saveliev},
       DOI = {10.1016/j.aim.2019.01.026},
       URL = {https://doi.org/10.1016/j.aim.2019.01.026},
}

@article {Baldwin-Sivek:2021,
    AUTHOR = {Baldwin, John A. and Sivek, Steven},
     TITLE = {Framed instanton homology and concordance},
   JOURNAL = {J. Topol.},
  FJOURNAL = {Journal of Topology},
    VOLUME = {14},
      YEAR = {2021},
    NUMBER = {4},
     PAGES = {1113--1175},
      ISSN = {1753-8416},
   MRCLASS = {57K10 (57K18 57K31 57R58)},
  MRNUMBER = {4332488},
MRREVIEWER = {Nikolai N. Saveliev},
       DOI = {10.1112/topo.12207},
       URL = {https://doi.org/10.1112/topo.12207},
}

@article {BarNatan:2004,
    AUTHOR = {Bar-Natan, Dror},
     TITLE = {Khovanov's homology for tangles and cobordisms},
   JOURNAL = {Geom. Topol.},
  FJOURNAL = {Geometry and Topology},
    VOLUME = {9},
      YEAR = {2005},
     PAGES = {1443--1499},
      ISSN = {1465-3060},
   MRCLASS = {57M27 (57M25 57R56)},
  MRNUMBER = {2174270},
MRREVIEWER = {Justin Sawon},
       DOI = {10.2140/gt.2005.9.1443},
       URL = {https://doi.org/10.2140/gt.2005.9.1443},
}

@article {BarNatan:2007,
    AUTHOR = {Bar-Natan, Dror},
     TITLE = {Fast {K}hovanov homology computations},
   JOURNAL = {J. Knot Theory Ramifications},
  FJOURNAL = {Journal of Knot Theory and its Ramifications},
    VOLUME = {16},
      YEAR = {2007},
    NUMBER = {3},
     PAGES = {243--255},
      ISSN = {0218-2165},
   MRCLASS = {57M25},
  MRNUMBER = {2320156},
       DOI = {10.1142/S0218216507005294},
       URL = {https://doi.org/10.1142/S0218216507005294},
}

@misc{Belousov:2025,
      title={Explicit Formulas for the Alexander Polynomial of Pretzel Knots}, 
      author={Y. Belousov},
      year={2025},
      eprint={2502.10370},
      archivePrefix={arXiv},
      primaryClass={math.GT},
      url={https://arxiv.org/abs/2502.10370}, 
}

@article {BHP:2023,
    AUTHOR = {Beliakova, Anna and Hogancamp, Matthew and Putyra, Krzysztof K. and Wehrli, Stephan M.},
     TITLE = {On the functoriality of {$\mathfrak{sl}_2$} tangle homology},
   JOURNAL = {Algebr. Geom. Topol.},
  FJOURNAL = {Algebraic \& Geometric Topology},
    VOLUME = {23},
      YEAR = {2023},
    NUMBER = {3},
     PAGES = {1303--1361},
      ISSN = {1472-2747,1472-2739},
   MRCLASS = {57K18 (18N25)},
  MRNUMBER = {4598808},
MRREVIEWER = {Paola\ Cristofori},
       DOI = {10.2140/agt.2023.23.1303},
       URL = {https://doi.org/10.2140/agt.2023.23.1303},
}

@article {BBG:2019,
    AUTHOR = {Boileau, Michel and Boyer, Steven and Gordon, Cameron McA.},
     TITLE = {Branched covers of quasi-positive links and {L}-spaces},
   JOURNAL = {J. Topol.},
  FJOURNAL = {Journal of Topology},
    VOLUME = {12},
      YEAR = {2019},
    NUMBER = {2},
     PAGES = {536--576},
      ISSN = {1753-8416,1753-8424},
   MRCLASS = {57K10 (57M12)},
  MRNUMBER = {4072174},
MRREVIEWER = {Mattia\ Mecchia},
       DOI = {10.1112/topo.12092},
       URL = {https://doi.org/10.1112/topo.12092},
}

@article {CS:1993,
    AUTHOR = {Carter, J. Scott and Saito, Masahico},
     TITLE = {Reidemeister moves for surface isotopies and their
              interpretation as moves to movies},
   JOURNAL = {J. Knot Theory Ramifications},
  FJOURNAL = {Journal of Knot Theory and its Ramifications},
    VOLUME = {2},
      YEAR = {1993},
    NUMBER = {3},
     PAGES = {251--284},
      ISSN = {0218-2165},
   MRCLASS = {57M25 (57N05)},
  MRNUMBER = {1238875},
       DOI = {10.1142/S0218216593000167},
       URL = {https://doi.org/10.1142/S0218216593000167},
}

@article {Cavallo:2020,
    AUTHOR = {Cavallo, Alberto and Collari, Carlo},
     TITLE = {Slice-torus concordance invariants and {W}hitehead doubles of
              links},
   JOURNAL = {Canad. J. Math.},
  FJOURNAL = {Canadian Journal of Mathematics. Journal Canadien de
              Math\'{e}matiques},
    VOLUME = {72},
      YEAR = {2020},
    NUMBER = {6},
     PAGES = {1423--1462},
      ISSN = {0008-414X},
   MRCLASS = {57K10},
  MRNUMBER = {4176697},
MRREVIEWER = {Patrick Orson},
       DOI = {10.4153/s0008414x19000294},
       URL = {https://doi.org/10.4153/s0008414x19000294},
}

@article {CMW:2009,
    AUTHOR = {Clark, David and Morrison, Scott and Walker, Kevin},
     TITLE = {Fixing the functoriality of {K}hovanov homology},
   JOURNAL = {Geom. Topol.},
  FJOURNAL = {Geometry \& Topology},
    VOLUME = {13},
      YEAR = {2009},
    NUMBER = {3},
     PAGES = {1499--1582},
      ISSN = {1465-3060,1364-0380},
   MRCLASS = {57M27},
  MRNUMBER = {2496052},
       DOI = {10.2140/gt.2009.13.1499},
       URL = {https://doi.org/10.2140/gt.2009.13.1499},
}

@misc{DISST:2022,
  archivePrefix = {arXiv},
  eprint = {2209.05400},
  primaryClass = "math.GT",
  doi = {10.48550/ARXIV.2209.05400},
  url = {https://arxiv.org/abs/2209.05400},
  author = {Daemi, Aliakbar and Imori, Hayato and Sato, Kouki and Scaduto, Christopher and Taniguchi, Masaki},
  title = {Instantons, special cycles, and knot concordance},
  publisher = {arXiv},
  year = {2022},
  copyright = {arXiv.org perpetual, non-exclusive license}
}

@article {Feller:2022,
    AUTHOR = {Feller, Peter and Lewark, Lukas and Lobb, Andrew},
     TITLE = {On the values taken by slice torus invariants},
   JOURNAL = {Math. Proc. Cambridge Philos. Soc.},
  FJOURNAL = {Mathematical Proceedings of the Cambridge Philosophical
              Society},
    VOLUME = {176},
      YEAR = {2024},
    NUMBER = {1},
     PAGES = {55--63},
      ISSN = {0305-0041,1469-8064},
   MRCLASS = {57K10 (57K18)},
  MRNUMBER = {4680480},
       DOI = {10.1017/s0305004123000403},
       URL = {https://doi.org/10.1017/s0305004123000403},
}

@article {Fintushel-Stern:1985,
    AUTHOR = {Fintushel, Ronald and Stern, Ronald J.},
     TITLE = {Pseudofree orbifolds},
   JOURNAL = {Ann. of Math. (2)},
  FJOURNAL = {Annals of Mathematics. Second Series},
    VOLUME = {122},
      YEAR = {1985},
    NUMBER = {2},
     PAGES = {335--364},
      ISSN = {0003-486X,1939-8980},
   MRCLASS = {57R19 (32C40 55N22 58G10)},
  MRNUMBER = {808222},
MRREVIEWER = {N.\ J.\ Hitchin},
       DOI = {10.2307/1971306},
       URL = {https://doi.org/10.2307/1971306},
}

@article {FGMW:2010,
    AUTHOR = {Freedman, Michael and Gompf, Robert and Morrison, Scott and
              Walker, Kevin},
     TITLE = {Man and machine thinking about the smooth 4-dimensional
              {P}oincar\'e{} conjecture},
   JOURNAL = {Quantum Topol.},
  FJOURNAL = {Quantum Topology},
    VOLUME = {1},
      YEAR = {2010},
    NUMBER = {2},
     PAGES = {171--208},
      ISSN = {1663-487X,1664-073X},
   MRCLASS = {57R60 (57M25 57N13)},
  MRNUMBER = {2657647},
       DOI = {10.4171/QT/5},
       URL = {https://doi.org/10.4171/QT/5},
}

@book {Gompf-Stipsicz:1999,
    AUTHOR = {Gompf, Robert E. and Stipsicz, Andr\'as I.},
     TITLE = {{$4$}-manifolds and {K}irby calculus},
    SERIES = {Graduate Studies in Mathematics},
    VOLUME = {20},
 PUBLISHER = {American Mathematical Society, Providence, RI},
      YEAR = {1999},
     PAGES = {xvi+558},
      ISBN = {0-8218-0994-6},
       DOI = {10.1090/gsm/020},
       URL = {https://doi.org/10.1090/gsm/020},
}

@article {Greene:2010,
    AUTHOR = {Greene, Joshua},
     TITLE = {Homologically thin, non-quasi-alternating links},
   JOURNAL = {Math. Res. Lett.},
  FJOURNAL = {Mathematical Research Letters},
    VOLUME = {17},
      YEAR = {2010},
    NUMBER = {1},
     PAGES = {39--49},
      ISSN = {1073-2780},
   MRCLASS = {57M25},
  MRNUMBER = {2592726},
       DOI = {10.4310/MRL.2010.v17.n1.a4},
       URL = {https://doi.org/10.4310/MRL.2010.v17.n1.a4},
}

@article {Jabuka:2010,
    AUTHOR = {Jabuka, Stanislav},
     TITLE = {Rational {W}itt classes of pretzel knots},
   JOURNAL = {Osaka J. Math.},
  FJOURNAL = {Osaka Journal of Mathematics},
    VOLUME = {47},
      YEAR = {2010},
    NUMBER = {4},
     PAGES = {977--1027},
      ISSN = {0030-6126},
   MRCLASS = {57M27 (11E12 11E81 57M25)},
  MRNUMBER = {2791566},
MRREVIEWER = {Nikolai\ N.\ Saveliev},
       URL = {http://projecteuclid.org/euclid.ojm/1292854315},
}

@misc{Iida-Taniguchi:2024,
      title={Monopoles and transverse knots}, 
      author={Nobuo Iida and Masaki Taniguchi},
      year={2024},
      eprint={2403.15763},
      archivePrefix={arXiv},
      primaryClass={math.GT},
      url={https://arxiv.org/abs/2403.15763}, 
}

@misc{ISST:2024,
      title={On the slice-torus invariant $q_M$ from $\mathbb{Z}_2$-equivariant Seiberg--Witten Floer cohomology}, 
      author={Nobuo Iida and Taketo Sano and Kouki Sato and Masaki Taniguchi},
      year={2025},
      eprint={2501.07788},
      archivePrefix={arXiv},
      primaryClass={math.GT},
      url={https://arxiv.org/abs/2501.07788}, 
}

@article {Ito-Yoshida:2021,
    AUTHOR = {Ito, Noboru and Yoshida, Jun},
     TITLE = {A cobordism realizing crossing change on {$\mathfrak{sl}_2$}
              tangle homology and a categorified {V}assiliev skein relation},
   JOURNAL = {Topology Appl.},
  FJOURNAL = {Topology and its Applications},
    VOLUME = {296},
      YEAR = {2021},
     PAGES = {Paper No. 107646, 31},
      ISSN = {0166-8641,1879-3207},
   MRCLASS = {57K18 (57K16)},
  MRNUMBER = {4243402},
MRREVIEWER = {William\ Rushworth},
       DOI = {10.1016/j.topol.2021.107646},
       URL = {https://doi.org/10.1016/j.topol.2021.107646},
}

@article {Khovanov:2000,
    AUTHOR = {Khovanov, Mikhail},
     TITLE = {A categorification of the {J}ones polynomial},
   JOURNAL = {Duke Math. J.},
  FJOURNAL = {Duke Mathematical Journal},
    VOLUME = {101},
      YEAR = {2000},
    NUMBER = {3},
     PAGES = {359--426},
      ISSN = {0012-7094},
   MRCLASS = {57M27 (57R56)},
  MRNUMBER = {1740682},
       DOI = {10.1215/S0012-7094-00-10131-7},
       URL = {https://doi.org/10.1215/S0012-7094-00-10131-7},
}

@article {Khovanov:2003,
    AUTHOR = {Khovanov, Mikhail},
     TITLE = {Patterns in knot cohomology. {I}},
   JOURNAL = {Experiment. Math.},
  FJOURNAL = {Experimental Mathematics},
    VOLUME = {12},
      YEAR = {2003},
    NUMBER = {3},
     PAGES = {365--374},
      ISSN = {1058-6458,1944-950X},
   MRCLASS = {57M27 (18G60 57M25 57R56)},
  MRNUMBER = {2034399},
MRREVIEWER = {Jacob\ Andrew\ Rasmussen},
       URL = {http://projecteuclid.org/euclid.em/1087329238},
}

@article {Khovanov:2004,
    AUTHOR = {Khovanov, Mikhail},
     TITLE = {Link homology and {F}robenius extensions},
   JOURNAL = {Fund. Math.},
  FJOURNAL = {Fundamenta Mathematicae},
    VOLUME = {190},
      YEAR = {2006},
     PAGES = {179--190},
      ISSN = {0016-2736},
   MRCLASS = {57M27 (57R56)},
  MRNUMBER = {2232858},
MRREVIEWER = {Jacob Andrew Rasmussen},
       DOI = {10.4064/fm190-0-6},
       URL = {https://doi.org/10.4064/fm190-0-6},
}

@article {Khovanov:2022,
    AUTHOR = {Khovanov, Mikhail and Robert, Louis-Hadrien},
     TITLE = {Link homology and {F}robenius extensions {II}},
   JOURNAL = {Fund. Math.},
  FJOURNAL = {Fundamenta Mathematicae},
    VOLUME = {256},
      YEAR = {2022},
    NUMBER = {1},
     PAGES = {1--46},
      ISSN = {0016-2736},
   MRCLASS = {57K18 (13B02 57K16)},
  MRNUMBER = {4361584},
       DOI = {10.4064/fm912-6-2021},
       URL = {https://doi.org/10.4064/fm912-6-2021},
}

@article {Kim-Lee:2007,
    AUTHOR = {Kim, Dongseok and Lee, Jaeun},
     TITLE = {Some invariants of pretzel links},
   JOURNAL = {Bull. Austral. Math. Soc.},
  FJOURNAL = {Bulletin of the Australian Mathematical Society},
    VOLUME = {75},
      YEAR = {2007},
    NUMBER = {2},
     PAGES = {253--271},
      ISSN = {0004-9727},
   MRCLASS = {57M25},
  MRNUMBER = {2312569},
MRREVIEWER = {Stefan\ K.\ Friedl},
       DOI = {10.1017/S0004972700039198},
       URL = {https://doi.org/10.1017/S0004972700039198},
}

@article{Kim-Lee-Song:2022,
    author = {Kim, Min Hoon and Lee, Changhee and Song, Minkyoung},
    title = {Non-slice 3-stranded pretzel knots},
    journal = {Journal of Knot Theory and Its Ramifications},
    volume = {31},
    number = {03},
    pages = {2250018},
    year = {2022},
    doi = {10.1142/S0218216522500183},
    URL = {https://doi.org/10.1142/S0218216522500183},
    eprint = {https://doi.org/10.1142/S0218216522500183},
}

@article {KM:1993,
    AUTHOR = {Kronheimer, P. B. and Mrowka, T. S.},
     TITLE = {Gauge theory for embedded surfaces. {I}},
   JOURNAL = {Topology},
  FJOURNAL = {Topology. An International Journal of Mathematics},
    VOLUME = {32},
      YEAR = {1993},
    NUMBER = {4},
     PAGES = {773--826},
      ISSN = {0040-9383},
   MRCLASS = {57R57 (57N13 57R40 57R55 58D29)},
  MRNUMBER = {1241873},
MRREVIEWER = {Ronald J. Stern},
       DOI = {10.1016/0040-9383(93)90051-V},
       URL = {https://doi.org/10.1016/0040-9383(93)90051-V},
}

@article {Lecuona:2015,
    AUTHOR = {Lecuona, Ana G.},
     TITLE = {On the slice-ribbon conjecture for pretzel knots},
   JOURNAL = {Algebr. Geom. Topol.},
  FJOURNAL = {Algebraic \& Geometric Topology},
    VOLUME = {15},
      YEAR = {2015},
    NUMBER = {4},
     PAGES = {2133--2173},
      ISSN = {1472-2747,1472-2739},
   MRCLASS = {57M25},
  MRNUMBER = {3402337},
MRREVIEWER = {Christopher\ William\ Davis},
       DOI = {10.2140/agt.2015.15.2133},
       URL = {https://doi.org/10.2140/agt.2015.15.2133},
}

@article {Lee:2005,
    AUTHOR = {Lee, Eun Soo},
     TITLE = {An endomorphism of the {K}hovanov invariant},
   JOURNAL = {Adv. Math.},
  FJOURNAL = {Advances in Mathematics},
    VOLUME = {197},
      YEAR = {2005},
    NUMBER = {2},
     PAGES = {554--586},
      ISSN = {0001-8708},
   MRCLASS = {57M27},
  MRNUMBER = {2173845},
MRREVIEWER = {Paola Cristofori},
       DOI = {10.1016/j.aim.2004.10.015},
       URL = {https://doi.org/10.1016/j.aim.2004.10.015},
}

@article {Lewark:2014,
    AUTHOR = {Lewark, Lukas},
     TITLE = {Rasmussen's spectral sequences and the
              {$\mathfrak{sl}_N$}-concordance invariants},
   JOURNAL = {Adv. Math.},
  FJOURNAL = {Advances in Mathematics},
    VOLUME = {260},
      YEAR = {2014},
     PAGES = {59--83},
      ISSN = {0001-8708},
   MRCLASS = {57M27},
  MRNUMBER = {3209349},
MRREVIEWER = {Mark A. C. Powell},
       DOI = {10.1016/j.aim.2014.04.003},
       URL = {https://doi.org/10.1016/j.aim.2014.04.003},
}

@article{LZ:2024,
   title={Rasmussen invariants of Whitehead doubles and other satellites},
   ISSN={1435-5345},
   url={http://dx.doi.org/10.1515/crelle-2024-0061},
   DOI={10.1515/crelle-2024-0061},
   journal={Journal für die reine und angewandte Mathematik (Crelles Journal)},
   publisher={Walter de Gruyter GmbH},
   author={Lewark, Lukas and Zibrowius, Claudius},
   year={2024},
   month=sep 
}

@misc{ILM:2021,
      title={Khovanov homology and rational unknotting}, 
      author={Damian Iltgen and Lukas Lewark and Laura Marino},
      year={2021},
      eprint={2110.15107},
      archivePrefix={arXiv},
      primaryClass={math.GT},
      url={https://arxiv.org/abs/2110.15107}, 
}

@article {Livingston:2004,
    AUTHOR = {Livingston, Charles},
     TITLE = {Computations of the {O}zsv\'{a}th-{S}zab\'{o} knot concordance
              invariant},
   JOURNAL = {Geom. Topol.},
  FJOURNAL = {Geometry and Topology},
    VOLUME = {8},
      YEAR = {2004},
     PAGES = {735--742},
      ISSN = {1465-3060},
   MRCLASS = {57M27 (57M25 57Q60)},
  MRNUMBER = {2057779},
MRREVIEWER = {James Roger Conant},
       DOI = {10.2140/gt.2004.8.735},
       URL = {https://doi.org/10.2140/gt.2004.8.735},
}

@article {Lobb:2009,
    AUTHOR = {Lobb, Andrew},
     TITLE = {A slice genus lower bound from {${\rm sl}(n)$}
              {K}hovanov-{R}ozansky homology},
   JOURNAL = {Adv. Math.},
  FJOURNAL = {Advances in Mathematics},
    VOLUME = {222},
      YEAR = {2009},
    NUMBER = {4},
     PAGES = {1220--1276},
      ISSN = {0001-8708},
   MRCLASS = {57M27 (18G60 55N35)},
  MRNUMBER = {2554935},
MRREVIEWER = {Paola Cristofori},
       DOI = {10.1016/j.aim.2009.06.001},
       URL = {https://doi.org/10.1016/j.aim.2009.06.001},
}

@article {Lobb:2012,
    AUTHOR = {Lobb, Andrew},
     TITLE = {A note on {G}ornik's perturbation of {K}hovanov-{R}ozansky
              homology},
   JOURNAL = {Algebr. Geom. Topol.},
  FJOURNAL = {Algebraic \& Geometric Topology},
    VOLUME = {12},
      YEAR = {2012},
    NUMBER = {1},
     PAGES = {293--305},
      ISSN = {1472-2747},
   MRCLASS = {57M25},
  MRNUMBER = {2916277},
MRREVIEWER = {Paola Cristofori},
       DOI = {10.2140/agt.2012.12.293},
       URL = {https://doi.org/10.2140/agt.2012.12.293},
}

@article {Manion:2014,
    AUTHOR = {Manion, Andrew},
     TITLE = {The rational {K}hovanov homology of 3-strand pretzel links},
   JOURNAL = {J. Knot Theory Ramifications},
  FJOURNAL = {Journal of Knot Theory and its Ramifications},
    VOLUME = {23},
      YEAR = {2014},
    NUMBER = {8},
     PAGES = {1450040, 40},
      ISSN = {0218-2165,1793-6527},
   MRCLASS = {57M27},
  MRNUMBER = {3261952},
MRREVIEWER = {Hao\ Wu},
       DOI = {10.1142/S0218216514500400},
       URL = {https://doi.org/10.1142/S0218216514500400},
}

@inproceedings {Manolescu-Ozvath:2008,
    AUTHOR = {Manolescu, Ciprian and Ozsv\'ath, Peter},
     TITLE = {On the {K}hovanov and knot {F}loer homologies of
              quasi-alternating links},
 BOOKTITLE = {Proceedings of {G}\"okova {G}eometry-{T}opology {C}onference
              2007},
     PAGES = {60--81},
 PUBLISHER = {G\"okova Geometry/Topology Conference (GGT), G\"okova},
      YEAR = {2008},
      ISBN = {978-1-57146-107-0},
   MRCLASS = {57M27},
  MRNUMBER = {2509750},
}

@article {Miller:2017,
    AUTHOR = {Miller, Allison N.},
     TITLE = {The topological sliceness of 3-strand pretzel knots},
   JOURNAL = {Algebr. Geom. Topol.},
  FJOURNAL = {Algebraic \& Geometric Topology},
    VOLUME = {17},
      YEAR = {2017},
    NUMBER = {5},
     PAGES = {3057--3079},
      ISSN = {1472-2747,1472-2739},
   MRCLASS = {57M25 (57N70)},
  MRNUMBER = {3704252},
MRREVIEWER = {Ana\ G.\ Lecuona},
       DOI = {10.2140/agt.2017.17.3057},
       URL = {https://doi.org/10.2140/agt.2017.17.3057},
}

@book {Milnor:1965,
    AUTHOR = {Milnor, John},
     TITLE = {Lectures on the {$h$}-cobordism theorem},
      NOTE = {Notes by L. Siebenmann and J. Sondow},
 PUBLISHER = {Princeton University Press, Princeton, NJ},
      YEAR = {1965},
     PAGES = {v+116},
}

@book {Milnor:1968,
    AUTHOR = {Milnor, John},
     TITLE = {Singular points of complex hypersurfaces},
    SERIES = {Annals of Mathematics Studies, No. 61},
 PUBLISHER = {Princeton University Press, Princeton, N.J.; University of Tokyo Press, Tokyo},
      YEAR = {1968},
     PAGES = {iii+122},
   MRCLASS = {57.20 (14.00)},
  MRNUMBER = {0239612},
MRREVIEWER = {J. P. Levine},
}

@article {Naot:2006,
    AUTHOR = {Naot, Gad},
     TITLE = {The universal {K}hovanov link homology theory},
   JOURNAL = {Algebr. Geom. Topol.},
  FJOURNAL = {Algebraic \& Geometric Topology},
    VOLUME = {6},
      YEAR = {2006},
     PAGES = {1863--1892},
      ISSN = {1472-2747,1472-2739},
   MRCLASS = {57M27 (57M25)},
  MRNUMBER = {2263052},
MRREVIEWER = {Stephen\ M.\ Budden},
       DOI = {10.2140/agt.2006.6.1863},
       URL = {https://doi.org/10.2140/agt.2006.6.1863},
}

@article {OS:2003,
    AUTHOR = {Ozsv\'{a}th, Peter and Szab\'{o}, Zolt\'{a}n},
     TITLE = {Knot {F}loer homology and the four-ball genus},
   JOURNAL = {Geom. Topol.},
  FJOURNAL = {Geometry and Topology},
    VOLUME = {7},
      YEAR = {2003},
     PAGES = {615--639},
      ISSN = {1465-3060},
   MRCLASS = {57R58 (57M25 57M27)},
  MRNUMBER = {2026543},
MRREVIEWER = {Stanislav Jabuka},
       DOI = {10.2140/gt.2003.7.615},
       URL = {https://doi.org/10.2140/gt.2003.7.615},
}

@article {Piccirillo:2020,
    AUTHOR = {Piccirillo, Lisa},
     TITLE = {The {C}onway knot is not slice},
   JOURNAL = {Ann. of Math. (2)},
  FJOURNAL = {Annals of Mathematics. Second Series},
    VOLUME = {191},
      YEAR = {2020},
    NUMBER = {2},
     PAGES = {581--591},
      ISSN = {0003-486X},
   MRCLASS = {57K10 (57R65)},
  MRNUMBER = {4076631},
MRREVIEWER = {Laurence R. Taylor},
       DOI = {10.4007/annals.2020.191.2.5},
       URL = {https://doi.org/10.4007/annals.2020.191.2.5},
}

@article {Rasmussen:2010,
    AUTHOR = {Rasmussen, Jacob},
     TITLE = {Khovanov homology and the slice genus},
   JOURNAL = {Invent. Math.},
  FJOURNAL = {Inventiones Mathematicae},
    VOLUME = {182},
      YEAR = {2010},
    NUMBER = {2},
     PAGES = {419--447},
      ISSN = {0020-9910},
   MRCLASS = {57M27},
  MRNUMBER = {2729272},
MRREVIEWER = {William D. Gillam},
       DOI = {10.1007/s00222-010-0275-6},
       URL = {https://doi.org/10.1007/s00222-010-0275-6},
}

@article {Rudolph:1993,
    AUTHOR = {Rudolph, Lee},
     TITLE = {Quasipositivity as an obstruction to sliceness},
   JOURNAL = {Bull. Amer. Math. Soc. (N.S.)},
  FJOURNAL = {American Mathematical Society. Bulletin. New Series},
    VOLUME = {29},
      YEAR = {1993},
    NUMBER = {1},
     PAGES = {51--59},
      ISSN = {0273-0979,1088-9485},
   MRCLASS = {57M25 (32S55)},
  MRNUMBER = {1193540},
MRREVIEWER = {Charles\ Livingston},
       DOI = {10.1090/S0273-0979-1993-00397-5},
       URL = {https://doi.org/10.1090/S0273-0979-1993-00397-5},
}

@article {Sano:2020,
    AUTHOR = {Sano, Taketo},
     TITLE = {A description of {R}asmussen's invariant from the divisibility
              of {L}ee's canonical class},
   JOURNAL = {J. Knot Theory Ramifications},
  FJOURNAL = {Journal of Knot Theory and its Ramifications},
    VOLUME = {29},
      YEAR = {2020},
    NUMBER = {6},
     PAGES = {2050037, 39},
      ISSN = {0218-2165},
   MRCLASS = {57K16 (57K18)},
  MRNUMBER = {4125174},
MRREVIEWER = {Mark A. C. Powell},
       DOI = {10.1142/S0218216520500376},
       URL = {https://doi.org/10.1142/S0218216520500376},
}

@article {Sano:2020-b,
    AUTHOR = {Sano, Taketo},
     TITLE = {Fixing the functoriality of {K}hovanov homology: a simple
              approach},
   JOURNAL = {J. Knot Theory Ramifications},
  FJOURNAL = {Journal of Knot Theory and its Ramifications},
    VOLUME = {30},
      YEAR = {2021},
    NUMBER = {11},
     PAGES = {Paper No. 2150074, 12},
      ISSN = {0218-2165,1793-6527},
   MRCLASS = {57K18},
  MRNUMBER = {4376719},
       DOI = {10.1142/S0218216521500747},
       URL = {https://doi.org/10.1142/S0218216521500747},
}

@misc{Sano:2025,
      title={Involutive Khovanov homology and equivariant knots}, 
      author={Taketo Sano},
      year={2025},
      eprint={2404.08568},
      archivePrefix={arXiv},
      primaryClass={math.GT},
      url={https://arxiv.org/abs/2404.08568}, 
}

@article{Sano-Sato:2023,
    title = {A family of slice-torus invariants from the divisibility of Lee classes},
    journal = {Topology and its Applications},
    pages = {109059},
    year = {2024},
    issn = {0166-8641},
    doi = {https://doi.org/10.1016/j.topol.2024.109059},
    url = {https://www.sciencedirect.com/science/article/pii/S016686412400244X},
    author = {Taketo Sano and Kouki Sato},
}

@article {Schuetz:2021,
    AUTHOR = {Sch\"{u}tz, Dirk},
     TITLE = {A fast algorithm for calculating {$S$}-invariants},
   JOURNAL = {Glasg. Math. J.},
  FJOURNAL = {Glasgow Mathematical Journal},
    VOLUME = {63},
      YEAR = {2021},
    NUMBER = {2},
     PAGES = {378--399},
      ISSN = {0017-0895},
   MRCLASS = {57K18 (57R56)},
  MRNUMBER = {4244204},
MRREVIEWER = {Lukas Lewark},
       DOI = {10.1017/S0017089520000257},
       URL = {https://doi.org/10.1017/S0017089520000257},
}

@article {Suzuki:2010,
    AUTHOR = {Suzuki, Ryohei},
     TITLE = {Khovanov homology and {R}asmussen's {$s$}-invariants for
              pretzel knots},
   JOURNAL = {J. Knot Theory Ramifications},
  FJOURNAL = {Journal of Knot Theory and its Ramifications},
    VOLUME = {19},
      YEAR = {2010},
    NUMBER = {9},
     PAGES = {1183--1204},
      ISSN = {0218-2165,1793-6527},
   MRCLASS = {57M25 (57M27)},
  MRNUMBER = {2726564},
MRREVIEWER = {Sabin\ Cautis},
       DOI = {10.1142/S0218216510008376},
       URL = {https://doi.org/10.1142/S0218216510008376},
}

@misc{Thompson:2018,
      title={Khovanov complexes of rational tangles}, 
      author={Benjamin Thompson},
      year={2018},
      eprint={1701.07525},
      archivePrefix={arXiv},
      primaryClass={math.GT},
      url={https://arxiv.org/abs/1701.07525}, 
}

@article {Turner:2008,
    AUTHOR = {Turner, Paul},
     TITLE = {A spectral sequence for {K}hovanov homology with an
              application to {$(3,q)$}-torus links},
   JOURNAL = {Algebr. Geom. Topol.},
  FJOURNAL = {Algebraic \& Geometric Topology},
    VOLUME = {8},
      YEAR = {2008},
    NUMBER = {2},
     PAGES = {869--884},
      ISSN = {1472-2747,1472-2739},
   MRCLASS = {57M25 (57M27)},
  MRNUMBER = {2443099},
MRREVIEWER = {Ronald\ M.\ Dotzel},
       DOI = {10.2140/agt.2008.8.869},
       URL = {https://doi.org/10.2140/agt.2008.8.869},
}

@article {Webster:2007,
    AUTHOR = {Webster, Ben},
     TITLE = {Khovanov-{R}ozansky homology via a canopolis formalism},
   JOURNAL = {Algebr. Geom. Topol.},
  FJOURNAL = {Algebraic \& Geometric Topology},
    VOLUME = {7},
      YEAR = {2007},
     PAGES = {673--699},
      ISSN = {1472-2747,1472-2739},
   MRCLASS = {57M27 (13D02)},
  MRNUMBER = {2308960},
       DOI = {10.2140/agt.2007.7.673},
       URL = {https://doi.org/10.2140/agt.2007.7.673},
}

@article {Wu:2009,
    AUTHOR = {Wu, Hao},
     TITLE = {On the quantum filtration of the {K}hovanov-{R}ozansky
              cohomology},
   JOURNAL = {Adv. Math.},
  FJOURNAL = {Advances in Mathematics},
    VOLUME = {221},
      YEAR = {2009},
    NUMBER = {1},
     PAGES = {54--139},
      ISSN = {0001-8708},
   MRCLASS = {57M27 (57M25)},
  MRNUMBER = {2509322},
MRREVIEWER = {Stanislav Jabuka},
       DOI = {10.1016/j.aim.2008.12.003},
       URL = {https://doi.org/10.1016/j.aim.2008.12.003},
}
    \addresses
\end{document}